\documentclass[reqno,10pt]{amsart}
\usepackage{amssymb,amsmath,amsthm}
\usepackage{a4wide}
\usepackage{mathrsfs}
\usepackage{color}
\usepackage{esint}
\usepackage{stmaryrd}
\usepackage[colorlinks,linkcolor=blue,anchorcolor=blue,citecolor=blue]{hyperref}

\newtheorem{theorem}{Theorem}

\newtheorem{lemma}[theorem]{Lemma}
\newtheorem{proposition}[theorem]{Proposition}
\newtheorem{remark}[theorem]{Remark}

\def\beq{\begin{equation}}
\def\eeq{\end{equation}}
\def\beqs{\begin{equation*}}
\def\eeqs{\end{equation*}}
\def\bal#1\eal{\begin{align}#1\end{align}}
\def\bals#1\eals{\begin{align*}#1\end{align*}}
\def\bsp#1\esp{\begin{split}#1\end{split}}

\def\d{{\mathrm{d}}}

\let\e=\varepsilon
\numberwithin{equation}{section}
\numberwithin{theorem}{section}

\allowdisplaybreaks[3]

\begin{document}
\date{}
\title[Diffusive Limit of  the Non-cutoff Vlasov--Poisson--Boltzmann system]{
\large{Diffusive Limit of the  Vlasov--Poisson--Boltzmann System without Angular Cutoff
}}

\author{Yuan Xu$^\dagger$, Fujun Zhou$^*$ and Yongsheng Li$^\ddagger$}

\address[Yuan Xu$^\dagger$]{School of Mathematics, South China University of Technology, Guangzhou 510640, China}
\email{yuanxu2019@163.com}

\address[Fujun Zhou$^*$, Corresponding author]{School of Mathematics, South China University of Technology, Guangzhou 510640, China}
\email{fujunht@scut.edu.cn}

\address[Yongsheng Li$^\ddagger$]{School of Mathematics, South China University of Technology, Guangzhou 510640, China}
\email{yshli@scut.edu.cn}

\begin{abstract}
Diffusive limit of the non-cutoff Vlasov--Poisson--Boltzmann system in perturbation framework still remains open.
By employing a new weight function and some novel treatments,  we solve this issue comprehensively for
the full range of potentials $\gamma > -3$ and $0 < s < 1$. This result marks the first comprehensive coverage of the full range
of potentials in diffusive limit of the non-cutoff Boltzmann type equations incorporating electric or electromagnetic field.
Uniform weighted estimate with respect to the Knudsen number $\varepsilon\in (0,1]$ is established globally in time, which
eventually establishes global solutions to the  Vlasov--Poisson--Boltzmann system and hydrodynamic limit to the
two-fluid incompressible Navier--Stokes--Fourier--Poisson system with Ohm's law for the full potential range $\gamma>-3$ and $0<s<1$.
 \\[2mm]
\\[1mm]
\end{abstract}

%
%
\maketitle

\tableofcontents

\section{Introduction}

\subsection{Description of the Problem}
\hspace*{\fill}

In this paper, we study hydrodynamic limit of the two-species Vlasov--Poisson--Boltzmann (VPB, for short) system without angular cutoff
\begin{equation}\label{GG1}
\left\{
	\begin{array}{ll}
	\displaystyle \varepsilon \partial_t F^\varepsilon_{+}+v\cdot\nabla_xF^\varepsilon_{+}
    - \varepsilon\nabla_x\phi^\varepsilon\cdot\nabla_v F^\varepsilon_{+} =\frac{1}{\varepsilon}Q(F^\varepsilon_{+},F^\varepsilon_{+})+\frac{1}{\varepsilon}Q(F^\varepsilon_{-},F^\varepsilon_{+}),    \\[2mm]
		
	\displaystyle \varepsilon \partial_tF^\varepsilon_{-}+v\cdot\nabla_xF^\varepsilon_{-}
    + \varepsilon\nabla_x\phi^\varepsilon\cdot\nabla_v F^\varepsilon_{-}
    =\frac{1}{\varepsilon}Q(F^\varepsilon_{-},F^\varepsilon_{-})+\frac{1}{\varepsilon}Q(F^\varepsilon_{+},F^\varepsilon_{-}),   \\
		
    \displaystyle -\Delta_x\phi^\varepsilon=\frac{1}{\varepsilon}\int_{\mathbb{R}^3}(F^\varepsilon_{+}-F^\varepsilon_{-})\d v, \quad
		
	\lim_{|x|\rightarrow \infty}\phi^\varepsilon=0, \\ [3mm]
	
    \displaystyle F_{\pm}^\varepsilon(0,x,v)=F^\varepsilon_{\pm, 0}(x,v).
\end{array}\right.
\end{equation}
The VPB system (\ref{GG1}) describes the dynamics of two species charged dilute particles (for example, electrons and ions), under the influence of the interactions with themselves through collisions and their self-consistent electrostatic field. More precisely, $F_{\pm}^{\varepsilon}(t,x,v)\geq 0$ are the density distribution functions for the ions $(+)$ and electrons $(-)$, respectively, at time $t \geq 0$, position $x=(x_1, x_2, x_3)\in \mathbb{R}^{3}$ and velocity $v =(v_1, v_2, v_3) \in \mathbb{R}^{3}$.
The self-consistent electric potential $\phi^\varepsilon(t, x)$ is coupled with the density function $F_{\pm}^{\varepsilon}(t,x,v)$ through the
Poisson equation in (\ref{GG1}). The positive parameter $\varepsilon\in (0,1]$ is the Knudsen number, which equals to the ratio of the mean free path to the macroscopic length scale. The Boltzmann collision operator $Q(F, G)$ is a bilinear operator,
\begin{align*}
Q(F, G)(v) :=
\iint_{\mathbb{R}^{3} \times \mathbb{S}^2} B(v-v_*,\sigma)[F(v_*^{\prime})G(v^{\prime})-F(v_*)G(v)]\d v_* \d \sigma.
\end{align*}
Here, $v$, $v_*$ and $v^{\prime}$, $v_*^{\prime}$ are the velocities of a pair of particles before and after collision, respectively,
and they are connected through the formulas
$$
v^{\prime}=\frac{v+v_*}{2}+\frac{|v-v_*|}{2} \sigma, \quad v_*^{\prime}=\frac{v+v_*}{2}-\frac{|v-v_*|}{2} \sigma, \quad \sigma \in \mathbb{S}^2.
$$

The Boltzmann collision kernel $B(v-v_*, \sigma)$ depends only on the relative velocity $|v-v_*|$ and on the deviation angle $\theta$ given by $\cos \theta=\langle\sigma,(v-v_*) /|v-v_*|\rangle$, where $\langle\cdot, \cdot\rangle$ is the usual dot product in $\mathbb{R}^3$. As in \cite{GS2011}, without loss of generality, we suppose that $B(v-v_*, \sigma)$ is supported on $\cos \theta \geq 0$.
Throughout this paper, the collision kernel $B(v-v_*, \sigma) $ is supposed to satisfy the following assumptions:
\begin{itemize}
\item
$B(v-v_*, \sigma)$ takes the product form in its argument, that is,
$$
B(v-v_*, \sigma)=\Phi(|v-v_*|) b(\cos \theta)
$$
with non-negative functions $\Phi$ and $b$.

\item
The angular function $\sigma \rightarrow b(\langle\sigma,(v-v_*) /|v-v_*|\rangle)$ is not integrable on $\mathbb{S}^2$, i.e.
$$
\int_{\mathbb{S}^2} b(\cos \theta) \d \sigma=2 \pi \int_0^{\pi / 2} \sin \theta b(\cos \theta) \d \theta=\infty,
$$
 and there are two positive constants $c_b>0$ and $0<s<1$ such that
$$
\frac{c_b}{\theta^{1+2 s}} \leq \sin \theta b(\cos \theta) \leq \frac{1}{c_b \theta^{1+2 s}}.
$$

\item
The kinetic function $z \rightarrow \Phi(|z|)$ satisfies
$$
\Phi(|z|)=C_{\Phi}|z|^\gamma
$$
for some positive constant $C_{\Phi}>0$, where the exponent $\gamma>-3$ is determined by the intermolecular interactive mechanism.
\end{itemize}

It is convenient to call ``soft potentials'' when $-3 < \gamma <-2s$ and $0 < s < 1$, and call ``hard potentials'' when $\gamma \geq -2s$ and $0 < s < 1$. The current work includes both cases, namely
$$
\gamma>-3 \; \text{ and }\; 0 < s < 1.
$$
Moreover, $0 < s < \frac{1}{2}$ is usually called weak angular singularity, and $ \frac{1}{2}\leq  s < 1$ is called strong angular singularity \cite{DL2013}.
Recall that when the intermolecular interactive potential takes the inverse power law in the form of $U(|x|)=|x|^{-(\ell-1)}$ with $2 < \ell < \infty$, the collision kernel $B(v-v_*, \sigma)$ in three space dimensions satisfies the above assumptions with $\gamma=\frac{\ell-5}{\ell-1}$ and $s=\frac{1}{\ell-1}$. Note that $\gamma \to -3$ and $s \to 1$ as $\ell \to 2$ in the limiting case, for which the grazing collisions between particles are dominated and the Boltzmann collision term has to be replaced by the classical Landau collision term for the Coulomb potential, cf. \cite{V}.

Much significant progress has been made on global existence of solutions to the VPB system. Firstly, for the cutoff VPB system,
that is, the collision kernel $B(v-v_*, \sigma)$ inside the Boltzmann collision operator $Q(F,G)$ is assumed to be integrable and taken the form
$$
 B(v-v_*, \sigma)=|v-v_*|^{\gamma}b(\cos\theta)\;\;\text{ with } \; 0\leq b(\cos\theta)\leq C|\cos\theta|,
$$
there have been many results.
In the framework of renormalized solutions, Lions \cite{Lions1994-1, Lions1994-2} developed global renormalized solutions for general large initial data,
while Mischler \cite{Mischler2000} investigated global renormalized solutions to the initial and boundary value problem. In the framework of perturbation around
the global Maxwellian, by employing the robust energy method, Guo \cite{guo2002cpam} constructed global existence of classical solutions in periodic box $\mathbb{T}^3$ for the hard sphere case $\gamma=1$. Duan--Strain \cite{DS11} considered global classical solutions and optimal time decay in the whole space $\mathbb{R}^3$ for the hard sphere case $\gamma=1$.
Subsequently, Duan--Yang--Zhao \cite{DYZ2002, DYZ2013} established global existence and optimal time decay of classical solutions in the whole space $\mathbb{R}^3$ for hard potentials $0\leq \gamma<1$ and moderately soft potentials $-2<\gamma<0$, respectively.
Eventually, global existence for the full range of cutoff soft potentials $-3<\gamma<0$ was solved by \cite{XXZ2017}.
Secondly, concerning the non-cutoff VPB system, some research results have also emerged. After the breakthrough of
the non-cutoff Boltzmann equation by Gressman--Strain \cite{GS2011} and Alexandre--Morimoto--Ukai--Xu--Yang
\cite{AMUXY2011AA, AMUXY2012JFA} independently, Duan--Liu \cite{DL2013} constructed global solutions to
the non-cutoff VPB system for soft potentials $-3<\gamma< -2s$ with strong angular singularity $\frac{1}{2} \leq s < 1$.
Later on, Deng \cite{D2021-1} investigated the regularity of the non-cutoff VPB system for hard potentials $\gamma\geq -2s$ and $0< s < 1$.

Hydrodynamic limit of the VPB system has also received extensive research.
In fact, most existing research has focused on the cutoff VPB system. Firstly,  Ars\'{e}nio--Saint-Raymond \cite{AS2019} proved, within the framework of renormalized solutions, that renormalized solutions of the VPB system (\ref{GG1}) converged to dissipative solutions of the incompressible Navier--Stokes--Fourier--Poisson (NSFP, for short) system. Secondly, in the framework of perturbation around the global Maxwellian, Guo--Jiang--Luo \cite{GJL2020} justified the incompressible NSFP limit of classical solutions
for the cutoff hard sphere case $\gamma=1$, Gong--Zhou--Wu \cite{GZW-2021} considered the same limit of strong solutions by Hilbert expansion,
and Li--Yang--Zhong \cite{LYZ2020} investigated the cutoff hard potential case $0\leq \gamma\leq 1$. More recently, Wu--Zhou--Li
\cite{WZL2023ARXIV} justified the incompressible NSFP limit of strong solutions by $H^2\cap W^{2,\infty}$ approach for the full cutoff potentials $-3<\gamma\leq 1$.
In addition to diffusive limit of the VPB system mentioned above, Guo--Jang \cite{guo2010cmp} studied the compressible Euler--Poisson limit of the VPB system in hyperbolic regime by Hilbert expansion \cite{Ca} for the cutoff hard sphere case $\gamma=1$, with further extensions to the cutoff soft potential case $-3<\gamma<0$ by \cite{LW2023}.

In addition to the aforementioned VPB system, the Vlasov--Maxwell--Boltzmann system (VMB, for short), as a more general model compared with
the VPB system, has also attracted much attention. The first breakthrough on global solutions to the VMB system was achieved by \cite{Guo2003} in perturbation framework near the global Maxwellian, where Guo constructed classical solutions to the cutoff VMB system for the hard sphere case $\gamma=1$ in $\mathbb{T}^3$. Subsequently, the analysis on global classical solutions and time decay estimates in this perturbation framework was extended to the whole space $\mathbb{R}^3$ for the hard sphere case $\gamma=1$ \cite{Strain2006CMP, DS2011}, the soft potential case $-3 < \gamma < 0$ \cite{DLYZ2017}, as well as the non-cutoff VMB system \cite{DLYZ2013,FLLZ2018SCM}. We also mention the literature \cite{J2009}, which justified
the incompressible NSFP limit from the cutoff VMB system for the hard sphere case $\gamma=1$ by using a special scaling.
Until recently, global renormalized solutions with or without angular cutoff for large initial data were established by Ars\'{e}nio--Saint-Raymond \cite{AS2019}, where incompressible Navier--Stokes--Fourier--Maxwell (NSFM, for short) limit was also justified. This vital work opened up a series of subsequent progress on hydrodynamic limit of the VMB system, mainly in perturbation framework near the global Maxwellian. Jiang--Luo \cite{JL2019} justified the incompressible NSFM limit of classical solutions in perturbation framework for the cutoff hard sphere case $\gamma=1$, and also considered Hilbert expansion for the cutoff hard potential case $0\leq \gamma\leq 1$ \cite{JLZ2023ARMA}. More recently, Jiang--Lei \cite{JL2023ARXIV} investigated the incompressible NSFM limit for the cutoff VMB system in the range of soft potentials $-1\leq \gamma<0$, and for the non-cutoff VMB system in the range of soft potentials $\max\{-3, -\frac{3}{2}-2s\}<\gamma<-2s$ with strong angular singularity $\frac{1}{2}\leq s<1$. We also note the recent compressible Euler--Maxwell limit of one-species VMB system in hyperbolic regime, given by \cite{DYY2023M3AS} for the cutoff hard-sphere potential $\gamma=1$ and \cite{LLXZ2023ARXIV} for the non-cutoff soft potentials $\max\{-3, -\frac{3}{2}-2s\}<\gamma<-2s$ and $0<s<1$, respectively.

The above mentioned VPB system and VMB system both belong to the Boltzmann type equations with electric or electromagnetic field.
For more related works on global existence of solutions to kinetic equations, we refer to \cite{DY10, LYZ2016, wang2013JDE, Yangyu2011CMP, YYZ-06, YZ-06}. For the study on hydrodynamic limit of kinetic equations, we refer
the readers to \cite{BGL91, BGL93, BU91, CIP, DEL, Esposito2018, GS, guo2006, GX2020, J2009, JXZ2018, LM, Saint} and  the references cited therein.

As shown by the research progress mentioned above, for diffusive limit of the non-cutoff VPB system in perturbation framework,
the existing result is only within the range of soft potentials $\max\{-3, -\frac{3}{2}-2s\}<\gamma<-2s$ with strong angular singularity
$\frac{1}{2}\leq s<1$, given
by \cite{JL2023ARXIV}. The more challenging case of soft potentials $-3 < \gamma <-2s$ with weak angular singularity $0<s<\frac{1}{2}$ still remains open. The purpose of this paper is to fill in this gap and investigate diffusive limit of the non-cutoff VPB system \eqref{GG1} uniformly for the full potential range
$\gamma>-3$ and $0<s<1$. By introducing a new weight function and some novel treatments, we establish the  uniform weighted estimates with respect to the Knudsen number $\varepsilon\in (0,1]$ globally in time, which eventually justifies the convergence to the two-fluid incompressible NSFP system with Ohm's law for the full potential range.
The approach of this paper completely solves diffusive limit of the non-cutoff VPB system for the full potential range $\gamma>-3$ and $0<s<1$,
and is expected to solve the more challenging diffusive limit of the non-cutoff VMB system.
\medskip

\subsection{Notations}
\hspace*{\fill}

Throughout the paper, $C$ denotes a generic positive constant independent of $\varepsilon$.
We use $X \lesssim Y$ to denote $X \leq CY$, where $C$ is a constant independent of $X$, $Y$. We also use the notation $X \approx Y$ to represent $X\lesssim Y$ and $Y\lesssim X$. The notation $X \ll 1$ means that $X$ is a positive constant small enough.

The multi-indexs $\alpha = [\alpha_1, \alpha_2, \alpha_3]\in\mathbb{N}^3$ and $\beta = [\beta_1, \beta_2, \beta_3]\in\mathbb{N}^3$ will be used to record space and velocity derivatives, respectively. We  denote the $\alpha$-th order space partial derivatives by $\partial_x^\alpha=\partial_{x_1}^{\alpha_1}\partial_{x_2}^{\alpha_2}\partial_{x_3}^{\alpha_3}$,
and the $\beta$-th order velocity partial derivatives by $\partial_v^\beta=\partial_{v_1}^{\beta_1}\partial_{v_2}^{\beta_2}\partial_{v_3}^{\beta_3}$.
In addition, $\partial^\alpha_\beta=\partial^\alpha_x\partial^\beta_v
=\partial^{\alpha_1}_{x_1}\partial^{\alpha_2}_{x_2}\partial^{\alpha_3}_{x_3}
\partial^{\beta_1}_{v_1}\partial^{\beta_2}_{v_2}\partial^{\beta_3}_{v_3}$ stands for the mixed space-velocity derivative.
The length of $\alpha$ is denoted by $|\alpha|=\alpha_1+\alpha_2+\alpha_3$.
If each component of $\theta$ is not greater than that of $\bar{\theta}$, we denote by $\theta \leq \bar{\theta}$. $\theta < \bar{\theta}$ means $\theta \leq \bar{\theta}$ and $|\theta| < |\bar{\theta}|$. We shall also denote $D_v=\frac{1}{i}\partial_v$. Moreover, we use $C_b$ to denote the set of continuous and bounded functions.

We use $| \cdot |_{L^p}$ to denote the $L^p$ norm in $\mathbb{R}^3_v$, and use $\|\cdot\|_{L^p}$ to denote the $L^p$ norm in $\mathbb{R}^3_x \times \mathbb{R}^3_v$ or in $\mathbb{R}^3_x$. In particular, if $p=2$, then we use $\|\cdot\|$ to denote $L^2$ norm in $\mathbb{R}^3_x \times \mathbb{R}^3_v$ or in $\mathbb{R}^3_x$.
Besides, $\| \cdot \|_{L^p_x L^q_v}$ denotes $\| | \cdot |_{L^q_v} \|_{L^p_x}$.
$\langle \cdot, \cdot\rangle$ denotes the $L^2$ inner product in $\mathbb{R}_v^3$ and $( \cdot, \cdot ) $ denotes the $L^2$ inner product in $\mathbb{R}_x^3 \times \mathbb{R}_v^3$ or in $\mathbb{R}_x^3$.
The norm of a vector means the sum of the norms for all components
of this vector. Also, the norm of $\nabla^k f$ means the sum of the norms for all functions $\partial^\alpha f$ with $|\alpha| = k$.

We use $\Lambda^{-\varrho}f(x)$ to denote
\begin{align*}
\Lambda^{-\varrho} f(x):= {(2\pi)^{-3/2}} \int_{\mathbb{R}^3} |y|^{-\varrho} \widehat{f}(y) e^{ix \cdot y} \d y,
\end{align*}
where $\widehat{f}(y) :=  {(2\pi)^{-3/2}} \int_{\mathbb{R}^3} f(x) e^{-ix \cdot y}\d x$
represents the Fourier transform of $f(x)$. For later use, we denote $\xi$ the dual variable of $v$. Let $q \in \mathscr{S}^{\prime} (\mathbb{R}_v^3 \times \mathbb{R}_\xi^3 )$ be a tempered distribution and let $t \in \mathbb{R}$, the operator $\operatorname{op}_t q$ is an operator from $\mathscr{S}  (\mathbb{R}_v^3 )$ to $\mathscr{S}^{\prime} (\mathbb{R}_v^3 )$, whose Schwartz kernel $K_t$ is defined by the oscillatory integral
$$
K_t\left(z, z^{\prime}\right):=(2 \pi)^{-3} \int_{\mathbb{R}^3} e^{i\left(z-z^{\prime}\right) \cdot \zeta} q\left((1-t) z+t z^{\prime}, \zeta\right) d \zeta.
$$
In particular we denote $q\left(v, D_v\right)=\mathrm{op}_0 q$ and $q^w=\mathrm{op}_{1 / 2} q$. Here $q^w$ is called the Weyl quantization of symbol $q$.
Now we define
$$
q(v,\xi):=\langle v \rangle^\ell \langle \xi \rangle^\theta
$$
for $\ell, \theta \in \mathbb{R}$. Then $q$ is a $\Gamma$-admissible weight function as well as a symbol in $S(q)$, with $\Gamma=|\d v|^2+ |\d \xi|^2$.
One can refer to \cite{AHL2019, L2010} for more information about pseudo-differential calculus.

We now list series of notations introduced in \cite{GS2011}. Let $\mathscr{S}^{\prime}(\mathbb{R}_v^3)$ be the space of the tempered distribution functions. $N_\gamma^s$ denotes the weighted geometric fractional Sobolev space
\begin{align*}
N_\gamma^s:=\left\{f \in \mathscr{S}^{\prime}(\mathbb{R}_v^3):\;|f|_{N_\gamma^s}<\infty \right\}
\end{align*}
with the anisotropic norm
\begin{align*}
|f|_{N_\gamma^s}^2:=|f|_{L_{ \gamma/2+ s}^2}^2+\iint_{\mathbb{R}^3 \times \mathbb{R}^3}\left(\langle v\rangle \langle v^{\prime}\rangle\right)^{\frac{\gamma+2 s+1}{2}} \frac{\left(f(v)-f(v^{\prime}\right))^2}{d(v, v^{\prime})^{3+2 s}} \chi_{d(v, v^{\prime}) \leq 1} \d v \d v^{\prime},
\end{align*}
where $\langle v \rangle := \sqrt{1+|v|^2}$, $L_{\ell}^2$ ($\ell \in \mathbb{R}$) denotes the weighted space with the norm
\begin{align*}
|f|_{L_{\ell}^2}^2:=\int_{\mathbb{R}^3}\langle v\rangle^{2\ell}|f(v)|^2 \d v,
\end{align*}
the anisotropic metric $d(v, v^{\prime})$ measuring the fractional differentiation effects is given by
$$
d(v, v^{\prime}):=\sqrt{|v-v^{\prime}|^2+\frac{1}{4}\Big(|v|^2-|v^{\prime}|^2\Big)^2},
$$
and $\chi_A$ is the indicator function of a set $A$. Furthermore, in $\mathbb{R}_x^3 \times \mathbb{R}_v^3$, we use $\|\cdot\|_{L^2_\ell}=\big\| |\cdot|_{L^2_\ell}\big\|_{L_x^2}$.
In this paper, the velocity weight function $w=w(v)$ always denotes
\begin{align}\label{w define}
w(v):= \langle v\rangle=\sqrt{1+|v|^2}.
\end{align}
For $\ell \in \mathbb{R}$, the $w^{\ell}$ weighted $|w^{\ell}\cdot|_{N_\gamma^s}$ norm is defined by
\begin{align*}
|w^{\ell} f|_{N_\gamma^s}^2:=|w^{\ell} f|_{L_{\gamma/2+s}^2}^2+\iint_{\mathbb{R}^3 \times \mathbb{R}^3}(\langle v\rangle \langle v^{\prime}\rangle)^{\frac{\gamma+2 s+1}{2}} w^{2 \ell}(v) \frac{(f(v)-f(v^{\prime}))^2}{d(v, v^{\prime})^{3+2 s}} \chi_{d(v, v^{\prime})
\leq 1} \d v \d v^{\prime}.
\end{align*}

For $\ell \in \mathbb{R}$, define the weighted fractional Sobolev norm
\begin{align*}
|w^{\ell} f|_{H_{\gamma/2}^s}^2:=|w^{\ell} f|_{L_{\gamma/2}^2}^2+\iint_{\mathbb{R}^3 \times \mathbb{R}^3} \frac{\big[\langle v\rangle^{\frac{\gamma}{2}} w^{\ell}(v) f(v)-\langle v^{\prime}\rangle^{\frac{\gamma}{2}} w^{\ell}(v^{\prime}) f(v^{\prime})\big]^2}{|v-v^{\prime}|^{3+2 s}} \chi_{|v-v^{\prime}| \leq 1} \d v \d v^{\prime},
\end{align*}
which turns out to be equivalent with
\begin{align*}
|w^{\ell} f|_{H_{\gamma/2}^s}^2:=\int_{\mathbb{R}^3}\langle v\rangle^\gamma \big|(1-\Delta_v)^{\frac{s}{2}}(w^{\ell}(v) f(v))\big|^2 \d v.
\end{align*}
Notice that
\begin{align*}
|w^{\ell} f |_{L_{\gamma/2 + s}^2}^2+|w^{\ell} f |_{H_{\gamma/2}^s}^2 \lesssim |w^{\ell} f |_{N_\gamma^s}^2
\lesssim |w^{\ell} f |_{H_{\gamma/2+  s}^s}^2,
\end{align*}
 cf. \cite{GS2011}.
In $\mathbb{R}_x^3 \times \mathbb{R}_v^3$, we use the notations
$$
  \|w^{\ell} f \|_{H_{\gamma/2}^s}=\big\| |w^{\ell} f |_{H_{\gamma/2}^s}\big\|_{L_x^2},
  \quad
 \|w^{\ell} f \|_{N_\gamma^s}=\big\| |w^{\ell} f |_{N_\gamma^s}\big\|_{L_x^2},
 \quad
 \| w^{\ell} f  \|_{L^p_x N_\gamma^s}=\big\| | w^{\ell} f |_{N_\gamma^s}\big\|_{L_x^p}.
$$

For an integer $N \geq 0$, we define the Sobolev space
\begin{align*}
|f|_{H^N}:=\sum_{|\beta| \leq N} |\partial_\beta f |_{L^2 (\mathbb{R}_v^3 )}, \;\;
\|f\|_{H^N_{x}}:=\sum_{|\alpha| \leq N} \|\partial^\alpha f \|_{L^2 (\mathbb{R}_x^3)}, \;\;
\|f\|_{H^N_{x,v}}:=\sum_{|\alpha|+|\beta| \leq N} \|\partial_\beta^\alpha f \|_{L^2 (\mathbb{R}_x^3 \times \mathbb{R}_v^3 )} .
\end{align*}
For an integer $N \geq 0$ and $\ell \in \mathbb{R}$, we define the weighted Sobolev space
\begin{align*}
|w^{\ell} f |_{H^N}:=&\sum_{|\beta| \leq N} |\partial_\beta(w^{\ell}  f) |_{L^2 (\mathbb{R}_v^3)}
\approx \sum_{|\beta| \leq N} |w^{\ell} \partial_\beta f |_{L^2 (\mathbb{R}_v^3)}, \\
\|w^{\ell} f \|_{H^N_{x}}:=&\sum_{|\alpha| \leq N} \| w^{\ell} \partial^\alpha f \|_{L^2 (\mathbb{R}_x^3 )},\\
\|w^{\ell} f \|_{H^N_{x,v}}:=&\sum_{|\alpha|+|\beta| \leq N} \|\partial_\beta^\alpha (w^{\ell} f) \|_{L^2 (\mathbb{R}_x^3 \times \mathbb{R}_v^3)}
\approx \sum_{|\alpha|+|\beta| \leq N} \|w^{\ell} \partial_\beta^\alpha f \|_{L^2 (\mathbb{R}_x^3 \times \mathbb{R}_v^3)},
\end{align*}
and
\begin{align*}
|w^{\ell} f |_{H^N_{\gamma/2}}:=&\sum_{|\beta| \leq N} | \langle v \rangle ^{\frac{\gamma}{2}}\partial_\beta(w^{\ell}  f) |_{L^2 (\mathbb{R}_v^3)}
\approx \sum_{|\beta| \leq N} |\langle v \rangle ^{\frac{\gamma}{2}} w^{\ell} \partial_\beta f |_{L^2 (\mathbb{R}_v^3)}, \\
\|w^{\ell} f \|_{H^N_{\gamma/2}}:=&\sum_{|\alpha|+|\beta| \leq N} \| \langle v \rangle ^{\frac{\gamma}{2}} \partial_\beta^\alpha (w^{\ell} f) \|_{L^2 (\mathbb{R}_x^3 \times \mathbb{R}_v^3)}
\approx \sum_{|\alpha|+|\beta| \leq N} \|\langle v \rangle ^{\frac{\gamma}{2}} w^{\ell} \partial_\beta^\alpha f \|_{L^2 (\mathbb{R}_x^3 \times \mathbb{R}_v^3)},
\end{align*}
where the equivalent norm holds due to the algebra weight $w^{\ell}$.

Finally, we define $B_C \subset \mathbb{R}^3$ to be the ball with center origin and radius $C$, and use $L^2 (B_C )$ to denote the space $L^2$ over $B_C$ and likewise for other spaces.


\subsection{Main Results}
\hspace*{\fill}

We investigate diffusive limit of the two-species VPB system (\ref{GG1}) in the framework of perturbation around the global Maxwellian.
It is well-known that $[\mu(v),\mu(v)]$ forms the global equilibrium of the two-species VPB system (\ref{GG1}), where
$$
\mu=\mu(v) :=(2 \pi)^{-3 / 2} \mathrm{e}^{-|v|^{2} / 2}
$$
represents the global Maxwellian.
By writing
$$
F_{\pm}^{\varepsilon}(t,x,v)=\mu+\varepsilon \mu^{1/2}f_{\pm}^{\varepsilon}(t,x,v),\quad
F^\varepsilon_{\pm, 0}(x,v)=\mu+\varepsilon \mu^{1/2}f_{\pm,0}^{\varepsilon}(x,v)
$$
with the fluctuations $f_{\pm}^{\varepsilon}$ and $f_{\pm,0}^{\varepsilon}$,
the VPB system
(\ref{GG1}) is transformed into the following equivalent perturbed VPB system
\begin{align}\label{rVPB}
	\left\{\begin{array}{l}
\displaystyle \partial_{t} f^{\varepsilon}+\frac{1}{\varepsilon}v \cdot \nabla_{x} {f}^{\varepsilon}+\frac{1}{\varepsilon} \nabla_{x} \phi^{\varepsilon} \cdot v \mu^{1/2}q_1+\frac{1}{\varepsilon^{2}} L {f}^{\varepsilon}
\\[2mm]
\displaystyle\;\;\;\;\;\;\;\;\;\;\;\;\;\;\;\;\;\;\;\;\;\;\;\;\;\;\;\; =q_0\nabla_{x} \phi^{\varepsilon} \cdot \nabla_{v} {f}^{\varepsilon}-\frac{1}{2}q_0 \nabla_{x} \phi^{\varepsilon} \cdot v {f}^{\varepsilon}+\frac{1}{\varepsilon} \Gamma\left({f}^{\varepsilon}, {f}^{\varepsilon}\right),  \\ [2mm]
\displaystyle -\Delta_{x} \phi^{\varepsilon}=\int_{\mathbb{R}^3}f^{\varepsilon}\cdot q_1 \mu^{1/2}\d v ,
\quad \lim_{|x|\rightarrow \infty}\phi^\varepsilon=0,  \\ [3mm]
\displaystyle f^\varepsilon(0,x,v)=f^\varepsilon_{0}(x,v),
	\end{array}\right.
\end{align}
where $f^{\varepsilon}=[f_{+}^{\varepsilon}, f_{-}^{\varepsilon}]$ represents the vector in $\mathbb{R}^2$ with the components $f_{\pm}^{\varepsilon}$, $q_0=\mathrm{diag}(1, -1)$ represents the $2 \times 2$ diagonal matrix, $q_1=[1, -1]$, the linearized Boltzmann collision operator $Lf^{\varepsilon}$ and the nonlinear collision term $\Gamma(f^{\varepsilon}, f^{\varepsilon})$ are respectively defined by
$$Lf^{\varepsilon}:=[L_{+}f^{\varepsilon}, L_{-}f^{\varepsilon}],\qquad
\Gamma(f^{\varepsilon}, f^{\varepsilon}):=\left[\Gamma_{+}(f^{\varepsilon}, f^{\varepsilon}), \Gamma_{-}(f^{\varepsilon}, f^{\varepsilon})\right].
$$
For any given $f=\left[f_{+}, f_{-}\right]$ and $g=\left[g_{+}, g_{-}\right]$,
define
\begin{align}
\begin{split}
L_{\pm}f&:=-\mu^{-1/2}\left\{2Q\left(\mu, \mu^{1/2}f_{\pm}\right)+Q\left(\mu^{1/2}(f_{\pm}+f_{\mp}),\mu\right)\right\},  \\
\Gamma_{\pm}(f,g)&:=\mu^{-1/2}\left\{Q\left(\mu^{1/2}f_{\pm},\mu^{1/2}g_{\pm}\right)+Q\left(\mu^{1/2}f_{\mp},\mu^{1/2}g_{\pm}\right)\right\}.
\end{split}
\end{align}

It is well-known that the operator $L$ is non-negative with null space
$$
\mathcal{N}(L):=\mathrm{span}\left\{[1,0]\mu^{1/2},[0,1]\mu^{1/2},[v_{i},v_{i}]\mu^{1/2}(1\leq i\leq 3),\big[|v|^2, |v|^2\big]\mu^{1/2}\right\},
$$
cf. \cite{Guo2003}.
For given vector valued function $f(t,x,v)$, the following macro-micro decomposition  was introduced in \cite{Guo2004}
\begin{equation}\label{f decomposition}
f=\mathbf{P}f+\{\mathbf{I}-\mathbf{P}\}f.
\end{equation}
Here, $\mathbf{P}$ denotes the orthogonal projection from $L^2(\mathbb{R}_{v}^3) \times L^2(\mathbb{R}_{v}^3)$ to $\mathcal{N}(L)$, defined by
\begin{equation}\label{Pf define}
\mathbf{P}f:=\left\{a_{+}(t,x)[1,0]+a_{-}(t,x)[0,1]+v \cdot b(t,x)[1,1]+(|v|^2-3)c(t,x)[1,1]\right\}\mu^{1/2},
\end{equation}
or equivalently $\mathbf{P}=[\mathbf{P}_{+},\mathbf{P}_{-}]$ with
$$
\mathbf{P}_{\pm}f:=\left\{a_{\pm}(t,x)+v \cdot b(t,x)+ (|v|^2-3)c(t,x)\right\}\mu^{1/2},
$$
where
\begin{align}
a_{\pm}(t,x):=\;&\langle \mu^{1/2}, f_{\pm} \rangle =\langle \mu^{1/2}, \mathbf{P}_{\pm}f \rangle, \nonumber \\
b(t,x):=\;&\frac{1}{2}\langle v\mu^{1/2},f_{+}+f_{-}\rangle =\langle v\mu^{1/2}, \mathbf{P}_{\pm}f\rangle, \nonumber \\
c(t,x):=\;&\frac{1}{12} \big \langle (|v|^2-3)\mu^{1/2},f_{+}+f_{-} \big \rangle= \frac{1}{6} \big\langle(|v|^2-3)\mu^{1/2},\mathbf{P}_{\pm}f \big\rangle. \nonumber
\end{align}

To state the main results of this paper, we introduce the following fundamental instant energy functional and the corresponding dissipation rate functional
\begin{align}
\label{without weight energy functional}
{\mathcal{E}}_{N} (t) \sim\;
&\sum_{|\alpha|\leq N} \left\|\partial^{\alpha}f^{\varepsilon}\right\|^2
+\sum_{|\alpha|\leq N+1} \left\|\partial^{\alpha}\nabla_x{\phi}^{\varepsilon}\right\|^2,  \\
\label{without weight dissipation functional}
{\mathcal{D}}_{N} (t)\sim \;
&\sum_{1 \leq |\alpha| \leq N} \left \| \partial^{\alpha} \mathbf{P}f^\varepsilon \right \|^2
+ \frac{1}{\varepsilon^2} \sum_{|\alpha| \leq N}\left\|\partial^{\alpha}
\{\mathbf{I}-\mathbf{P}\}f^\varepsilon\right\|^2_{N^s_{\gamma}}
+\sum_{|\alpha| \leq N+1} \left\|\partial^{\alpha}\nabla_x{\phi}^{\varepsilon}\right\|^2,
\end{align}
respectively, where the integer $N\in\mathbb{Z}$ will be determined later.

Due to the weaker dissipation of the linearized Boltzmann operator $L$ for soft potentials rather than the hard sphere case, in order to deal with the external force term brought by the self-consistent electric potential, we need to introduce the following velocity weight function for hard and soft potentials
\begin{align}\label{weight function}
w_{l}(\alpha,\beta):=
\begin{cases}
\langle v \rangle^{l-|\alpha|-|\beta|}, & \quad \text{ for } \gamma \geq -2s, \\
\langle v \rangle^{l-(-\frac{3\gamma}{s}+\gamma)|\alpha|-(-\frac{3\gamma}{s})|\beta|}, &  \quad \text{ for } -3 < \gamma < -2s,
\end{cases}
\end{align}
where the constant $l \geq 0$ will be determined later. For convenience, sometimes we write $w_{l}(|\alpha|,|\beta|)$ to denote $w_{l}(\alpha, \beta)$.
Then the weighted instant energy functional ${\mathcal{E}}_{N,l} (t)$, the weighted instant high-order energy functional ${\mathcal{E}}^h_{N,l} (t)$ and the corresponding dissipation rate functional ${\mathcal{D}} _{N,l} (t)$ are defined as
\begin{align}\label{energy functional}
{\mathcal{E}}_{N,l} (t) \sim\;
& {\mathcal{E}}_{N} (t) + \sum_{\substack{|\alpha|+|\beta| \leq N \\ |\alpha| \leq N-1}}\left\|w_l (\alpha, \beta) \partial_{\beta}^{\alpha} \{\mathbf{I}-\mathbf{P}\}f^{\varepsilon}\right\|^2
+\varepsilon\sum_{|\alpha|=N}\left\|w_l (\alpha, 0) \partial^{\alpha} f^{\varepsilon}\right\|^2 , \\
\label{high energy functional}
{\mathcal{E}}^h_{N,l} (t) \sim\;
& \sum_{1 \leq |\alpha|\leq N} \left\|\partial^{\alpha} \mathbf{P}f^{\varepsilon}\right\|^2
+ \sum_{ |\alpha|\leq N} \left\|\partial^{\alpha} \{\mathbf{I}-\mathbf{P}\}f^{\varepsilon}\right\|^2
+\sum_{|\alpha|\leq N+1} \left\|\partial^{\alpha}\nabla_x{\phi}^{\varepsilon}\right\|^2
 \nonumber\\
&+ \sum_{\substack{|\alpha|+|\beta| \leq N \\ |\alpha| \leq N-1}}\left\|w_l (\alpha, \beta) \partial_{\beta}^{\alpha} \{\mathbf{I}-\mathbf{P}\}f^{\varepsilon}\right\|^2
+\varepsilon\sum_{|\alpha|=N}\left\|w_l (\alpha, 0) \partial^{\alpha} f^{\varepsilon}\right\|^2 , \\
\label{dissipation functional}
{\mathcal{D}} _{N,l} (t)\sim \;& {\mathcal{D}}_{N} (t)
+\frac{1}{\varepsilon^{2}} \sum_{\substack{|\alpha|+|\beta| \leq N \\
|\alpha| \leq N-1}}\left\|w_l (\alpha, \beta) \partial_{\beta}^{\alpha}
\{\mathbf{I}-\mathbf{P}\} f^{\varepsilon}\right\|^2_{N^s_{\gamma}} \nonumber\\
&+\frac{1}{\varepsilon}\sum_{|\alpha|=N}\!\! \left\|w_l (\alpha, 0) \partial^{\alpha}
\{\mathbf{I}-\mathbf{P}\} f^{\varepsilon}\right\|^2_{N^s_{\gamma}},
\end{align}
respectively.

Besides, to seek for the desired time decay rate to close the energy estimate, we introduce the following energy functional with the lowest $k$-order space derivative $\mathcal{E}^k_{N}(t)$ and the corresponding dissipation rate functional with the lowest $k$-order space derivative ${\mathcal{D}}^k_{N} (t)$
\begin{align}
\label{low k energy}
\mathcal{E}^k_{N}(t)\sim\;
& \sum_{k \leq |\alpha| \leq N} \left\|\partial^{\alpha}f^{\varepsilon}\right\|^2
+\sum_{k \leq |\alpha| \leq N+1} \left\|\partial^{\alpha}\nabla_x{\phi}^{\varepsilon}\right\|^2,  \\
\label{low k dissipation}
{\mathcal{D}}^k_{N} (t)\sim \;
& \sum_{k+1 \leq |\alpha| \leq N} \left \| \partial^{\alpha} \mathbf{P}f^\varepsilon \right \|^2
+ \frac{1}{\varepsilon^2} \sum_{k \leq |\alpha| \leq N} \left\|\partial^{\alpha}
\{\mathbf{I}-\mathbf{P}\}f^\varepsilon\right\|^2_{N^s_{\gamma}}
+\sum_{k \leq |\alpha| \leq N+1} \left\|\partial^{\alpha}\nabla_x{\phi}^{\varepsilon}\right\|^2,
\end{align}
respectively.

Moreover, due to the singularity brought by the Knudsen number $\varepsilon$, the time decay estimate of the nonlinear problem can not be obtained by the Duhamel principle and semigroup estimate of the linearized problem, so that we have to resort the Sobolev space with negative index. For this, we define the following instant energy functional $\widetilde{\mathcal{E}}_{N,l}(t)$ and the corresponding dissipation rate functional $\widetilde{\mathcal{D}}_{N,l} (t)$
\begin{align}
\label{negative sobolev energy}
\widetilde{\mathcal{E}}_{N,l} (t) \sim\;
& \mathcal{E}_{N,l} (t) +\left\| \Lambda^{-\varrho}f^\varepsilon\right\|^2
+\left\| \Lambda^{-\varrho}\nabla_x \phi^\varepsilon\right\|^2_{H^1},\\
\label{negative sobolev dissipation}
\widetilde{\mathcal{D}}_{N,l} (t) \sim\;& \mathcal{D}_{N,l} (t)
+\left\| \Lambda^{1-\varrho} \mathbf{P}f^\varepsilon\right\|^2
+\frac{1}{\varepsilon^2}\left\| \Lambda^{-\varrho} \{\mathbf{I}-\mathbf{P}\}f^\varepsilon\right\|_{N^s_\gamma}^2
+ \left\| \Lambda^{-\varrho}\nabla_x \phi^\varepsilon\right\|_{H^2}^2,
\end{align}
respectively, where the positive constant $\varrho>0$ will be determined later.
\medskip

Our first main result gives the uniform estimate with respect to $\varepsilon\in (0,1]$ of global solutions to the VPB system (\ref{rVPB}) for soft potentials $-3<\gamma<-2s$ and $0<s<1$.

\begin{theorem}[Soft potentials]  \label{mainth1}
Let $-3<\gamma<-2s$, $0 < s <1$ and $0<\varepsilon \leq 1$. Introduce the following constants in sequence
\begin{equation}
\begin{split}\label{soft assumption}
 &N\geq 5\; (N\in \mathbb{Z}), \quad  l_0 \geq -\frac{3\gamma}{s}N+1,\\
 & 1 < \varrho < 3/2\;\text{ and }\; 1/2 < p <1 \; \text{ with }  \; \varrho+p>2,\;
 \; l_1 = -\frac{\varrho+p}{2(1-p)}\left(\gamma+2s\right).
\end{split}
\end{equation}
If there exists a small constant $\delta_1>0$ independent of $\varepsilon$ such that the initial data
\begin{equation}\nonumber
\widetilde{\mathcal{E}}_{N,l_0+l_1}(0) \leq \delta_1,
\end{equation}
  then the VPB system \eqref{rVPB} admits a unique global solution
$(f^{\varepsilon},\nabla_{x}\phi^\varepsilon)$  satisfying
\begin{align}
&\mathcal{E}_{N,l_0}(t) \leq C\widetilde{\mathcal{E}}_{N,l_0+l_1}(0) (1+t)^{-\varrho}, \label{thm1 N l0 decay}\\
&\mathcal{E}^h_{N,l_0}(t) \leq C\widetilde{\mathcal{E}}_{N,l_0+l_1}(0)(1+t)^{-(\varrho+p)}, \label{thm1 N h l0 decay}\\
&\widetilde{\mathcal{E}}_{N,l_0+l_1}(t)+\int_{0}^t \widetilde{\mathcal{D}}_{N,l_0+l_1}(\tau) \d \tau \leq C\widetilde{\mathcal{E}}_{N,l_0+l_1}(0) \label{thm1 widetilde N l0+l1 decay}
\end{align}
for any $t \geq 0$ and some positive constant $C>0$ independent of $\varepsilon$.
\end{theorem}
\medskip

Then we state our second main result on the uniform estimate with respect to $\varepsilon \in (0,1]$  of global solutions to the VPB system (\ref{rVPB}) for hard potentials $\gamma+2s \geq 0$ and $0<s<1$.
\begin{theorem}[Hard potentials]   \label{mainth2} \
Let $\gamma\geq -2s$, $0 < s < 1$ and $0<\varepsilon \leq 1$.
Define the following constants
\begin{equation}
\begin{split}\label{hard assumption}
 &N\geq 2\; (N\in \mathbb{Z}), \quad  l_2 \geq N+\max\left\{1, \gamma+2s\right \} \; \text{ and } \; 1 < \varrho < 3/2.
\end{split}
\end{equation}
If there exists a small constant $\delta_2>0$ independent of $\varepsilon$ such that the initial data
\begin{equation}\nonumber
\widetilde{\mathcal{E}}_{N,l_2}(0)\leq \delta_2,
\end{equation}
then the VPB system \eqref{rVPB} admits a unique global solution
$\left(f^{\varepsilon},\nabla_{x}\phi^\varepsilon\right)$  satisfying
\begin{align}
&\mathcal{E}_{N,l_2}(t) \leq C\widetilde{\mathcal{E}}_{N,l_2}(0)(1+t)^{-\varrho}, \\
&\widetilde{\mathcal{E}}_{N,l_2}(t) + \int_0^t \widetilde{\mathcal{D}}_{N,l_2}(\tau) \d \tau \leq C\widetilde{\mathcal{E}}_{N,l_2}(0)
\end{align}
for any $t \geq 0$ and some positive constant $C>0$ independent of $\varepsilon$.
\end{theorem}
\medskip

The third main result is on hydrodynamic limit from the VPB system \eqref{rVPB}
to the two-fluid incompressible NSFP system with Ohm's law.
\begin{theorem}[Hydrodynamic limit]\label{mainth3}
Let $(f^\varepsilon, \nabla_{x}\phi^\varepsilon)$  be the
global solutions to the VPB system \eqref{rVPB} constructed in Theorem
\ref{mainth1} or Theorem \ref{mainth2} with initial data $f_0^\e={f}_0^\e(x,v)$. Suppose that there exist scalar functions $(\rho_0, n_0, \theta_0, \omega_0)=(\rho_0(x), n_0(x), \theta_0(x), \omega_0(x))$ and vector-valued functions $(u_0, \nabla_x \phi_0, j_0)=(u_0(x),\nabla_x \phi_0(x), j_0(x))$ such that
\begin{align}
\begin{split}\label{theorem1.3 1}
&{f}_0^\varepsilon \to {f}_0\;\text{ strongly in } H^{N}_x L^2_v,
\;\; \nabla_{x}\phi_0^\varepsilon \to \nabla_{x}\phi_0 \; \text{ strongly in } H^{N}_x,\\
&\frac{1}{\varepsilon}\big\langle f_0^{\varepsilon}, q_1v\mu^{1/2} \big\rangle \to j_0, \;\;
\frac{1}{\varepsilon}\big\langle f_0^{\varepsilon}, q_1\big(\frac{|v|^2}{3}-1\big)\mu^{1/2} \big\rangle \to \omega_0\; \text{ strongly in } H^{N}_x
\end{split}
\end{align}
as $\varepsilon \to 0$ and ${f}_0={f}_0(x, v)$ is of the form
\begin{align}\label{theorem1.3 2}
\begin{split}
\!\!\!f_0=\;&\big(\rho_0+\frac{1}{2}n_0\big)\frac{q_1+q_2}{2}\mu^{1/2} +\! \big(\rho_0-\frac{1}{2}n_0\big)\frac{q_2-q_1}{2}\mu^{1/2}+ u_0 \cdot v q_2 \mu^{1/2} + \theta_0 \big(\frac{|v|^2}{2}-\frac{3}{2}\big)q_2 \mu^{1/2},
\end{split}
\end{align}
where $q_{1}=[1,-1], q_{2}=[1,1]$.

Then there hold
\begin{equation}\label{theorem1.3 4}
\begin{split}
&{f}^\varepsilon \to {f}
\; \text{ weakly}\!-\!* \text{ in }  L^\infty(\mathbb{R}^+; H^N_x L^2_v)
 \text{ and strongly in } C(\mathbb{R}^+; H^{N-1}_x L^2_v), \\
&\nabla_x\phi^\varepsilon \to \nabla_x\phi \;
\text{ weakly}\!-\!*  \text{ in }   L^\infty(\mathbb{R}^+;H^{N+1}_{x})
\text{ and strongly in } C(\mathbb{R}^+; H^{N}_x),\\
&\frac{1}{\varepsilon}\big\langle f^{\varepsilon}, q_1v\mu^{1/2} \big\rangle \to j\;
\text{ weakly}\!-\!* \text{ in } L^\infty(\mathbb{R}^+; H^N_{x}) \text{ and strongly in } C(\mathbb{R}^+; H^{N-1}_x), \\
&\frac{1}{\varepsilon}\big\langle f^{\varepsilon}, q_1\big(\frac{|v|^2}{3}-1\big)\mu^{1/2} \big\rangle \to \omega \;
\text{ weakly}\!-\!*  \text{ in } L^\infty(\mathbb{R}^+; H^N_{x}) \text{ and strongly in } C(\mathbb{R}^+; H^{N-1}_x)
\end{split}
\end{equation}
as $\varepsilon \to 0$, where $f={f}(t, x, v)$ has the form
\begin{align}
\begin{split}\label{theorem1.3 5}
f=\;&\big(\rho+\frac{1}{2}n\big)\frac{q_1+q_2}{2}\mu^{1/2} + \big(\rho-\frac{1}{2}n\big)\frac{q_2-q_1}{2}\mu^{1/2} + u \cdot v q_2 \mu^{1/2} + \theta\big(\frac{|v|^2}{2}-\frac{3}{2}\big)q_2 \mu^{1/2}.
\end{split}
\end{align}
Moreover, the above mass density $\rho=\rho(t,x)$, the bulk velocity $u=u(t,x)$, the temperature $\theta=\theta(t,x)$, the electric charge $n=n(t,x)$, the internal electric energy $\omega=\omega(t,x)$, the electric field $\nabla_x \phi=\nabla_x \phi(t,x)$ and the electric current $j=j(t,x)$ satisfy
 \begin{align}
\begin{split}\label{jixian solution space}
&\left(\rho, u, \theta, n,\omega\right) \in  C(\mathbb{R}^+; H^{N}_{x}),\;\; \nabla_x \phi\in C(\mathbb{R}^+; H^{N+1}_{x}),\;\; j \in C(\mathbb{R}^+; H^{N-1}_{x})
\end{split}
\end{align}
 and the following two-fluid incompressible NSFP system with Ohm's law
\begin{equation}\label{INSFP limit}
\left\{
\begin{array}{ll}
\displaystyle \partial_{t} u+u \cdot \nabla_{x} u-\nu \Delta_{x} u+\nabla_{x}P=-\frac{1}{2} n \nabla_{x} \phi,  &\nabla_{x} \cdot u=0,\\[2mm]
\displaystyle \partial_{t} \theta+u \cdot \nabla_{x} \theta-\kappa \Delta_{x} \theta=0, &\rho+\theta=0,\\[2mm]
\displaystyle \partial_{t} {n}+  u \cdot \nabla_x n - \frac{\sigma}{2} \Delta_x n + \sigma n=0, &-\Delta_{x} \phi=n,\\ [2mm]
\displaystyle j=n u-\sigma\Big(\nabla_{x} \phi+\frac{1}{2} \nabla_{x} n\Big), &\omega=n \theta,\\[2mm]
\displaystyle u(0)=\mathcal{P}u_0,\; \theta (0)=\frac{3}{5}\theta_0-\frac{2}{5}\rho_0,\;  n(0)=n_0,&
\end{array} \right.
\end{equation}
where $\mathcal{P}$ is the Leray projection, and the viscosity coefficient $\nu$, the heat conductivity coefficient $\kappa$ and the electrical conductivity coefficient $\sigma$ are defined in \eqref{nu define}, \eqref{kappa define}, \eqref{sigma define} respectively.
\end{theorem}
\medskip

\begin{remark}\label{remark-0}
The results given in Theorems \ref{mainth1}, \ref{mainth2} and \ref{mainth3} indicate that
\begin{equation*}
  \begin{split}
    &\sup_{0\leq t\leq \infty}   \left\|  \left[ F_{+}^{\varepsilon},  F_{-}^{\varepsilon} \right](t,x,v)- \left[ \mu(v),  \mu(v) \right]-\varepsilon \mu^{1/2}(v)f(t,x,v)    \right\|_{H^{N-1}_xL^2_v} =o(\varepsilon), \\
    & \sup_{0\leq t\leq \infty} \left\| \nabla_x \phi^\varepsilon(t,x) - \nabla_x \phi(t,x)  \right\|_{H^{N}_x}  =o(1),
  \end{split}
\end{equation*}
that is, the two-fluid incompressible NSFP system with Ohm's law (\ref{INSFP limit}) is the first-order approximation of
the original VPB system (\ref{GG1}).
\end{remark}
\medskip

\begin{remark}\label{remark-1.4}
This paper solves the open problem of diffusive limit of the non-cutoff VPB system
in perturbation framework for the full range of collision potentials $\gamma> -3$ and $0<s<1$, by establishing weighted uniform estimate
with respect to $\varepsilon\in (0,1]$ globally in time.

As a byproduct, this uniform estimate with respect to $\varepsilon\in (0,1]$ also extends the global existence result of the non-cutff VPB system
for soft potentials from strong angular singularity $\frac{1}{2} \leq s<1$ \cite{DL2013} to the full range $0<s<1$. Besides, weaker initial conditions are required in this work. In fact, for the soft potential case $-3<\gamma<-2s$, only $N \geq 5$ is required in Theorem \ref{mainth1}, rather than the requirement $N \geq 8$ in \cite{DL2013}; for the hard potential case  $\gamma+2s\geq 0$, $N \geq 2$ is enough in Theorem \ref{mainth2}, instead of $N \geq i+1$ $(i=1,2 ~determined~by~ s)$ needed in \cite{D2021-1}.
\end{remark}
\medskip

\begin{remark}\label{remark-1.5}

In fact, by combining the uniform estimate with respect to $\varepsilon\in (0,1]$, semigroup estimate of the linearized VPB system  and the Duhamel principle (cf. \cite{DYZ2013,WZL2023ARXIV}), we can obtain the same optimal $L^2$ time decay rate of the VPB system \eqref{rVPB} for all $\varepsilon\in (0,1]$ and full range of collision potentials $\gamma> -3$ and $0<s<1$, namely,
\begin{align*}
\|f^\varepsilon(t)\|+\|\nabla_x \phi^\varepsilon(t)\| \lesssim (1+t)^{-\frac{3}{4}},\quad \forall \;t\geq 0.
\end{align*}
 But this approach falls to obtain the time decay of the first-order and higher-order derivatives for the soft potential case (still applicable to the hard potential case), due to the singularity $\frac{1}{\varepsilon}$ when deriving the time decay of the corresponding nonlinear problem. This is also the reason that we have to resort negative Sobolev space \cite{GW2012CPDE} to obtain sufficient time decay
 of the first-order derivative, which is essential to close the weighted energy estimate, cf. Section \ref{Global Existence}.
\end{remark}
\medskip

\subsection{Difficulties  and Innovations}
\hspace*{\fill}

In this subsection, we outline the difficulties and innovations proposed in this paper.

Our analysis is based on the uniform weighted energy estimate with respect to $\varepsilon\in (0,1]$ for the VPB system \eqref{rVPB} globally in time, which eventually leads to the incompressible NSFP limit.
To deal with the full soft range $-3<\gamma<-2s$ and $0<s<1$,  we are faced with considerable difficulties induced by the weak dissipation of the
soft linearized Boltzmann operator and the singularity $\frac{1}{\varepsilon}$.
In what follows, we point out the critical technical points in our treatment.

\subsubsection {Control of the velocity growth in the transport $\frac{1}{\varepsilon} v \cdot \nabla_x f^\varepsilon$ and electric force  $\nabla_{x} \phi^{\varepsilon} \cdot v {f}^{\varepsilon}$}
\hspace*{\fill}

 As is well known, the weighted energy estimate is necessary for the Boltzmann equation with soft potentials, due to the weak dispassion of the linearized Boltzmann operator.
 To handle the velocity growth in the transport term $\frac{1}{\varepsilon} v \cdot \nabla_x f^\varepsilon$ as well as the singular factor $\frac{1}{\varepsilon}$, we design a new velocity weight function
\begin{align}\label{weight-original}
w_l(\alpha,\beta)=\langle v \rangle^{l-(m+\gamma)|\alpha|-m|\beta|},
\end{align}
where the constants $l$ and $m$ will be determined later. With the help of this weight function $w_l(\alpha,\beta)$, we estimate the transport term
$\frac{1}{\varepsilon} \partial^\alpha_{\beta}( v \cdot \nabla_x f^\varepsilon)$
for $|\alpha|+|\beta| \leq N$ and $|\beta|\geq 1$ as
\begin{align*}
&\frac{1}{\varepsilon} \left(\partial^{\alpha+e_i}_{\beta-e_i} \{\mathbf{I}-\mathbf{P}\}f^\varepsilon, w_l^2(\alpha,\beta)\partial^\alpha_{\beta} \{\mathbf{I}-\mathbf{P}\}f^\varepsilon\right)\\
\lesssim\;&\frac{1}{\varepsilon}\left\| w_l(\alpha,\beta) \left\langle v \right\rangle^\frac{\gamma}{2} \partial^\alpha_{\beta} \{\mathbf{I}-\mathbf{P}\}f^\varepsilon \right\|
\left\| w_l(\alpha+e_i,\beta-e_i) \left\langle v \right\rangle^\frac{\gamma}{2} \partial^{\alpha+e_i}_{\beta-e_i} \{\mathbf{I}-\mathbf{P}\}f^\varepsilon \right\| \nonumber\\
\lesssim\;&\frac{\eta}{\varepsilon^2}\left\| w_l(\alpha,\beta)  \partial^\alpha_{\beta} \{\mathbf{I}-\mathbf{P}\}f^\varepsilon \right\|^2_{N^s_\gamma}
+C_\eta\sum_{\substack{{|\alpha^\prime|=|\alpha|+1}\\{|\beta^\prime|=|\beta|-1}}}\left\| w_l(\alpha^\prime,\beta^\prime) \partial^{\alpha^\prime}_{\beta^\prime} \{\mathbf{I}-\mathbf{P}\}f^\varepsilon \right\|^2_{N^s_\gamma},
\end{align*}
where we have used the notation $e_i$ to denote the multi-index with the $i$-th element unit and the rest zero, as well as the fact
\begin{align*}
w_l(\alpha,\beta)=w_l(\alpha+e_i,\beta-e_i)\langle v \rangle^\gamma.
\end{align*}
Therefore, by proper linear combination of the weighted energy estimates for each order, the transport term can be treated with the help of this weight.
Note that this design of weight function $w_l(\alpha,\beta)$ can eliminate the limitation of strong angular singularity, i.e. $\frac{1}{2} \leq s < 1$ required in \cite{DL2013}.

To overcome the velocity growth of the electric potential term $\frac{1}{2}\partial^\alpha_\beta(\nabla_x \phi^\varepsilon \cdot v f_{\pm}^\varepsilon)$, inspired by \cite{Guo2012JAMS},
we introduce an $\varepsilon$ dependent exponential factor $e^{\pm \varepsilon \phi^\varepsilon}$ to cancel this term together with the transport term $\frac{1}{\varepsilon} v \cdot \nabla_x f_{\pm}^\varepsilon$, namely
\begin{align*}
\left( \frac{1}{\varepsilon} v_i \partial^{\alpha+e_i}_\beta f_{\pm}^\varepsilon \pm \frac{1}{2} v_i \partial^{e_i} \phi^\varepsilon\partial^\alpha_\beta f_{\pm}^\varepsilon, e^{\pm \varepsilon \phi^\varepsilon } w^2_l(\alpha,\beta)\partial^\alpha_\beta f_{\pm}^\varepsilon\right)=0
\end{align*}
when all derivatives $\partial^\alpha_\beta$ act on $f_{\pm}^\varepsilon$.
Besides, we also need to deal with the remaining terms when not all derivatives act on $f^\varepsilon_{\pm}$, such as the term $ v_i \partial^{\alpha_1+e_i}\phi^\varepsilon \partial^{\alpha-\alpha_1} \{\mathbf{I}_{\pm}-\mathbf{P}_{\pm}\}f^\varepsilon$ for $1\leq |\alpha_1| \leq |\alpha|$.
By setting $m \geq -2\gamma+1$ in the definition of $w_l(\alpha,\beta)$, we get
$$
  \langle v \rangle w_l(\alpha,0) \leq \langle v \rangle^{\gamma} w_l(|\alpha|-1,0),
$$
 which eventually leads to the following estimate
\begin{align*}
&\sum_{1 \leq | \alpha_1 | \leq |\alpha|} \left( v_i \partial^{\alpha_1+e_i}\phi^\varepsilon \partial^{\alpha-\alpha_1} \{\mathbf{I}_{\pm}-\mathbf{P}_{\pm}\} f^{\varepsilon}, e^{\pm \varepsilon \phi^{\varepsilon}} w^2_l(\alpha, 0) \partial^\alpha  \{\mathbf{I}_{\pm}-\mathbf{P}_{\pm}\} f^{\varepsilon} \right) \\
\lesssim\;& \sum_{|\alpha_1|=1}\left\|\partial^{\alpha_1}\nabla_x \phi^\varepsilon \right\|_{L^\infty}
\left\| \langle v \rangle^{\frac{\gamma}{2}}w_l(|\alpha|-1, 0) \partial^{\alpha-\alpha_1}  \{\mathbf{I}-\mathbf{P}\} f^{\varepsilon} \right\|
\left\| \langle v \rangle^{\frac{\gamma}{2}}w_l(\alpha, 0) \partial^\alpha  \{\mathbf{I}-\mathbf{P}\} f^{\varepsilon} \right\| \\
&+\!\sum_{2 \leq |\alpha_1| \leq N}\!\left\|\partial^{\alpha_1}\nabla_x \phi^\varepsilon \right\|_{L^3}
\left\| \langle v \rangle^{\frac{\gamma}{2}}w_l(|\alpha|-1, 0) \partial^{\alpha-\alpha_1}  \{\mathbf{I}-\mathbf{P}\} f^{\varepsilon} \right\|_{L^6_xL^2_v}
\left\| \langle v \rangle^{\frac{\gamma}{2}}w_l(\alpha, 0) \partial^\alpha  \{\mathbf{I}-\mathbf{P}\} f^{\varepsilon} \right\| \\
\lesssim\;& \mathcal{E}^{1/2}_{N,l}(t) \mathcal{D}_{N,l}(t),
\end{align*}
 cf. the proof of Lemma \ref{softnonlinearterm1}. It is also worth noting that $N \geq 2$ is sufficient in this proof.

\subsubsection{Treatment of mismatch between derivatives and
weight functions}
\hspace*{\fill}

Noticing that higher order derivatives are associated with weaker velocity weights, we employ an interpolation inequality to balance the mismatch between derivatives and weight functions.

Let us take the electric force term $\partial^\alpha(\nabla_x \phi^\varepsilon \cdot \nabla_v f^\varepsilon_{\pm})$ (for $|\alpha|\leq N-1$) as an example to illustrate our novel processing trick. On one hand, if all derivatives act on $f_{\pm}^\varepsilon$, then we use the integration by parts with respect to velocity, namely,
\begin{align*}
\left|\left( \partial^{e_i} \phi^\varepsilon \partial^{\alpha}_{e_i} f_{\pm}^{\varepsilon},
e^{\pm \varepsilon \phi^\varepsilon } w^2_l(\alpha,0)\partial^\alpha f_{\pm}^\varepsilon\right)\right|
\lesssim \;
& \left|\left( \partial^{e_i} \phi^\varepsilon (\partial^{\alpha} f_{\pm}^{\varepsilon})^2,
e^{\pm \varepsilon \phi^\varepsilon }\partial_{e_i}w^2_l(\alpha,0)\right)\right|.
\end{align*}
In this case, the microscopic part of the term on the right-hand side can be controlled by $\|\nabla_x \phi^\varepsilon\|_{L^\infty}\mathcal{E}_{N,l}(t)$ or $\|\nabla_x \phi^\varepsilon\|_{L^\infty} \mathcal{E}^h_{N,l}(t)$, different with \cite{DL2013} where $\mathcal{E}^{1/2}_{N,l}(t)\mathcal{D}_{N,l}(t)$ was used to control the above estimate.
On the other hand, if not all derivatives act on $f^\varepsilon_{\pm}$, then we have
\begin{align*}
&\left( \partial^{\alpha_1+e_i} \phi^\varepsilon \partial^{\alpha-\alpha_1}_{e_i} f_{\pm}^{\varepsilon},
e^{\pm \varepsilon \phi^\varepsilon } w^2_l(\alpha,0)\partial^\alpha f_{\pm}^\varepsilon\right)\\
\lesssim
&\;\int_{\mathbb{R}^3}|\partial^{\alpha_1}\nabla_x\phi^\varepsilon | \underbrace{\left|\langle v\rangle^{-\frac{\gamma}{2}} w_l(\alpha,0)\partial^{\alpha-\alpha_1}_{e_i} f^{\varepsilon}\right|_{L^2} } \left|\langle v\rangle^{\frac{\gamma}{2}} w_l(\alpha,0)\partial^{\alpha} f^{\varepsilon}\right|_{L^2} \d x,
\end{align*}
for $1\leq |\alpha_1| \leq |\alpha|$. Here, the weight function $\langle v\rangle^{-\frac{\gamma}{2}} w_l(\alpha,0)$ before the term $\partial^{\alpha-\alpha_1}_{e_i} f^{\varepsilon}$ mismatches with the desired weight $w_l(\alpha-\alpha_1,e_i)$, since the inequality $\langle v \rangle ^{-\frac{\gamma}{2}}w_l(\alpha,0) \leq \langle v \rangle ^{\frac{\gamma}{2}}w_l(\alpha-\alpha_1,e_i)$ for $|\alpha_1| =1$ does not hold.
To control the underbraced term, inspired by \cite{CDL2024SIAM}, we employ the following interpolation inequality
\begin{align*}
\left| f \right|_{H^1} \lesssim \left| \langle v \rangle ^\ell f\right|_{H^s} + \left| \langle v \rangle ^{-\frac{\ell s}{1-s}} f\right|_{H^{1+s}}.
\end{align*}
By setting
\begin{align*}
\langle v \rangle ^\ell =\{w_l(|\alpha|-1,0)\}^{1-s}\{w_l(|\alpha|-1,1)\}^{-(1-s)},
\end{align*}
 one has
\begin{align*}
\langle v \rangle ^{-\frac{\ell s}{1-s}} =\{w_l(|\alpha|-1,0)\}^{-s}\{w_l(|\alpha|-1,1)\}^s.
\end{align*}
Substituting these two terms into the above inequality, the microscopic part of the underbraced term
can be controlled as
\begin{align*}
&\left|\langle v\rangle^{-\frac{\gamma}{2}} w_l(\alpha, 0) \partial^{\alpha-\alpha_1}_{e_i}\{\mathbf{I}-\mathbf{P}\} f^{\varepsilon}\right|_{L^2} \\
\lesssim\;&\left|\langle v\rangle^{-\frac{\gamma}{2}} w_l(\alpha, 0) \partial^{\alpha-\alpha_1}\{\mathbf{I}-\mathbf{P}\} f^{\varepsilon}\right|_{H^1}\\
\lesssim\;&\left|\langle v \rangle ^\ell\left(\langle v\rangle^{-\frac{\gamma}{2}} w_l(\alpha, 0) \partial^{\alpha-\alpha_1}\{\mathbf{I}-\mathbf{P}\} f^{\varepsilon}\right)\right|_{H^{s}}
+\left|\langle v \rangle ^{-\frac{\ell s}{1-s}} \left(\langle v\rangle^{-\frac{\gamma}{2}} w_l(\alpha, 0) \partial^{\alpha-\alpha_1}\{\mathbf{I}-\mathbf{P}\} f^{\varepsilon}\right)\right|_{H^{1+s}}
 \\
\lesssim \;& \left| w_l(|\alpha|-1,0) \partial^{\alpha-\alpha_1}\{\mathbf{I}-\mathbf{P}\} f^\varepsilon \right|_{H^s_{\gamma/2}}
+\left| w_l(|\alpha|-1,1) \partial^{\alpha-\alpha_1}\{\mathbf{I}-\mathbf{P}\} f^{\varepsilon} \right|_{H^{1+s}_{\gamma/2}},
\end{align*}
where in the third inequality above we used the decomposition
\begin{align*}
w_l(\alpha,0)\leq \{w_l(|\alpha|-1,0)\}^s \{w_l(|\alpha|-1,1)\}^{1-s}\langle v \rangle^{\gamma},
\end{align*}
 by further setting $m \geq -\frac{2\gamma}{s}$ in the definition of $w_l(\alpha, \beta)$. Then derivatives and weight functions can be matched well with above treatment, which further leads to the control of the electric force term $\partial^\alpha(\nabla_x \phi^\varepsilon \cdot \nabla_v f^\varepsilon_{\pm})$, cf. the proof of Lemma \ref{softnonlinearterm2} for more details.

In the estimate of $\frac{1}{\varepsilon}\Gamma(f^\varepsilon, f^\varepsilon)$, there is an additional two order velocity derivative, cf. Lemma \ref{Gamma1}. In fact, for $|\alpha^\prime|+|\beta^\prime| = |\alpha|+|\beta| $ and $|\beta^\prime| > |\beta|$, the inequality $w_{l}(\alpha,\beta) \leq w_l(\alpha^\prime,\beta^\prime)$ does not hold due to the structure of the weight function $w_l(\alpha,\beta)$. In other words, there
  are also mismatches between derivatives and weight functions. To briefly illustrate such mismatches, we use the following term as an example,
\begin{align*}
&\frac{1}{\varepsilon} \int_{\mathbb{R}^3}
\min \Big\{ \sum_{|\beta^{\prime}| \leq 2}\left|  \partial^{\alpha_1}_{\beta_1+\beta^{\prime}} \{\mathbf{I}-\mathbf{P}\}f^{\varepsilon}
\right|_{L^2} \left| w_l(\alpha,\beta) \partial^{\alpha_2}_{\beta_2} \{\mathbf{I}-\mathbf{P}\} f^{\varepsilon}\right|_{N^s_\gamma}, \nonumber \\
&\;\;\;\;\;\;\;\;\;\;\;\;\;\;\;\;\;\;\;\;\;\;\;\;\;\;\;\;\;\;\;\;\;\;\;\;\;\;\;\;\;\sum_{|\beta^{\prime}| \leq 2} \left|  \partial^{\alpha_1}_{\beta_1} \{\mathbf{I}-\mathbf{P}\}f^{\varepsilon} \right|_{L^2}
\left| w_l(\alpha,\beta) \partial^{\alpha_2}_{\beta_2+\beta^{\prime}} \{\mathbf{I}-\mathbf{P}\} f^{\varepsilon}\right|_{N^s_\gamma} \Big\}\nonumber \\
&\;\;\;\;\;\;\;\;\;\;\;\;\;\;\;\;\;\;\;\;\;\times \left| w_l(\alpha,\beta) \partial^{\alpha}_{\beta} \{\mathbf{I}-\mathbf{P}\} f^{\varepsilon}\right|_{N^s_\gamma} \d x
\end{align*}
for $|\alpha_1|+|\beta_1|=N$ and $|\alpha_2|+|\beta_2|=0$. For this situation, the only possible way is to select the second term inside the minimum function and take $L^2$--$L^\infty$--$L^2$ norm. To match the derivatives with corresponding weight functions, it is necessary to use the inequality $w_l(\alpha,\beta) \leq w_l(\alpha^\prime,\beta^\prime)$ for $1\leq |\alpha^\prime|\leq 2$ and $|\beta^\prime|\leq 2$. However, for $N=4$, $|\alpha_1|=|\alpha|=N$ and $|\beta_1|=|\beta|=0$, this inequality does not hold. This is the main reason why
$N \geq 5$ is required in the soft potential case. For $N\geq 5$, in order to make the inequality $w_{l}(\alpha,\beta) \leq w_l(\alpha^\prime,\beta^\prime)$  hold for $1 \leq |\alpha^\prime| \leq 2$ and $ |\beta^\prime|\leq 2$, it is necessary to require $m \geq -3\gamma$. For more details, please see the proof of Lemma \ref{soft Gamma}.

Combining the three ranges of $m$ mentioned above, it can be inferred that $m \geq -\frac{3\gamma}{s}$ is sufficient. Therefore we let $m=-\frac{3\gamma}{s}$
in the original definition \eqref{weight-original} and eventually design the weight function \eqref{weight function}.

\subsubsection{Treatment of the singularities induced by $\frac{1}{\varepsilon}$}
\hspace*{\fill}

The usage of the weight function $w_l(\alpha, \beta)$ also generates a severe singularity when we estimate the $N$-th order space derivative of the linearized Boltzmann operator $L$, that is
\begin{align*}
\frac{1}{\varepsilon^2} \left(L \partial^\alpha f^\varepsilon,
w^2_l(\alpha,0)\partial^\alpha f^\varepsilon\right)
\geq\;&\frac{\lambda}{\varepsilon^2} \left\| w_l(\alpha,0) \partial^\alpha f^\varepsilon\right\|^2_{N^s_\gamma}
-\frac{1}{\varepsilon^2}\left\|\partial^\alpha f^\varepsilon\right\|^2_{L^2(B_C)},
\quad |\alpha|=N,
\end{align*}
where the last term includes the singular macroscopic quantity $\frac{1}{\varepsilon^2}\left\|\partial^\alpha \mathbf{P} f^{\varepsilon}\right\|^2$ and is out of control.
We solve this difficulty by first multiplying the equation with $\varepsilon$ and then making $w_l(\alpha,0)$-weighted energy estimate, namely
\begin{align*}
&\frac{\varepsilon}{2}\frac{\d}{\d t}\sum_{\pm} \Big\|e^{\pm \frac{\varepsilon \phi^\varepsilon}{2}} w_l(\alpha,0) \partial^\alpha f_{\pm}^\varepsilon\Big\|^2 + \frac{1}{\varepsilon}\left( L \partial^\alpha f^{\varepsilon}, w_l^2(\alpha,0) \partial^\alpha f^{\varepsilon}\right)\\
=\;&\frac{\varepsilon}{2}\frac{\d}{\d t}\sum_{\pm} \Big\|e^{\pm \frac{\varepsilon \phi^\varepsilon}{2}} w_l(\alpha,0) \partial^\alpha f_{\pm}^\varepsilon\Big\|^2
+\frac{1}{\varepsilon}\left( L \partial^\alpha \{\mathbf{I}-\mathbf{P}\}f^{\varepsilon}, w_l^2(\alpha,0) \partial^\alpha \{\mathbf{I}-\mathbf{P}\}f^{\varepsilon}\right) \\
&+\frac{1}{\varepsilon}\left( L \partial^\alpha \{\mathbf{I}-\mathbf{P}\}f^{\varepsilon}, w_l^2(\alpha,0) \partial^\alpha \mathbf{P}f^{\varepsilon}\right)\\
\gtrsim\;&\frac{\varepsilon}{2}\frac{\d}{\d t}\sum_{\pm} \Big\|e^{\pm \frac{\varepsilon \phi^\varepsilon}{2}} w_l(\alpha,0) \partial^\alpha f_{\pm}^\varepsilon\Big\|^2
+ \frac{\lambda}{\varepsilon} \left\| w_l(\alpha,0) \partial^\alpha \{\mathbf{I}-\mathbf{P}\}f^\varepsilon\right\|^2_{N^s_\gamma}
-\frac{C}{\varepsilon}\left\|\partial^\alpha \{\mathbf{I}-\mathbf{P}\}f^\varepsilon\right\|^2_{L^2(B_C)}\\
&-\frac{1}{\varepsilon^2} \left\| \partial^\alpha \{\mathbf{I}-\mathbf{P}\}f^\varepsilon\right\|^2_{N^s_\gamma}
-\left\|\partial^\alpha \mathbf{P} f^{\varepsilon}\right\|^2,
\end{align*}
where the last three terms on the right-hand side can be controlled by the dissipation.

 Finally, in order to close the a priori estimates, we need sufficient time decay of $\|\mathbf{P}f^\varepsilon\|^2$ and $\|\nabla_x \mathbf{P}f^\varepsilon\|^2$. Due to the singularity $\frac{1}{\varepsilon}$ in front of the nonlinear terms, the method of semigroup estimate for linearized equation combined with the Duhamel principle would result in $O(\frac{1}{\varepsilon})$ singularity in the final time decay rate
 of the nonlinear problem. To overcome this difficulty, inspired by \cite{GW2012CPDE}, we employ the interpolation inequality and energy estimate in negative Sobolev space $\|\Lambda^{-\varrho} \big(f^\varepsilon, \nabla_x \phi^\varepsilon\big) \|$ to obtain the estimate
\begin{align*}
\mathcal{E}^k_{N}(t) \lesssim  (1+t)^{-(k+\varrho)}\sup_{0 \leq \tau \leq t}\widetilde{\mathcal{E}}_{N,l}(\tau)  \;\; \text{ for } k=0,1.
\end{align*}
 By combining the above strategies, we eventually close the global a priori estimates successfully.
\medskip

The rest of this paper is organized as follows. In Section \ref{Nonlinear Estimates}, we list basic lemmas concerning the properties of $L$ and $\Gamma$ in the framework of \cite{GS2011} and present the $w_l(\alpha,\beta)$-weighted estimates for all the nonlinear terms. In Section \ref{The a Priori Estimate}, we establish a series of a priori estimates by the weighted energy method. In Section \ref{Global Existence}, we first obtain the time decay rate and close the a priori estimates, and then we give the proof of Theorem \ref{mainth1} and Theorem \ref{mainth2}. In section \ref{Limit section}, based on the uniform energy estimate with respect to $\e \in (0,1]$ globally in time, we justify
the limit of the VPB system \eqref{rVPB} to the two-fluid incompressible NSFP system with Ohm's law \eqref{INSFP limit}, that is, give the proof of Theorem \ref{mainth3}.
\medskip

\section{Nonlinear Estimates}\label{Nonlinear Estimates}
\hspace*{\fill}

In this section, we list some basic lemmas concerning the properties of the linearized Boltzmann operator $L$ and the nonlinear collision operator $\Gamma$ in the functional framework of \cite{GS2011}, and also present the $w_l(\alpha,\beta)$-weighted estimates for all the nonlinear terms  in \eqref{rVPB}.

\subsection{Preliminary Lemmas}
\hspace*{\fill}

In this subsection, we list some basic results to be used in the weighted energy estimates of the nonlinear terms in \eqref{rVPB}.

Firstly, we introduce the bilinear operation $\mathcal{T}$ by
\begin{align}
\begin{split}\label{T define}
\mathcal{T}(g_1, g_2) :=\;& \mu^{-1/2} Q\left(\mu^{1/2}g_1, \mu^{1/2}g_2 \right) \\
=\;& \iint_{\mathbb{R}^3 \times \mathbb{S}^2} B(v-v_{*}, \sigma) \mu^{1/2}(v_*) [g_1(v_*^{\prime}) g_2(v^{\prime})-g_1(v_{*}) g_2(v)] \d v_{*} \d \sigma
\end{split}
\end{align}
for two scalar functions $g_1, g_2$. Hence, it follows that
\begin{align}
 L_{ \pm} f =\;&-\left\{2 \mathcal{T}\left(\mu^{1 / 2}, f_{ \pm}\right)+\mathcal{T}\left(f_{ \pm}+f_{\mp}, \mu^{1 / 2}\right)\right\}, \nonumber\\
 \Gamma_{ \pm}(f, g) =\;&\mathcal{T}\left(f_{ \pm}, g_{ \pm}\right)+\mathcal{T}\left(f_{\mp}, g_{ \pm}\right). \nonumber
\end{align}

Recall that $w \equiv w(v)=\langle v \rangle$ in \eqref{w define}.
The first lemma concerns the coercivity estimate on the linearized operator $L$, proved by \cite{GS2011}.
\begin{lemma}
Let $\gamma > -3$ and $0 < s < 1$.
\begin{itemize}
\setlength{\leftskip}{-5mm}
\item[(1)]  There holds
\begin{align}\label{L coercive1}
\langle Lf, f \rangle &\gtrsim \left| \{\mathbf{I}-\mathbf{P}\}f\right|^2_{N_\gamma^s}.
\end{align}
\item[(2)] There exists a constant $C\geq 0$ such that the uniform coercive lower bound estimate
\begin{equation}\label{L coercive2}
\langle w^{2\ell}Lf,f \rangle \gtrsim \left| w^{\ell}f\right|^2_{N_\gamma^s}-C\left| f \right|^2_{L^2(B_C)}
\end{equation}
holds for any $\ell \in \mathbb{R}$ (if $\ell = 0$, we may take $C = 0$).
\item[(3)]  For $|\beta| \geq 1$, we have
\begin{equation}\label{L coercive3}
\langle w^{2\ell}\partial_{\beta}Lf, \partial_{\beta}f \rangle \gtrsim \left| w^{\ell}\partial_{\beta}f \right|^2_{N_\gamma^s}
-\eta \sum_{|\beta^{\prime}| < |\beta|} \left| w^{\ell}\partial_{\beta^{\prime}}f\right|^2_{N_\gamma^s}-C_\eta\left| f \right|^2_{L^2(B_C)}
\end{equation}
for any $\ell \in \mathbb{R}$ and small $\eta>0$.
\end{itemize}
\end{lemma}
\medskip

The following two lemmas focus on  the estimates of the nonlinear collision operator $\Gamma$, which can be found in \cite{D2021-1}, \cite{DL2013} and \cite{Strain2012}.
\begin{lemma}\label{Gamma1}
Let $0 < s < 1$ and $\ell \geq 0$.
\begin{itemize}
\setlength{\leftskip}{-6mm}
\item[(1)]
For $\gamma+2s \geq 0$, there holds
\begin{align}\label{hard gamma1}
\!\!\!\left| \left\langle w^{2\ell} \partial_{\beta}^{\alpha} \Gamma_{\pm}(f, g), \partial_{\beta}^{\alpha} h_{\pm}\right\rangle \right|
\lesssim  \sum \left\{ \left|\partial_{\beta_1}^{\alpha_1} f\right|_{L^2} \left|w^{\ell} \partial_{\beta_2}^{\alpha_2} g\right|_{N^s_\gamma}
+\left|w^{\ell} \partial_{\beta_1}^{\alpha_1} f\right|_{L^2}\left|\partial_{\beta_2}^{\alpha_2} g\right|_{N^s_\gamma}\right\}
\left|w^{\ell} \partial_\beta^\alpha h\right|_{N^s_\gamma}.
\end{align}
For $-3 < \gamma <-2s$, there holds
\begin{align}\label{soft gamma1}
&\left| \left\langle w^{2\ell} \partial_{\beta}^{\alpha} \Gamma_{\pm}(f, g), \partial_{\beta}^{\alpha} h_{\pm}\right\rangle \right| \nonumber\\
\lesssim\; & \sum \left\{ \left|\partial_{\beta_1}^{\alpha_1} f\right|_{L^2}
\left|w^{\ell} \partial_{\beta_2}^{\alpha_2} g\right|_{N^s_\gamma}
+\left|w^{\ell} \partial_{\beta_1}^{\alpha_1} f\right|_{L^2} \left|\partial_{\beta_2}^{\alpha_2} g\right|_{N^s_\gamma}\right\}
\left|w^{\ell} \partial_\beta^\alpha h\right|_{N^s_\gamma} \\
&\!\!\!+\sum \min \Big\{\sum_{\left|\beta^{\prime}\right| \leq 2}\left|w^{-m} \partial_{\beta_1+\beta^{\prime}}^{\alpha_1} f \right|_{L^2} \left|w^{\ell} \partial_{\beta_2}^{\alpha_2} g\right|_{N^s_\gamma},
\left|w^{-m} \partial_{\beta_1}^{\alpha_1} f\right|_{L^2} \sum_{\left|\beta^{\prime}\right| \leq 2}\left|w^{\ell} \partial_{\beta_2+\beta^{\prime}}^{\alpha_2} g\right|_{N^s_\gamma}\Big\}\left|w^{\ell} \partial_\beta^\alpha h\right|_{N^s_\gamma} \nonumber
\end{align}
for any $m \geq 0$, where the summation $\sum$ is taken over $\alpha_1+\alpha_2=\alpha$ and $\beta_1+\beta_2 \leq \beta$.
\item[(2)]
For $\gamma+2s \geq 0$, there holds
\begin{equation}\label{hard gamma2}
\left| w^{-\ell}  \Gamma (f, g) \right|_{L^2} \lesssim \left| w^{-\ell} f\right|_{L^2_{\gamma/2+s}}
\left| g\right|_{H^i_{\gamma+2s}}.
\end{equation}
For $-3 < \gamma < -2s$, there holds
\begin{equation}\label{soft gamma2}
\left| w^{\ell}  \Gamma (f, g) \right|_{L^2} \lesssim \min\left\{ \left| w^{\ell}f \right|_{H^2_{\gamma/2+s}}
\left| w^{\ell} g \right|_{H^i_{\gamma/2+s}}, \left| w^{\ell}f \right|_{L^2_{\gamma/2+s}}
\left| w^{\ell} g \right|_{H^{i+2}_{\gamma/2+s}} \right\}.
\end{equation}
Here $i=1$ if $s \in (0, 1/2)$, and $i=2$ if $s \in [1/2, 1)$.
\end{itemize}
\end{lemma}
\medskip

\begin{lemma}\label{Gamma2}
Let $\zeta(v)$ be a function satisfying
\begin{equation}\label{zeta}
|\zeta(v)| \approx e^{-c_0|v|^2}
\end{equation}
for some $c_0>0$. Let $g_1$ and $g_2$ be two scalar functions. For any $\ell \in \mathbb{R}$ and $m \geq 0$, there holds
\begin{align}
&\left| \left\langle w^{2\ell}\mathcal{T}(g_1,g_2),\zeta \right\rangle \right|
\lesssim |g_1|_{L^2_{-m}} |g_2|_{L^2_{-m}}, \label{Gamma zeta1}\\
&\left| \left\langle  w^{2\ell}\mathcal{T}(\zeta,g_1),g_2 \right\rangle \right|
\lesssim |w^\ell g_1|_{N^s_\gamma}|w^\ell g_2|_{N^s_\gamma}, \label{Gamma zeta2}\\
&\left| \left\langle  w^{2\ell}\mathcal{T}(g_1,\zeta),g_2 \right\rangle\right|
\lesssim |w^\ell g_1|_{N^s_\gamma}|w^\ell g_2|_{N^s_\gamma},\label{Gamma zeta3}
\end{align}
where $\mathcal{T}$ is defined in \eqref{T define}.
\end{lemma}
\medskip

In what follows, we collect some basic inequalities to be used throughout this paper. Firstly, the following Sobolev interpolation inequalities have been shown in \cite{GW2012CPDE}.

\begin{lemma}\label{sobolev interpolation}
 Let $2 \leq p<\infty$ and $k, \ell, m \in \mathbb{R}$. Then for any $f \in C_0^\infty(\mathbb{R}^3)$, we have
\begin{align}\label{sobolev interpolation1}
\left\|\nabla^k f\right\|_{L^p} \lesssim\left\|\nabla^{\ell} f\right\|^\theta\left\|\nabla^m f\right\|^{1-\theta},
\end{align}
where $0 \leq \theta \leq 1$ and $\ell$ satisfies
\begin{align}\nonumber
\frac{1}{p}-\frac{k}{3}=\left(\frac{1}{2}-\frac{\ell}{3}\right) \theta+\left(\frac{1}{2}-\frac{m}{3}\right)(1-\theta).
\end{align}
For the case $p=+\infty$, we have
\begin{align}\label{sobolev interpolation2}
\left\|\nabla^k f\right\|_{L^{\infty}} \lesssim\left\|\nabla^{\ell} f\right\|^\theta\left\|\nabla^m f\right\|^{1-\theta},
\end{align}
where $\ell \leq k+1$, $m \geq k+2$, $0 \leq \theta \leq 1$ and $\ell$ satisfies
\begin{align}\nonumber
-\frac{k}{3}=\left(\frac{1}{2}-\frac{\ell}{3}\right) \theta+\left(\frac{1}{2}-\frac{m}{3}\right)(1-\theta).
\end{align}
\end{lemma}
\medskip

If $\varrho \in (0,3)$, then $\Lambda^{-\varrho}$ is the Riesz potential operator. The
Hardy--Littlewood--Sobolev theorem implies the following inequality for the Riesz potential operator $\Lambda^{-\varrho}$.

\begin{lemma}\label{negative embedding theorem}
Let $0<\varrho<3$, $1< p < q < \infty$, $1/q + \varrho/3=1/p$, there holds
\begin{align}\label{negative embed 1}
\left\|\Lambda^{-\varrho} f\right\|_{L^q} \lesssim\|f\|_{L^p}.
\end{align}
\end{lemma}
\medskip

Based on Lemma \ref{sobolev interpolation}, we have the following lemma to be used frequently in later analysis.

\begin{lemma}\label{negative embedding theorem2}
Let $0<\varrho<3/2$. Then we have
\begin{align}\label{negative embed 2}
\|f\|_{L^{\frac{12}{3+2 \varrho}}} \lesssim\left\|\Lambda^{\frac{3}{4}-\frac{\varrho}{2}} f\right\|, \quad\|f\|_{L^{\frac{3}{\varrho}}} \lesssim \left\|\Lambda^{\frac{3}{2}-\varrho} f\right\|.
\end{align}
\end{lemma}
\medskip

In many places, we will use the following Minkowski inequality to interchange the order of a multiple integral, cf. \cite{GW2012CPDE}.
\begin{lemma}\label{minkowski theorem}
For $1 \leq p \leq q \leq \infty$, there holds
\begin{align}\label{minkowski}
\|f\|_{L_x^q L_v^p} \leq\|f\|_{L_v^p L_x^q}.
\end{align}
\end{lemma}
\medskip

The following lemma concerns the equivalence of weight function and differential operator up to commutation, whose proof can be found in \cite{AMSY2023} and \cite{HMUY2008}.
\begin{lemma}\label{equivalent norm}
Let $1 \leq p \leq \infty$ and $\ell, \theta \in \mathbb{R}$. Then there exists a generic constant $C$ independent of $f$ such that
$$
\frac{1}{C}\left|\langle v\rangle^{\ell}\left\langle D_v\right\rangle^\theta f\right|_{L^p} \leq\left|\left\langle D_v\right\rangle^\theta\langle v\rangle^{\ell} f\right|_{L^p} \leq C\left|\langle v\rangle^{\ell}\left\langle D_v\right\rangle^\theta f\right|_{L^p},
$$
that is, these two norms are equivalent. Here $\left\langle D_v\right\rangle$ denotes the Fourier multiplier with symbol $\langle\xi\rangle$.
\end{lemma}
\medskip

Finally, we give the proof of an interpolation formula stated in \cite{CDL2024SIAM}.
\begin{lemma}\label{interpolation1}
Let $0 < s < 1$. For any $\ell \in \mathbb{R}$, there holds
\begin{equation}\label{interpolation}
| f |_{H^1_v} \lesssim \big| \langle v \rangle ^\ell f\big|_{H^s_v} + \big| \langle v \rangle ^{-\frac{\ell s}{1-s}} f\big|_{H^{1+s}_v}.
\end{equation}
\end{lemma}

\begin{proof}
By Young's inequality, one has
$$
\langle \xi \rangle =\big(\langle v \rangle^{\ell} \langle \xi \rangle^{s}\big)^s \big(\langle v \rangle^{-\frac{\ell s}{1-s}} \langle \xi \rangle^{1+s}\big)^{1-s} \lesssim \langle v \rangle^{\ell} \langle \xi \rangle^{s}+ \langle v \rangle^{-\frac{\ell s}{1-s}} \langle \xi \rangle^{1+s}.
$$
Define
$$
a(v, \xi)= m(v, \xi)= \langle v \rangle^{\ell} \langle \xi \rangle^{s}+ \langle v \rangle^{-\frac{\ell s}{1-s}}\langle \xi \rangle^{1+s}
+K\big( \langle v \rangle^{\ell}+ \langle v \rangle^{-\frac{\ell s}{1-s}}\big),
$$
where the constant $K$ is to be determined later. Obviously, $m$ is a $\Gamma$-admissible weight.
By using Lemma 2.1 in \cite{Deng2020}, there exists $K>1$ such that $a^w: H\big(\langle v \rangle^{\ell} \langle \xi \rangle^{s}+ \langle v \rangle^{-\frac{\ell s}{1-s}}\langle \xi \rangle^{1+s}\big) \to L^2$ is invertible. Hence, it follows from Lemma 2.3 and Corollary 2.5 in \cite{Deng2020} that
\begin{align*}
|\langle D_v \rangle f |_{L^2_v} &\lesssim\; |a^w(v,D_v)f|_{L^2_v} \\
&\lesssim\;\big|\langle v \rangle^\ell \langle D_v \rangle^s f\big|_{L^2_v}
+\big|\langle v \rangle^{-\frac{\ell s}{1-s}} \langle D_v \rangle^{1+s} f\big|_{L^2_v}
+\big|\langle v \rangle^\ell f\big|_{L^2_v}
+\big|\langle v \rangle^{-\frac{\ell s}{1-s}}  f\big|_{L^2_v}\\
&\lesssim\; \big| \langle v \rangle ^\ell f\big|_{H^s_v} + \big| \langle v \rangle ^{-\frac{\ell s}{1-s}} f\big|_{H^{1+s}_v}.
\end{align*}
This completes the proof.
\end{proof}
\medskip

\subsection{Estimates on the Nonlinear Terms}
\hspace*{\fill}

The goal of this subsection is to make the weighted energy estimates for all the nonlinear terms in \eqref{rVPB}.

In this subsection, we always assume
\begin{align*}
&N \geq 2,\; l \geq N, \qquad\quad\;\, \text{~if~} \gamma+2s \geq 0;\\
&N \geq 5,\; l \geq -\frac{3\gamma}{s} N, \quad \,~~ \text{~~if~} -3 < \gamma < -2s.
\end{align*}
Recall the definition of $w_l(\alpha,\beta)$ in \eqref{weight function}, the definitions of $\mathcal{E}_{N,l}(t)$ and $\mathcal{D}_{N,l}(t)$ given in \eqref{energy functional} and \eqref{dissipation functional}, respectively. Besides, the decomposition
\begin{equation}\label{gamma decomposition}
\Gamma(f^\varepsilon, f^\varepsilon)=  \Gamma(\mathbf{P} f^\varepsilon, \mathbf{P} f^\varepsilon)+\Gamma(\mathbf{P} f^\varepsilon,\{\mathbf{I}-\mathbf{P}\} f^\varepsilon)+\Gamma(\{\mathbf{I}-\mathbf{P}\} f^\varepsilon, \mathbf{P} f^\varepsilon)
+\Gamma(\{\mathbf{I}-\mathbf{P}\} f^\varepsilon,\{\mathbf{I}-\mathbf{P}\} f^\varepsilon)
\end{equation}
 will be frequently used in later context.

\medskip

The following two lemmas are about the estimates on the nonlinear collision term $\Gamma(f^\varepsilon,f^\varepsilon)$.

\begin{lemma}\label{soft Gamma}
Let $-3 < \gamma <-2s $. Assume that $\mathcal{E}^{1/2}_{N,l}(t) \leq \delta$ for some positive constant $\delta >0$.
\begin{itemize}
\setlength{\leftskip}{-6mm}
\item[(1)] For $| \alpha | \leq N$, there holds
\begin{align}\label{soft gamma3}
\bigg|\frac{1}{\varepsilon}\sum_{\pm}\left(\partial^\alpha \Gamma_{\pm}( f^{\varepsilon}, f^{\varepsilon} ),
e^{\pm \varepsilon \phi^{\varepsilon}} \partial^\alpha f_{\pm}^{\varepsilon}\right)\bigg|
\lesssim \left\{\mathcal{E}^{1/2}_{N,l}(t)+\mathcal{E}_{N,l}(t)\right\} \mathcal{D}_{N,l}(t).
\end{align}
\item[(2)] For $| \alpha | \leq N-1$, there holds
\begin{align}\label{soft gamma4}
\bigg|\frac{1}{\varepsilon} \left( \partial^\alpha \Gamma_{\pm}( f^{\varepsilon}, f^{\varepsilon} ),
e^{\pm \varepsilon \phi^{\varepsilon}}
w^2_{l}(\alpha, 0) \partial^\alpha \{\mathbf{I}_{\pm}-\mathbf{P}_{\pm}\} f^{\varepsilon}\right)\bigg|
\lesssim \mathcal{E}^{1/2}_{N,l}(t) \mathcal{D}_{N,l}(t).
\end{align}
\item[(3)] For $|\alpha|=N$, there holds
\begin{align}\label{soft gamma5}
\Big|\left( \partial^\alpha \Gamma_{\pm}( f^{\varepsilon}, f^{\varepsilon} ), e^{\pm \varepsilon \phi^{\varepsilon}}
w^2_{l}(\alpha, 0) \partial^\alpha  f_{\pm}^{\varepsilon}\right)\Big|
\lesssim \mathcal{E}^{1/2}_{N,l}(t) \mathcal{D}_{N,l}(t).
\end{align}
\item[(4)] For $| \alpha |+ | \beta | \leq N$, $| \beta | \geq 1$ and $| \alpha | \leq N-1$, there holds
\begin{align}\label{soft gamma6}
\bigg| \frac{1}{\varepsilon}\left( \partial^\alpha_\beta \Gamma_{\pm}( f^{\varepsilon}, f^{\varepsilon} ),
e^{\pm \varepsilon \phi^{\varepsilon}} w^2_{l}(\alpha, \beta) \partial^\alpha_\beta \{\mathbf{I}_{\pm}-\mathbf{P}_{\pm}\} f^{\varepsilon}\right) \bigg|
\lesssim \mathcal{E}^{1/2}_{N,l}(t) \mathcal{D}_{N,l}(t).
\end{align}
\end{itemize}
\end{lemma}

\begin{proof}
For brevity, we only give the proof of \eqref{soft gamma6}. The proof of \eqref{soft gamma4} and \eqref{soft gamma5} can be obtained similarly.
Besides, by the collision invariant property, we have
\begin{align}
\begin{split}\label{soft gamma3 decomposition}
&\frac{1}{\varepsilon}\sum_{\pm}\left( \partial^\alpha \Gamma_{\pm}( f^{\varepsilon}, f^{\varepsilon} ), e^{\pm \varepsilon \phi^{\varepsilon}} \partial^\alpha f_{\pm}^{\varepsilon}\right) \\
=\;&\frac{1}{\varepsilon}\sum_{\pm}\left( \partial^\alpha \Gamma_{\pm}( f^{\varepsilon}, f^{\varepsilon} ),
(e^{\pm \varepsilon \phi^{\varepsilon}}-1) \partial^\alpha f_{\pm}^{\varepsilon}\right)
+\frac{1}{\varepsilon}\sum_{\pm}\left( \partial^\alpha \Gamma_{\pm}( f^{\varepsilon}, f^{\varepsilon} ), \partial^\alpha f_{\pm}^{\varepsilon}\right) \\
=\;&\frac{1}{\varepsilon}\sum_{\pm}\left( \partial^\alpha \Gamma_{\pm}( f^{\varepsilon}, f^{\varepsilon} ),
(e^{\pm \varepsilon \phi^{\varepsilon}}-1) \partial^\alpha f_{\pm}^{\varepsilon}\right)
+\frac{1}{\varepsilon}\left( \partial^\alpha \Gamma( f^{\varepsilon}, f^{\varepsilon} ), \partial^\alpha \{\mathbf{I}-\mathbf{P}\} f^{\varepsilon}\right).
\end{split}
\end{align}
Noticing that the inequality
\begin{equation}\label{phi estimate1}
\left|e^{\pm \varepsilon \phi^{\varepsilon}}-1\right| \lesssim \varepsilon \left\| \phi^{\varepsilon} \right\|_{L^\infty} \lesssim \varepsilon \left\| \nabla_x \phi^{\varepsilon} \right\|_{H^1}
\end{equation}
 can absorb the singularity $\frac{1}{\varepsilon}$ in the first term on the right-hand side of \eqref{soft gamma3 decomposition} and $\| \nabla_x \phi^{\varepsilon} \|_{H^1}$ can be contained in $\mathcal{E}^{1/2}_{N,l}(t)$ or $\mathcal{D}^{1/2}_{N,l}(t)$,
 then the estimate \eqref{soft gamma3} can also be proved similar to \eqref{soft gamma6}.

To prove \eqref{soft gamma6}, we set
$$
I_1\equiv  \frac{1}{\varepsilon}\left( \partial^\alpha_\beta \Gamma_{\pm}( f^{\varepsilon}, f^{\varepsilon} ),
e^{\pm \varepsilon \phi^{\varepsilon}}
w^2_{l}(\alpha, \beta) \partial^\alpha_\beta \{\mathbf{I}_{\pm}-\mathbf{P}_{\pm}\} f^{\varepsilon}\right) :=I_{1,1}+ I_{1,2}+ I_{1,3}+ I_{1,4},
$$
where $I_{1,1}$, $I_{1,2}$, $I_{1,3}$ and $I_{1,4}$ are the terms corresponding to the decomposition \eqref{gamma decomposition}, respectively.
In the following, we estimate these four terms one by one.

Firstly, for $I_{1,1}$, recalling \eqref{Pf define} and applying \eqref{soft gamma1}, one has \begin{align}
\begin{split}\nonumber
I_{1,1} \lesssim\;& \frac{1}{\varepsilon}\sum_{|\alpha_1| \leq |\alpha|}
\int_{\mathbb{R}^3} \left| \partial^{\alpha_1} (a^{\varepsilon}_{\pm},b^{\varepsilon},c^{\varepsilon})\right| \left| \partial^{\alpha-\alpha_1} (a^{\varepsilon}_{\pm},b^{\varepsilon},c^{\varepsilon})\right|
\left| w_{l}(\alpha, \beta) \partial^\alpha_\beta \{\mathbf{I}-\mathbf{P}\} f^{\varepsilon}\right|_{N^s_\gamma} \d x \\
\lesssim \;& \frac{1}{\varepsilon}\left\|(a^{\varepsilon}_{\pm},b^{\varepsilon},c^{\varepsilon})\right\|_{L^\infty}
\left\|\partial^{\alpha} (a^{\varepsilon}_{\pm},b^{\varepsilon},c^{\varepsilon})\right\|_{L^2}
\left\| w_{l}(\alpha, \beta) \partial^\alpha_\beta \{\mathbf{I}-\mathbf{P}\} f^{\varepsilon}\right\|_{N^s_\gamma}\\
&+\frac{1}{\varepsilon} \sum_{1 \leq |\alpha_1| \leq N-1} \left\| \partial^{\alpha_1} (a^{\varepsilon}_{\pm},b^{\varepsilon},c^{\varepsilon})\right\|_{L^6} \left\| \partial^{\alpha-\alpha_1} (a^{\varepsilon}_{\pm},b^{\varepsilon},c^{\varepsilon})\right\|_{L^3}
\left\| w_{l}(\alpha, \beta) \partial^\alpha_\beta \{\mathbf{I}-\mathbf{P}\} f^{\varepsilon}\right\|_{N^s_\gamma}\\
\lesssim \;& \mathcal{E}^{1/2}_{N,l}(t) \mathcal{D}_{N,l}(t),
\end{split}
\end{align}
where we made use of the fact $|e^{\pm \varepsilon \phi^\varepsilon}| \approx 1$ and the Sobolev inequalities
\begin{align}
\label{Sobolev-ineq}
\left\|f\right\|_{L^3_x}\lesssim \left\|f\right\|_{H^1_x}, \quad \left\|f\right\|_{L^6_x}\lesssim \left\| \nabla_x f \right\|_{L^2_x}, \quad \left\|f\right\|_{L^\infty_x}\lesssim \left\| \nabla_x f\right\|_{H^1_x}.
\end{align}

Secondly, for $I_{1,2}$, by using \eqref{soft gamma1} and selecting the first one inside the minimum function in \eqref{soft gamma1}, we have
\begin{align}
I_{1,2}\lesssim\;&\frac{1}{\varepsilon}\!\sum_{\substack{{|\alpha_1|\leq |\alpha|} \\ {|\beta_2|  \leq |\beta|}}}
\int_{\mathbb{R}^3}
\left| \partial^{\alpha_1} (a^{\varepsilon}_{\pm},b^{\varepsilon},c^{\varepsilon})\right| \left| w_{l}(\alpha,\beta) \partial^{\alpha-\alpha_1}_{\beta_2} \{\mathbf{I}-\mathbf{P}\} f^{\varepsilon}\right|_{N^s_\gamma} \left| w_{l}(\alpha,\beta) \partial^{\alpha}_{\beta} \{\mathbf{I}-\mathbf{P}\} f^{\varepsilon}\right|_{N^s_\gamma} \d x \nonumber \\
\lesssim\;&\frac{1}{\varepsilon}\sum_{|\beta_2|  \leq |\beta|} \left\| (a^{\varepsilon}_{\pm},b^{\varepsilon},c^{\varepsilon})\right\|_{L^\infty}
\left\|w_{l}(\alpha,\beta) \partial^{\alpha}_{\beta_2} \{\mathbf{I}-\mathbf{P}\} f^{\varepsilon}\right\|_{N^s_\gamma}
\left\|w_{l}(\alpha,\beta) \partial^{\alpha}_\beta \{\mathbf{I}-\mathbf{P}\} f^{\varepsilon}\right\|_{N^s_\gamma} \nonumber \\
+&\frac{1}{\varepsilon}\!\sum_{\substack{{1 \leq |\alpha_1| \leq N-1} \\ {|\beta_2|  \leq |\beta|}}}\!\!\!\!
\left\| \partial^{\alpha_1} (a^{\varepsilon}_{\pm},b^{\varepsilon},c^{\varepsilon})\right\|_{L^6}
\left\|w_{l}(\alpha,\beta) \partial^{\alpha-\alpha_1}_{\beta_2} \{\mathbf{I}-\mathbf{P}\} f^{\varepsilon}\right\|_{L^3_x N^s_\gamma}
\left\|w_{l}(\alpha,\beta) \partial^{\alpha}_\beta \{\mathbf{I}-\mathbf{P}\} f^{\varepsilon}\right\|_{N^s_\gamma} \nonumber\\
\lesssim\;&\mathcal{E}^{1/2}_{N,l}(t) \mathcal{D}_{N,l}(t), \nonumber
\end{align}
where we have used the Sobolev inequalities (\ref{Sobolev-ineq}), the Minkowski inequality \eqref{minkowski}, $w_{l}(\alpha,\beta)\geq 1$ and the fact $w_l(\alpha,\beta) \leq w_l(\alpha-\alpha_1+\alpha^\prime,\beta_2)$ for $|\alpha-\alpha_1+\alpha^\prime| \leq |\alpha|$, $ |\beta_2| \leq |\beta|$, $|\alpha_1| \geq 1$ and $|\alpha^\prime| \leq 1$. Here and hereafter, the spatial derivative $\alpha^\prime$ is generated by the Sobolev inequalities \eqref{Sobolev-ineq}.
In the same way, by selecting the second one inside the minimum function in \eqref{soft gamma1}, $I_{1,3}$ has the same upper bound as $I_{1,2}$.

Finally, applying \eqref{soft gamma1} again, $I_{1,4}$ is decomposed as
\begin{align*}
&\frac{1}{\varepsilon}\sum_{\substack{{\alpha_1+\alpha_2=\alpha}\\{\beta_1+\beta_2  \leq \beta}}}\int_{\mathbb{R}^3}
\left| \partial^{\alpha_1}_{\beta_1} \{\mathbf{I}-\mathbf{P}\}f^{\varepsilon} \right|_{L^2} \left| w_{l}(\alpha,\beta) \partial^{\alpha_2}_{\beta_2} \{\mathbf{I}-\mathbf{P}\} f^{\varepsilon}\right|_{N^s_\gamma} \left| w_{l}(\alpha,\beta) \partial^{\alpha}_{\beta} \{\mathbf{I}-\mathbf{P}\} f^{\varepsilon}\right|_{N^s_\gamma} \d x  \\
+&\frac{1}{\varepsilon} \sum_{\substack{{\alpha_1+\alpha_2=\alpha}\\{\beta_1+\beta_2  \leq \beta}}} \int_{\mathbb{R}^3}
\left| w_{l}(\alpha,\beta) \partial^{\alpha_1}_{\beta_1} \{\mathbf{I}-\mathbf{P}\}f^{\varepsilon} \right|_{L^2}
\left| \partial^{\alpha_2}_{\beta_2} \{\mathbf{I}-\mathbf{P}\} f^{\varepsilon}\right|_{N^s_\gamma} \left| w_{l}(\alpha,\beta) \partial^{\alpha}_{\beta} \{\mathbf{I}-\mathbf{P}\} f^{\varepsilon}\right|_{N^s_\gamma} \d x \\
+&\frac{1}{\varepsilon} \sum_{\substack{{\alpha_1+\alpha_2=\alpha}\\{\beta_1+\beta_2  \leq \beta}}}\int_{\mathbb{R}^3}
\min \Big\{ \sum_{|\beta^{\prime}| \leq 2}\left|  \partial^{\alpha_1}_{\beta_1+\beta^{\prime}} \{\mathbf{I}-\mathbf{P}\}f^{\varepsilon}
\right|_{L^2} \left| w_l(\alpha,\beta) \partial^{\alpha_2}_{\beta_2} \{\mathbf{I}-\mathbf{P}\} f^{\varepsilon}\right|_{N^s_\gamma},  \\
&\;\;\;\;\;\;\;\;\;\;\;\;\;\;\;\;\;\;\;\;\;\;\;\;\;\;\;\;\;\;\;\;\sum_{|\beta^{\prime}| \leq 2} \left|  \partial^{\alpha_1}_{\beta_1} \{\mathbf{I}-\mathbf{P}\}f^{\varepsilon} \right|_{L^2}
\left| w_l(\alpha,\beta) \partial^{\alpha_2}_{\beta_2+\beta^{\prime}} \{\mathbf{I}-\mathbf{P}\} f^{\varepsilon}\right|_{N^s_\gamma} \Big\} \\
&\;\;\;\;\;\;\;\;\;\;\;\;\;\;\;\;\;\;\;\;\;\;\;\;\times \left| w_l(\alpha, \beta) \partial^\alpha_\beta \{\mathbf{I}-\mathbf{P}\} f^\varepsilon \right|_{N^s_\gamma} \d x \\
:=& \;I_{1,4}^\prime + I_{1,4}^{\prime\prime} +  I_{1,4}^{\prime\prime\prime}.
\end{align*}
Then we estimate $I_{1,4}^\prime, I_{1,4}^{\prime\prime}$ and $I_{1,4}^{\prime\prime\prime}$ term by term.

Firstly, for $I_{1,4}^\prime$, when $|\alpha_1|+|\beta_1|=0$, we take $L^{\infty}$--$L^2$--$L^2$ for the space integral. When $|\alpha_1|+|\beta_1|=1$, we take $L^6$--$L^3$--$L^2$. When $2 \leq |\alpha_1|+|\beta_1| \leq N$, we take $L^2$--$L^{\infty}$--$L^2$. Thus, $I_{1,4}^\prime \lesssim \mathcal{E}^{1/2}_{N,l}(t) \mathcal{D}_{N,l}(t)$. Secondly, in the same way we can get the estimate $I_{1,4}^{\prime\prime} \lesssim \mathcal{E}^{1/2}_{N,l}(t) \mathcal{D}_{N,l}(t)$.
Finally, $I_{1,4}^{\prime\prime\prime}$ should be estimated very carefully
due to the change of the weight function $w_{l}(\alpha,\beta)$. For this, we split $I_{1,4}^{\prime\prime\prime}$ into the following four cases.

{\emph {Case 1.}}
$|\alpha_1|+|\beta_1|=0$ $\mathrm{or}$ $|\alpha_1|+|\beta_1|=1$. \; By selecting
the first term inside the minimum function in $I_{1,4}^{\prime\prime\prime}$ and taking $L^{\infty}$--$L^2$--$L^2$, one has
\begin{align}
I_{1,4}^{\prime\prime\prime}\lesssim\;& \frac{1}{\varepsilon} \sum_{|\beta^{\prime}| \leq 2}\left\| \partial^{\alpha_1}_{
\beta_1+\beta^{\prime}} \{\mathbf{I}-\mathbf{P}\}f^{\varepsilon}\right\|_{L^\infty_x L^2_v}
\left\| w_l(\alpha,\beta) \partial^{\alpha_2}_{\beta_2} \{\mathbf{I}-\mathbf{P}\} f^{\varepsilon}\right\|_{N^s_\gamma}
\left\| w_l(\alpha,\beta) \partial^{\alpha}_{\beta} \{\mathbf{I}-\mathbf{P}\} f^{\varepsilon}\right\|_{N^s_\gamma} \nonumber \\
\lesssim \;&\mathcal{E}^{1/2}_{N,l}(t) \mathcal{D}_{N,l}(t), \nonumber
\end{align}
where we have used the Sobolev inequalities \eqref{Sobolev-ineq}, the Minkowski inequality \eqref{minkowski}, $w_l(\alpha, \beta) \geq 1$ and the fact $w_l(\alpha,\beta) \leq w_l(\alpha_2,\beta_2)$ for $|\alpha_2| \leq |\alpha|$ and $ |\beta_2| \leq |\beta|$.

{\emph {Case 2.}}
$|\alpha_1|+|\beta_1|=2$ $\mathrm{or}$ $|\alpha_1|+|\beta_1|=3$. \; By selecting the first term inside the minimum function in $I_{1,4}^{\prime\prime\prime}$ and taking $L^2$--$L^{\infty}$--$L^2$, we obtain
\begin{align}
I_{1,4}^{\prime\prime\prime}\lesssim\;&\frac{1}{\varepsilon} \sum_{|\beta^{\prime}| \leq 2} \left\| \partial^{\alpha_1}_{
\beta_1+\beta^{\prime}} \{\mathbf{I}-\mathbf{P}\}f^{\varepsilon}\right\|
\left\|  w_l(\alpha,\beta) \partial^{\alpha_2}_{\beta_2} \{\mathbf{I}-\mathbf{P}\} f^{\varepsilon}\right\|_{L^\infty_x N^s_\gamma}
\left\| w_l(\alpha,\beta) \partial^{\alpha}_{\beta} \{\mathbf{I}-\mathbf{P}\} f^{\varepsilon}\right\|_{N^s_\gamma} \nonumber \\
\lesssim\;& \mathcal{E}^{1/2}_{N,l}(t) \mathcal{D}_{N,l}(t), \nonumber
\end{align}
where we have used the Sobolev inequalities \eqref{Sobolev-ineq}, the Minkowski inequality \eqref{minkowski}, $w_l(\alpha, \beta) \geq 1$  and the fact $w_l(\alpha,\beta) \leq w_l(\alpha_2+\alpha^\prime,\beta_2)$ for $|\alpha_2+\alpha^{\prime}|+ |\beta_2|  \leq |\alpha| + |\beta|$, $ |\beta_2| \leq |\beta|$ and $1 \leq |\alpha^\prime| \leq 2$.

{\emph {Case 3.}} $4 \leq |\alpha_1|+|\beta_1| \leq N-1$. \; By choosing the second term inside the minimum function in $I_{1,4}^{\prime\prime\prime}$ and taking $L^6$--$L^3$--$L^2$, we have
$$
I_{1,4}^{\prime\prime\prime}\lesssim\;\frac{1}{\varepsilon} \left\| \partial^{\alpha_1}_{\beta_1} \{\mathbf{I}-\mathbf{P}\}f^{\varepsilon}\right\|_{L^6_x L^2_v}
\sum_{|\beta^{\prime}| \leq 2}\left\| w_l(\alpha,\beta) \partial^{\alpha_2}_{\beta_2+\beta^{\prime}} \{\mathbf{I}-\mathbf{P}\} f^{\varepsilon}\right\|_{L^3_x N^s_\gamma}
\left\| w_l(\alpha,\beta) \partial^{\alpha}_{\beta} \{\mathbf{I}-\mathbf{P}\} f^{\varepsilon}\right\|_{N^s_\gamma}. \nonumber
$$
Owing to $4 \leq |\alpha_1|+|\beta_1| \leq N-1$, one has $|\alpha_2|+|\beta_2| \leq |\alpha|+|\beta|-4$. For $| \alpha^\prime | \leq 1$ and
$| \beta^\prime | \leq 2$, we have $|\alpha_2 + \alpha^\prime |+|\beta_2 + \beta^\prime | \leq |\alpha|+|\beta|-1$, and then
$w_l(\alpha,\beta) \leq w_l(\alpha_2 + \alpha^\prime, \beta_2 + \beta^\prime)$. Thus,
there holds $I_{1,4}^{\prime\prime\prime} \lesssim \mathcal{E}^{1/2}_{N,l}(t) \mathcal{D}_{N,l}(t)$.
In fact, in the case when $|\alpha^\prime|+|\beta^\prime| = |\alpha|+|\beta| $ and $|\beta^\prime| > |\beta|$, $w_{l}(\alpha,\beta) \leq w_l(\alpha^\prime,\beta^\prime)$ does not necessarily hold. This is the main reason for the requirement $N \geq 5$ in the soft potential case.

{\emph {Case 4.}} $|\alpha_1|+|\beta_1| = N$. \; By choosing the second term inside the minimum function in $I_{1,4}^{\prime\prime\prime}$ and taking $L^2$--$L^{\infty}$--$L^2$, we have
$$
I_{1,4}^{\prime\prime\prime}\lesssim\; \frac{1}{\varepsilon} \left\| \partial^{\alpha}_{\beta} \{\mathbf{I}-\mathbf{P}\}f^{\varepsilon}\right\|
\sum_{|\beta^{\prime}| \leq 2}\left\| w_l(\alpha,\beta) \partial_{\beta^{\prime}} \{\mathbf{I}-\mathbf{P}\} f^{\varepsilon}\right\|_{L^\infty_x N^s_\gamma}
\left\| w_l(\alpha,\beta) \partial^{\alpha}_{\beta} \{\mathbf{I}-\mathbf{P}\} f^{\varepsilon}\right\|_{N^s_\gamma}. \nonumber
$$
By $|\alpha_1|+|\beta_1| = N \geq 5$, we have $|\alpha^\prime |+ | \beta^\prime | \leq |\alpha|+|\beta|-1$, and then
$w_l(\alpha,\beta) \leq w_l( \alpha^\prime, \beta^\prime)$ for $1 \leq | \alpha^\prime | \leq 2$ and
$| \beta^\prime | \leq 2$. Therefore, there holds
$I_{1,4}^{\prime\prime\prime} \lesssim \mathcal{E}^{1/2}_{N,l}(t) \mathcal{D}_{N,l}(t)$.

Finally, combining all the above estimates of $I_{1,1}$, $I_{1,2}$, $I_{1,3}$ and $I_{1,4}$, the desired estimate \eqref{soft gamma6} is thus proved. This completes the proof of Lemma \ref{soft Gamma}.
\end{proof}
\medskip

\begin{lemma}\label{hard Gamma}
Let $ \gamma+2s \geq 0$. Assume that $\mathcal{E}^{1/2}_{N,l}(t) \leq \delta$ for some positive constant $\delta >0$.
\begin{itemize}
\setlength{\leftskip}{-6mm}
\item[(1)] For $| \alpha | \leq N$, there holds
\begin{align}\label{hard gamma3}
\bigg|\frac{1}{\varepsilon}\sum_{\pm}\left( \partial^\alpha \Gamma_{\pm}( f^{\varepsilon}, f^{\varepsilon} ),
e^{\pm \varepsilon \phi^{\varepsilon}} \partial^\alpha f_{\pm}^{\varepsilon}\right)\bigg|
\lesssim \left\{\mathcal{E}^{1/2}_{N,l}(t)+\mathcal{E}_{N,l}(t)\right\} \mathcal{D}_{N,l}(t).
\end{align}
\item[(2)] For $| \alpha | \leq N-1$, there holds
\begin{align}\label{hard gamma4}
\bigg|\frac{1}{\varepsilon} \left( \partial^\alpha \Gamma_{\pm}( f^{\varepsilon}, f^{\varepsilon} ),
e^{\pm \varepsilon \phi^{\varepsilon}}
w^2_{l}(\alpha, 0) \partial^\alpha \{\mathbf{I}_{\pm}-\mathbf{P}_{\pm}\} f^{\varepsilon}\right)\bigg|
\lesssim \mathcal{E}^{1/2}_{N,l}(t) \mathcal{D}_{N,l}(t).
\end{align}
\item[(3)] For $|\alpha|=N$, there holds
\begin{align}\label{hard gamma5}
\Big|\left( \partial^\alpha \Gamma_{\pm}( f^{\varepsilon}, f^{\varepsilon} ), e^{\pm \varepsilon \phi^{\varepsilon}}
w^2_{l}(\alpha, 0) \partial^\alpha  f_{\pm}^{\varepsilon}\right)\Big|
\lesssim \mathcal{E}^{1/2}_{N,l}(t) \mathcal{D}_{N,l}(t).
\end{align}
\item[(4)] For $| \alpha |+ | \beta | \leq N$, $| \beta | \geq 1$ and $| \alpha | \leq N-1$, there holds
\begin{align}\label{hard gamma6}
\bigg| \frac{1}{\varepsilon}\left( \partial^\alpha_\beta \Gamma_{\pm}( f^{\varepsilon}, f^{\varepsilon} ),
e^{\pm \varepsilon \phi^{\varepsilon}} w^2_{l}(\alpha, \beta) \partial^\alpha_\beta \{\mathbf{I}_{\pm}-\mathbf{P}_{\pm}\} f^{\varepsilon}\right) \bigg|
\lesssim \mathcal{E}^{1/2}_{N,l}(t) \mathcal{D}_{N,l}(t).
\end{align}
\end{itemize}
\end{lemma}

\begin{proof}
In fact, comparing \eqref{hard gamma1} and \eqref{soft gamma1}, we find that \eqref{hard gamma1} can be a part of \eqref{soft gamma1} to a certain extent. In other words, \eqref{hard gamma1} and the first two terms on the right-hand side of \eqref{soft gamma1} are equivalent. Therefore, the estimates on $\Gamma(f^\varepsilon,f^\varepsilon)$ for the hard potential case can be proved by using the same approach as that for the soft potential case in Lemma \ref{soft Gamma}. For the estimates \eqref{hard gamma3}--\eqref{hard gamma6}, we can see that $N \geq 2$ is enough by Lemma \ref{soft Gamma}. The details are omitted for brevity.
\end{proof}

The following two lemmas concern the estimates on the nonlinear term $v \cdot \nabla_x \phi^{\varepsilon} f^{\varepsilon}_{\pm}$.
\begin{lemma}\label{softnonlinearterm1}
Let $ -3 < \gamma < -2s$. Assume that $\mathcal{E}^{1/2}_{N,l}(t) \leq \delta$ for some positive constant $\delta >0$.
\begin{itemize}
\setlength{\leftskip}{-6mm}
\item[(1)] For $1 \leq | \alpha | \leq N$, there holds
\begin{align}\label{softnonlinearterm1-1}
\sum_{1 \leq | \alpha_1 | \leq |\alpha|}\left(v_i \partial^{\alpha_1+e_i}\phi^\varepsilon \partial^{\alpha-\alpha_1}f_{\pm}^{\varepsilon}, e^{\pm \varepsilon \phi^{\varepsilon}} \partial^\alpha f_{\pm}^{\varepsilon} \right)
\lesssim \mathcal{E}^{1/2}_{N,l}(t) \mathcal{D}_{N,l}(t).
\end{align}
\item[(2)] For $1 \leq | \alpha | \leq N-1$,  there holds
\begin{align}\label{softnonlinearterm1-2}
\!\!\!\! \sum_{1 \leq | \alpha_1 | \leq | \alpha |}\!\!\! \left(v_i \partial^{\alpha_1+e_i}\phi^\varepsilon \partial^{\alpha-\alpha_1}\{\mathbf{I}_{\pm}-\mathbf{P}_{\pm}\} f^{\varepsilon}, e^{\pm \varepsilon \phi^{\varepsilon}}
w^2_{l}(\alpha, 0) \partial^\alpha \{\mathbf{I}_{\pm}-\mathbf{P}_{\pm}\} f^{\varepsilon} \right)
\lesssim \mathcal{E}^{1/2}_{N,l}(t) \mathcal{D}_{N,l}(t).
\end{align}
\item[(3)] For $ | \alpha | = N$,  there holds
\begin{align}\label{softnonlinearterm1-3}
\sum_{1 \leq | \alpha_1 | \leq | \alpha |}\left(v_i \partial^{\alpha_1+e_i}\phi^\varepsilon \partial^{\alpha-\alpha_1}f_{\pm}^{\varepsilon}, e^{\pm \varepsilon \phi^{\varepsilon}} w^2_l(\alpha,0) \partial^\alpha f_{\pm}^{\varepsilon} \right)
\lesssim \mathcal{E}^{1/2}_{N,l}(t) \mathcal{D}_{N,l}(t).
\end{align}
\item[(4)] For $1 \leq | \alpha | + | \beta | \leq N$, $| \beta | \geq 1$ and
$| \alpha | \leq N-1$,  there holds
\begin{align}\label{softnonlinearterm1-4}
 \sum_{\substack{{ | \alpha_1 | + | \beta_1 | \geq 1}\\ {| \beta_1 | \leq 1}}} \left(\partial_{\beta_1} v_i \partial^{\alpha_1+e_i}\phi^\varepsilon \partial^{\alpha-\alpha_1}_{\beta-\beta_1} \{\mathbf{I}_{\pm}-\mathbf{P}_{\pm}\} f^{\varepsilon}, e^{\pm \varepsilon \phi^{\varepsilon}} w^2_l(\alpha,\beta) \partial^\alpha_\beta \{\mathbf{I}_{\pm}-\mathbf{P}_{\pm}\} f^{\varepsilon} \right)
\lesssim \mathcal{E}^{1/2}_{N,l}(t) \mathcal{D}_{N,l}(t).
\end{align}
\end{itemize}
\end{lemma}

\begin{proof}
For brevity, we only prove \eqref{softnonlinearterm1-4} and
the estimate
\begin{align}\label{softnonlinearterm1-5}
\sum_{1 \leq | \alpha_1 | \leq | \alpha |}\left(v_i \partial^{\alpha_1+e_i}\phi^\varepsilon \partial^{\alpha-\alpha_1}f_{\pm}^{\varepsilon}, e^{\pm \varepsilon \phi^{\varepsilon}} w^2_l(\alpha,0) \partial^\alpha f_{\pm}^{\varepsilon} \right)
\lesssim \mathcal{E}^{1/2}_{N,l}(t) \mathcal{D}_{N,l}(t)\;
\text{ for } 1 \leq | \alpha | \leq N,
\end{align}
 because the estimates \eqref{softnonlinearterm1-1}--\eqref{softnonlinearterm1-3} can be obtained from \eqref{softnonlinearterm1-5}.

 Firstly, we estimate \eqref{softnonlinearterm1-5}. In terms of $f_{\pm}^{\varepsilon}=\{\mathbf{I}_{\pm}-\mathbf{P}_{\pm}\} f^{\varepsilon} + \mathbf{P}_{\pm}f^{\varepsilon}$, we set
\begin{align}
\begin{split}\nonumber
I_2\equiv \;& \sum_{1 \leq |\alpha_1| \leq|\alpha|}\left(v_i \partial^{\alpha_1+e_i} \phi^{\varepsilon} \partial^{\alpha-\alpha_1} \mathbf{P}_{ \pm} f^{\varepsilon}, e^{ \pm \varepsilon \phi^{\varepsilon} } w_l^2(\alpha, 0) \partial^\alpha f^{\varepsilon}_{ \pm}\right) \\
& \quad+\sum_{1 \leq |\alpha_1| \leq|\alpha|}\left(v_i \partial^{\alpha_1+e_i} \phi^{\varepsilon} \partial^{\alpha-\alpha_1}\{\mathbf{I}_{ \pm}-\mathbf{P}_{ \pm}\} f^{\varepsilon}, e^{ \pm \varepsilon \phi^{\varepsilon}} w_l^2(\alpha, 0) \partial^\alpha \mathbf{P}_{ \pm} f^{\varepsilon}\right) \\
& \quad+\sum_{1 \leq |\alpha_1| \leq|\alpha|}\left(v_i \partial^{\alpha_1+e_i} \phi^{\varepsilon} \partial^{\alpha-\alpha_1}\{\mathbf{I}_{ \pm}-\mathbf{P}_{ \pm}\} f^{\varepsilon}, e^{ \pm \varepsilon \phi^{\varepsilon}} w_l^2(\alpha, 0) \partial^\alpha \{\mathbf{I}_{ \pm}-\mathbf{P}_{ \pm} \} f^{\varepsilon}\right) \\
:=\;&I_{2,1}+I_{2,2}+I_{2,3}.
\end{split}
\end{align}
For $I_{2,1}$, due to $1 \leq | \alpha | \leq N$, one has
\begin{align}
\begin{split}\nonumber
|I_{2,1}| \lesssim\;& \sum_{1 \leq |\alpha_1| \leq|\alpha|}\int_{\mathbb{R}^3} \left|\partial^{\alpha_1} \nabla_x \phi^{\varepsilon}\right|
\left| \partial^{\alpha-\alpha_1} (a^{\varepsilon}_{\pm},b^{\varepsilon},c^{\varepsilon})\right|
\left|\partial^\alpha f^{\varepsilon}\right|_{L^2_{\gamma/2+s}} \d x \\
\lesssim\;& \sum_{|\alpha_1|=1} \left\|\partial^{\alpha_1} \nabla_x \phi^{\varepsilon}\right\|_{L^3}
\left\| \partial^{\alpha-\alpha_1} (a^{\varepsilon}_{\pm},b^{\varepsilon},c^{\varepsilon})\right\|_{L^6}
\left\|\partial^\alpha f^{\varepsilon}\right\|_{L^2_{\gamma/2+s}}\\
&+ \sum_{2 \leq |\alpha_1| \leq N} \left\|\partial^{\alpha_1} \nabla_x \phi^{\varepsilon}\right\|
\left\| \partial^{\alpha-\alpha_1} (a^{\varepsilon}_{\pm},b^{\varepsilon},c^{\varepsilon})\right\|_{L^\infty}
\left\|\partial^\alpha f^{\varepsilon}\right\|_{L^2_{\gamma/2+s}}\\
\lesssim \;&\mathcal{E}^{1/2}_{N,l}(t) \mathcal{D}_{N,l}(t),
\end{split}
\end{align}
where we have used $|e^{\pm \varepsilon \phi^\varepsilon}| \approx 1$, $\| \cdot \|_{L^2_{\gamma/2+s}} \lesssim \| \cdot \|_{N^s_\gamma}$ and the Sobolev inequalities \eqref{Sobolev-ineq}. Similarly, for $I_{2,2}$, we have
\begin{align}
\begin{split}\nonumber
|I_{2,2}|\lesssim\;& \sum_{1 \leq |\alpha_1| \leq|\alpha|}\int_{\mathbb{R}^3} \left|\partial^{\alpha_1} \nabla_x \phi^{\varepsilon}\right|
\left| \partial^{\alpha-\alpha_1} \{\mathbf{I}-\mathbf{P}\} f^{\varepsilon} \right|_{L^2_{\gamma/2+s}}
\left|\partial^\alpha (a^{\varepsilon}_{\pm},b^{\varepsilon},c^{\varepsilon})\right| \d x \\
\lesssim\;& \sum_{|\alpha_1|=1} \left\|\partial^{\alpha_1} \nabla_x \phi^{\varepsilon}\right\|_{L^3}
\left\| \partial^{\alpha-\alpha_1} \{\mathbf{I}-\mathbf{P}\} f^{\varepsilon} \right\|_{L^6_x L^2_{\gamma/2+s}}
\left\|\partial^\alpha (a^{\varepsilon}_{\pm},b^{\varepsilon},c^{\varepsilon})\right\|\\
&+ \!\!\!\sum_{2 \leq |\alpha_1| \leq N} \!\!\! \left\|\partial^{\alpha_1} \nabla_x \phi^{\varepsilon}  \right\|
\left\| \partial^{\alpha-\alpha_1} \{\mathbf{I}-\mathbf{P}\} f^{\varepsilon} \right\|_{L^\infty_x L^2_{\gamma/2+s}}
\left\|\partial^\alpha (a^{\varepsilon}_{\pm},b^{\varepsilon},c^{\varepsilon})\right\|\\
\lesssim \;&\mathcal{E}^{1/2}_{N,l}(t) \mathcal{D}_{N,l}(t).
\end{split}
\end{align}
As for $I_{2,3}$, thanks to the fact $w_l(\alpha,0) \leq w_l(|\alpha|-1,0) \langle v \rangle ^ {-1+\gamma}$, we have
\begin{align}
\begin{split}\nonumber
|I_{2,3}|\lesssim\;& \sum_{1 \leq |\alpha_1| \leq|\alpha|}\int_{\mathbb{R}^3} \left|\partial^{\alpha_1} \nabla_x \phi^{\varepsilon}\right|
\left| w_l(|\alpha|-1,0) \partial^{\alpha-\alpha_1} \{\mathbf{I}-\mathbf{P}\} f^{\varepsilon} \right|_{L^2_{\gamma/2}}
\left| w_l(\alpha,0) \partial^{\alpha} \{\mathbf{I}-\mathbf{P}\} f^{\varepsilon}\right|_{L^2_{\gamma/2}} \d x \\
\lesssim\;& \sum_{|\alpha_1|=1} \left\|\partial^{\alpha_1} \nabla_x \phi^{\varepsilon}\right\|_{L^\infty}
\left\| w_l(|\alpha|-1,0) \partial^{\alpha-\alpha_1} \{\mathbf{I}-\mathbf{P}\} f^{\varepsilon} \right\|_{L^2_{\gamma/2}}
\left\| w_l(\alpha,0) \partial^{\alpha} \{\mathbf{I}-\mathbf{P}\} f^{\varepsilon} \right\|_{L^2_{\gamma/2}}\\
+&\!\!\! \sum_{2 \leq |\alpha_1| \leq N-1} \!\!\!\left\|\partial^{\alpha_1} \nabla_x \phi^{\varepsilon}  \right\|_{L^6}
\left\| w_l(|\alpha|-1,0) \partial^{\alpha-\alpha_1} \{\mathbf{I}-\mathbf{P}\} f^{\varepsilon} \right\|_{L^3_x L^2_{\gamma/2}}
\left\| w_l(\alpha,0) \partial^{\alpha} \{\mathbf{I}-\mathbf{P}\} f^{\varepsilon} \right\|_{L^2_{\gamma/2}}\\
+& \sum_{ |\alpha_1| = N } \left\|\partial^{\alpha_1} \nabla_x \phi^{\varepsilon}  \right\|
\left\| w_l(|\alpha|-1,0) \{\mathbf{I}-\mathbf{P}\} f^{\varepsilon} \right\|_{L^\infty_x L^2_{\gamma/2}}
\left\| w_l(\alpha,0) \partial^{\alpha} \{\mathbf{I}-\mathbf{P}\} f^{\varepsilon} \right\|_{L^2_{\gamma/2}}\\
\lesssim \;&\mathcal{E}^{1/2}_{N,l}(t) \mathcal{D}_{N,l}(t),
\end{split}
\end{align}
where we have used $\| \cdot \|_{L^2_{\gamma/2}} \lesssim \| \cdot \|_{N^s_\gamma}$, the Sobolev inequalities \eqref{Sobolev-ineq}, the fact $w_l(|\alpha|-1,0) \leq w_l(\alpha-\alpha_1+\alpha^{\prime},0)$ in the case when $2 \leq |\alpha_1| \leq |\alpha| \leq N-1$, $|\alpha^\prime| \leq 1$ and $w_l(|\alpha|-1,0) \leq w_l(\alpha^{\prime},0)$ in the case when $|\alpha|=|\alpha_1| = N$, $1 \leq |\alpha^\prime| \leq 2$. Collecting all the estimates of $I_{2,1}$, $I_{2,2}$ and $I_{2,3}$, \eqref{softnonlinearterm1-5} is thus proved.

Then, we turn to estimate \eqref{softnonlinearterm1-4}. For brevity, we denote
$$
I_3\equiv \text {~the left-hand side of~} \eqref{softnonlinearterm1-4}.
$$
We prove \eqref{softnonlinearterm1-4} by considering the following two cases.

{\em Case 1.} $|\beta_1|=1$, $|\alpha_1| \leq |\alpha| \leq N-1$. \;
 By virtue of $w_l(\alpha, \beta) \leq w_l(\alpha,\beta-e_i) \langle v \rangle ^ \gamma$, one has
\begin{align}
\begin{split}\nonumber
|I_3| \lesssim\;& \sum_{|\alpha_1| \leq |\alpha|} \int_{\mathbb{R}^3} \left|\partial^{\alpha_1} \nabla_x \phi^{\varepsilon}\right|
\left|w_l(\alpha,\beta-e_i) \partial^{\alpha-\alpha_1}_{\beta-e_i} \{\mathbf{I}-\mathbf{P}\}f^{\varepsilon}\right|_{L^2_{\gamma/2}}
\left|w_l(\alpha,\beta) \partial^\alpha_\beta \{\mathbf{I}-\mathbf{P}\}f^{\varepsilon}\right|_{L^2_{\gamma/2}} \d x \\
\lesssim\;& \left\| \nabla_x \phi^{\varepsilon}\right\|_{L^\infty}
\left\|w_l(\alpha,\beta-e_i) \partial^{\alpha}_{\beta-e_i} \{\mathbf{I}-\mathbf{P}\}f^{\varepsilon}\right\|_{L^2_{\gamma/2}}
\left\|w_l(\alpha,\beta) \partial^\alpha_\beta \{\mathbf{I}-\mathbf{P}\}f^{\varepsilon}\right\|_{L^2_{\gamma/2}} \\
+&\!\!\!\!\sum_{1 \leq |\alpha_1| \leq N-1}\!\!\!\!\left\| \partial^{\alpha_1} \nabla_x \phi^{\varepsilon}\right\|_{L^6}
\left\|w_l(\alpha,\beta-e_i) \partial^{\alpha-\alpha_1}_{\beta-e_i} \{\mathbf{I}-\mathbf{P}\}f^{\varepsilon}\right\|_{L^3_x L^2_{\gamma/2}}
\!\!\!\left\|w_l(\alpha,\beta) \partial^\alpha_\beta \{\mathbf{I}-\mathbf{P}\}f^{\varepsilon}\right\|_{L^2_{\gamma/2}} \\
\lesssim \;&\mathcal{E}^{1/2}_{N,l}(t) \mathcal{D}_{N,l}(t),
\end{split}
\end{align}
where in the last inequality we have used the Sobolev inequalities \eqref{Sobolev-ineq} and the fact $w_l(\alpha, \beta-e_i) \leq w_l(\alpha-\alpha_1+\alpha^{\prime}, \beta-e_i)$ when $1 \leq |\alpha_1| \leq |\alpha| \leq  N-1$ and $|\alpha^\prime| \leq 1$.

{\em Case 2.} $|\beta_1|=0$, $1 \leq |\alpha_1| \leq |\alpha| \leq N-1$. \;
By virtue of $w_l(\alpha, \beta) \leq w_l(|\alpha|-1, |\beta|) \langle v \rangle ^ {-1+\gamma}$, we obtain
\begin{align}
\begin{split}\nonumber
\! |I_3| \lesssim\;& \!\!\!\sum_{1 \leq |\alpha_1| \leq |\alpha|} \int_{\mathbb{R}^3} \left|\partial^{\alpha_1} \nabla_x \phi^{\varepsilon}\right|
\left|w_l(|\alpha|-1,|\beta|) \partial^{\alpha-\alpha_1}_{\beta} \{\mathbf{I}-\mathbf{P}\}f^{\varepsilon}\right|_{L^2_{\gamma/2}}
\!\!\left|w_l(\alpha,\beta) \partial^\alpha_\beta \{\mathbf{I}-\mathbf{P}\}f^{\varepsilon}\right|_{L^2_{\gamma/2}} \d x \\
\lesssim\;& \sum_{|\alpha_1|=1} \left\| \partial^{\alpha_1} \nabla_x \phi^{\varepsilon}\right\|_{L^\infty}
\left\|w_l(|\alpha|-1,|\beta|) \partial^{\alpha-\alpha_1}_{\beta} \{\mathbf{I}-\mathbf{P}\}f^{\varepsilon}\right\|_{L^2_{\gamma/2}}
\left\|w_l(\alpha,\beta) \partial^\alpha_\beta \{\mathbf{I}-\mathbf{P}\}f^{\varepsilon}\right\|_{L^2_{\gamma/2}} \\
+&\!\!\!\sum_{2 \leq |\alpha_1| \leq N-1} \!\!\!\! \left\| \partial^{\alpha_1} \nabla_x \phi^{\varepsilon}\right\|_{L^6}
\left\|w_l(|\alpha|-1,|\beta|) \partial^{\alpha-\alpha_1}_{\beta} \{\mathbf{I}-\mathbf{P}\}f^{\varepsilon}\right\|_{L^3_x L^2_{\gamma/2}}
\!\!\left\|w_l(\alpha,\beta) \partial^\alpha_\beta \{\mathbf{I}-\mathbf{P}\}f^{\varepsilon}\right\|_{L^2_{\gamma/2}} \\
\lesssim \;&\mathcal{E}^{1/2}_{N,l}(t) \mathcal{D}_{N,l}(t),
\end{split}
\end{align}
where in the last inequality we have used $w_l(|\alpha|-1, |\beta|) \leq w_l(\alpha-\alpha_1+\alpha^{\prime}, \beta)$ when $2 \leq |\alpha_1| \leq |\alpha| \leq N-1$ and $|\alpha^\prime| \leq 1$.
Collecting the above estimates, we prove \eqref{softnonlinearterm1-4}. This completes the proof of Lemma \ref{softnonlinearterm1}.
\end{proof}
\medskip

\begin{lemma}\label{hardnonlinearterm1}
Let $\gamma+2s \geq 0$. Assume that $\mathcal{E}^{1/2}_{N,l}(t) \leq \delta$ for some
positive constant $\delta >0$.
\begin{itemize}
\setlength{\leftskip}{-6mm}
\item[(1)] For $1 \leq | \alpha | \leq N$, we have
\begin{align}\label{hardnonlinearterm1-1}
\sum_{1 \leq | \alpha_1 | \leq |\alpha|}\left(v_i \partial^{\alpha_1+e_i}\phi^\varepsilon \partial^{\alpha-\alpha_1}f_{\pm}^{\varepsilon}, e^{\pm \varepsilon \phi^{\varepsilon}} \partial^\alpha f_{\pm}^{\varepsilon} \right)
\lesssim \mathcal{E}^{1/2}_{N,l}(t) \mathcal{D}_{N,l}(t).
\end{align}
\item[(2)] For $1 \leq | \alpha | \leq N-1$, we have
\begin{align}\label{hardnonlinearterm1-2}
\!\!\!\! \sum_{1 \leq | \alpha_1 | \leq | \alpha |}\!\!\! \left(v_i \partial^{\alpha_1+e_i}\phi^\varepsilon \partial^{\alpha-\alpha_1}\{\mathbf{I}_{\pm}-\mathbf{P}_{\pm}\} f^{\varepsilon}, e^{\pm \varepsilon \phi^{\varepsilon}}
w^2_{l}(\alpha, 0) \partial^\alpha \{\mathbf{I}_{\pm}-\mathbf{P}_{\pm}\} f^{\varepsilon} \right)
\lesssim \mathcal{E}^{1/2}_{N,l}(t) \mathcal{D}_{N,l}(t).
\end{align}
\item[(3)] For $ | \alpha | = N$, we have
\begin{align}\label{hardnonlinearterm1-3}
\sum_{1 \leq | \alpha_1 | \leq | \alpha |}\left(v_i \partial^{\alpha_1+e_i}\phi^\varepsilon \partial^{\alpha-\alpha_1}f_{\pm}^{\varepsilon}, e^{\pm \varepsilon \phi^{\varepsilon}} w^2_l(\alpha,0) \partial^\alpha f_{\pm}^{\varepsilon} \right)
\lesssim \mathcal{E}^{1/2}_{N,l}(t) \mathcal{D}_{N,l}(t).
\end{align}
\item[(4)] For $1 \leq | \alpha | + | \beta | \leq N$, $| \beta | \geq 1$ and
$| \alpha | \leq N-1$, we have
\begin{align}\label{hardnonlinearterm1-4}
 \sum_{\substack{{ | \alpha_1 | + | \beta_1 | \geq 1}\\ {| \beta_1 | \leq 1}}} \left(\partial_{\beta_1} v_i \partial^{\alpha_1+e_i}\phi^\varepsilon \partial^{\alpha-\alpha_1}_{\beta-\beta_1} \{\mathbf{I}_{\pm}-\mathbf{P}_{\pm}\} f^{\varepsilon}, e^{\pm \varepsilon \phi^{\varepsilon}} w^2_l(\alpha,\beta) \partial^\alpha_\beta \{\mathbf{I}_{\pm}-\mathbf{P}_{\pm}\} f^{\varepsilon} \right)
\lesssim \mathcal{E}^{1/2}_{N,l}(t) \mathcal{D}_{N,l}(t).
\end{align}
\end{itemize}
\end{lemma}

\begin{proof}
In this lemma, we explain that $N \geq 2$ is sufficient for the hard potential case.
For brevity, we only give the proof of the most complicated case
\begin{align}\label{hardnonlinearterm1-5}
\!\!\! \sum_{1 \leq |\alpha_1| \leq|\alpha|} \!\!\! \left(v_i \partial^{\alpha_1+e_i} \phi^{\varepsilon} \partial^{\alpha-\alpha_1}\{\mathbf{I}_{ \pm}-\mathbf{P}_{ \pm}\} f^{\varepsilon}, e^{ \pm \varepsilon \phi^{\varepsilon}} w_l^2(\alpha, 0) \partial^\alpha \{\mathbf{I}_{ \pm}-\mathbf{P}_{ \pm} \} f^{\varepsilon}\right) \lesssim \mathcal{E}^{1/2}_{N,l}(t) \mathcal{D}_{N,l}(t)
\end{align}
 for $1 \leq |\alpha| \leq N$.
The estimate \eqref{hardnonlinearterm1-4} can be proved in a similar way, and other cases concerning $\mathbf{P}_{\pm}f^\varepsilon$ can be proved by the similar argument as that of $I_{2,1}$ and $I_{2,2}$ in Lemma \ref{softnonlinearterm1}.

To prove \eqref{hardnonlinearterm1-5},  for brevity, we denote
$$
I_4\equiv \text { the left-hand side of } \eqref{hardnonlinearterm1-5}.
$$
Thanks to the relation $w_l(\alpha,0)=w_l(|\alpha|-1,0)\langle v \rangle ^{-1}$ for the hard potential case $\gamma+2s \geq 0$, one has
\begin{align}
\begin{split}\nonumber
| I_4 | \lesssim\;& \!\!\!\sum_{1 \leq |\alpha_1| \leq |\alpha|} \int_{\mathbb{R}^3} \left|\partial^{\alpha_1} \nabla_x \phi^{\varepsilon}\right|
\left|w_l(|\alpha|-1,0) \partial^{\alpha-\alpha_1}\{\mathbf{I}-\mathbf{P}\} f^{\varepsilon}\right|_{L^2}
\left|w_l(\alpha, 0) \partial^\alpha \{\mathbf{I}-\mathbf{P} \} f^{\varepsilon}\right|_{L^2} \d x \\
\lesssim\;& \sum_{|\alpha_1|=1 }  \left|\partial^{\alpha_1} \nabla_x \phi^{\varepsilon}\right\|_{L^\infty}
\left\|w_l(|\alpha|-1,0) \partial^{\alpha-\alpha_1}\{\mathbf{I}-\mathbf{P}\} f^{\varepsilon}\right\|
\left\| w_l(\alpha, 0) \partial^\alpha \{\mathbf{I}-\mathbf{P} \} f^{\varepsilon}\right\| \\
&\!\!\!+\sum_{2\leq |\alpha_1| \leq N } \!\!\! \left\|\partial^{\alpha_1} \nabla_x \phi^{\varepsilon}\right\|_{L^3}
\left\|w_l(|\alpha|-1,0) \partial^{\alpha-\alpha_1}\{\mathbf{I}-\mathbf{P}\} f^{\varepsilon}\right\|_{L^6_x L^2_v}
\left\| w_l(\alpha, 0) \partial^\alpha \{\mathbf{I}-\mathbf{P} \} f^{\varepsilon}\right\| \\
\lesssim\;& \mathcal{E}^{1/2}_{N,l}(t) \mathcal{D}_{N,l}(t).
\end{split}
\end{align}
Here, we have used $\| \cdot \|\lesssim \| \cdot \|_{N^s_\gamma}$ and the fact that
\begin{equation}\label{reason N geq 2}
\sum_{|\alpha| = N+1} \left\| \partial^\alpha \nabla_x \phi^{\varepsilon} \right\|
=  \sum_{|\alpha| = N} \left\| \partial^\alpha (a_{+}^\varepsilon - a_{-}^\varepsilon) \right\|
\lesssim \mathcal{E}^{1/2}_{N,l}(t).
\end{equation}
This is the reason that $N \geq 2$ is enough for the hard potential case. Therefore, the proof of Lemma \ref{hardnonlinearterm1} is completed.
\end{proof}
\medskip

The following two lemmas concern the estimates on the nonlinear term $\nabla_x \phi^{\varepsilon} \cdot \nabla_v f^{\varepsilon}_{\pm}$.
\begin{lemma}\label{softnonlinearterm2}
Let $ -3 < \gamma < -2s$. Assume that $\mathcal{E}^{1/2}_{N,l}(t) \leq \delta$ for some positive constant $\delta >0$.
\begin{itemize}
\setlength{\leftskip}{-6mm}
\item[(1)] For $1 \leq | \alpha | \leq N$, we have
\begin{align}\label{softnonlinearterm2-1}
\sum_{1 \leq | \alpha_1 | \leq |\alpha|}\left( \partial^{\alpha_1+e_i}\phi^\varepsilon \partial_{e_i}^{\alpha-\alpha_1} f_{\pm}^{\varepsilon}, e^{\pm \varepsilon \phi^{\varepsilon}} \partial^\alpha f_{\pm}^{\varepsilon} \right)
\lesssim \mathcal{E}^{1/2}_{N,l}(t) \mathcal{D}_{N,l}(t).
\end{align}
\item[(2)] For $ | \alpha | \leq N-1$,  we have
\begin{align}
\begin{split}\label{softnonlinearterm2-2}
& \sum_{ | \alpha_1 | \leq | \alpha |} \left( \partial^{\alpha_1+e_i}\phi^\varepsilon \partial^{\alpha-\alpha_1}_{e_i}\{\mathbf{I}_{\pm}-\mathbf{P}_{\pm}\} f^{\varepsilon}, e^{\pm \varepsilon \phi^{\varepsilon}}
w^2_{l}(\alpha, 0) \partial^\alpha \{\mathbf{I}_{\pm}-\mathbf{P}_{\pm}\} f^{\varepsilon} \right) \\
\lesssim\;& \mathcal{E}^{1/2}_{N,l}(t) \mathcal{D}_{N,l}(t)+\left\|\nabla_x \phi^{\varepsilon}\right\|_{L^\infty}
\sum_{\pm}\left\|e^{\pm \frac{\varepsilon \phi^{\varepsilon}}{2}}w_l(\alpha,0)\partial^{\alpha}\{\mathbf{I}_{\pm}-\mathbf{P}_{\pm}\} f^{\varepsilon}\right\|^2.
\end{split}
\end{align}
\item[(3)] For $ | \alpha | = N$,  we have
\begin{align}\label{softnonlinearterm2-3}
\begin{split}
&\varepsilon \sum_{| \alpha_1 | \leq | \alpha |} \left(\partial^{\alpha_1+e_i}\phi^\varepsilon \partial^{\alpha-\alpha_1}_{e_i} f_{\pm}^{\varepsilon}, e^{\pm \varepsilon \phi^{\varepsilon}} w^2_l(\alpha,0) \partial^\alpha f_{\pm}^{\varepsilon} \right) \\
\lesssim\;&\mathcal{E}^{1/2}_{N,l}(t) \mathcal{D}_{N,l}(t)+\varepsilon\left\|\nabla_x \phi^{\varepsilon}\right\|_{L^\infty}
\sum_{|\alpha|= N}\sum_{\pm}\left\|e^{\pm \frac{\varepsilon \phi^{\varepsilon}}{2}}w_l(\alpha,0)\partial^{\alpha} f^{\varepsilon}_{\pm}\right\|^2.
\end{split}
\end{align}
\item[(4)] For $1 \leq | \alpha | + | \beta | \leq N$, $| \beta | \geq 1$
 and $| \alpha | \leq N-1$,  we have
\begin{align}
\begin{split}\label{softnonlinearterm2-4}
&\sum_{| \alpha_1 | \leq | \alpha |}  \left( \partial^{\alpha_1+e_i}\phi^\varepsilon \partial^{\alpha-\alpha_1}_{\beta+e_i} \{\mathbf{I}_{\pm}-\mathbf{P}_{\pm}\} f^{\varepsilon}, e^{\pm \varepsilon \phi^{\varepsilon}} w^2_l(\alpha,\beta) \partial^\alpha_\beta \{\mathbf{I}_{\pm}-\mathbf{P}_{\pm}\} f^{\varepsilon} \right) \\
\lesssim\;& \mathcal{E}^{1/2}_{N,l}(t) \mathcal{D}_{N,l}(t)+\left\|\nabla_x \phi^{\varepsilon}\right\|_{L^\infty}
\sum_{\pm}\left\|e^{\pm \frac{\varepsilon \phi^{\varepsilon}}{2}}w_l(\alpha,\beta)\partial^{\alpha}_{\beta}\{\mathbf{I}_{\pm}-\mathbf{P}_{\pm}\} f^{\varepsilon}\right\|^2.
\end{split}
\end{align}
\end{itemize}
\end{lemma}

\begin{proof}
For brevity, we only prove \eqref{softnonlinearterm2-2}. The estimate \eqref{softnonlinearterm2-4} can be proved in a similar way, and the remaining other estimates concerning $\mathbf{P}_{\pm}f^{\varepsilon}$ can be proved by the similar arguments as $I_{2,1}$ and $I_{2,2}$ in Lemma \ref{softnonlinearterm1}.

To prove \eqref{softnonlinearterm2-2}, for simplicity, we write
$$
I_5 \equiv \text{~the left-hand side of~} \eqref{softnonlinearterm2-2}.
$$
Firstly, when $|\alpha_1|=0$, one has
\begin{align}
\begin{split}\label{alpha1=0estimate}
I_5=\;&\left( \partial^{e_i}\phi^\varepsilon \partial^{\alpha}_{e_i} \{\mathbf{I}_{\pm}-\mathbf{P}_{\pm}\} f^{\varepsilon}, e^{\pm \varepsilon \phi^{\varepsilon}} w^2_l(\alpha,0) \partial^\alpha \{\mathbf{I}_{\pm}-\mathbf{P}_{\pm}\} f^{\varepsilon} \right)\\
=\;&-\frac{1}{2}\left( \partial^{e_i}\phi^\varepsilon \partial^{\alpha} \{\mathbf{I}_{\pm}-\mathbf{P}_{\pm}\} f^{\varepsilon}, e^{\pm \varepsilon \phi^{\varepsilon}} \partial_{e_i}w^2_l(\alpha,0) \partial^\alpha \{\mathbf{I}_{\pm}-\mathbf{P}_{\pm}\} f^{\varepsilon} \right)\\
\lesssim\;&\left\|\nabla_x \phi^{\varepsilon}\right\|_{L^\infty}
\sum_{\pm}\left\|e^{\pm \frac{\varepsilon \phi^{\varepsilon}}{2}}w_l(\alpha,0)\partial^{\alpha}\{\mathbf{I}_{\pm}-\mathbf{P}_{\pm}\} f^{\varepsilon}\right\|^2,
\end{split}
\end{align}
where we have used the velocity integration by part and the inequality $|\partial_{e_i}w^2_l(\alpha,0)| \lesssim w^2_l(\alpha,0)$.

Secondly, when $1 \leq | \alpha_1 | \leq | \alpha | \leq N-1$, we need to make use of the interpolation inequality \eqref{interpolation}. More precisely, we first choose
$$
\langle v \rangle ^\ell =\{w_l(|\alpha|-1,0)\}^{1-s}\{w_l(|\alpha|-1,1)\}^{-(1-s)},
$$
 then we have
$$
\langle v \rangle ^{-\frac{\ell s}{1-s}} =\{w_l(|\alpha|-1,0)\}^{-s}\{w_l(|\alpha|-1,1)\}^s.
$$
Substituting them into \eqref{interpolation}, we get
\begin{align}\label{inter-H1}
&\left|\langle v\rangle^{-\frac{\gamma}{2}} w_l(\alpha, 0) \partial^{\alpha-\alpha_1}\{\mathbf{I}-\mathbf{P}\} f^{\varepsilon}\right|_{H^1} \nonumber\\
\lesssim\;&\left|\langle v \rangle ^\ell\left(\langle v\rangle^{-\frac{\gamma}{2}} w_l(\alpha, 0) \partial^{\alpha-\alpha_1}\{\mathbf{I}-\mathbf{P}\} f^{\varepsilon}\right)\right|_{H^{s}}
\!\!+\left|\langle v \rangle ^{-\frac{\ell s}{1-s}} \left(\langle v\rangle^{-\frac{\gamma}{2}} w_l(\alpha, 0) \partial^{\alpha-\alpha_1}\{\mathbf{I}-\mathbf{P}\} f^{\varepsilon}\right)\right|_{H^{1+s}} \nonumber\\
 \lesssim \;& \left| w_l(|\alpha|-1,0) \partial^{\alpha-\alpha_1}\{\mathbf{I}-\mathbf{P}\} f^\varepsilon\right|_{H^s_{\gamma/2}}
+\left| w_l(|\alpha|-1,1) \partial^{\alpha-\alpha_1}\{\mathbf{I}-\mathbf{P}\} f^{\varepsilon} \right|_{H^{1+s}_{\gamma/2}},
\end{align}
where in the last inequality we used the fact
$$
w_l(\alpha,0)\leq \{w_l(|\alpha|-1,0)\}^s\{w_l(|\alpha|-1,1)\}^{1-s}\langle v \rangle^{\gamma},
$$
and the definition of the weight function $w_l(\alpha,\beta)$ in \eqref{weight function}.
By the property of the algebra weight $w_l(\alpha,\beta)$ and the above inequality
\eqref{inter-H1}, one has
\begin{align}
\begin{split}\nonumber
I_5 \lesssim \;& \int_{\mathbb{R}^3}\left|\partial^{\alpha_1} \nabla_x \phi^\varepsilon \right|
\left|w_l(\alpha,0)\langle v \rangle ^{-\frac{\gamma}{2}} \partial^{\alpha-\alpha_1}\{\mathbf{I}-\mathbf{P}\}f^{\varepsilon}\right|_{H^1}
\left|w_l(\alpha,0)\partial^{\alpha}\{\mathbf{I}-\mathbf{P}\}f^{\varepsilon}\right|_{L^2_{\gamma/2}} \d x \\
\lesssim\;& \int_{\mathbb{R}^3}\left|\partial^{\alpha_1} \nabla_x \phi^\varepsilon \right|
\left| w_l(|\alpha|-1,0) \partial^{\alpha-\alpha_1}\{\mathbf{I}-\mathbf{P}\} f^\varepsilon \right|_{H^s_{\gamma/2}}
\left|w_l(\alpha,0)\partial^{\alpha}\{\mathbf{I}-\mathbf{P}\}f^{\varepsilon} \right|_{L^2_{\gamma/2}} \d x\\
&+ \int_{\mathbb{R}^3}\left|\partial^{\alpha_1} \nabla_x \phi^\varepsilon \right|
\left| w_l(|\alpha|-1,1)\partial^{\alpha-\alpha_1}\{\mathbf{I}-\mathbf{P}\} f^{\varepsilon} \right|_{H^{1+s}_{\gamma/2}}
\left|w_l(\alpha,0)\partial^{\alpha}\{\mathbf{I}-\mathbf{P}\}f^{\varepsilon}\right|_{L^2_{\gamma/2}} \d x \\
:=\; & I_{5,1}+I_{5,2}.
\end{split}
\end{align}
For $I_{5,1}$, one has from the H\"{o}lder inequality that
\begin{align}
\begin{split}\nonumber
I_{5,1} \lesssim \;& \sum_{|\alpha_1|=1} \left\| \partial^{\alpha_1} \nabla_x \phi^\varepsilon \right\|_{L^\infty}
\left\| w_l(|\alpha|-1,0) \partial^{\alpha-\alpha_1}\{\mathbf{I}-\mathbf{P}\} f^\varepsilon\right\|_{H^s_{\gamma/2}}
\left\|w_l(\alpha,0)\partial^{\alpha}\{\mathbf{I}-\mathbf{P}\}f^{\varepsilon}\right\|_{L^2_{\gamma/2}} \\
& +\!\!\!\sum_{2 \leq |\alpha_1| \leq N-1}\!\!\! \left\|\partial^{\alpha_1} \nabla_x \phi^\varepsilon \right\|_{L^6}
\left\| w_l(|\alpha|-1,0) \partial^{\alpha-\alpha_1}\{\mathbf{I}-\mathbf{P}\} f^\varepsilon\right\|_{L^3_x H^s_{\gamma/2}}
\left\|w_l(\alpha,0)\partial^{\alpha}\{\mathbf{I}-\mathbf{P}\}f^{\varepsilon}\right\|_{L^2_{\gamma/2}} \\
\lesssim\;&\mathcal{E}^{1/2}_{N,l}(t) \mathcal{D}_{N,l}(t),
\end{split}
\end{align}
where we have used the fact $\| \cdot \|_{H^s_{\gamma/2}} \lesssim \| \cdot \|_{N^s_\gamma}$ and $w_l(|\alpha|-1,0) \leq w_l(\alpha-\alpha_1+\alpha^{\prime},0)$ in the case when $2 \leq |\alpha_1| \leq |\alpha| \leq N-1$ and $|\alpha^\prime| \leq 1$.
Similarly, for $I_{5,2}$, thanks to $w_l(|\alpha|-1,1) \leq w_l(\alpha-\alpha_1+\alpha^{\prime},\beta^{\prime})$
in the case when $2 \leq |\alpha_1| \leq |\alpha| \leq N-1$, $|\alpha^\prime| \leq 1$ and $|\beta^{\prime}| \leq 1$, we get
$I_{5,2} \lesssim\;\mathcal{E}^{1/2}_{N,l}(t) \mathcal{D}_{N,l}(t)$. This completes the proof of Lemma \ref{softnonlinearterm2}.
\end{proof}
\medskip

\begin{lemma}\label{hardnonlinearterm2}
Let $ \gamma+2s \geq 0$. Assume that $\mathcal{E}^{1/2}_{N,l}(t) \leq \delta$ for some positive constant $\delta >0$.
\begin{itemize}
\setlength{\leftskip}{-6mm}
\item[(1)] For $1 \leq \left| \alpha \right| \leq N$, we have
\begin{align}\label{hardnonlinearterm2-1}
\sum_{1 \leq | \alpha_1 | \leq |\alpha|}\left( \partial^{\alpha_1+e_i}\phi^\varepsilon \partial_{e_i}^{\alpha-\alpha_1} f_{\pm}^{\varepsilon}, e^{\pm \varepsilon \phi^{\varepsilon}} \partial^\alpha f_{\pm}^{\varepsilon} \right)
\lesssim \mathcal{E}^{1/2}_{N,l}(t) \mathcal{D}_{N,l}(t).
\end{align}
\item[(2)] For $ \left| \alpha \right| \leq N-1$,  we have
\begin{align}
\begin{split}\label{hardnonlinearterm2-2}
& \sum_{ | \alpha_1 | \leq | \alpha |} \left( \partial^{\alpha_1+e_i}\phi^\varepsilon \partial^{\alpha-\alpha_1}_{e_i}\{\mathbf{I}_{\pm}-\mathbf{P}_{\pm}\} f^{\varepsilon}, e^{\pm \varepsilon \phi^{\varepsilon}}
w^2_{l}(\alpha, 0) \partial^\alpha \{\mathbf{I}_{\pm}-\mathbf{P}_{\pm}\} f^{\varepsilon} \right)
\lesssim \mathcal{E}^{1/2}_{N,l}(t) \mathcal{D}_{N,l}(t).
\end{split}
\end{align}
\item[(3)] For $ \left| \alpha \right| = N$,  we have
\begin{align}\label{hardnonlinearterm2-3}
\sum_{| \alpha_1 | \leq | \alpha |} \left(\partial^{\alpha_1+e_i}\phi^\varepsilon \partial^{\alpha-\alpha_1}_{e_i} f_{\pm}^{\varepsilon}, e^{\pm \varepsilon \phi^{\varepsilon}} w^2_l(\alpha,0) \partial^\alpha f_{\pm}^{\varepsilon} \right)
\lesssim \mathcal{E}^{1/2}_{N,l}(t) \mathcal{D}_{N,l}(t).
\end{align}
\item[(4)] For $1 \leq \left| \alpha \right| + \left| \beta \right| \leq N$, $\left| \beta \right| \geq 1$ and $\left| \alpha \right| \leq N-1$,  we have
\begin{align}
\begin{split}\label{hardnonlinearterm2-4}
&\sum_{| \alpha_1 | \leq | \alpha |}  \left( \partial^{\alpha_1+e_i}\phi^\varepsilon \partial^{\alpha-\alpha_1}_{\beta+e_i} \{\mathbf{I}_{\pm}-\mathbf{P}_{\pm}\} f^{\varepsilon}, e^{\pm \varepsilon \phi^{\varepsilon}} w^2_l(\alpha,\beta) \partial^\alpha_\beta \{\mathbf{I}_{\pm}-\mathbf{P}_{\pm}\} f^{\varepsilon} \right)
\lesssim \mathcal{E}^{1/2}_{N,l}(t) \mathcal{D}_{N,l}(t).
\end{split}
\end{align}
\end{itemize}
\end{lemma}

\begin{proof}
This lemma can be proved similarly as that of Lemma \ref{softnonlinearterm2}. For brevity, we use the same notations as in Lemma \ref{softnonlinearterm2} and only give the estimate of $I_5$ when $|\alpha| \leq N$ for the hard potential case.

When $|\alpha_1|=0$, with the velocity integration by parts, $I_5$ can be controlled by $\mathcal{E}^{1/2}_{N,l}(t) \mathcal{D}_{N,l}(t)$.
When $|\alpha_1| \geq 1$, we have $|\alpha| \geq 1$. By $w_l(\alpha,0)=w_l(|\alpha|-1,1)$ and $\| \cdot \|\lesssim \| \cdot \|_{N^s_\gamma}$, one has
\begin{align}
\begin{split}\nonumber
I_5 \lesssim \;& \sum_{1 \leq |\alpha_1| \leq |\alpha| }\int_{\mathbb{R}^3} \left|\partial^{\alpha_1} \nabla_x \phi^\varepsilon\right|
\left|w_l(|\alpha|-1,1) \partial^{\alpha-\alpha_1}_{e_i}\{\mathbf{I}-\mathbf{P}\}f^{\varepsilon}\right|_{L^2}
\left|w_l(\alpha,0)\partial^{\alpha}\{\mathbf{I}-\mathbf{P}\}f^{\varepsilon}\right|_{L^2}\d x \\
\lesssim\;& \sum_{ |\alpha_1|=1 }\left\|\partial^{\alpha_1} \nabla_x \phi^\varepsilon\right\|_{L^\infty}
\left\|w_l(|\alpha|-1,1) \partial^{\alpha-\alpha_1}_{e_i}\{\mathbf{I}-\mathbf{P}\}f^{\varepsilon}\right\|
\left\|w_l(\alpha,0)\partial^{\alpha}\{\mathbf{I}-\mathbf{P}\}f^{\varepsilon}\right\|\\
&+\sum_{2 \leq |\alpha_1| \leq N }\left\|\partial^{\alpha_1} \nabla_x \phi^\varepsilon\right\|_{L^6}
\left\|w_l(|\alpha|-1,1) \partial^{\alpha-\alpha_1}_{e_i}\{\mathbf{I}-\mathbf{P}\}f^{\varepsilon}\right\|_{L^3_x L^2_v}
\left\|w_l(\alpha,0)\partial^{\alpha}\{\mathbf{I}-\mathbf{P}\}f^{\varepsilon}\right\| \\
\lesssim \;&\mathcal{E}^{1/2}_{N,l}(t) \mathcal{D}_{N,l}(t).
\end{split}
\end{align}
Here, $N \geq 2 $ is sufficient because of \eqref{reason N geq 2} and the fact $w_l(|\alpha|-1,1) \leq w_l(\alpha-\alpha_1+\alpha^{\prime}, e_i)$ in the case when $2 \leq |\alpha_1| \leq |\alpha| \leq N$ and $|\alpha^\prime| \leq 1 $. This completes the proof of Lemma \ref{hardnonlinearterm2}.
\end{proof}
\medskip

\section{The a Priori Estimates}\label{The a Priori Estimate}

In this section, we deduce the uniform a priori estimates  of the VPB system \eqref{rVPB} with respect to $\varepsilon\in (0,1]$
globally in time.

For this purpose, we define the following time-weighted energy norm $X(t)$ by
\begin{equation}\label{X define}
\!\!\!\!X(t):=\left\{
\begin{aligned}& \!\!\sup_{0 \leq \tau \leq t} \!\widetilde{\mathcal{E}}_{N,l_2}(\tau)
+\!\!\sup_{0 \leq \tau \leq t}\! (1+\tau)^{\varrho} \mathcal{E}_{N,l_2}(\tau),
\qquad \qquad \qquad \qquad \qquad \qquad \;\; \gamma+2s\geq 0, \\
&\!\!\sup_{0 \leq \tau \leq t} \! \widetilde{\mathcal{E}}_{N,l_0+l_1}(\tau)+ \!\!\sup_{0 \leq \tau \leq t} \!(1+\tau)^{\varrho} \mathcal{E}_{N,l_0}(\tau)
+\!\!\sup_{0 \leq \tau \leq t} \!(1+\tau)^{\varrho+p}\mathcal{E}^h_{N,l_0}(\tau),
 -3 < \gamma < -2s.
\end{aligned}\right.
\end{equation}
Here, all the involved parameters are fixed to satisfy \eqref{hard assumption} for the hard potential case  $\gamma+2s \geq 0$ with $0 < s < 1$ and \eqref{soft assumption} for the soft potential case $-3 < \gamma <-2s$ with $0 < s < 1$, respectively.

The construction of the weighted instant energy functionals $\mathcal{E}_{N,l}(t)$, $\widetilde{\mathcal{E}}_{N,l}(t)$ and the  dissipation rate functionals $\mathcal{D}_{N,l}(t)$, $\widetilde{\mathcal{D}}_{N,l}(t)$ will be given in the following. Suppose that the VPB system  \eqref{rVPB} admits a smooth solution $\left(f^\varepsilon,\nabla_x \phi^\varepsilon \right)$ over $0 \leq t \leq T$ for $0 < T \leq \infty$, and  the solution
$\left(f^\varepsilon,\nabla_x \phi^\varepsilon \right)$ also satisfies
\begin{align}\label{priori assumption}
\sup_{0 \leq t \leq T} X(t) \leq \delta_0,
\end{align}
where $\delta_0$ is a suitable small positive constant to be determined later.
Moreover, in the remaining parts of this paper, we always assume
\begin{align*}
&\;l = l_2, \;\qquad\text{~if~} \gamma+2s \geq 0;\\
l_0 \leq &\;l \leq l_0+l_1,  \;\text{~if~} -3 < \gamma -2s.
\end{align*}

\subsection{Macroscopic Estimate}
\hspace*{\fill}

In this subsection, we provide the macroscopic estimate of the VPB system with $\varepsilon$, which was also given in \cite{DL2013} without proof for $\varepsilon=1$.

\begin{lemma}\label{macroscopic estimate}
Let $\left(f^\varepsilon,\nabla_x \phi^\varepsilon \right)$ be the solution to the VPB system \eqref{rVPB}.
Then there exist two interactive functionals $\mathcal{E}^N_{int}(t)$  and $\mathcal{E}^{N,h}_{int}(t)$ defined in \eqref{macro energy2} and \eqref{macro energy high} respectively satisfying
\begin{align}
\begin{split}\nonumber
\left|\mathcal{E}^N_{int}(t)\right| \lesssim\;& \sum_{|\alpha| \leq N+1} \left\| \partial^\alpha \nabla_x \phi^{\varepsilon}\right\|^2
+\sum_{|\alpha| \leq N} \left\| \partial^\alpha f^{\varepsilon}\right\|^2, \\
\left|\mathcal{E}^{N, h}_{int}(t)\right| \lesssim\;& \sum_{|\alpha| \leq N+1} \left\| \partial^\alpha \nabla_x \phi^{\varepsilon}\right\|^2
+\sum_{1 \leq |\alpha| \leq N} \left\| \partial^\alpha \mathbf{P}f^{\varepsilon}\right\|^2
+\sum_{ |\alpha| \leq N} \left\| \partial^\alpha \{\mathbf{I}-\mathbf{P}\}f^{\varepsilon}\right\|^2,\\
\end{split}
\end{align}
such that for any $t \geq 0$, there hold
\begin{align}\label{macro estimate}
\begin{split}
&\varepsilon \frac{\d}{\d t}\mathcal{E}^N_{int}(t)+ \sum_{1 \leq |\alpha| \leq N}
\left\|  \partial^\alpha (a_{\pm}^{\varepsilon}, b^{\varepsilon}, c^{\varepsilon})\right\|^2
+ \sum_{|\alpha| \leq N+1} \left\| \partial^\alpha \nabla_x \phi^{\varepsilon} \right\|^2 \\
\lesssim \;& \frac{1}{\varepsilon^2} \sum_{|\alpha| \leq N}\left\| \partial^\alpha \{\mathbf{I}-\mathbf{P}\}f^{\varepsilon}\right\|^2_{N^s_\gamma}
+\mathcal{E}_{N}(t) \mathcal{D}_{N}(t)
\end{split}
\end{align}
and
\begin{align}
\begin{split}\label{macro estimate high}
&\varepsilon \frac{\d}{\d t}\mathcal{E}^{N,h}_{int}(t)+ \sum_{2 \leq |\alpha| \leq N}
\left\|  \partial^\alpha (a_{\pm}^{\varepsilon}, b^{\varepsilon}, c^{\varepsilon})\right\|^2
+ \sum_{|\alpha| \leq N+1} \left\| \partial^\alpha \nabla_x \phi^{\varepsilon} \right\|^2 \\
\lesssim \;& \frac{1}{\varepsilon^2} \sum_{|\alpha| \leq N}\left\| \partial^\alpha \{\mathbf{I}-\mathbf{P}\}f^{\varepsilon}\right\|^2_{N^s_\gamma}
+\mathcal{E}_{N}(t) \mathcal{D}_{N}(t).
\end{split}
\end{align}
\end{lemma}

\begin{proof}
By applying the macro-micro decomposition \eqref{f decomposition} introduced in \cite{Guo2004} and defining moment functions
$$
\Theta_{ij}(f^{\varepsilon}):=\int_{\mathbb{R}^3}(v_i v_j -1)\mu^{1/2}f^{\varepsilon} \d v, \qquad
\Lambda_{i}(f^{\varepsilon}):=\frac{1}{10}\int_{\mathbb{R}^3}(|v|^2 -5)v_i \mu^{1/2}f^{\varepsilon} \d v,
$$
one can derive from \eqref{rVPB} the following two fluid-type systems of equations
\begin{equation}\label{macro equation 1}
\left\{\begin{aligned}
& \partial_t a^{\varepsilon}_{ \pm}+ \frac{1}{\varepsilon} \nabla_x \cdot b^{\varepsilon}
+\frac{1}{\varepsilon} \nabla_x \cdot \langle v \mu^{1 / 2}, \{ \mathbf{I}_{ \pm}-\mathbf{P}_{ \pm}\} f^{\varepsilon} \rangle=0, \\
& \partial_t \big(b^{\varepsilon}_i+\langle v_i \mu^{1 / 2},\{\mathbf{I}_{ \pm}-\mathbf{P}_{ \pm}\} f^{\varepsilon}\rangle \big)
+ \frac{1}{\varepsilon} \partial_i(a^{\varepsilon}_{ \pm}+2 c^{\varepsilon}) \pm \frac{1}{\varepsilon} \partial_i \phi^{\varepsilon} \\
&\qquad
=-\frac{1}{\varepsilon} \langle v_i \mu^{1 / 2}, v \cdot \nabla_x \{\mathbf{I}_{ \pm}-\mathbf{P}_{ \pm}\} f^\varepsilon \rangle
+\langle g^\varepsilon_{ \pm}-\frac{1}{\varepsilon^2}L_{ \pm} f^\varepsilon, v_i \mu^{1 / 2}\rangle, \\
& \partial_t \bigg(c^\varepsilon+\frac{1}{6}\langle (|v|^2-3) \mu^{1 / 2},\{\mathbf{I}_{ \pm}-\mathbf{P}_{ \pm} \} f^\varepsilon \rangle \bigg)
+\frac{1}{3\varepsilon} \nabla_x \cdot b^\varepsilon \\
&\qquad
=-\frac{1}{6\varepsilon}\langle (|v|^2-3) \mu^{1 / 2}, v \cdot \nabla_x \{\mathbf{I}_{ \pm}-\mathbf{P}_{ \pm}\} f^\varepsilon \rangle
+\frac{1}{6}\langle g^\varepsilon_{ \pm}- \frac{1}{\varepsilon^2} L_{ \pm} f^\varepsilon ,(|v|^2-3) \mu^{1 / 2}\rangle
\end{aligned}\right.
\end{equation}
and
\begin{equation}\label{macro equation 2}
\left\{\begin{aligned}
\!\!\!&\partial_t (\Theta_{i i}(\{\mathbf{I}_{ \pm}-\mathbf{P}_{ \pm}\} f^{\varepsilon})+2 c^{\varepsilon} )+ \frac{2}{\varepsilon} \partial_i b^\varepsilon_i =\Theta_{i i}(g^\varepsilon_{ \pm}+h^\varepsilon_{ \pm}), \\
&\partial_t \Theta_{i j} ( \{\mathbf{I}_{ \pm}-\mathbf{P}_{ \pm} \} f^{\varepsilon} )+\frac{1}{\varepsilon} \partial_i b^\varepsilon_j
+\frac{1}{\varepsilon} \partial_j b^\varepsilon_i+\frac{1}{\varepsilon} \nabla_x \cdot \langle v \mu^{1 / 2},
\{\mathbf{I}_{ \pm}-\mathbf{P}_{ \pm}\} f^{\varepsilon} \rangle=\Theta_{i j}(g^\varepsilon_{ \pm}+h^\varepsilon_{ \pm} ), i \neq j, \\
&\partial_t \Lambda_i (\{\mathbf{I}_{ \pm}-\mathbf{P}_{ \pm}\} f^{\varepsilon})+\frac{1}{\varepsilon} \partial_i c^{\varepsilon}
=\Lambda_i(g^\varepsilon_{ \pm}+h^\varepsilon_{ \pm}),
\end{aligned}\right.
\end{equation}
where
\begin{align}
\begin{split}\label{g,h define}
g^\varepsilon_{\pm}:=&\pm \nabla_x \phi^{\varepsilon} \cdot \nabla_v f_{\pm}^\varepsilon \mp \frac{1}{2}\nabla_x \phi^{\varepsilon} \cdot vf_{\pm}^\varepsilon + \frac{1}{\varepsilon} \Gamma_{\pm}(f^\varepsilon, f^\varepsilon), \\
h^\varepsilon_{\pm}:=&-\frac{1}{\varepsilon}v \cdot \nabla_x \{\mathbf{I}_{\pm}-\mathbf{P}_{\pm}\}f^{\varepsilon}-\frac{1}{\varepsilon^2}L_{\pm} f^{\varepsilon}.
\end{split}
\end{align}
By taking the mean value of every two equations with $\pm$ sign for \eqref{macro equation 1}--\eqref{macro equation 2}, one has
\begin{equation}\label{macro equation 3}
\left\{\begin{aligned}
& \partial_t \bigg(\frac{a^{\varepsilon}_{+}+a^{\varepsilon}_{-}}{2}\bigg) + \frac{1}{\varepsilon} \nabla_x \cdot b^{\varepsilon}=0, \\
& \partial_t b^{\varepsilon}_i
+ \frac{1}{\varepsilon} \partial_i \bigg(\frac{a^{\varepsilon}_{+}+a^{\varepsilon}_{-}}{2}+2 c^{\varepsilon}\bigg)
+\frac{1}{2\varepsilon} \sum_{j=1}^{3} \partial_j \Theta_{ij} ( \{\mathbf{I}-\mathbf{P}\} f^\varepsilon \cdot [1,1] )
= \frac{1}{2} \langle g^\varepsilon_{+}+g^\varepsilon_{-}, v_i \mu^{1 / 2}\rangle, \\
& \partial_t c^\varepsilon +\frac{1}{3\varepsilon} \nabla_x \cdot b^\varepsilon
+\frac{5}{6\varepsilon} \sum_{i=1}^{3} \partial_i \Lambda_i ( \{\mathbf{I}-\mathbf{P}\} f^\varepsilon \cdot [1,1] )
=\frac{1}{12}\langle g^\varepsilon_{+}+g^\varepsilon_{-}, (|v|^2-3) \mu^{1 / 2}\rangle
\end{aligned}\right.
\end{equation}
and
\begin{equation}\label{macro equation 4}
\left\{\begin{aligned}
&\partial_t \bigg(\frac{1}{2}\Theta_{i j}(\{\mathbf{I}-\mathbf{P}\} f^{\varepsilon} \cdot [1,1])+2 c^{\varepsilon}\delta_{ij} \bigg)+ \frac{1}{\varepsilon} \partial_i b^\varepsilon_j + \frac{1}{\varepsilon} \partial_j b^\varepsilon_i
=\frac{1}{2}\Theta_{i j}(g^\varepsilon_{+}+g^\varepsilon_{-}+h^\varepsilon_{+}+h^\varepsilon_{-}), \\
&\frac{1}{2}\partial_t \Lambda_i (\{\mathbf{I}-\mathbf{P}\} f^{\varepsilon} \cdot [1,1])+\frac{1}{\varepsilon} \partial_i c^{\varepsilon}
=\frac{1}{2}\Lambda_i(g^\varepsilon_{+}+g^\varepsilon_{-}+h^\varepsilon_{+}+h^\varepsilon_{-}),
\end{aligned}\right.
\end{equation}
where $1 \leq i, j \leq 3$ and $\delta_{ij}$ denotes the Kronecker delta.
Moreover, in order to further obtain the dissipation rate related to $a_{\pm}^\varepsilon$ from the formula
$$
|a_{+}^\varepsilon|^2+|a_{-}^\varepsilon|^2=\frac{|a_{+}^\varepsilon+a_{-}^\varepsilon|^2}{2}+\frac{|a_{+}^\varepsilon-a_{-}^\varepsilon|^2}{2},
$$
by taking the difference of the first two equations with $\pm$ sign for \eqref{macro equation 1}, we have
\begin{equation}\label{macro equation 5}
\!\!\!\!\!\left\{\begin{aligned}
&\! \partial_t (a^{\varepsilon}_{+}-a^{\varepsilon}_{-}) + \frac{1}{\varepsilon} \nabla_x \cdot G^\varepsilon=0, \\
&\! \partial_t G^\varepsilon + \frac{1}{\varepsilon} \nabla_x(a^{\varepsilon}_{+}-a^{\varepsilon}_{-})+\frac{2}{\varepsilon}\nabla_x \phi^{\varepsilon}
+\frac{1}{\varepsilon} \nabla_x \cdot \Theta ( \{\mathbf{I}-\mathbf{P}\} f^\varepsilon \cdot q_1 )
=  \Big\langle g^\varepsilon\!-\!\frac{1}{\varepsilon^2}Lf^\varepsilon, [v,-v] \mu^{1 / 2} \Big\rangle,
\end{aligned}\right.
\end{equation}
where
\begin{align}\label{G define}
G^\varepsilon:=\langle v\mu^{1/2}, \{\mathbf{I}-\mathbf{P}\}f^{\varepsilon} \cdot q_1\rangle.
\end{align}
Notice that the second equation of \eqref{rVPB} gives
\begin{equation}\label{a equality}
-\Delta_x \phi^\varepsilon=a_{+}^\varepsilon-a_{-}^\varepsilon.
\end{equation}

First of all, let's estimate $c^\varepsilon$. By applying $\partial^\alpha$ to the second equation of \eqref{macro equation 4}, multiplying the identity with $\varepsilon \partial^\alpha \partial_i c^\varepsilon$, and integrating the identity result over $\mathbb{R}_x^3$, one has
\begin{align}\label{c estimate 1}
& \left\|\partial^\alpha \nabla_x c^{\varepsilon}\right\|^2=\sum_{i=1}^3 \left(\frac{1}{\varepsilon} \partial^\alpha \partial_i c^{\varepsilon}, \varepsilon \partial^\alpha \partial_i c^{\varepsilon}\right) \nonumber \\
=\; & -\sum_{i=1}^3 \left(\frac{1}{2}\partial_t \partial^{\alpha} \Lambda_i (\{\mathbf{I}-\mathbf{P}\} f^{\varepsilon} \cdot [1,1]), \varepsilon \partial^\alpha \partial_i c^{\varepsilon}\right)+\sum_{i=1}^3 \left(\frac{1}{2} \partial^\alpha \Lambda_i (g^\varepsilon_{+}+g^\varepsilon_{-}+h^\varepsilon_{+}+h^\varepsilon_{-}), \varepsilon \partial^\alpha \partial_i c^{\varepsilon}\right) \nonumber \\
=\; & -\frac{\d}{{\d} t} \sum_{i=1}^3 \left(\frac{1}{2} \partial^\alpha \Lambda_i (\{\mathbf{I}-\mathbf{P}\} f^{\varepsilon} \cdot[1,1]), \varepsilon \partial^\alpha \partial_i c^{\varepsilon}\right)+\sum_{i=1}^3 \left(\frac{1}{2} \partial^\alpha \Lambda_i (\{\mathbf{I}-\mathbf{P}\} f^{\varepsilon} \cdot[1,1]), \varepsilon \partial^\alpha \partial_i \partial_t c^{\varepsilon}\right) \nonumber \\
& +\sum_{i=1}^3 \left(\frac{1}{2} \partial^\alpha \Lambda_i (g^\varepsilon_{+}+g^\varepsilon_{-}+h^\varepsilon_{+}+h^\varepsilon_{-}), \varepsilon \partial^\alpha \partial_i c^{\varepsilon}\right).
\end{align}
For the second term on the right-hand side of the second equality in \eqref{c estimate 1}, we obtain from the third equation of \eqref{macro equation 3} that
\begin{align}
\begin{split}\nonumber
&\sum_{i=1}^3 \left(\frac{1}{2} \partial^\alpha \Lambda_i (\{\mathbf{I}-\mathbf{P}\} f^{\varepsilon} \cdot[1,1]), \varepsilon \partial^\alpha \partial_i \partial_t c^{\varepsilon}\right) \\
=\;&\sum_{i=1}^3 \bigg(
   \varepsilon \partial^\alpha \partial_i \Big\{ -\frac{1}{3\varepsilon} \nabla_x \cdot b^\varepsilon
-\frac{5}{6\varepsilon} \sum_{j=1}^{3} \partial_j \Lambda_j ( \{\mathbf{I}-\mathbf{P}\} f^\varepsilon \cdot [1,1] )
+\frac{1}{12}\langle g^\varepsilon_{+}+g^\varepsilon_{-}, (|v|^2-3) \mu^{1 / 2} \rangle \Big\},\\
 &\qquad\;\; \frac{1}{2} \partial^\alpha \Lambda_i (\{\mathbf{I}-\mathbf{P}\} f^{\varepsilon} \cdot[1,1]) \bigg) \\
\lesssim\;&
\eta \left\|\partial^\alpha \nabla_x b^\varepsilon \right\|^2+\left\| \partial^\alpha \nabla_x \{\mathbf{I}-\mathbf{P}\}f^\varepsilon \right\|_{N^s_\gamma}^2+ \varepsilon^2 \left\| \langle \partial^\alpha (g^\varepsilon_{+}+g^\varepsilon_{-}), \zeta \rangle \right\|^2_{L^2_x},
\end{split}
\end{align}
where we have employed integration by parts with respect to $x_i$, \eqref{zeta} and the Cauchy--Schwarz inequality with $\eta$.
For the last term on the right-hand side of \eqref{c estimate 1}, we have
\begin{align}
\begin{split}
&\sum_{i=1}^3 \left(\frac{1}{2} \partial^\alpha \Lambda_i (g^\varepsilon_{+}+g^\varepsilon_{-}+h^\varepsilon_{+}+h^\varepsilon_{-}), \varepsilon \partial^\alpha \partial_i c^{\varepsilon}\right)\\
\lesssim\;&\eta \left\| \partial^\alpha \nabla_x c^{\varepsilon}\right\|^2+ \varepsilon^2\left\| \langle \partial^\alpha (g^\varepsilon_{+}+g^\varepsilon_{-}+h^\varepsilon_{+}+h^\varepsilon_{-}), \zeta \rangle \right\|^2_{L^2_x}.
\end{split}
\end{align}
As a consequence, we can obtain
\begin{align}\label{c estimate 2}
\begin{split}
&\varepsilon \frac{\d }{\d t} \mathcal{E}^{(\alpha)}_{c^\varepsilon}(t) + \left\|\partial^\alpha \nabla_x c^{\varepsilon}\right\|^2 \\
\lesssim \;& \eta \left\|\partial^\alpha \nabla_x b^\varepsilon \right\|^2+\left\| \partial^\alpha \nabla_x \{\mathbf{I}-\mathbf{P}\}f^\varepsilon \right\|_{N^s_\gamma}^2+ \varepsilon^2 \left\| \langle \partial^\alpha (g^\varepsilon_{+}+g^\varepsilon_{-}+h^\varepsilon_{+}+h^\varepsilon_{-}), \zeta \rangle \right\|^2_{L^2_x},
\end{split}
\end{align}
where $\mathcal{E}^{(\alpha)}_{c^\varepsilon}(t)$ is given by
$$
\mathcal{E}^{(\alpha)}_{c^\varepsilon}(t) := \sum_{i=1}^3 \left(\frac{1}{2} \partial^{\alpha} \Lambda_i (\{\mathbf{I}-\mathbf{P}\} f^{\varepsilon} \cdot [1,1]), \partial^\alpha \partial_i c^{\varepsilon}\right).
$$

Secondly, for the estimate of $b^\varepsilon$, applying $\partial^\alpha$ to the first equation of \eqref{macro equation 4}, multiplying it with $\varepsilon ( \partial^\alpha \partial_i b_j^{\varepsilon}+\partial^\alpha \partial_j b_i^{\varepsilon})$, and integrating the resulting equation over $\mathbb{R}^3_x$, we can derive
\begin{align}\label{b estimate 1}
&2\left\| \partial^\alpha \nabla_x b^\varepsilon \right\|^2+2\left\| \partial^\alpha \nabla_x \cdot b^\varepsilon\right\|^2
=\sum_{i,j=1}^3\left(\frac{1}{\varepsilon}( \partial^\alpha \partial_i b^\varepsilon_j +  \partial^\alpha \partial_j b^\varepsilon_i),
\varepsilon (\partial^\alpha \partial_i b^\varepsilon_j + \partial^\alpha \partial_j b^\varepsilon_i) \right) \nonumber\\
=\;&\sum_{i,j=1}^3\left(- \partial_t \Big(\frac{1}{2}\partial^\alpha \Theta_{i j}(\{\mathbf{I}-\mathbf{P}\} f^{\varepsilon} \cdot [1,1])
+2 \partial^\alpha c^{\varepsilon}\delta_{ij}\Big), \varepsilon (\partial^\alpha \partial_i b^\varepsilon_j + \partial^\alpha \partial_j b^\varepsilon_i) \right) \nonumber \\
&+\sum_{i,j=1}^3\left( \frac{1}{2}\partial^\alpha \Theta_{i j}(g^\varepsilon_{+}+g^\varepsilon_{-}+h^\varepsilon_{+}+h^\varepsilon_{-}),
\varepsilon (\partial^\alpha \partial_i b^\varepsilon_j + \partial^\alpha \partial_j b^\varepsilon_i) \right) \nonumber \\
=\;&-\frac{\d}{\d t}  \sum_{i,j=1}^3\left( \frac{1}{2}\partial^\alpha \Theta_{i j}(\{\mathbf{I}-\mathbf{P}\} f^{\varepsilon} \cdot [1,1])
+2 \partial^\alpha c^{\varepsilon}\delta_{ij}, \varepsilon (\partial^\alpha \partial_i b^\varepsilon_j + \partial^\alpha \partial_j b^\varepsilon_i) \right)\\
&+\sum_{i,j=1}^3\left( \frac{1}{2}\partial^\alpha \Theta_{i j}(\{\mathbf{I}-\mathbf{P}\} f^{\varepsilon} \cdot [1,1])
+2 \partial^\alpha c^{\varepsilon}\delta_{ij}, \varepsilon (\partial^\alpha \partial_i \partial_t b^\varepsilon_j + \partial^\alpha \partial_j \partial_t b^\varepsilon_i) \right) \nonumber\\
&+\sum_{i,j=1}^3\left( \frac{1}{2}\partial^\alpha \Theta_{i j}(g^\varepsilon_{+}+g^\varepsilon_{-}+h^\varepsilon_{+}+h^\varepsilon_{-}),
\varepsilon (\partial^\alpha \partial_i b^\varepsilon_j + \partial^\alpha \partial_j b^\varepsilon_i) \right).  \nonumber
\end{align}
For the second term on the right-hand side of \eqref{b estimate 1}, it follows from the second equation of \eqref{macro equation 3} that
\begin{align}
\begin{split}\nonumber
&\sum_{i,j=1}^3\left( \frac{1}{2}\partial^\alpha \Theta_{i j}(\{\mathbf{I}-\mathbf{P}\} f^{\varepsilon} \cdot [1,1])
+2 \partial^\alpha c^{\varepsilon}\delta_{ij} , \varepsilon (\partial^\alpha \partial_i \partial_t b^\varepsilon_j + \partial^\alpha \partial_j \partial_t b^\varepsilon_i) \right)\\
\lesssim\;& \eta \left\|\partial^\alpha \nabla_x (a^\varepsilon_{+}+a^\varepsilon_{-})\right\|^2+\left\|\partial^\alpha \nabla_x c^\varepsilon \right\|^2
+\left\| \partial^\alpha \nabla_x \{\mathbf{I}-\mathbf{P}\}f^\varepsilon \right\|_{N^s_\gamma}^2
+\varepsilon^2\left\| \langle \partial^\alpha (g^\varepsilon_{+}+g^\varepsilon_{-}), \zeta \rangle \right\|^2_{L^2_x}.
\end{split}
\end{align}
For the last term on the right-hand side of \eqref{b estimate 1}, we can deduce that
\begin{align}
\begin{split}\nonumber
&\sum_{i,j=1}^3\left( \frac{1}{2}\partial^\alpha \Theta_{i j}(g^\varepsilon_{+}+g^\varepsilon_{-}+h^\varepsilon_{+}+h^\varepsilon_{-}),
\varepsilon (\partial^\alpha \partial_i b^\varepsilon_j + \partial^\alpha \partial_j b^\varepsilon_i) \right)\\
\lesssim\;& \eta \left\|\partial^\alpha \nabla_x  b^\varepsilon \right\|^2+\varepsilon^2 \left\| \langle \partial^\alpha (g^\varepsilon_{+}+g^\varepsilon_{-}+h^\varepsilon_{+}+h^\varepsilon_{-}), \zeta \rangle \right\|^2_{L^2_x}.
\end{split}
\end{align}
Consequently, one has
\begin{align}\label{b estimate 2}
\begin{split}
&\varepsilon\frac{\d}{\d t}  \mathcal{E}^{(\alpha)}_{b^\varepsilon}(t)+ \left\| \partial^\alpha \nabla_x b^\varepsilon \right\|^2+ \left\| \partial^\alpha \nabla_x \cdot b^\varepsilon \right\|^2 \\
\lesssim\; &\eta \left\|\partial^\alpha \nabla_x (a^\varepsilon_{+}+a^\varepsilon_{-}) \right\|^2+ \left\|\partial^\alpha \nabla_x c^\varepsilon \right\|^2+ \left\| \partial^\alpha \nabla_x \{\mathbf{I}-\mathbf{P}\}f^\varepsilon \right\|_{N^s_\gamma}^2\\
&+ \varepsilon^2 \left\| \langle \partial^\alpha (g^\varepsilon_{+}+g^\varepsilon_{-}+h^\varepsilon_{+}+h^\varepsilon_{-}), \zeta \rangle \right\|^2_{L^2_x},
\end{split}
\end{align}
where $\mathcal{E}^{(\alpha)}_{b^\varepsilon}(t)$ is defined as
$$
\mathcal{E}^{(\alpha)}_{b^\varepsilon}(t) := \sum_{i,j=1}^3\left( \frac{1}{2}\partial^\alpha \Theta_{i j}(\{\mathbf{I}-\mathbf{P}\} f^{\varepsilon} \cdot [1,1])+2 \partial^\alpha c^{\varepsilon}\delta_{ij} , \partial^\alpha \partial_i b^\varepsilon_j + \partial^\alpha \partial_j b^\varepsilon_i \right).
$$

Next, for $a_{+}^\varepsilon+a_{-}^\varepsilon$, applying $\partial^\alpha$ to the second equation of \eqref{macro equation 3}, multiplying the identity by $2 \varepsilon \partial^\alpha \partial_i (a^\varepsilon_{+}+a^\varepsilon_{-})$, and integrating the identity result over $\mathbb{R}^3_x$, we arrive at
\begin{align}\label{a estimate 1}
&\left\| \partial^\alpha \nabla_x (a^\varepsilon_{+}+a^\varepsilon_{-}) \right\|^2
= \sum_{i=1}^3 \left( \frac{1}{\varepsilon} \partial^\alpha \partial_i \Big(\frac{a^\varepsilon_{+}+a^\varepsilon_{-}}{2}\Big), 2\varepsilon \partial^\alpha \partial_i (a^\varepsilon_{+}+a^\varepsilon_{-})\right)  \nonumber\\
=\;&\sum_{i=1}^3 \bigg( -\partial_t \partial^\alpha b^{\varepsilon}_i-\frac{2}{\varepsilon} \partial^\alpha \partial_i c^{\varepsilon}
-\frac{1}{2\varepsilon} \sum_{j=1}^{3} \partial^\alpha \partial_j \Theta_{ij} ( \{\mathbf{I}-\mathbf{P}\} f^\varepsilon \cdot [1,1] ),
2\varepsilon \partial^\alpha \partial_i (a^\varepsilon_{+}+a^\varepsilon_{-})\bigg)  \nonumber\\
&+\sum_{i=1}^3 \bigg( \frac{1}{2} \langle \partial^\alpha( g^\varepsilon_{+}+g^\varepsilon_{-}), v_i \mu^{1 / 2}\rangle,
2\varepsilon \partial^\alpha \partial_i (a^\varepsilon_{+}+a^\varepsilon_{-})\bigg)  \\
=\;&-\frac{\d}{\d t} \sum_{i=1}^3 \left(  \partial^\alpha b^{\varepsilon}_i, 2\varepsilon \partial^\alpha \partial_i (a^\varepsilon_{+}+a^\varepsilon_{-})\right) + \sum_{i=1}^3 \left(  \partial^\alpha b^{\varepsilon}_i, 2\varepsilon \partial^\alpha \partial_i \partial_t (a^\varepsilon_{+}+a^\varepsilon_{-})\right) \nonumber\\
&+\sum_{i=1}^3 \bigg(-\frac{2}{\varepsilon} \partial^\alpha \partial_i c^{\varepsilon}-\frac{1}{2\varepsilon} \sum_{j=1}^{3} \partial^\alpha \partial_j \Theta_{ij} ( \{\mathbf{I}-\mathbf{P}\} f^\varepsilon \cdot [1,1] ), 2\varepsilon \partial^\alpha \partial_i (a^\varepsilon_{+}+a^\varepsilon_{-})\bigg) \nonumber\\
&+\sum_{i=1}^3 \bigg( \frac{1}{2} \langle \partial^\alpha( g^\varepsilon_{+}+g^\varepsilon_{-}), v_i \mu^{1 / 2}\rangle,
2\varepsilon \partial^\alpha \partial_i (a^\varepsilon_{+}+a^\varepsilon_{-})\bigg). \nonumber
\end{align}
For the  second term on the right-hand side of \eqref{a estimate 1}, we can deduce from the first equation of \eqref{macro equation 3} that
\begin{align}\nonumber
\sum_{i=1}^3 \left(  \partial^\alpha b^{\varepsilon}_i, 2\varepsilon \partial^\alpha \partial_i \partial_t (a^\varepsilon_{+}+a^\varepsilon_{-})\right)=4\left\|\partial^\alpha \nabla_x \cdot b^\varepsilon\right\|^2.
\end{align}
Other terms can be dominated by
$$
\eta \left\|\partial^\alpha \nabla_x (a^\varepsilon_{+}+a^\varepsilon_{-}) \right\|^2
+\left\|\partial^\alpha \nabla_x c^\varepsilon \right\|^2
+\left\|\partial^\alpha \nabla_x \{\mathbf{I}-\mathbf{P}\} f^\varepsilon \right\|^2_{N^s_\gamma}
+\varepsilon^2 \left\| \langle \partial^\alpha (g^\varepsilon_{+}+g^\varepsilon_{-}), \zeta \rangle \right\|^2_{L^2_x}.
$$
Therefore, we have
\begin{align}\label{a estimate 2}
\begin{split}
&\varepsilon\frac{\d}{\d t} \mathcal{E}^{(\alpha)}_{a^\varepsilon_{+}+a^\varepsilon_{-}}(t)+ \left\| \partial^\alpha \nabla_x (a^\varepsilon_{+}+a^\varepsilon_{-}) \right\|^2\\
\lesssim \;& \left\|\partial^\alpha \nabla_x \cdot b^\varepsilon\right\|^2
+ \left\|\partial^\alpha \nabla_x c^\varepsilon\right\|^2
+ \left\|\partial^\alpha \nabla_x \{\mathbf{I}-\mathbf{P}\} f^\varepsilon \right\|^2_{N^s_\gamma}
+\varepsilon^2 \left\| \langle \partial^\alpha (g^\varepsilon_{+}+g^\varepsilon_{-}), \zeta \rangle \right\|^2_{L^2_x},
\end{split}
\end{align}
where $\mathcal{E}^{(\alpha)}_{a^\varepsilon_{+}+a^\varepsilon_{-}}(t)$ is defined by
$$
\mathcal{E}^{(\alpha)}_{a^\varepsilon_{+}+a^\varepsilon_{-}}(t) := \sum_{i=1}^3 \left(  \partial^\alpha b^{\varepsilon}_i, 2 \partial^\alpha \partial_i (a^\varepsilon_{+}+a^\varepsilon_{-})\right).
$$
Defining
$$
\mathcal{E}^{(\alpha)}_{a^\varepsilon_{+}+a^\varepsilon_{-},b^\varepsilon,c^\varepsilon}(t):=
\mathcal{E}^{(\alpha)}_{c^\varepsilon}(t)+ \kappa_1 \mathcal{E}^{(\alpha)}_{b^\varepsilon}(t)+
\kappa_2 \mathcal{E}^{(\alpha)}_{a^\varepsilon_{+}+a^\varepsilon_{-}}(t),\quad 0 < \kappa_2 \ll \kappa_1 \ll 1,
$$
we can obtain from \eqref{c estimate 2}, \eqref{b estimate 2} and \eqref{a estimate 2} that
\begin{align} \label{macro estimate 1}
\begin{split}
&\varepsilon \frac{\d}{\d t} \mathcal{E}^{(\alpha)}_{a^\varepsilon_{+}+a^\varepsilon_{-},b^\varepsilon,c^\varepsilon}(t)
+\left\|\partial^\alpha \nabla_x (a^\varepsilon_{+}+a^\varepsilon_{-},b^\varepsilon,c^\varepsilon)\right\|^2 \\
\lesssim \;& \left\|\partial^\alpha \nabla_x \{\mathbf{I}-\mathbf{P}\} f^\varepsilon \right\|^2_{N^s_\gamma}
+\varepsilon^2\left\| \langle \partial^\alpha (g^\varepsilon_{+}+g^\varepsilon_{-}+h^\varepsilon_{+}+h^\varepsilon_{-}), \zeta \rangle \right\|^2_{L^2_x}.
\end{split}
\end{align}

Moreover, for $a^\varepsilon_{+}-a^\varepsilon_{-}$, applying $\partial^\alpha$ to the second equation of \eqref{macro equation 5}, multiplying the equality with $\varepsilon \partial^\alpha \nabla_x (a^\varepsilon_{+}-a^\varepsilon_{-})$, and integrating the identity result over $\mathbb{R}_x^3$, one has
\begin{align}\label{a estimate 3}
&\left\|\partial^\alpha \nabla_x (a^\varepsilon_{+}-a^\varepsilon_{-})\right\|^2+2\left\|\partial^\alpha (a^\varepsilon_{+}-a^\varepsilon_{-})\right\|^2 \nonumber\\
=\;&\left( \frac{1}{\varepsilon} \big(\partial^\alpha \nabla_x (a^\varepsilon_{+}-a^\varepsilon_{-})+2\partial^\alpha \nabla_x \phi^{\varepsilon}\big), \varepsilon \partial^\alpha \nabla_x (a^\varepsilon_{+}-a^\varepsilon_{-})\right)\\
=\;&\left(-\partial^\alpha \partial_t G^\varepsilon
-\frac{1}{\varepsilon} \partial^\alpha \nabla_x \cdot \Theta ( \{\mathbf{I}-\mathbf{P}\} f^\varepsilon \cdot q_1 )
+\Big\langle \partial^\alpha g^\varepsilon-\frac{1}{\varepsilon^2}L\partial^\alpha f^\varepsilon, [v,-v] \mu^{1 / 2}\Big\rangle,
\varepsilon \partial^\alpha \nabla_x (a^\varepsilon_{+}-a^\varepsilon_{-})\right)\nonumber\\
=\;& -\frac{\d}{\d t} \left(\partial^\alpha G^\varepsilon, \varepsilon \partial^\alpha \nabla_x (a^\varepsilon_{+}-a^\varepsilon_{-})\right)
+\left(\partial^\alpha G^\varepsilon, \varepsilon \partial^\alpha \nabla_x \partial_t(a^\varepsilon_{+}-a^\varepsilon_{-})\right)\nonumber\\
&+\left(-\frac{1}{\varepsilon} \partial^\alpha \nabla_x \cdot \Theta ( \{\mathbf{I}-\mathbf{P}\} f^\varepsilon \cdot q_1 )
+\Big\langle \partial^\alpha g^\varepsilon-\frac{1}{\varepsilon^2}L\partial^\alpha f^\varepsilon, [v,-v] \mu^{1 / 2}\Big\rangle,
\varepsilon \partial^\alpha \nabla_x (a^\varepsilon_{+}-a^\varepsilon_{-})\right). \nonumber
\end{align}
For the second term on the right-hand side of \eqref{a estimate 3}, we deduce from the first equation of \eqref{macro equation 5} that
\begin{align}
\left(\partial^\alpha G^\varepsilon, \varepsilon \partial^\alpha \nabla_x \partial_t(a^\varepsilon_{+}-a^\varepsilon_{-})\right)
=\left\|\partial^\alpha \nabla_x \cdot G^\varepsilon \right\|^2 \lesssim \left\| \partial^\alpha \nabla_x \{\mathbf{I}-\mathbf{P}\}f^\varepsilon \right\|^2_{N^s_\gamma}.
\end{align}
The last two terms on the right-hand side of \eqref{a estimate 3} can be controlled by
$$
\eta \left\|\partial^\alpha \nabla_x (a^\varepsilon_{+}-a^\varepsilon_{-})\right\|^2
+ \left\| \partial^\alpha \nabla_x \{\mathbf{I}-\mathbf{P}\}f^\varepsilon\right\|^2_{N^s_\gamma}
+\frac{1}{\varepsilon^2} \left\| \partial^\alpha \{\mathbf{I}-\mathbf{P}\}f^\varepsilon\right\|^2_{N^s_\gamma}
+\varepsilon^2 \left\| \langle \partial^\alpha (g^\varepsilon_{+}+g^\varepsilon_{-}), \zeta \rangle \right\|^2_{L^2_x}.
$$
As a consequence, we have
\begin{align}\label{macro estimate 2}
\begin{split}
&\varepsilon\frac{\d}{\d t} \mathcal{E}^{(\alpha)}_{a^\varepsilon_{+}-a^\varepsilon_{-}}(t)
+\left\|\partial^\alpha \nabla_x (a^\varepsilon_{+}-a^\varepsilon_{-})\|^2+\|\partial^\alpha (a^\varepsilon_{+}-a^\varepsilon_{-})\right\|^2 \\
\lesssim\;&\left\| \partial^\alpha \nabla_x \{\mathbf{I}-\mathbf{P}\}f^\varepsilon\right\|^2_{N^s_\gamma}
+\frac{1}{\varepsilon^2}\left\| \partial^\alpha \{\mathbf{I}-\mathbf{P}\}f^\varepsilon\right\|^2_{N^s_\gamma}
+\varepsilon^2\left\| \langle \partial^\alpha (g^\varepsilon_{+}+g^\varepsilon_{-}), \zeta \rangle \right\|^2_{L^2_x}.
\end{split}
\end{align}
Here,
$$
\mathcal{E}^{(\alpha)}_{a^\varepsilon_{+}-a^\varepsilon_{-}}(t) := \left(\partial^\alpha G^\varepsilon, \partial^\alpha \nabla_x (a^\varepsilon_{+}-a^\varepsilon_{-})\right).
$$

Finally, to estimate $\nabla_x \phi^\varepsilon$, applying $\partial^\alpha$ to the second equation of \eqref{macro equation 5}, multiplying it by $\varepsilon \partial^\alpha \nabla_x \phi^\varepsilon$, and integrating the resulting equation over $\mathbb{R}_x^3$, we have
\begin{align}\label{phi estimate 2}
\begin{split}
&2\left\|\partial^\alpha \nabla_x \phi^\varepsilon\right\|^2
=\left(\frac{2}{\varepsilon} \partial^\alpha \nabla_x \phi^\varepsilon, \varepsilon \partial^\alpha \nabla_x \phi^\varepsilon\right)\\
=\;& \left( -\partial^\alpha \partial_t G^\varepsilon- \frac{1}{\varepsilon} \partial^\alpha \nabla_x(a^{\varepsilon}_{+}-a^{\varepsilon}_{-})-\frac{1}{\varepsilon} \partial^\alpha \nabla_x \cdot \Theta ( \{\mathbf{I}-\mathbf{P}\} f^\varepsilon \cdot q_1 ) , \varepsilon \partial^\alpha \nabla_x \phi^\varepsilon\right)\\
&+\left( \Big\langle \partial^\alpha g^\varepsilon-\frac{1}{\varepsilon^2}L\partial^\alpha f^\varepsilon, [v,-v] \mu^{1 / 2} \Big\rangle,
\varepsilon \partial^\alpha \nabla_x \phi^\varepsilon\right)\\
=\;& -\frac{\d }{\d t}\left( \partial^\alpha G^\varepsilon, \varepsilon \partial^\alpha \nabla_x \phi^\varepsilon\right)
+\left( \partial^\alpha G^\varepsilon, \varepsilon \partial^\alpha \partial_t \nabla_x \phi^\varepsilon\right)\\
&+\left(- \frac{1}{\varepsilon} \partial^\alpha \nabla_x(a^{\varepsilon}_{+}-a^{\varepsilon}_{-})-\frac{1}{\varepsilon} \partial^\alpha \nabla_x \cdot \Theta ( \{\mathbf{I}-\mathbf{P}\} f^\varepsilon \cdot q_1 ) , \varepsilon \partial^\alpha \nabla_x \phi^\varepsilon\right)\\
&+\left( \Big\langle \partial^\alpha g^\varepsilon-\frac{1}{\varepsilon^2}L\partial^\alpha f^\varepsilon, [v,-v] \mu^{1 / 2} \Big\rangle,
\varepsilon \partial^\alpha \nabla_x \phi^\varepsilon\right).
\end{split}
\end{align}
For the second term on the right-hand side of \eqref{phi estimate 2}, we can deduce from the first equation of \eqref{macro equation 5} and \eqref{a equality} that
$$
\left( \partial^\alpha G^\varepsilon, \varepsilon \partial^\alpha \partial_t \nabla_x \phi^\varepsilon\right) \lesssim \|\partial^\alpha G^\varepsilon\|^2
\lesssim \| \partial^\alpha \{\mathbf{I}-\mathbf{P}\}f^\varepsilon\|^2_{N^s_\gamma}.
$$
Other terms can be controlled by
\begin{align}
\begin{split}\nonumber
&\eta \left\|\partial^\alpha \nabla_x \phi^\varepsilon\right\|^2
+\left\|\partial^\alpha \nabla_x(a^{\varepsilon}_{+}-a^{\varepsilon}_{-})\right\|^2
+\left\| \partial^\alpha \nabla_x \{\mathbf{I}-\mathbf{P}\}f^\varepsilon\right\|^2_{N^s_\gamma} \\
&+\frac{1}{\varepsilon^2} \left\| \partial^\alpha \{\mathbf{I}-\mathbf{P}\}f^\varepsilon\right\|^2_{N^s_\gamma}
+\varepsilon^2\left\| \langle \partial^\alpha (g^\varepsilon_{+}+g^\varepsilon_{-}), \zeta \rangle \right\|^2_{L^2_x}.
\end{split}
\end{align}
Thereby, we can get
\begin{align}\label{macro estimate 3}
\begin{split}
\varepsilon \frac{\d }{\d t} \mathcal{E}^{(\alpha)}_{\phi^\varepsilon}(t)
+\left\|\partial^\alpha \nabla_x \phi^\varepsilon\right\|^2
\lesssim\;& \left\|\partial^\alpha \nabla_x(a^{\varepsilon}_{+}-a^{\varepsilon}_{-})\right\|^2
+\left\| \partial^\alpha \nabla_x \{\mathbf{I}-\mathbf{P}\}f^\varepsilon\right\|^2_{N^s_\gamma} \\
&+\frac{1}{\varepsilon^2} \left\| \partial^\alpha \{\mathbf{I}-\mathbf{P}\}f^\varepsilon\right\|^2_{N^s_\gamma}
+\varepsilon^2\left\| \langle \partial^\alpha (g^\varepsilon_{+}+g^\varepsilon_{-}), \zeta \rangle \right\|^2_{L^2_x},
\end{split}
\end{align}
where $\mathcal{E}^{(\alpha)}_{\phi^\varepsilon}(t)$ is given by
$$
\mathcal{E}^{(\alpha)}_{\phi^\varepsilon}(t) := \left( \partial^\alpha G^\varepsilon,  \partial^\alpha \nabla_x \phi^\varepsilon\right).
$$
Setting
$$
\mathcal{E}^{(\alpha)}_{int}(t) := \mathcal{E}^{(\alpha)}_{a^\varepsilon_{+}
+a^\varepsilon_{-},b^\varepsilon,c^\varepsilon}(t)+\mathcal{E}^{(\alpha)}_{a^\varepsilon_{+}-a^\varepsilon_{-}}(t)
+\kappa_3 \mathcal{E} ^{(\alpha)}_{\phi^\varepsilon}(t), \quad 0 < \kappa_3 \ll 1,
$$
it follows from \eqref{macro estimate 1}, \eqref{macro estimate 2} and \eqref{macro estimate 3} that
\begin{align}\label{macro energy1}
\begin{split}
&\varepsilon \frac{\d}{\d t} \mathcal{E}^{(\alpha)}_{int}(t)
+ \left\|\partial^\alpha \nabla_x (\alpha^\varepsilon_{\pm}, b^\varepsilon, c^\varepsilon)\right\|^2
+ \left\|\partial^\alpha \nabla_x \phi^\varepsilon\right\|^2
+ \left\|\partial^\alpha (a^\varepsilon_{+}-a^\varepsilon_{-})\right\|^2
+ \left\|\partial^\alpha \nabla_x (a^\varepsilon_{+}-a^\varepsilon_{-})\right\|^2 \\
\lesssim \;& \left\| \partial^\alpha \nabla_x \{\mathbf{I}-\mathbf{P}\}f^\varepsilon\right\|^2_{N^s_\gamma}
+\frac{1}{\varepsilon^2}\!\left\| \partial^\alpha \{\mathbf{I}-\mathbf{P}\}f^\varepsilon\right\|^2_{N^s_\gamma}
+\varepsilon^2\! \left\| \langle \partial^\alpha (g^\varepsilon_{+}+g^\varepsilon_{-}+h^\varepsilon_{+}+h^\varepsilon_{-}), \zeta \rangle \right\|^2_{L^2_x}.
\end{split}
\end{align}

Now, we define
\begin{align}\label{macro energy2}
\mathcal{E}^{N}_{int}(t):= \sum_{|\alpha| \leq N-1}\mathcal{E}^{(\alpha)}_{int}(t).
\end{align}
By referring back to \eqref{g,h define} for $g^\varepsilon_{\pm}$ and $h^\varepsilon_{\pm}$, and employing \eqref{Gamma zeta1}, we obtain
\begin{align}\label{g estimate}
&\sum_{|\alpha| \leq N-1}\varepsilon^2\left\| \langle \partial^\alpha (g^\varepsilon_{+}+g^\varepsilon_{-}), \zeta \rangle \right\|^2_{L^2_x}
\lesssim \mathcal{E}_{N}(t)\mathcal{D}_{N}(t),\\
&\sum_{|\alpha| \leq N-1}\varepsilon^2 \left\| \langle \partial^\alpha (h^\varepsilon_{+}+h^\varepsilon_{-}), \zeta \rangle \right\|^2_{L^2_x}
\lesssim \frac{1}{\varepsilon^2}\sum_{|\alpha|\leq N} \left\| \partial^\alpha \{\mathbf{I}-\mathbf{P}\}f^\varepsilon \right\|^2_{N^s_\gamma}. \label{h estimate}
\end{align}
Hence, we can obtain \eqref{macro estimate} from \eqref{a equality} and \eqref{macro energy1}--\eqref{h estimate}.

Furthermore, we set
\begin{align}\label{macro energy high}
\mathcal{E}_{int}^{N,h}(t) := \sum_{1 \leq |\alpha|\leq N-1}\left\{\mathcal{E}^{(\alpha)}_{a^\varepsilon_{+}
+a^\varepsilon_{-},b^\varepsilon,c^\varepsilon}(t)+\mathcal{E}^{(\alpha)}_{a^\varepsilon_{+}-a^\varepsilon_{-}}(t)\right\}
+\kappa_3 \sum_{|\alpha|\leq N-1}\mathcal{E} ^{(\alpha)}_{\phi^\varepsilon}(t), \quad 0 < \kappa_3 \ll 1.
\end{align}
As a result, combining \eqref{a equality}, \eqref{macro estimate 1}, \eqref{macro estimate 2}, \eqref{macro estimate 3} with \eqref{g estimate}--\eqref{macro energy high}, we conclude \eqref{macro estimate high}.
This completes the proof of Lemma \ref{macroscopic estimate}.
\end{proof}
\medskip

\subsection{Energy Estimate Without Weight}
\hspace*{\fill}

This subsection provides the energy estimate without weight, for both hard and soft potentials.
\begin{lemma}\label{energy estimate without weight}
There holds
\begin{align}\label{estimate without weight}
\begin{split}
&\frac{\d}{\d t}\bigg\{\sum_{|\alpha| \leq N}\sum_{\pm}\left\|e^{\pm \frac{\varepsilon \phi^\varepsilon}{2}}\partial^\alpha f_{\pm}^\varepsilon\right\|^2
+\sum_{|\alpha| \leq N}\left\|\partial^\alpha \nabla_x\phi^\varepsilon\right\|^2\bigg\}+\frac{\lambda}{\varepsilon^2}
\sum_{|\alpha| \leq N}\left\|\partial^\alpha \{\mathbf{I}-\mathbf{P}\}f^\varepsilon\right\|^2_{N^s_\gamma} \\
\lesssim\;& \varepsilon \left\|\partial_t \phi^\varepsilon \right\|_{L^\infty} \sum_{|\alpha| \leq N}\sum_{\pm}\left\|e^{\pm \frac{\varepsilon \phi^\varepsilon}{2}}\partial^\alpha f_{\pm}^\varepsilon\right\|^2 +\left\{\mathcal{E}^{1/2}_{N,l}(t)+\mathcal{E}_{N,l}(t)\right\}\mathcal{D}_{N,l}(t).
\end{split}
\end{align}
\end{lemma}

\begin{proof}
Notice that the first equation of \eqref{rVPB} can be rewritten as
\begin{align}\label{rrVPB}
\partial_t f_{\pm}^\varepsilon + \frac{1}{\varepsilon}v_i \partial^{e_i} f_{ \pm}^\varepsilon
\mp \partial^{e_i} \phi^\varepsilon \partial_{e_i} f_{ \pm}^\varepsilon
\pm \frac{1}{2}\partial^{e_i} \phi^\varepsilon v_i f_{ \pm}^\varepsilon
\pm\frac{1}{\varepsilon}\partial^{e_i} \phi^\varepsilon v_i \mu^{1/2}+\frac{1}{\varepsilon^2}L_{ \pm} f^\varepsilon=\frac{1}{\varepsilon}\Gamma_{ \pm}(f^\varepsilon, f^\varepsilon) .
\end{align}
Applying $\partial^\alpha$ with $| \alpha | \leq N $ to \eqref{rrVPB} and taking the inner product with $ e^{\pm \varepsilon {\phi^\varepsilon}} \partial^\alpha f_{\pm}^\varepsilon $ over $\mathbb{R}_x^3$ $\times$  $\mathbb{R}_v^3$, we obtain
\begin{align}\label{without weight}
&\left(\partial_t \partial^\alpha f^\varepsilon_{ \pm}, e^{ \pm \varepsilon\phi^\varepsilon} \partial^\alpha f^\varepsilon_{ \pm}\right) +\left(\frac{1}{\varepsilon}v_i \partial^{\alpha+e_i} f^\varepsilon_{ \pm}, e^{ \pm \varepsilon \phi^\varepsilon} \partial^\alpha f^\varepsilon_{ \pm}\right)
\pm \frac{1}{2}\left(\partial^{e_i} \phi^\varepsilon v_i \partial^\alpha f^\varepsilon_{ \pm}, e^{ \pm \varepsilon \phi^\varepsilon} \partial^\alpha f^\varepsilon_{ \pm}\right) \nonumber \\
& \pm\left(\frac{1}{\varepsilon}\partial^{\alpha+e_i} \phi^\varepsilon v_i \mu^{1/2}, \partial^\alpha f^\varepsilon_{ \pm}\right)
+\left(\frac{1}{\varepsilon^2}L_{ \pm} \partial^\alpha f^\varepsilon, \partial^\alpha f^\varepsilon_{ \pm}\right)  \nonumber\\
=\;&\mp \frac{1}{2}\chi_{|\alpha|} \sum_{1 \leq\left|\alpha_1\right| \leq|\alpha|} C_\alpha^{\alpha_1}\left(\partial^{\alpha_1+e_i} \phi^\varepsilon v_i \partial^{\alpha-\alpha_1} f^\varepsilon_{ \pm}, e^{ \pm \varepsilon \phi^\varepsilon} \partial^\alpha f^\varepsilon_{ \pm}\right) \\
& \pm \sum_{|\alpha_1| \leq |\alpha|}  C_\alpha^{\alpha_1} \left(\partial^{\alpha_1+e_i} \phi^\varepsilon \partial^{\alpha-\alpha_1}_{e_i}f^\varepsilon_{ \pm}, e^{ \pm \varepsilon \phi^\varepsilon} \partial^\alpha f^\varepsilon_{ \pm}\right)
 \mp \left(\frac{1}{\varepsilon}\partial^{\alpha+e_i} \phi^\varepsilon v_i \mu^{1/2},
\left(e^{ \pm \varepsilon \phi^\varepsilon}-1\right) \partial^\alpha f^\varepsilon_{ \pm}\right)  \nonumber\\
&-\left(\frac{1}{\varepsilon^2}L_{ \pm} \partial^\alpha f^\varepsilon,\left(e^{ \pm \varepsilon \phi^\varepsilon}-1\right)
\partial^\alpha f^\varepsilon_{ \pm}\right)+\left(\frac{1}{\varepsilon}\partial^\alpha \Gamma_{ \pm}(f^\varepsilon, f^\varepsilon), e^{ \pm \varepsilon \phi^\varepsilon} \partial^\alpha f^\varepsilon_{ \pm}\right). \nonumber
\end{align}
Here and hereafter, $\chi_{|\alpha|}=1$ if $|\alpha| > 0$, and $\chi_{|\alpha|}=0$ if $|\alpha| = 0 $.

Now we denote the ten terms on both sides of \eqref{without weight} with summation $\sum_{\pm}$ by $J_{1,1}$ to $J_{1,10}$ and estimate them term by term.
For the first term $J_{1,1}$, we have the following equality
\begin{align}\nonumber
J_{1,1}=\frac{1}{2}\frac{\d}{\d t} \sum_{\pm}\left\|e^{\pm \frac{\varepsilon \phi^\varepsilon}{2}} \partial^\alpha f_{\pm}^\varepsilon \right\|^2
-\frac{1}{2}\sum_{\pm} \left(\pm \varepsilon \partial_t \phi^\varepsilon e^{\pm \varepsilon \phi^\varepsilon} \partial^\alpha f_{\pm}^\varepsilon, \partial^\alpha f_{\pm}^\varepsilon\right).
\end{align}
The second term on the right-hand side of the above equality can be governed by
\begin{align}\nonumber
\varepsilon \left\|\partial_t \phi^\varepsilon \right\|_{L^\infty} \sum_{\pm}\left\|e^{\pm \frac{\varepsilon \phi^\varepsilon}{2}}\partial^\alpha f_{\pm}^\varepsilon\right\|^2.
\end{align}
As in \cite{Guo2012JAMS}, the extra factor $e^{\pm \varepsilon \phi^\varepsilon}$ is designed to treat $J_{1,3}$. In fact, we can deduce from  integration by parts with respect to $x_i$ that
\begin{align}\nonumber
\left(\frac{1}{\varepsilon}v_i \partial^{\alpha+e_i} f^\varepsilon_{ \pm}, e^{ \pm \varepsilon \phi^\varepsilon} \partial^\alpha f^\varepsilon_{ \pm}\right)
\pm \frac{1}{2}\left(\partial^{e_i} \phi^\varepsilon v_i \partial^\alpha f^\varepsilon_{ \pm}, e^{ \pm \varepsilon \phi^\varepsilon} \partial^\alpha f^\varepsilon_{ \pm}\right)=0.
\end{align}
That is $J_{1,2}+J_{1,3}=0$. For the term $J_{1,4}$, recalling the first equation of \eqref{macro equation 5}  and \eqref{a equality}, we have
\begin{align}\nonumber
J_{1,4}=\sum_{\pm} \left(\pm\frac{1}{\varepsilon}\partial^{\alpha+e_i} \phi^\varepsilon v_i \mu^{1/2}, \partial^\alpha f^\varepsilon_{ \pm}\right)
=\frac{1}{2}\frac{\d}{\d t}\left\| \partial^\alpha \nabla_x \phi^\varepsilon \right\|^2.
\end{align}
For the term $J_{1,5}$, it follows from \eqref{L coercive1} that
\begin{align}\nonumber
J_{1,5}=\sum_{\pm}\left(\frac{1}{\varepsilon^2}L_{ \pm} \partial^\alpha f^\varepsilon, \partial^\alpha f^\varepsilon_{ \pm}\right)
=\left( \frac{1}{\varepsilon^2} L\partial^\alpha \{\mathbf{I}-\mathbf{P}\}f^\varepsilon, \partial^\alpha \{\mathbf{I}-\mathbf{P}\}f^\varepsilon\right)
\geq \frac{\lambda}{\varepsilon^2}\left\|\partial^\alpha \{\mathbf{I}-\mathbf{P}\}f^\varepsilon\right\|^2_{N^s_\gamma}.
\end{align}
For the term $J_{1,6}$, thanks to \eqref{softnonlinearterm1} and \eqref{hardnonlinearterm1-1}, one has $J_{1,6} \lesssim \mathcal{E}^{1/2}_{N,l}(t)\mathcal{D}_{N,l}(t)$.
Next, we estimate the term $J_{1,7}$. When $| \alpha_1 |=0$, making use of integration by parts with respect to $v_i$, we get
\begin{align}\nonumber
J_{1,7}=\sum_{\pm} \left(\pm \partial^{e_i} \phi^\varepsilon \partial^{\alpha}_{e_i}f^\varepsilon_{ \pm}, e^{ \pm \varepsilon \phi^\varepsilon} \partial^\alpha f^\varepsilon_{ \pm}\right) =0.
\end{align}
When $|\alpha_1| \geq 1$, using \eqref{softnonlinearterm2-1} and \eqref{hardnonlinearterm2-1}, we obtain $J_{1,7} \lesssim \mathcal{E}^{1/2}_{N,l}(t)\mathcal{D}_{N,l}(t)$. For the term $J_{1,8}$, utilizing \eqref{phi estimate1}, one has
\begin{align}\nonumber
J_{1,8} =\sum_{\pm} \left(\mp\frac{1}{\varepsilon}\partial^{\alpha+e_i} \phi^\varepsilon v_i \mu^{1/2},
\left(e^{ \pm \varepsilon \phi^\varepsilon}-1\right) \partial^\alpha f^\varepsilon_{ \pm}\right)
\lesssim \mathcal{E}^{1/2}_{N,l}(t)\mathcal{D}_{N,l}(t).
\end{align}
For the term $J_{1,9}$, combining \eqref{Gamma zeta2}, \eqref{Gamma zeta3} and \eqref{phi estimate1}, we deduce that
\begin{align}
\begin{split}\nonumber
J_{1,9}=\;&-\sum_{\pm}\left(\frac{1}{\varepsilon^2}L_{ \pm} \partial^\alpha f^\varepsilon,\left(e^{ \pm \varepsilon \phi^\varepsilon}-1\right)
\partial^\alpha f^\varepsilon_{ \pm}\right)\\
\lesssim\;&\frac{1}{\varepsilon}\left\|\partial^\alpha \{\mathbf{I}-\mathbf{P}\}f^\varepsilon \right\|_{N^s_\gamma}
\left\| \nabla_x \phi^\varepsilon\right\|_{H^1}\left\|\partial^\alpha f^\varepsilon\right\|_{N^s_\gamma} \\
\lesssim\;&\frac{1}{\varepsilon}\left\|\partial^\alpha \{\mathbf{I}-\mathbf{P}\}f^\varepsilon \right\|_{N^s_\gamma}
\left\| \nabla_x \phi^\varepsilon\right\|_{H^1} \left(\left\|\partial^\alpha \{\mathbf{I}-\mathbf{P}\} f^\varepsilon\right\|_{N^s_\gamma}
+\left\|\partial^\alpha \mathbf{P}f^\varepsilon\right\|\right).
\end{split}
\end{align}
It should be noted that we treat $\| \mathbf{P} f^\varepsilon\|$ as $\mathcal{E}^{1/2}_{N,l}(t)$ when $| \alpha |=0$. Hence, $J_{1,9} \lesssim \mathcal{E}^{1/2}_{N,l}(t)\mathcal{D}_{N,l}(t)$. As for the term $J_{1,10}$, utilizing \eqref{soft gamma3} and \eqref{hard gamma3}, we obtain that $J_{1,10} \lesssim \left\{\mathcal{E}^{1/2}_{N,l}(t)+\mathcal{E}_{N,l}(t)\right\} \mathcal{D}_{N,l}(t)$.

Collecting all the estimates of $J_{1,1}$--$J_{1,10}$ and taking summation over $|\alpha| \leq N$, we get \eqref{estimate without weight}. This completes the proof of Lemma \ref{energy estimate without weight}.
\end{proof}
\medskip

\subsection{Weighted Energy Estimate}
\hspace*{\fill}

In this subsection, we give the energy estimates with the weight function $w_l(\alpha, \beta)$ for both hard and soft potentials. The following lemma is the energy estimate about the pure spatial derivatives $|\alpha| \leq N-1$.
\begin{lemma}\label{weighted 1}
There holds
\begin{align}\label{weighted estimate1}
&\frac{\d}{\d t} \sum_{|\alpha|\leq N-1} \sum_{\pm}\left\|e^{\pm \frac{\varepsilon \phi^\varepsilon}{2}} w_l(\alpha,0) \partial^\alpha \{\mathbf{I}_\pm-\mathbf{P}_\pm\} f^\varepsilon \right\|^2
+ \frac{\lambda}{\varepsilon^2}\sum_{|\alpha| \leq N-1} \left\| w_l(\alpha,0)\partial^\alpha \{\mathbf{I}-\mathbf{P}\} f^\varepsilon\right\|_{N^s_\gamma}^2 \nonumber\\
\lesssim\;&\left\{ \varepsilon \left\| \partial_t \phi^\varepsilon\right\|_{L^\infty} + \chi_{\gamma}\left\| \nabla_x \phi^\varepsilon \right\|_{L^\infty} \right\} \sum_{|\alpha|\leq N-1} \sum_{\pm}\left\|e^{\pm \frac{\varepsilon \phi^\varepsilon}{2}} w_l(\alpha,0) \partial^\alpha \{\mathbf{I}_\pm-\mathbf{P}_\pm\} f^\varepsilon \right\|^2 \nonumber \\
&+\frac{1}{\varepsilon^2}\sum_{|\alpha| \leq N}\left\|\partial^\alpha \{\mathbf{I}-\mathbf{P}\}f^\varepsilon\right\|^2_{N^s_\gamma} +\sum_{|\alpha|\leq N-1}\!\!\!\left\| \partial^\alpha \nabla_x \phi^\varepsilon \right\|^2
+\sum_{1 \leq |\alpha| \leq N}\!\!\!\left\|\partial^\alpha(a^\varepsilon_{\pm},b^\varepsilon,c^\varepsilon)\right\|^2 \\ &+\left\{\mathcal{E}^{1/2}_{N,l}(t)+\mathcal{E}_{N,l}(t)\right\}\mathcal{D}_{N,l}(t). \nonumber
\end{align}
Here and hereafter, $\chi_{\gamma}=1$ if $ -3 < \gamma < -2s$, and $\chi_{\gamma}=0$ if $\gamma+2s \geq 0$.
\end{lemma}
\begin{proof}
One can rewrite \eqref{rrVPB} as
\begin{align}\label{rrrVPB}
\begin{split}
& {\left[\partial_t+\frac{1}{\varepsilon}v_i \partial^{e_i} \mp \partial^{e_i} \phi^\varepsilon \partial_{e_i}\right]\{\mathbf{I}_{ \pm}-\mathbf{P}_{ \pm}\} f^\varepsilon \pm \frac{1}{2}\partial^{e_i} \phi^\varepsilon v_i \{\mathbf{I}_{ \pm}-\mathbf{P}_{ \pm}\} f^\varepsilon+\frac{1}{\varepsilon^2}L_{ \pm}\{\mathbf{I}-\mathbf{P}\}f^\varepsilon} \\
 = &-\left[\partial_t+\frac{1}{\varepsilon}v_i \partial^{e_i} \mp \partial^{e_i} \phi^\varepsilon \partial_{e_i}\right] \mathbf{P}_{ \pm} f^\varepsilon \mp \frac{1}{2}\partial^{e_i} \phi^\varepsilon v_i \mathbf{P}_{ \pm} f^\varepsilon \mp \frac{1}{\varepsilon}\partial^{e_i} \phi^\varepsilon v_i \mu^{1/2}
+\frac{1}{\varepsilon}\Gamma_{ \pm}(f^\varepsilon, f^\varepsilon).
\end{split}
\end{align}
By applying $\partial^\alpha$ to \eqref{rrrVPB}, taking the inner product with $e^{\pm \varepsilon \phi^\varepsilon} w^2_l(\alpha,0)\partial^\alpha \{\mathbf{I}_\pm-\mathbf{P}_\pm\} f^\varepsilon$ over $\mathbb{R}_x^3 \times \mathbb{R}_v^3$ and summing with $\pm$, one has
\begin{align}
&\sum_{\pm}\left( \partial_t \partial^\alpha \{\mathbf{I}_{\pm}-\mathbf{P}_{\pm}\} f^\varepsilon,
 e^{ \pm \varepsilon \phi^\varepsilon} w_l^2(\alpha,0)\partial^\alpha \{\mathbf{I}_{ \pm}-\mathbf{P}_{ \pm}\}f^\varepsilon\right) \nonumber\\
&+\sum_{ \pm}\left(\frac{1}{\varepsilon^2}L_{\pm}\partial^\alpha \{\mathbf{I}-\mathbf{P}\} f^\varepsilon, w_l^2(\alpha ,0)\partial^\alpha \{ \mathbf{I}_{\pm}-\mathbf{P}_{\pm} \} f^\varepsilon \right) \nonumber\\
=\;&-\sum_{ \pm}\left( \frac{1}{\varepsilon^2} L_{ \pm}\partial^\alpha\{\mathbf{I}-\mathbf{P}\} f^\varepsilon, \left(e^{ \pm \varepsilon \phi^\varepsilon }-1\right) w_l^2(\alpha,0)\partial^\alpha\{\mathbf{I}_{ \pm}-\mathbf{P}_{ \pm}\} f^\varepsilon \right) \nonumber \\
&+ \sum_{ \pm}\sum_{|\alpha_1| \leq |\alpha|} C_\alpha^{\alpha_1} \left( \pm \partial^{\alpha_1+e_i} \phi^\varepsilon \partial^{\alpha-\alpha_1}_{e_i}\{\mathbf{I}_{ \pm}-\mathbf{P}_{ \pm}\} f^\varepsilon, e^{ \pm \varepsilon \phi^\varepsilon} w_l^2(\alpha,0)\partial^\alpha\{\mathbf{I}_{ \pm}-\mathbf{P}_{ \pm}\} f^\varepsilon\right) \nonumber\\
&+\frac{1}{2} \chi_{|\alpha|}\sum_{\pm} \sum_{1 \leq |\alpha_1| \leq |\alpha|} C_\alpha^{\alpha_1} \left(\mp \partial^{\alpha_1+e_i} \phi^\varepsilon v_i \partial^{\alpha-\alpha_1} \{\mathbf{I}_{ \pm}-\mathbf{P}_{ \pm}\}f^\varepsilon, e^{ \pm \varepsilon \phi^\varepsilon} w_l^2(\alpha,0)\partial^\alpha\{\mathbf{I}_{ \pm}-\mathbf{P}_{ \pm}\} f^\varepsilon\right) \nonumber \\
&+  \sum_{ \pm}\left( \mp \frac{1}{\varepsilon} \partial^{\alpha+e_i} \phi^\varepsilon v_i \mu^{1/2}, e^{ \pm \varepsilon \phi^\varepsilon} w_l^2(\alpha,0)\partial^\alpha \{\mathbf{I}_{ \pm}-\mathbf{P}_{ \pm}\} f^\varepsilon\right) \label{with weight 1}\\
&-\sum_{ \pm}\left(\partial^\alpha \left[\partial_t+\frac{1}{\varepsilon}v_i \partial^{e_i} \mp \partial^{e_i} \phi^\varepsilon \partial_{e_i}\right] \mathbf{P}_{ \pm} f^\varepsilon, e^{ \pm \varepsilon \phi^\varepsilon} w_l^2(\alpha,0)\partial^\alpha\{\mathbf{I}_{ \pm}-\mathbf{P}_{ \pm}\} f^\varepsilon\right) \nonumber\\
&+\frac{1}{2}\sum_{\pm}\sum_{|\alpha_1| \leq |\alpha|}C_{\alpha}^{\alpha_1}\left( \mp \partial^{\alpha_1+e_i} \phi^\varepsilon v_i \partial^{\alpha-\alpha_1}\mathbf{P}_{\pm}f^\varepsilon, e^{ \pm \varepsilon \phi^\varepsilon} w_l^2(\alpha,0)\partial^\alpha\{\mathbf{I}_{ \pm}-\mathbf{P}_{ \pm}\} f^\varepsilon \right) \nonumber\\
&+\sum_{ \pm}\left(\frac{1}{\varepsilon}\partial^\alpha\Gamma_{ \pm}(f^\varepsilon, f^\varepsilon), e^{ \pm \varepsilon \phi^\varepsilon} w_l^2(\alpha,0)\partial^\alpha\{\mathbf{I}_{ \pm}-\mathbf{P}_{ \pm}\} f^\varepsilon \right) \nonumber
\end{align}
for $|\alpha| \leq N-1$, where we have used the identity
\begin{align}\label{e phi reason}
\left( \frac{1}{\varepsilon} v_i \partial^{\alpha+e_i} \{\mathbf{I}_{\pm}-\mathbf{P}_{\pm}\}f^\varepsilon \pm \frac{1}{2}\partial^{e_i} \phi^\varepsilon v_i \partial^\alpha\{\mathbf{I}_{\pm}-\mathbf{P}_{\pm}\}f^\varepsilon, e^{\pm \varepsilon \phi^\varepsilon } w^2_l(\alpha,0)\partial^\alpha\{\mathbf{I}_{\pm}-\mathbf{P}_{\pm}\}f^\varepsilon\right)=0.
\end{align}

Now we denote the nine terms on both sides of \eqref{with weight 1} by $J_{2,1}$ to $J_{2,9}$ and estimate them one by one. Firstly, for the term $J_{2,1}$, taking integration by parts with respect to $t$, we have
\begin{align}
J_{2,1}=\;&
\frac{1}{2} \frac{\d}{\d t} \sum_{ \pm}\left\| e^{\pm \frac{ \varepsilon \phi^\varepsilon}{2}}
w_l(\alpha,0)\partial^\alpha\{\mathbf{I}_{ \pm}-\mathbf{P}_{ \pm}\} f^\varepsilon\right\|^2 \nonumber\\
&-\frac{1}{2} \sum_{ \pm}\left( \pm \varepsilon \partial_t \phi^\varepsilon \partial^\alpha \{\mathbf{I}_{ \pm}-\mathbf{P}_{ \pm}\} f^\varepsilon, e^{ \pm \varepsilon \phi^\varepsilon} w_l^2(\alpha,0)\partial^\alpha\{\mathbf{I}_{ \pm}-\mathbf{P}_{ \pm}\} f^\varepsilon \right). \nonumber
\end{align}
The second term on the right-hand side of the above equality is dominated by
\begin{equation}\nonumber
\varepsilon \left\|\partial_t \phi^\varepsilon \right\|_{L^\infty} \sum_{\pm}\left\|e^{\pm \frac{\varepsilon \phi^\varepsilon}{2}}w_l(\alpha,0)\partial^\alpha \{\mathbf{I}_{\pm}-\mathbf{P}_{\pm}\}f^\varepsilon\right\|^2.
\end{equation}
For the term $J_{2,2}$, it follows from \eqref{L coercive2} that
\begin{align}
\begin{split}\nonumber
J_{2,2}=\;&\left(\frac{1}{\varepsilon^2} L \partial^\alpha \{\mathbf{I}-\mathbf{P}\}f^\varepsilon,
w^2_l(\alpha,0)\partial^\alpha \{\mathbf{I}-\mathbf{P}\}f^\varepsilon\right) \\
\geq\;&\frac{\lambda}{\varepsilon^2} \left\| w_l(\alpha,0) \partial^\alpha \{\mathbf{I}-\mathbf{P}\}f^\varepsilon\right\|^2_{N^s_\gamma}
-\frac{C}{\varepsilon^2}\left\|\partial^\alpha \{\mathbf{I}-\mathbf{P}\}f^\varepsilon\right\|^2_{L^2(B_C)}.
\end{split}
\end{align}
For the term $J_{2,3}$, we can deduce from \eqref{Gamma zeta2}, \eqref{Gamma zeta3} and \eqref{phi estimate1} that
\begin{align}\nonumber
J_{2,3} \lesssim \frac{1}{\varepsilon} \left\|\nabla_x \phi^\varepsilon\right\|_{H^1}\left\| w_l(\alpha,0) \partial^\alpha \{\mathbf{I}-\mathbf{P}\}f^\varepsilon\right\|^2_{N^s_\gamma}
\lesssim \mathcal{E}^{1/2}_{N,l}(t)\mathcal{D}_{N,l}(t).
\end{align}
As for $J_{2,4}$ and $J_{2,5}$, it follows from \eqref{softnonlinearterm1-2}, \eqref{hardnonlinearterm1-2}, \eqref{softnonlinearterm2-2} and \eqref{hardnonlinearterm2-2} that
\begin{align}\nonumber
J_{2,4}+J_{2,5} \lesssim \mathcal{E}^{1/2}_{N,l}(t)\mathcal{D}_{N,l}(t)
+ \chi_{\gamma} \left\|\nabla_x \phi^{\varepsilon}\right\|_{L^\infty} \sum_{\pm}\left\|e^{\pm \frac{\varepsilon \phi^\varepsilon}{2}} w_l(\alpha,0) \partial^\alpha \{\mathbf{I}_{\pm}-\mathbf{P}_{\pm}\}f^\varepsilon \right\|^2.
\end{align}
For the terms $J_{2,6}$--$J_{2,8}$ with $|\alpha| \leq N-1$, from the Cauchy--Schwarz inequality with $\eta$, the local conservation laws \eqref{macro equation 3} and the first equation of \eqref{macro equation 5}, we can bound these terms by
\begin{align*}
&C\left\{ \left\|\partial^\alpha \nabla_x \phi^\varepsilon \right\|^2
+ \left\|\partial^\alpha \nabla_x (a_{\pm}^\varepsilon, b^\varepsilon, c^\varepsilon) \right\|^2
+ \left\|\partial^\alpha \nabla_x \{\mathbf{I}-\mathbf{P}\}f^\varepsilon \right\|_{N^s_{\gamma}}^2
+ \mathcal{E}_{N}(t)\mathcal{D}_{N}(t) \right\}\\
& +\frac{\eta}{\varepsilon^2} \left\|w_l(\alpha,0)\partial^\alpha \{\mathbf{I}-\mathbf{P}\}f^\varepsilon \right\|_{N^s_\gamma}^2 \nonumber ,\nonumber
\end{align*}
where we have used \eqref{g estimate} and $\left| e^{\pm \varepsilon \phi^\varepsilon} \right| \approx 1$. Finally, for the term $J_{2,9}$, it follows from \eqref{soft gamma4} and \eqref{hard gamma4} that $J_{2,9} \lesssim \mathcal{E}^{1/2}_{N,l}(t)\mathcal{D}_{N,l}(t)$.

Collecting all the estimates of $J_{2,1}$--$J_{2,9}$ and taking summation over $|\alpha| \leq N-1$, \eqref{weighted estimate1} follows. This completes the proof of Lemma \ref{weighted 1}.
\end{proof}
\medskip

Next, we provide the weighted energy estimate with the pure spatial derivative $|\alpha|=N$.
\begin{lemma}\label{weighted 2}
There holds
\begin{align}\label{weighted estimate2}
&\varepsilon \frac{\d}{\d t} \sum_{|\alpha| = N} \sum_{\pm}\left\|e^{\pm \frac{\varepsilon \phi^\varepsilon}{2}} w_l(\alpha,0) \partial^\alpha f^\varepsilon_{\pm} \right\|^2
+ \frac{\lambda}{\varepsilon}\sum_{|\alpha| = N} \left\| w_l(\alpha,0)\partial^\alpha \{\mathbf{I}-\mathbf{P}\} f^\varepsilon\right\|_{N^s_\gamma}^2 \nonumber\\
\lesssim\;&\left\{\varepsilon \left\| \partial_t \phi^\varepsilon\right\|_{L^\infty} +\chi_{\gamma}\left\|\nabla_x \phi^{\varepsilon}\right\|_{L^\infty} \right\} \varepsilon \sum_{|\alpha| = N} \sum_{\pm}\left\|e^{\pm \frac{\varepsilon \phi^\varepsilon}{2}} w_l(\alpha,0) \partial^\alpha f^\varepsilon_{\pm} \right\|^2\\
&+\frac{1}{\varepsilon^2}\sum_{|\alpha| = N}\left\|\partial^\alpha \{\mathbf{I}-\mathbf{P}\}f^\varepsilon\right\|^2_{N^s_\gamma}
+\sum_{|\alpha| = N}\left\| \partial^\alpha \nabla_x \phi^\varepsilon \right\|^2
+\sum_{|\alpha| = N}\left\|\partial^\alpha(a^\varepsilon_{\pm},b^\varepsilon,c^\varepsilon)\right\|^2
+\mathcal{E}^{1/2}_{N,l}(t)\mathcal{D}_{N,l}(t). \nonumber
\end{align}
\end{lemma}

\begin{proof}
Applying $\partial^\alpha $ with $|\alpha|=N$ to \eqref{rrVPB}, taking the inner product with $\varepsilon e^{\pm \varepsilon \phi^\varepsilon} w^2_l(\alpha,0)\partial^\alpha f^\varepsilon_{\pm}$ over $\mathbb{R}_x^3 \times \mathbb{R}_v^3$ and summing with $\pm$, we have
\begin{align}\label{with weight 2}
&\sum_{\pm}\left(\partial_t \partial^\alpha f^\varepsilon_{ \pm},
\varepsilon e^{ \pm \varepsilon\phi^\varepsilon} w^2_l(\alpha,0) \partial^\alpha f^\varepsilon_{ \pm}\right)
+\sum_{\pm}\left(\frac{1}{\varepsilon^2}L_{ \pm} \partial^\alpha \{\mathbf{I}-\mathbf{P}\} f^\varepsilon,
\varepsilon w^2_l(\alpha,0) \partial^\alpha \{\mathbf{I}_{ \pm}-\mathbf{P}_{\pm}\}f^\varepsilon \right) \nonumber \\
=&-\sum_{\pm}\left(\frac{1}{\varepsilon^2}L_{ \pm} \partial^\alpha \{\mathbf{I}-\mathbf{P}\} f^\varepsilon,
\varepsilon \left(e^{ \pm \varepsilon\phi^\varepsilon}-1\right) w^2_l(\alpha,0) \partial^\alpha \{\mathbf{I}_{ \pm}-\mathbf{P}_{\pm}\}f^\varepsilon \right)  \nonumber\\
&-\sum_{\pm}\left(\frac{1}{\varepsilon^2}L_{ \pm} \partial^\alpha \{\mathbf{I}-\mathbf{P}\} f^\varepsilon,
\varepsilon e^{ \pm \varepsilon\phi^\varepsilon} w^2_l(\alpha,0) \partial^\alpha \mathbf{P}_{\pm} f^\varepsilon \right) \\
&+\sum_{\pm} \sum_{|\alpha_1| \leq |\alpha|}  C_\alpha^{\alpha_1} \left( \pm \partial^{\alpha_1+e_i} \phi^\varepsilon \partial^{\alpha-\alpha_1}_{e_i}f^\varepsilon_{ \pm},
\varepsilon e^{ \pm \varepsilon\phi^\varepsilon} w^2_l(\alpha,0) \partial^\alpha f^\varepsilon_{ \pm}\right) \nonumber\\
&+\frac{1}{2} \sum_{\pm}\sum_{1 \leq\left|\alpha_1\right| \leq|\alpha|} C_\alpha^{\alpha_1} \left( \mp \partial^{\alpha_1+e_i} \phi^\varepsilon v_i \partial^{\alpha-\alpha_1} f^\varepsilon_{ \pm},
\varepsilon e^{ \pm \varepsilon\phi^\varepsilon} w^2_l(\alpha,0) \partial^\alpha f^\varepsilon_{ \pm}\right) \nonumber\\
&+\sum_{\pm} \left(\mp \frac{1}{\varepsilon}\partial^{\alpha+e_i} \phi^\varepsilon v_i \mu^{1/2}, \varepsilon e^{ \pm \varepsilon\phi^\varepsilon} w^2_l(\alpha,0) \partial^\alpha f^\varepsilon_{ \pm}\right)
+\sum_{\pm}\left(\frac{1}{\varepsilon}\partial^\alpha \Gamma_{ \pm}(f^\varepsilon, f^\varepsilon), \varepsilon e^{ \pm \varepsilon \phi^\varepsilon} w_l(\alpha,0) \partial^\alpha f^\varepsilon_{ \pm}\right), \nonumber
\end{align}
where we used \eqref{f decomposition} and the identity \eqref{e phi reason}.

As before, we denote these terms in \eqref{with weight 2} by $J_{3,1}$ to $J_{3,8}$ and estimate them term by term. For the term $J_{3,1}$, from integration by parts with respect to $t$, one has
\begin{align}\nonumber
J_{3,1}=\frac{1}{2}\frac{\d}{\d t} \sum_{\pm}\left\|e^{\pm \frac{\varepsilon \phi^\varepsilon}{2}} w_l(\alpha,0) \partial^\alpha f_{\pm}^\varepsilon \right\|^2
-\frac{1}{2}\sum_{\pm} \left(\pm \varepsilon^2 \partial_t \phi^\varepsilon \partial^\alpha f_{\pm}^\varepsilon, e^{\pm \varepsilon \phi^\varepsilon} w^2_l(\alpha,0)\partial^\alpha f_{\pm}^\varepsilon\right).
\end{align}
The second term on the right-hand side of the above equality is bounded by
\begin{equation}\nonumber
\varepsilon^2 \left\|\partial_t \phi^\varepsilon \right\|_{L^\infty} \sum_{\pm}\left\|e^{\pm \frac{\varepsilon \phi^\varepsilon}{2}}w_l(\alpha,0)\partial^\alpha f_{\pm}^\varepsilon\right\|^2.
\end{equation}
For the term $J_{3,2}$, it follows from \eqref{L coercive2} that
\begin{align}
\begin{split}\nonumber
J_{3,2}=\;&\left(\frac{1}{\varepsilon} L \partial^\alpha \{\mathbf{I}-\mathbf{P}\}f^\varepsilon,
w^2_l(\alpha,0)\partial^\alpha \{\mathbf{I}-\mathbf{P}\}f^\varepsilon\right) \\
\geq\;&\frac{\lambda}{\varepsilon} \left\| w_l(\alpha,0) \partial^\alpha \{\mathbf{I}-\mathbf{P}\}f^\varepsilon\right\|^2_{N^s_\gamma}
-\frac{C}{\varepsilon}\left\|\partial^\alpha \{\mathbf{I}-\mathbf{P}\}f^\varepsilon\right\|^2_{L^2(B_C)}.
\end{split}
\end{align}
For the term $J_{3,3}$, applying \eqref{Gamma zeta2}, \eqref{Gamma zeta3} and \eqref{phi estimate1}, we can deduce
\begin{align}\nonumber
J_{3,3} \lesssim \left\|\nabla_x \phi^\varepsilon\right\|_{H^1}\left\| w_l(\alpha,0) \partial^\alpha \{\mathbf{I}-\mathbf{P}\}f^\varepsilon\right\|^2_{N^s_\gamma}
\lesssim \mathcal{E}^{1/2}_{N,l}(t)\mathcal{D}_{N,l}(t).
\end{align}
For the term $J_{3,4}$, one can derive from \eqref{Gamma zeta1} that
\begin{align}\nonumber
J_{3,4} \lesssim \frac{1}{\varepsilon} \left\| \partial^\alpha \{\mathbf{I}-\mathbf{P}\}f^\varepsilon \right\|_{N^s_\gamma}
\left\|\partial^\alpha(a_{\pm}^\varepsilon,b^\varepsilon,c^\varepsilon)\right\|
\lesssim \frac{1}{\varepsilon^2} \left\| \partial^\alpha \{\mathbf{I}-\mathbf{P}\}f^\varepsilon \right\|^2_{N^s_\gamma}
+\left\|\partial^\alpha(a_{\pm}^\varepsilon,b^\varepsilon,c^\varepsilon)\right\|^2.
\end{align}
As for $J_{3,5}$ and $J_{3,6}$, by using \eqref{softnonlinearterm1-3}, \eqref{hardnonlinearterm1-3}, \eqref{softnonlinearterm2-3} and \eqref{hardnonlinearterm2-3}, one has
\begin{align}\nonumber
J_{3,5}+ J_{3,6} \lesssim \mathcal{E}^{1/2}_{N,l}(t)\mathcal{D}_{N,l}(t)
+ \chi_{\gamma} \left\|\nabla_x \phi^{\varepsilon}\right\|_{L^\infty}
\varepsilon \sum_{|\alpha|= N}\sum_{\pm}\left\|e^{\pm \frac{\varepsilon \phi^{\varepsilon}}{2}}w_l(\alpha,0)\partial^{\alpha} f^{\varepsilon}_{\pm}\right\|^2.
\end{align}
For the term $J_{3,7}$, it is straightforward to see that
\begin{align}\nonumber
J_{3,7} \lesssim \eta \left(\left\|w_l(\alpha,0) \partial^\alpha \{\mathbf{I}-\mathbf{P}\}f^\varepsilon \right\|^2_{N^s_\gamma}
+\left\|\partial^\alpha(a_{\pm}^\varepsilon,b^\varepsilon,c^\varepsilon)\right\|^2\right)
+\left\|\partial^\alpha \nabla_x \phi^\varepsilon\right\|^2,
\end{align}
where we have used \eqref{f decomposition} and $\left| e^{\pm \varepsilon \phi^\varepsilon} \right| \approx 1$. Finally, for the term $J_{3,8}$, it follows from \eqref{soft gamma5} and \eqref{hard gamma5} that
$J_{3,8} \lesssim \mathcal{E}^{1/2}_{N,l}(t)\mathcal{D}_{N,l}(t)$.

Collecting all the estimates of $J_{3,1}$--$J_{3,8}$ with $|\alpha|=N$, \eqref{weighted estimate2} follows. This completes the proof of Lemma \ref{weighted 2}.
\end{proof}
\medskip

Finally, we present the weighted energy estimate with the mixed derivatives $|\alpha|+|\beta| \leq N$ and $|\beta| \geq 1$.
\begin{lemma}\label{weighted 3}
There holds
\begin{align}\label{weighted estimate3}
&\sum_{m=1}^N C_m \sum_{\substack{{|\alpha|+|\beta|\leq N}\\{|\beta|=m}}} \left\{\frac{\d}{\d t}\sum_{\pm}
\left\| e^{\pm \frac{\varepsilon \phi^\varepsilon}{2}} w_l(\alpha,\beta)\partial^\alpha_\beta \{\mathbf{I}_{\pm}-\mathbf{P}_{\pm}\}f^\varepsilon\right\|^2
+\frac{\lambda}{\varepsilon^2}\left\| w_l(\alpha,\beta)\partial^\alpha_\beta \{\mathbf{I}-\mathbf{P}\}f^\varepsilon\right\|_{N^s_\gamma}^2\right\} \nonumber\\
\lesssim\;&\left\{ \varepsilon \left\|\partial_t \phi^\varepsilon\right\|_{L^\infty} + \chi_{\gamma} \left\|\nabla_x \phi^{\varepsilon}\right\|_{L^\infty} \right\}\sum_{\substack{{|\alpha|+|\beta|\leq N}\\{|\beta|\geq 1}}}\sum_{\pm}
\left\| e^{\pm \frac{\varepsilon \phi^\varepsilon}{2}} w_l(\alpha,\beta)\partial^\alpha_\beta \{\mathbf{I}_{\pm}-\mathbf{P}_{\pm}\}f^\varepsilon\right\|^2 \nonumber\\
&+\frac{1}{\varepsilon^2} \sum_{|\alpha|\leq N-1}\|w_l(\alpha,0)\partial^\alpha \{\mathbf{I}-\mathbf{P}\}f^\varepsilon\|^2_{N^s_\gamma}
+\sum_{|\alpha|= N}\!\!\|w_l(\alpha,0)\partial^\alpha \{\mathbf{I}-\mathbf{P}\}f^\varepsilon\|^2_{N^s_\gamma} \nonumber\\
&+\sum_{|\alpha| \leq N-1}\left\| \partial^\alpha \nabla_x \phi^\varepsilon \right\|^2
+\sum_{1 \leq |\alpha| \leq N}\left\|\partial^\alpha(a^\varepsilon_{\pm},b^\varepsilon,c^\varepsilon)\right\|^2
+ \left\{\mathcal{E}^{1/2}_{N,l}(t)+\mathcal{E}_{N,l}(t)\right\}\mathcal{D}_{N,l}(t).
\end{align}
Here, $C_m$ is a fixed constant satisfying $C_m \gg C_{m+1}$.
\end{lemma}
\begin{proof}
Let $1 \leq m \leq N$. By applying $\partial^\alpha_\beta$ with $|\beta|=m$ and $|\alpha|+|\beta| \leq N$ to \eqref{rrrVPB}, taking the inner product with $e^{\pm \varepsilon \phi^\varepsilon}w^2_l(\alpha,\beta)\partial^\alpha_\beta \{\mathbf{I}_{\pm}-\mathbf{P}_{\pm}\}f^\varepsilon$ over $\mathbb{R}_x^3 \times \mathbb{R}_v^3$ and summing with $\pm$, one has
\begin{align}\label{with weight 3}
&\sum_{\pm}\left(\partial_t \partial^\alpha_\beta \{\mathbf{I}_{\pm}-\mathbf{P}_{\pm}\}f^\varepsilon,
e^{\pm \varepsilon \phi^\varepsilon} w^2_l(\alpha,\beta)\partial^\alpha_\beta \{\mathbf{I}_{\pm}-\mathbf{P}_{\pm}\}f^\varepsilon \right) \nonumber\\
&+\sum_{\pm}\left( \frac{1}{\varepsilon^2}\partial^\alpha_\beta L_{\pm}\{\mathbf{I}-\mathbf{P}\}f^\varepsilon,
w^2_l(\alpha,\beta)\partial^\alpha_\beta \{\mathbf{I}_{\pm}-\mathbf{P}_{\pm}\}f^\varepsilon \right)\nonumber\\
=\;&-\sum_{\pm}\left(\frac{1}{\varepsilon^2}\partial^\alpha_\beta L_{\pm}\{\mathbf{I}-\mathbf{P}\}f^\varepsilon,
\left(e^{\pm \varepsilon \phi^\varepsilon}-1\right) w^2_l(\alpha,\beta)\partial^\alpha_\beta \{\mathbf{I}_{\pm}-\mathbf{P}_{\pm}\}f^\varepsilon \right)\nonumber \\
&-\sum_{\pm}C_\beta^{e_i}\left( \frac{1}{\varepsilon} \partial^{\alpha+e_i}_{\beta-e_i}\{\mathbf{I}_{\pm}-\mathbf{P}_{\pm}\}f^\varepsilon,
e^{\pm \varepsilon \phi^\varepsilon} w^2_l(\alpha,\beta)\partial^\alpha_\beta \{\mathbf{I}_{\pm}-\mathbf{P}_{\pm}\}f^\varepsilon \right)\nonumber\\
&+ \sum_{\pm}\sum_{|\alpha_1|\leq |\alpha|} C^{\alpha_1}_{\alpha} \left(\pm \partial^{\alpha_1+e_i} \phi^\varepsilon \partial^{\alpha-\alpha_1}_{\beta+e_i} \{\mathbf{I}_{\pm}-\mathbf{P}_{\pm}\}f^\varepsilon,
e^{\pm \varepsilon \phi^\varepsilon} w^2_l(\alpha,\beta)\partial^\alpha_\beta \{\mathbf{I}_{\pm}-\mathbf{P}_{\pm}\}f^\varepsilon \right)\nonumber\\
&+\frac{1}{2}\sum_{\pm}\sum_{\substack{{|\alpha_1|+|\beta_1| \geq 1}\\{|\beta_1| \leq 1}}}C^{\alpha_1,\beta_1}_{\alpha,\beta}
\left( \mp \partial_{\beta_1}v_i \partial^{\alpha_1+e_i} \phi^\varepsilon \partial^{\alpha-\alpha_1}_{\beta-\beta_1} \{\mathbf{I}_{\pm}-\mathbf{P}_{\pm}\}f^\varepsilon,
e^{\pm \varepsilon \phi^\varepsilon} w^2_l(\alpha,\beta)\partial^\alpha_\beta \{\mathbf{I}_{\pm}-\mathbf{P}_{\pm}\}f^\varepsilon \right)\nonumber\\
&+\frac{1}{2} \sum_{\pm} \sum_{\substack{{|\alpha_1|\leq |\alpha|}\\{|\beta_1|\leq 1}}}
C^{\alpha_1,\beta_1}_{\alpha,\beta} \left( \mp \partial_{\beta_1}v_i \partial^{\alpha_1+e_i} \phi^\varepsilon \partial^{\alpha-\alpha_1}_{\beta-\beta_1} \mathbf{P}_{\pm} f^\varepsilon,
e^{\pm \varepsilon \phi^\varepsilon} w^2_l(\alpha,\beta)\partial^\alpha_\beta \{\mathbf{I}_{\pm}-\mathbf{P}_{\pm}\}f^\varepsilon \right)\nonumber\\
&-\sum_{\pm}\left(\partial^\alpha_{\beta} \left[\partial_t+\frac{1}{\varepsilon}v_i \partial^{e_i} \mp \partial^{e_i} \phi^\varepsilon \partial_{e_i}\right] \mathbf{P}_{ \pm} f^\varepsilon,
e^{\pm \varepsilon \phi^\varepsilon} w^2_l(\alpha,\beta)\partial^\alpha_\beta \{\mathbf{I}_{\pm}-\mathbf{P}_{\pm}\}f^\varepsilon \right)\nonumber\\
&+\sum_{\pm}\left(\mp \frac{1}{\varepsilon}\partial^{\alpha+e_i} \phi^\varepsilon \partial_\beta (v_i \mu^{1/2}),
e^{\pm \varepsilon \phi^\varepsilon} w^2_l(\alpha,\beta)\partial^\alpha_\beta \{\mathbf{I}_{\pm}-\mathbf{P}_{\pm}\}f^\varepsilon \right)\nonumber\\
&+\sum_{\pm} \left( \frac{1}{\varepsilon} \partial^\alpha_\beta \Gamma_{\pm}(f^\varepsilon,f^\varepsilon),
e^{\pm \varepsilon \phi^\varepsilon} w^2_l(\alpha,\beta)\partial^\alpha_\beta \{\mathbf{I}_{\pm}-\mathbf{P}_{\pm}\}f^\varepsilon \right).
\end{align}
Here, we have used the identity
\begin{align}\nonumber
\left(\frac{1}{\varepsilon}v_i\partial^{\alpha+e_i}_\beta \{\mathbf{I}_{\pm}-\mathbf{P}_{\pm}\}f^\varepsilon \pm \frac{1}{2}\partial^{e_i}\phi^\varepsilon v_i\partial^{\alpha}_\beta \{\mathbf{I}_{\pm}-\mathbf{P}_{\pm}\}f^\varepsilon,
e^{\pm \varepsilon \phi^\varepsilon} w^2_l(\alpha,\beta)\partial^\alpha_\beta \{\mathbf{I}_{\pm}-\mathbf{P}_{\pm}\}f^\varepsilon \right)=0.
\end{align}

Now we denote all terms in \eqref{with weight 3} by $J_{4,1}$ to $J_{4,10}$ and estimate them term by term. First of all, for the term $J_{4,1}$, employing integration by parts with respect to $t$, we obtain
\begin{align}
J_{4,1}=\;&\frac{1}{2}\frac{\d}{\d t} \sum_{\pm}\left\|e^{\pm \frac{\varepsilon \phi^\varepsilon}{2}} w_l(\alpha,\beta) \partial^\alpha_\beta \{\mathbf{I}_{\pm}-\mathbf{P}_{\pm}\}f^\varepsilon \right\|^2 \nonumber\\
&-\frac{1}{2}\sum_{\pm} \left(\pm \varepsilon \partial_t \phi^\varepsilon \partial^\alpha_\beta \{\mathbf{I}_{\pm}-\mathbf{P}_{\pm}\}f^\varepsilon, e^{\pm \varepsilon \phi^\varepsilon} w^2_l(\alpha,\beta)\partial^\alpha_\beta \{\mathbf{I}_{\pm}-\mathbf{P}_{\pm}\}f^\varepsilon\right). \nonumber
\end{align}
The second term  on the right-hand side of the above equality is dominated by
\begin{equation}\nonumber
\varepsilon \left\|\partial_t \phi^\varepsilon \right\|_{L^\infty} \sum_{\pm}\left\|e^{\pm \frac{\varepsilon \phi^\varepsilon}{2}}w_l(\alpha,\beta) \partial^\alpha_\beta \{\mathbf{I}_{\pm}-\mathbf{P}_{\pm}\}f^\varepsilon\right\|^2.
\end{equation}
For the term $J_{4,2}$, it follows from \eqref{L coercive3} that
\begin{align}
\begin{split}\nonumber
J_{4,2}\geq\;&\frac{\lambda}{\varepsilon^2} \left\| w_l(\alpha,\beta) \partial^\alpha_\beta \{\mathbf{I}-\mathbf{P}\}f^\varepsilon\right\|^2_{N^s_\gamma}
-\frac{\eta}{\varepsilon^2}\sum_{|\beta^\prime| < |\beta|}\left\|w_l(\alpha,\beta^\prime)\partial^\alpha_{\beta^\prime} \{\mathbf{I}-\mathbf{P}\}f^\varepsilon\right\|^2_{N^s_\gamma}\\
&-\frac{C}{\varepsilon^2}\left\|\partial^\alpha \{\mathbf{I}-\mathbf{P}\}f^\varepsilon\right\|^2_{L^2(B_C)}.
\end{split}
\end{align}
For the term $J_{4,3}$, with the help of \eqref{Gamma zeta2}, \eqref{Gamma zeta3} and \eqref{phi estimate1}, one obtains
\begin{align}
J_{4,3}\lesssim\;&\frac{1}{\varepsilon} \left\| \nabla_x \phi^\varepsilon\right\|_{H^1} \sum_{|\beta^\prime| \leq |\beta|} \left\|w_l(\alpha,\beta^\prime)\partial^\alpha_{\beta^\prime} \{\mathbf{I}-\mathbf{P}\}f^\varepsilon \right\|_{N^s_\gamma}
\left\|w_l(\alpha,\beta)\partial^\alpha_{\beta} \{\mathbf{I}-\mathbf{P}\}f^\varepsilon \right\|_{N^s_\gamma} \nonumber\\
\lesssim\;& \mathcal{E}^{1/2}_{N,l}(t)\mathcal{D}_{N,l}(t). \nonumber
\end{align}
For the term $J_{4,4}$, it can be inferred from \eqref{weight function} and the Cauchy--Schwarz inequality with $\eta$  that
\begin{align}
J_{4,4} \lesssim \;&\frac{1}{\varepsilon}\left\| w_l(\alpha,\beta) \left\langle v \right\rangle^\frac{\gamma+2s}{2} \partial^\alpha_{\beta} \{\mathbf{I}-\mathbf{P}\}f^\varepsilon \right\|
\left\| w_l(\alpha+e_i,\beta-e_i) \left\langle v \right\rangle^\frac{\gamma+2s}{2} \partial^{\alpha+e_i}_{\beta-e_i} \{\mathbf{I}-\mathbf{P}\}f^\varepsilon \right\| \nonumber\\
\lesssim\;&\frac{\eta}{\varepsilon^2}\left\| w_l(\alpha,\beta)  \partial^\alpha_{\beta} \{\mathbf{I}-\mathbf{P}\}f^\varepsilon \right\|^2_{N^s_\gamma}
+\sum_{\substack{{|\alpha^\prime|=|\alpha|+1}\\{|\beta^\prime|=|\beta|-1}}}\left\| w_l(\alpha^\prime,\beta^\prime) \partial^{\alpha^\prime}_{\beta^\prime} \{\mathbf{I}-\mathbf{P}\}f^\varepsilon \right\|^2_{N^s_\gamma}, \nonumber
\end{align}
where we used the fact $w_l(\alpha,\beta) \leq w_l(\alpha+e_i,\beta-e_i) \left\langle v \right\rangle^{\gamma+2s}$  for $\gamma > -3$.
For the terms $J_{4,5}$ and $J_{4,6}$, making use of \eqref{softnonlinearterm1-4}, \eqref{hardnonlinearterm1-4}, \eqref{softnonlinearterm2-4} and \eqref{hardnonlinearterm2-4}, we get
\begin{align}\nonumber
J_{4,5}+J_{4,6}\lesssim \mathcal{E}^{1/2}_{N,l}(t)\mathcal{D}_{N,l}(t)
+ \chi_{\gamma} \left\|\nabla_x \phi^{\varepsilon}\right\|_{L^\infty}
\sum_{\pm}\left\| e^{\pm \frac{\varepsilon \phi^\varepsilon}{2}} w_l(\alpha,\beta)\partial^\alpha_\beta \{\mathbf{I}_{\pm}-\mathbf{P}_{\pm}\}f^\varepsilon\right\|^2.
\end{align}
For the terms $J_{4,7}$--$J_{4,9}$ with $|\alpha|+|\beta|\leq N$, $|\beta|\geq 1$ and $|\alpha|\leq N-1$, applying \eqref{macro equation 3}, $\eqref{g estimate}$, the first equation of \eqref{macro equation 5}, the Sobolev inequalities \eqref{Sobolev-ineq} and the Cauchy--Schwarz inequality with $\eta$, they can be bounded by
\begin{align}\nonumber
&C\left\{\left\|\partial^\alpha \nabla_x \phi^\varepsilon\right\|^2
+ \left\|\partial^\alpha \nabla_x (a_{\pm}^\varepsilon, b^\varepsilon, c^\varepsilon) \right\|^2
+ \left\|\partial^\alpha \nabla_x \{\mathbf{I}-\mathbf{P}\}f^\varepsilon \right\|_{N^s_{\gamma}}^2
+ \mathcal{E}_{N,l}(t)\mathcal{D}_{N,l}(t) \right\}\\
&+\frac{\eta}{\varepsilon^2}\left\|w_l(\alpha,\beta)\partial^\alpha_\beta\{\mathbf{I}-\mathbf{P}\}f^\varepsilon\right\|^2_{N^s_\gamma}.\nonumber
\end{align}
Finally, for the term $J_{4,10}$, it follows from \eqref{soft gamma6} and \eqref{hard gamma6} that
$J_{4,10} \lesssim \mathcal{E}^{1/2}_{N,l}(t)\mathcal{D}_{N,l}(t)$.

As a consequence, by plugging all the estimates above into \eqref{with weight 3}, taking the summation over $\left\{ |\beta|=m,|\alpha|+|\beta| \leq N\right\}$ for each given $1 \leq m \leq N$, and then taking combination of those $N$ estimates with properly chosen constant $C_m > 0 $ $(1 \leq m \leq N)$ and $\eta$ small enough,
we obtain \eqref{weighted estimate3}. This completes the proof of Lemma \ref{weighted 3}.
\end{proof}
\medskip

\subsection{Energy Estimate in Negative Sobolev Space}
\hspace*{\fill}

This subsection aims to provide an energy estimate in negative Sobolev space, for both hard and soft potentials.
\begin{lemma}\label{negative sobolev lemma}
There holds
\begin{align}
\begin{split}\label{negative sobolev estimate}
&\frac{\d}{\d t} \left\{ \left\|\Lambda^{-\varrho}f^\varepsilon\right\|^2+\left\|\Lambda^{-\varrho}\nabla_x \phi^\varepsilon\right\|^2 \right\}
+\frac{\lambda}{\varepsilon^2}\left\|\Lambda^{-\varrho}\{\mathbf{I}-\mathbf{P}\}f^\varepsilon\right\|_{N^s_\gamma}^2 \\
\lesssim\;& \widetilde{\mathcal{E}}^{1/2}_{N,l}(t)\left\{ \left\| \Lambda^{\frac{3}{4}-\frac{\varrho}{2}} f^\varepsilon\right\|^2_{N^s_\gamma}
+ \left\| \Lambda^{\frac{3}{4}-\frac{\varrho}{2}} \nabla_x \phi^\varepsilon\right\|^2 \right\}
+\mathcal{E}_{N,l}(t)\left\{ \left\| \Lambda^{\frac{3}{2}-\varrho} f^\varepsilon\right\|^2
+ \left\| \Lambda^{\frac{3}{2}-\varrho} \nabla_x \phi^\varepsilon\right\|^2 \right\} .
\end{split}
\end{align}
\end{lemma}

\begin{proof}
Applying $\Lambda^{-\varrho}$ to the first equation of \eqref{rVPB} and taking the inner product with $\Lambda^{-\varrho} f^\varepsilon$ over $\mathbb{R}_x^3 \times \mathbb{R}_v^3$, one has
\begin{align}\label{negative sobolev estimate 1}
&\frac{1}{2}\frac{\d}{\d t} \left\| \Lambda^{-\varrho} f^\varepsilon \right\|^2
+\left( \frac{1}{\varepsilon} \Lambda^{-\varrho} \nabla_x \phi^\varepsilon \cdot v \mu^{1/2}q_1 ,\Lambda^{-\varrho} f^\varepsilon \right)
+\left( \frac{1}{\varepsilon^2} L\Lambda^{-\varrho} f^\varepsilon, \Lambda^{-\varrho} f^\varepsilon\right) \\
=\;&\left( q_0 \Lambda^{-\varrho}\left(\nabla_x \phi^\varepsilon \cdot \nabla_v f^\varepsilon \right), \Lambda^{-\varrho} f^\varepsilon \right)
-\left( \frac{q_0}{2} \Lambda^{-\varrho}\left(\nabla_x \phi^\varepsilon \cdot v f^\varepsilon \right), \Lambda^{-\varrho} f^\varepsilon \right)
+\left( \frac{1}{\varepsilon} \Lambda^{-\varrho}\Gamma(f^\varepsilon, f^\varepsilon), \Lambda^{-\varrho} f^\varepsilon \right). \nonumber
\end{align}

Now we further estimate \eqref{negative sobolev estimate 1} term by term. For the left-hand second term, it follows from the first equation of \eqref{macro equation 5} and \eqref{a equality} that
\begin{align}\nonumber
\left( \frac{1}{\varepsilon} \Lambda^{-\varrho} \nabla_x \phi^\varepsilon \cdot v \mu^{1/2}q_1 ,\Lambda^{-\varrho} f^\varepsilon \right)
=\frac{1}{2}\frac{\d}{\d t} \left\| \Lambda^{-\varrho} \nabla_x \phi^\varepsilon\right\|^2.
\end{align}
For the left-hand third term, one has from \eqref{L coercive1} that
\begin{align}\nonumber
\left( \frac{1}{\varepsilon^2} L\Lambda^{-\varrho} f^\varepsilon, \Lambda^{-\varrho} f^\varepsilon \right)
\geq \frac{\lambda}{\varepsilon^2} \left\| \Lambda^{-\varrho} \{\mathbf{I}-\mathbf{P}\}f^\varepsilon \right\|_{N^s_\gamma}^2.
\end{align}
To estimate the right-hand first term, owing to $1 < \varrho < \frac{3}{2}$, by using \eqref{negative embed 1}, \eqref{negative embed 2}, the Minkowski inequality \eqref{minkowski}, the H\"{o}lder inequality and the Cauchy--Schwarz inequality with $\eta$, one has
\begin{align}\nonumber
&\left( q_0 \Lambda^{-\varrho}\left(\nabla_x \phi^\varepsilon \cdot \nabla_v f^\varepsilon \right), \Lambda^{-\varrho} f^\varepsilon \right)\\
=\;&\left( q_0 \Lambda^{-\varrho}\left(\nabla_x \phi^\varepsilon \cdot \nabla_v f^\varepsilon \right), \Lambda^{-\varrho} \mathbf{P}f^\varepsilon \right)
+\left( q_0 \Lambda^{-\varrho}\left(\nabla_x \phi^\varepsilon \cdot \nabla_v f^\varepsilon \right), \Lambda^{-\varrho} \{\mathbf{I}-\mathbf{P}\}f^\varepsilon \right) \nonumber\\
\lesssim\;&\left\| \Lambda^{-\varrho}(\nabla_x \phi^\varepsilon \mu^{\delta}f^\varepsilon)\right\|
\left\| \Lambda^{-\varrho} \mathbf{P}f^\varepsilon \right\|
+\left\| \Lambda^{-\varrho}(\nabla_x \phi^\varepsilon \cdot \nabla_v f^\varepsilon\langle v \rangle ^{-\frac{\gamma+2s}{2}})\right\|
\left\| \langle v \rangle ^ \frac{\gamma+2s}{2}\Lambda^{-\varrho} \{\mathbf{I}-\mathbf{P}\}f^\varepsilon \right\| \nonumber\\
\lesssim\;& \left\| \nabla_x \phi^\varepsilon \mu^{\delta}f^\varepsilon\right\|_{{L^2_v}{L^{\frac{6}{3+2\varrho}}_x}}
\left\| \Lambda^{-\varrho} \mathbf{P}f^\varepsilon \right\|
+\left\| \nabla_x \phi^\varepsilon \cdot \nabla_v f^\varepsilon\langle v \rangle ^{-\frac{\gamma+2s}{2}}\right\|_{L^2_v L_x^{\frac{6}{3+2\varrho}}}
\left\| \Lambda^{-\varrho} \{\mathbf{I}-\mathbf{P}\}f^\varepsilon \right\|_{N^s_\gamma}   \label{negative sobolev estimate 2}\\
\lesssim\;& \left\| \nabla_x \phi^\varepsilon \right\|_{L^{\frac{12}{3+2\varrho}}}
\left\| \mu^{\delta}f^\varepsilon \right\|_{L^2_{v}{L^{\frac{12}{3+2\varrho}}_x}}
\left\| \Lambda^{-\varrho} \mathbf{P}f^\varepsilon \right\|
+\left\| \nabla_x \phi^\varepsilon \right\|_{L^{\frac{3}{\varrho}}}
\left\|  \langle v \rangle ^{-\frac{\gamma+2s}{2}} \nabla_v f^\varepsilon \right\|
\left\| \Lambda^{-\varrho} \{\mathbf{I}-\mathbf{P}\}f^\varepsilon \right\|_{N^s_\gamma} \nonumber\\
\lesssim\;& \left\| \Lambda^{\frac{3}{4}-\frac{\varrho}{2}} \nabla_x \phi^\varepsilon\right\|
\left\| \Lambda^{\frac{3}{4}-\frac{\varrho}{2}} (\mu^{\delta}f^\varepsilon)\right\|
\left\| \Lambda^{-\varrho} \mathbf{P}f^\varepsilon \right\|
+\left\| \Lambda^{\frac{3}{2}-\varrho} \nabla_x \phi^\varepsilon\right\|
\left\|  \langle v \rangle ^{-\frac{\gamma+2s}{2}} \nabla_v f^\varepsilon \right\|
\left\| \Lambda^{-\varrho} \{\mathbf{I}-\mathbf{P}\}f^\varepsilon \right\|_{N^s_\gamma} \nonumber\\
\lesssim\;& \widetilde{\mathcal{E}}^{1/2}_{N,l}(t)
\left\{  \left\| \Lambda^{\frac{3}{4}-\frac{\varrho}{2}} \nabla_x \phi^\varepsilon\right\|^2
+\left\| \Lambda^{\frac{3}{4}-\frac{\varrho}{2}} f^\varepsilon\right\|^2_{N^s_\gamma}\right\}
+\mathcal{E}_{N,l}(t)\left\| \Lambda^{\frac{3}{2}-\varrho} \nabla_x \phi^\varepsilon\right\|^2
+\eta \left\| \Lambda^{-\varrho} \{\mathbf{I}-\mathbf{P}\}f^\varepsilon \right\|^2_{N^s_\gamma}. \nonumber
\end{align}
Similarly,  the right-hand second term has the same upper bound as \eqref{negative sobolev estimate 2}. Finally, for the last term on the right-hand side of \eqref{negative sobolev estimate 1}, by the collision invariant property, we have for soft potentials that
\begin{align} \label{negative sobolev Gamma1}
\left( \frac{1}{\varepsilon} \Lambda^{-\varrho}\Gamma(f^\varepsilon, f^\varepsilon), \Lambda^{-\varrho} f^\varepsilon \right)
=\;&\left( \frac{1}{\varepsilon} \Lambda^{-\varrho}\Gamma(f^\varepsilon, f^\varepsilon),
\Lambda^{-\varrho} \{\mathbf{I}-\mathbf{P}\}f^\varepsilon \right) \nonumber\\
\lesssim \;& \frac{1}{\varepsilon} \left\|\Lambda^{-\varrho}\left( \langle v \rangle ^{-\frac{\gamma+2s}{2}}\Gamma(f^\varepsilon, f^\varepsilon)\right) \right\|
\left\| \Lambda^{-\varrho} \{\mathbf{I}-\mathbf{P}\}f^\varepsilon \right\|_{N^s_\gamma} \nonumber\\
\lesssim\;& \frac{1}{\varepsilon} \left\| \langle v \rangle ^{-\frac{\gamma+2s}{2}}\Gamma(f^\varepsilon, f^\varepsilon) \right\|_{L^2_{v}L_x^{\frac{6}{3+2\varrho}}}
\left\| \Lambda^{-\varrho} \{\mathbf{I}-\mathbf{P}\}f^\varepsilon \right\|_{N^s_\gamma} \nonumber\\
\lesssim\;&\frac{1}{\varepsilon} \left\| \langle v \rangle ^{-\frac{\gamma+2s}{2}}\Gamma(f^\varepsilon, f^\varepsilon) \right\|_{L_x^{\frac{6}{3+2\varrho}}L^2_{v}}
\left\| \Lambda^{-\varrho} \{\mathbf{I}-\mathbf{P}\}f^\varepsilon \right\|_{N^s_\gamma} \\
\lesssim\;&\frac{1}{\varepsilon} \left\| |f^\varepsilon|_{L^2} |f^\varepsilon|_{H^4} \right\|_{L_x^{\frac{6}{3+2\varrho}}}
\left\| \Lambda^{-\varrho} \{\mathbf{I}-\mathbf{P}\}f^\varepsilon \right\|_{N^s_\gamma} \nonumber\\
\lesssim\;&\frac{1}{\varepsilon} \left\|f^\varepsilon\right\|_{L_x^{\frac{3}{\varrho}}L^2_v} \left\|f^\varepsilon\right\|_{H^4_vL^2_x}
\left\| \Lambda^{-\varrho} \{\mathbf{I}-\mathbf{P}\}f^\varepsilon \right\|_{N^s_\gamma} \nonumber\\
\lesssim\;& \mathcal{E}_{N,l}(t) \left\| \Lambda^{\frac{3}{2}-\varrho}f^\varepsilon\right\|^2
+\frac{\eta}{\varepsilon^2} \left\| \Lambda^{-\varrho} \{\mathbf{I}-\mathbf{P}\} f^\varepsilon \right\|_{N^s_\gamma}^2, \nonumber
\end{align}
where we have used \eqref{soft gamma2}, \eqref{negative embed 1}, \eqref{negative embed 2}, the Minkowski inequality \eqref{minkowski},  the H\"{o}lder inequality and the Cauchy--Schwarz inequality with $\eta$. Similarly, for the hard potential case, we obtain the same upper bound as \eqref{negative sobolev Gamma1}.

As a result, combining all the estimates above, \eqref{negative sobolev estimate} follows. This completes the proof of Lemma \ref{negative sobolev lemma}.
\end{proof}
\medskip

\subsection{Uniform Energy Estimate}
\hspace*{\fill}

In this subsection, based on Lemma \ref{macroscopic estimate}--\ref{negative sobolev lemma}, we give the uniform a priori estimates
with respect to $\varepsilon\in (0,1]$, for both hard and soft potentials.
\begin{proposition}\label{energy estimates total}
There is $\mathcal{E}_{N,l}(t)$ satisfying \eqref{energy functional} such that
\begin{align}\label{a priori estimates}
\frac{\d}{\d t} \mathcal{E}_{N,l}(t) + \lambda  \mathcal{D}_{N,l}(t) \lesssim \left\{\varepsilon \left\| \partial_t \phi^\varepsilon \right\|_{L^\infty}+  \chi_{\gamma}\left\|\nabla_x \phi^{\varepsilon}\right\|_{L^\infty}\right\} \mathcal{E}_{N,l}(t).
\end{align}
\end{proposition}

\begin{proof}
To prove \eqref{a priori estimates}, we take the proper linear combination of those estimates obtained in the previous Lemmas as follows. For $\widetilde{C}_1 > 0$ and $\widetilde{C}_2 >0$ large enough, $0 < \kappa_4 \ll 1$, the combination $\widetilde{C}_2 \times \left\{ \widetilde{C}_1 \times \big[ \eqref{estimate without weight}+ \kappa_4 \times \eqref{macro estimate}\big]+ \eqref{weighted estimate1}+\eqref{weighted estimate2}\right\}+\eqref{weighted estimate3}$ gives
\begin{align}\label{a priori estimates 1}
\!\!\!\!\!\frac{\d}{\d t} \mathcal{E}_{N,l}(t) + \lambda  \mathcal{D}_{N,l}(t) \lesssim\left\{\varepsilon \left\| \partial_t \phi^\varepsilon \right\|_{L^\infty}\!+\! \chi_{\gamma} \left\|\nabla_x \phi^{\varepsilon}\right\|_{L^\infty}\right\} \mathcal{E}_{N,l}(t) + \left\{\mathcal{E}^{1/2}_{N,l}(t)+\mathcal{E}_{N,l}(t)\right\} \mathcal{D}_{N,l}(t),
\end{align}
where $\mathcal{E}_{N,l}(t)$ is given by
\begin{align}\label{energy functional define}
\begin{split}
\mathcal{E}_{N,l}(t) :=\; &\widetilde{C}_2\bigg\{ \widetilde{C}_1\Big[ \sum_{|\alpha|\leq N}\sum_{\pm}\left\| e^{\pm \frac{\varepsilon \phi^\varepsilon}{2}}\partial^\alpha f^\varepsilon_{\pm}\right\|^2+\sum_{|\alpha| \leq N}\left\| \partial^\alpha \nabla_x \phi^\varepsilon \right\|^2+ \kappa_4 \varepsilon \mathcal{E}_{int}^N(t)\Big]\\
&\qquad+\sum_{|\alpha| \leq N-1}\sum_{\pm}\left\| e^{\pm \frac{\varepsilon \phi^\varepsilon}{2}}w_l(\alpha,0)\partial^\alpha \{\mathbf{I}_{\pm}-\mathbf{P}_{\pm}\}f^\varepsilon\right\|^2\\
&\qquad+\sum_{|\alpha|=N}\sum_{\pm}\varepsilon \left\| e^{\pm \frac{\varepsilon \phi^\varepsilon}{2}}w_l(\alpha,0)\partial^\alpha f_{\pm}^\varepsilon \right\|^2\bigg\}\\
&+\sum_{m=1}^N C_m \sum_{\substack{{|\alpha|+|\beta|\leq N}\\{|\beta|=m}}} \sum_{\pm}
\left\| e^{\pm \frac{\varepsilon \phi^\varepsilon}{2}} w_l(\alpha,\beta)\partial^\alpha_\beta \{\mathbf{I}_{\pm}-\mathbf{P}_{\pm}\}f^\varepsilon\right\|^2.
\end{split}
\end{align}

Recalling the a priori assumption \eqref{priori assumption}, the desired estimate \eqref{a priori estimates} follows directly from \eqref{a priori estimates 1}. This completes the proof of Proposition \ref{energy estimates total}.
\end{proof}
\medskip

\begin{proposition}\label{high energy estimates total}
There is $\mathcal{E}^h_{N,l}(t)$ satisfying \eqref{high energy functional} such that
\begin{align}\label{a priori estimates2}
\frac{\d}{\d t} \mathcal{E}^h_{N,l}(t) + \lambda  \mathcal{D}_{N,l}(t) \lesssim
\left\|\nabla_x(a_{\pm}^\varepsilon, b^\varepsilon, c^\varepsilon)\right\|^2
+\left\{\varepsilon \left\| \partial_t \phi^\varepsilon \right\|_{L^\infty}+  \chi_{\gamma}\left\|\nabla_x \phi^{\varepsilon}\right\|_{L^\infty}\right\} \mathcal{E}^h_{N,l}(t).
\end{align}
\end{proposition}

\begin{proof}
Let $1 \leq |\alpha| \leq N$ in \eqref{without weight}. Repeating those computations in Lemma \ref{energy estimate without weight}, one can obtain
\begin{align}\label{1-N estimate}
&\frac{\d}{\d t}\left\{\sum_{1 \leq |\alpha| \leq N}\sum_{\pm}\left\|e^{\pm \frac{\varepsilon \phi^\varepsilon}{2}}\partial^\alpha f_{\pm}^\varepsilon\right\|^2
+\sum_{1 \leq |\alpha| \leq N}\left\|\partial^\alpha \nabla_x\phi^\varepsilon\right\|^2 \right\} +\frac{\lambda}{\varepsilon^2}
\sum_{1 \leq |\alpha| \leq N}\left\|\partial^\alpha \{\mathbf{I}-\mathbf{P}\}f^\varepsilon\right\|^2_{N^s_\gamma} \nonumber \\
\lesssim\;& \varepsilon \left\|\partial_t \phi^\varepsilon \right\|_{L^\infty} \sum_{1 \leq |\alpha| \leq N}\sum_{\pm}\left\|e^{\pm \frac{\varepsilon \phi^\varepsilon}{2}}\partial^\alpha f_{\pm}^\varepsilon\right\|^2
+\left\{\mathcal{E}^{1/2}_{N,l}(t)+\mathcal{E}_{N,l}(t)\right\}\mathcal{D}_{N,l}(t).
\end{align}

Moreover, taking the inner product of \eqref{rrrVPB} with $e^{\pm \varepsilon \phi^\varepsilon} \{\mathbf{I}_\pm-\mathbf{P}_\pm\}f^\varepsilon$ over $\mathbb{R}_x^3 \times \mathbb{R}_v^3$ and summing with $\pm$, we have
\begin{align}
&\sum_{\pm}\left( \partial_t \{\mathbf{I}_\pm-\mathbf{P}_\pm\}f^\varepsilon, e^{\pm \varepsilon \phi^\varepsilon} \{\mathbf{I}_\pm-\mathbf{P}_\pm\}f^\varepsilon\right)
+\sum_{\pm}\left( \pm\frac{1}{\varepsilon} \partial^{e_i}\phi^\varepsilon v_i \mu^{1/2},
\{\mathbf{I}_\pm-\mathbf{P}_\pm\}f^\varepsilon\right) \nonumber\\
&+\sum_{\pm}\left( \frac{1}{\varepsilon^2} L_{\pm} \{\mathbf{I}-\mathbf{P}\}f^\varepsilon,  \{\mathbf{I}_\pm-\mathbf{P}_\pm\}f^\varepsilon\right) \nonumber\\
=\;&-\sum_{\pm}\left(\frac{1}{\varepsilon^2} L_{\pm} \{\mathbf{I}-\mathbf{P}\}f^\varepsilon, \left(e^{\pm \varepsilon \phi^\varepsilon}-1\right) \{\mathbf{I}_\pm-\mathbf{P}_\pm\}f^\varepsilon\right) \label{E^h-N,l-formula} \\
&-\sum_{\pm}\left( \pm\frac{1}{\varepsilon} \partial^{e_i}\phi^\varepsilon v_i \mu^{1/2},  \left(e^{\pm \varepsilon \phi^\varepsilon}-1\right) \{\mathbf{I}_\pm-\mathbf{P}_\pm\}f^\varepsilon\right) \nonumber\\
&-\sum_{\pm}\left( \left[\partial_t+\frac{1}{\varepsilon}v_i \partial^{e_i} \mp \partial^{e_i} \phi^\varepsilon \partial_{e_i}\right] \mathbf{P}_{ \pm} f^\varepsilon, e^{\pm \varepsilon \phi^\varepsilon} \{\mathbf{I}_\pm-\mathbf{P}_\pm\}f^\varepsilon\right) \nonumber\\
&+\frac{1}{2}\sum_{\pm}\left( \mp \partial^{e_i}\phi^\varepsilon v_i \mathbf{P}_{\pm}f^\varepsilon, e^{\pm \varepsilon \phi^\varepsilon} \{\mathbf{I}_\pm-\mathbf{P}_\pm\}f^\varepsilon\right)
+\sum_{\pm}\left( \frac{1}{\varepsilon} \Gamma_{\pm}(f^\varepsilon,f^\varepsilon), e^{\pm \varepsilon \phi^\varepsilon} \{\mathbf{I}_\pm-\mathbf{P}_\pm\}f^\varepsilon \right),\nonumber
\end{align}
where we have used
\begin{align}\nonumber
\left(\frac{1}{\varepsilon}v_i\partial^{e_i} \{\mathbf{I}_{\pm}-\mathbf{P}_{\pm}\}f^\varepsilon \pm \frac{1}{2}\partial^{e_i}\phi^\varepsilon v_i \{\mathbf{I}_{\pm}-\mathbf{P}_{\pm}\}f^\varepsilon,
e^{\pm \varepsilon \phi^\varepsilon} \{\mathbf{I}_{\pm}-\mathbf{P}_{\pm}\}f^\varepsilon \right)=0
\end{align}
and
\begin{align}\nonumber
\left(\partial^{e_i}\phi^\varepsilon \partial_{e_i}\{\mathbf{I}_{\pm}-\mathbf{P}_{\pm}\}f^\varepsilon, e^{\pm \varepsilon \phi^\varepsilon} \{\mathbf{I}_{\pm}-\mathbf{P}_{\pm}\}f^\varepsilon \right)=0.
\end{align}

Now we denote all terms in \eqref{E^h-N,l-formula} by $J_{5,1}$ to $J_{5,8}$ and estimate them term by term.
For the term $J_{5,1}$, we acquire from integration by parts with respect to $t$ that
\begin{align}
J_{5,1}=\;&\frac{1}{2}\frac{\d}{\d t} \sum_{\pm}\left\|e^{\pm \frac{\varepsilon \phi^\varepsilon}{2}} \{\mathbf{I}_{\pm}-\mathbf{P}_{\pm}\}f^\varepsilon \right\|^2
-\frac{1}{2}\sum_{\pm} \left(\pm \varepsilon \partial_t \phi^\varepsilon  \{\mathbf{I}_{\pm}-\mathbf{P}_{\pm}\}f^\varepsilon, e^{\pm \varepsilon \phi^\varepsilon}  \{\mathbf{I}_{\pm}-\mathbf{P}_{\pm}\}f^\varepsilon\right). \nonumber
\end{align}
The second term  on the right-hand side of the above equality is bounded by
\begin{equation}\nonumber
\varepsilon \left\|\partial_t \phi^\varepsilon \right\|_{L^\infty} \sum_{\pm}\left\|e^{\pm \frac{\varepsilon \phi^\varepsilon}{2}}  \{\mathbf{I}_{\pm}-\mathbf{P}_{\pm}\}f^\varepsilon\right\|^2.
\end{equation}
For the term $J_{5,2}$, recalling the first equation of \eqref{macro equation 5} and \eqref{a equality}, one has
\begin{align}\nonumber
J_{5,2}=\sum_{\pm} \left(\pm\frac{1}{\varepsilon}\partial^{e_i} \phi^\varepsilon v_i \mu^{1/2},  \{\mathbf{I}_{\pm}-\mathbf{P}_{\pm}\}f^\varepsilon \right)
=\frac{1}{2}\frac{\d}{\d t}\left\| \nabla_x \phi^\varepsilon \right\|^2.
\end{align}
For the term $J_{5,3}$, it follows from \eqref{L coercive1} that
\begin{align}
\begin{split}\nonumber
J_{5,3}=\;&\left(\frac{1}{\varepsilon^2} L \partial^\alpha \{\mathbf{I}-\mathbf{P}\}f^\varepsilon,
\partial^\alpha \{\mathbf{I}-\mathbf{P}\}f^\varepsilon\right)
\geq\frac{\lambda}{\varepsilon^2} \left\| \{\mathbf{I}-\mathbf{P}\}f^\varepsilon\right\|^2_{N^s_\gamma}.
\end{split}
\end{align}
For the term $J_{5,4}$, combining \eqref{Gamma zeta2}, \eqref{Gamma zeta3} and \eqref{phi estimate1}, we have
\begin{align}
J_{5,4}\lesssim \frac{1}{\varepsilon} \left\| \nabla_x \phi^\varepsilon\right\|_{H^1} \left\|\{\mathbf{I}-\mathbf{P}\}f^\varepsilon \right\|^2_{N^s_\gamma}
\lesssim \mathcal{E}^{1/2}_{N,l}(t)\mathcal{D}_{N,l}(t). \nonumber
\end{align}
For the term $J_{5,5}$, we can obtain from \eqref{phi estimate1} that
\begin{align}\nonumber
J_{5,5}= \sum_{\pm} \left(\mp\frac{1}{\varepsilon}\partial^{e_i} \phi^\varepsilon v_i \mu^{1/2},
\left(e^{ \pm \varepsilon \phi^\varepsilon}-1\right)  \{\mathbf{I}_{\pm}-\mathbf{P}_{\pm}\}f^\varepsilon\right)
\lesssim \mathcal{E}^{1/2}_{N,l}(t)\mathcal{D}_{N,l}(t).
\end{align}
For the terms $J_{5,6}$ and $J_{5,7}$, by utilizing \eqref{macro equation 3}, $ \eqref{g estimate}$, the first equation of \eqref{macro equation 5}, the Sobolev inequalities \eqref{Sobolev-ineq} and the Cauchy--Schwarz inequality with $\eta$, they can be dominated by
\begin{align}\nonumber
\frac{\eta}{\varepsilon^2}\left\|\{\mathbf{I}-\mathbf{P}\}f^\varepsilon\right\|^2_{N^s_\gamma}
+ \left\| \nabla_x (a_{\pm}^\varepsilon, b^\varepsilon, c^\varepsilon) \right\|^2
+ \mathcal{E}_{N}(t)\mathcal{D}_{N}(t),\nonumber
\end{align}
where we used the definitions of $\mathbf{P}$ and $\mathbf{I}-\mathbf{P}$.
Finally, for the term $J_{5,8}$, it follows from \eqref{soft gamma3} and \eqref{hard gamma3} that $J_{5,8} \lesssim \mathcal{E}^{1/2}_{N,l}(t)\mathcal{D}_{N,l}(t)$. Therefore, combining all the estimates above, one has
\begin{align}\label{I-P estimate}
&\frac{\d}{\d t}\left\{\sum_{\pm}\left\|e^{\pm \frac{\varepsilon \phi^\varepsilon}{2}}\{\mathbf{I}_{\pm}-\mathbf{P}_{\pm}\}f^\varepsilon\right\|^2
+\left\|\nabla_x\phi^\varepsilon\right\|^2\right\}
+\frac{\lambda}{\varepsilon^2}\left\| \{\mathbf{I}-\mathbf{P}\}f^\varepsilon\right\|^2_{N^s_\gamma} \\
\lesssim\;& \varepsilon \left\|\partial_t \phi^\varepsilon \right\|_{L^\infty} \sum_{\pm}\left\|e^{\pm \frac{\varepsilon \phi^\varepsilon}{2}}\{\mathbf{I}_{\pm}-\mathbf{P}_{\pm}\}f^\varepsilon\right\|^2
+\left\| \nabla_x(a_{\pm}^\varepsilon, b^\varepsilon, c^\varepsilon)\right\|^2
+\left\{\mathcal{E}^{1/2}_{N,l}(t)+\mathcal{E}_{N,l}(t)\right\}\mathcal{D}_{N,l}(t). \nonumber
\end{align}

As in Proposition \ref{energy estimates total}, under the a priori assumption \eqref{priori assumption}, a suitable linear combination of \eqref{1-N estimate}, \eqref{I-P estimate}, \eqref{macro estimate high}, \eqref{weighted estimate1}, \eqref{weighted estimate2} and \eqref{weighted estimate3} yields the desired estimate \eqref{a priori estimates2}. This completes the proof of Proposition \ref{high energy estimates total}.
\end{proof}
\medskip

\begin{proposition}\label{negative sobolev energy estimates total}
There is $\widetilde{\mathcal{E}}_{N,l}(t)$ satisfying \eqref{negative sobolev energy} such that
\begin{align}\label{a priori estimates3}
\frac{\d}{\d t} \widetilde{\mathcal{E}}_{N,l}(t) + \lambda  \widetilde{\mathcal{D}}_{N,l}(t)
\lesssim \left\{\varepsilon \left\| \partial_t \phi^\varepsilon \right\|_{L^\infty}+  \chi_{\gamma}\left\|\nabla_x \phi^{\varepsilon}\right\|_{L^\infty}\right\} \mathcal{E}_{N,l}(t).
\end{align}
\end{proposition}

\begin{proof}
Similar to Lemma \ref{macroscopic estimate}, there exists an interactive functional $\mathcal{E}^{-\varrho}_{int}(t)$ satisfying
\begin{align}\nonumber
\left|\mathcal{E}^{-\varrho}_{int}(t)\right| \lesssim \left\| \Lambda^{-\varrho} f^\varepsilon \right\|^2
+ \left\| \Lambda^{1-\varrho} f^\varepsilon \right\|^2
+ \left\| \Lambda^{-\varrho} \nabla_x \phi^\varepsilon \right\|^2_{H^2},
\end{align}
such that
\begin{align}
\begin{split}\label{macro negative sobolev1}
&\varepsilon \frac{\d}{\d t} \mathcal{E}^{-\varrho}_{int}(t) + \left\| \Lambda^{1-\varrho} (a_{\pm}^\varepsilon, b^\varepsilon, c^\varepsilon) \right\|^2+  \left\| \Lambda^{-\varrho} \nabla_x \phi^ \varepsilon \right\|_{H^2}^2 \\
\lesssim\;& \left\| \Lambda^{1-\varrho}\{\mathbf{I}-\mathbf{P}\}f^\varepsilon \right\|^2_{N^s_\gamma}
+ \frac{1}{\varepsilon^2}\left\| \Lambda^{-\varrho}\{\mathbf{I}-\mathbf{P}\}f^\varepsilon \right\|^2_{N^s_\gamma}
+\varepsilon^2 \left\| \langle \Lambda^{-\varrho}(g^\varepsilon_{+}+g^\varepsilon_{-}+h^\varepsilon_{+}+h^\varepsilon_{-}), \zeta \rangle\right\|_{L^2_x}^2.
\end{split}
\end{align}
By the interpolation inequality with respect to the spatial derivatives, it is easy to derive that
\begin{align}
\begin{split}\label{interpolation negative}
&\left\| \Lambda^{1-\varrho}\{\mathbf{I}-\mathbf{P}\}f^\varepsilon \right\|^2_{N^s_\gamma}
\lesssim \left\| \Lambda^{-\varrho}\{\mathbf{I}-\mathbf{P}\}f^\varepsilon \right\|^2_{N^s_\gamma}
+ \left\|  \nabla_x \{\mathbf{I}-\mathbf{P}\}f^\varepsilon \right\|^2_{N^s_\gamma},  \\
&\left\| \Lambda^{\frac{3}{4}-\frac{\varrho}{2}} \nabla_x \phi^\varepsilon \right\|^2
+\left\| \Lambda^{\frac{3}{2}-\varrho} \nabla_x \phi^\varepsilon \right\|^2
\lesssim \left\| \Lambda^{-\varrho} \nabla_x \phi^\varepsilon \right\|^2
+\left\| \nabla^2_x \phi^\varepsilon \right\|^2,   \\
&\left\| \Lambda^{\frac{3}{4}-\frac{\varrho}{2}}f^\varepsilon  \right\|^2_{N^s_\gamma}
\lesssim  \left\| \Lambda^{1-\varrho}\mathbf{P}f^\varepsilon \right\|^2
+\left\| \nabla_x \mathbf{P}f^\varepsilon \right\|^2
+\left\| \Lambda^{-\varrho}\{\mathbf{I}-\mathbf{P}\}f^\varepsilon \right\|^2_{N^s_\gamma}
+\left\| \nabla_x\{\mathbf{I}-\mathbf{P}\}f^\varepsilon \right\|^2_{N^s_\gamma}, \\
&\left\| \Lambda^{\frac{3}{2}-\varrho}f^\varepsilon \right\|^2
\lesssim  \left\| \Lambda^{1-\varrho} \mathbf{P}f^\varepsilon \right\|^2
+\left\| \nabla_x \mathbf{P}f^\varepsilon \right\|^2
+\left\| \langle v \rangle ^{-\frac{\gamma+2s}{2}} \{\mathbf{I}-\mathbf{P}\}f^\varepsilon \right\|_{H^1_x N^s_\gamma}^2,
\end{split}
\end{align}
where we have used $\frac{3}{4}-\frac{\varrho}{2} > 1-\varrho$ for $1 < \varrho < \frac{3}{2}$.
Recalling \eqref{g,h define} for $g^\varepsilon_{\pm}$ and $h^\varepsilon_{\pm}$, we obtain from \eqref{Gamma zeta1} and \eqref{interpolation negative} that
\begin{align}
\begin{split}\label{g estimate negative sobolev}
\varepsilon \left\| \langle \Lambda^{-\varrho}(g^\varepsilon_{+}+g^\varepsilon_{-}), \zeta \rangle \right\|_{L^2_x}
\lesssim\;& \left\|  \Lambda^{-\varrho}(\nabla_x \phi^\varepsilon \mu^\delta f^\varepsilon )\right\|
+\left\|  \Lambda^{-\varrho} \langle \Gamma(f^\varepsilon, f^\varepsilon), \zeta \rangle\right\|_{L^2_x} \\
\lesssim\;& \left\| \nabla_x \phi^\varepsilon \mu^\delta f^\varepsilon \right\|_{L^2_v L_x^{\frac{6}{3+2\varrho}}}
+ \left\|  \langle \Gamma(f^\varepsilon, f^\varepsilon), \zeta \rangle\right\|_{L_x^{\frac{6}{3+2\varrho}}} \\
\lesssim\;& \left\| \nabla_x \phi^\varepsilon \right\|_{L^{\frac{12}{3+2\varrho}}}
\left\| \mu^\delta f^\varepsilon \right\|_{L^2_v L_x^{\frac{12}{3+2\varrho}}}
+\left\| |f^\varepsilon|^2_{L^2_{-m}} \right\|_{L_x^{\frac{6}{3+2\varrho}}} \\
\lesssim\;& \left\| \Lambda^{\frac{3}{4}-\frac{\varrho}{2}}\nabla_x \phi^\varepsilon \right\|
\left\| \Lambda^{\frac{3}{4}-\frac{\varrho}{2}} (\mu^\delta f^\varepsilon) \right\|
+\left\| |f^\varepsilon|_{L^2_{-m}} \right\|^2_{L_x^{\frac{12}{3+2\varrho}}}\\
\lesssim\;& \left\| \Lambda^{\frac{3}{4}-\frac{\varrho}{2}}\nabla_x \phi^\varepsilon \right\|
\left\| \Lambda^{\frac{3}{4}-\frac{\varrho}{2}} (\mu^\delta f^\varepsilon) \right\|
+\left\|  \Lambda^{\frac{3}{4}-\frac{\varrho}{2}} f^\varepsilon \right\|^2_{L^2_{-m}} \\
\lesssim \; & \widetilde{\mathcal{E}}^{1/2}_{N,l}(t)\widetilde{\mathcal{D}}^{1/2}_{N,l}(t),
\end{split}
\end{align}
and
\begin{align} \label{h estimate negative sobolev}
\varepsilon \left\| \langle \Lambda^{-\varrho}(h^\varepsilon_{+}+h^\varepsilon_{-}), \zeta \rangle \right\|_{L^2_x}
\lesssim \left\| \Lambda^{1-\varrho}\{\mathbf{I}-\mathbf{P}\}f^\varepsilon \right\|_{N^s_\gamma}
+ \frac{1}{\varepsilon}\left\| \Lambda^{-\varrho}\{\mathbf{I}-\mathbf{P}\}f^\varepsilon \right\|_{N^s_\gamma}.
\end{align}
Here, we used \eqref{negative embed 1}, \eqref{negative embed 2}, the Minkowski inequality \eqref{minkowski}  and the H\"{o}lder inequality. Hence, it follows from \eqref{macro negative sobolev1}--\eqref{h estimate negative sobolev} that
\begin{align}
\begin{split}\label{macro negative sobolev2}
&\varepsilon \frac{\d}{\d t} \mathcal{E}^{-\varrho}_{int}(t) + \left\| \Lambda^{1-\varrho} (a_{\pm}^\varepsilon, b^\varepsilon, c^\varepsilon) \right\|^2+  \left\| \Lambda^{-\varrho} \nabla_x \phi^ \varepsilon \right\|_{H^2}^2 \\
\lesssim\;&\frac{1}{\varepsilon^2}\left\| \Lambda^{-\varrho}\{\mathbf{I}-\mathbf{P}\}f^\varepsilon \right\|^2_{N^s_\gamma}
+ \left\| \nabla_x \{\mathbf{I}-\mathbf{P}\}f^\varepsilon \right\|^2_{N^s_\gamma}
+\widetilde{\mathcal{E}}_{N,l}(t)\widetilde{\mathcal{D}}_{N,l}(t).
\end{split}
\end{align}

Combining \eqref{interpolation negative} with the a priori assumption \eqref{priori assumption},  one can deduce \eqref{a priori estimates3} by a proper linear combination of \eqref{negative sobolev estimate}, \eqref{a priori estimates} and \eqref{macro negative sobolev2}. This completes the proof of Proposition \ref{negative sobolev energy estimates total}.
\end{proof}
\medskip

\section{Global Existence}\label{Global Existence}

In this section, we give the proof of Theorem \ref{mainth1} and \ref{mainth2} by first establishing the closed uniform estimates on $X(t)$ defined in \eqref{X define}.

 Noticing the estimate \eqref{a priori estimates2} and the difference between $\mathcal{E}_{N,l}(t)$ and $\mathcal{D}_{N,l}(t)$, we need to obtain the temporal decay estimates of $\left\| \mathbf{P}f^\varepsilon \right\|$ and $\left\| \nabla_x \mathbf{P}f^\varepsilon \right\|$.

\subsection{Time Decay Estimate}
\hspace*{\fill}

Our main idea to deduce the time decay estimate is based on the approach proposed by \cite{GW2012CPDE}.
First of all, we give a more refined energy estimate for the pure spatial derivatives than \eqref{estimate without weight}.

\begin{lemma}\label{0-N 1-N lemma}
There hold
\begin{align}
&\frac{\d}{\d t} \bigg\{ \sum_{k \leq N} \left\| \nabla^k f^\varepsilon \right\|^2 + \sum_{k \leq N} \left\| \nabla^k \nabla_x \phi^\varepsilon \right\|^2 \bigg\}
+ \frac{\lambda}{\varepsilon^2} \sum_{k \leq N }  \left\| \nabla^k \{\mathbf{I}-\mathbf{P}\} f^\varepsilon \right\|^2_{N^s_\gamma} \nonumber \\
\lesssim\;& \mathcal{E}^{1/2}_{N,l}(t)\bigg\{ \sum_{k \leq N+1} \left\| \nabla^k \nabla_x \phi^\varepsilon \right\|^2
+  \sum_{1 \leq k \leq N} \left\| \nabla^k \mathbf{P}f^\varepsilon \right\|^2 \bigg\}  \label{0-N estimate decay}
\end{align}
and
\begin{align}
&\frac{\d}{\d t} \bigg\{ \sum_{1 \leq k \leq N} \left\| \nabla^k f^\varepsilon \right\|^2 + \sum_{1 \leq k \leq N} \left\| \nabla^k \nabla_x \phi^\varepsilon \right\|^2 \bigg\}
+ \frac{\lambda}{\varepsilon^2} \sum_{1 \leq k \leq N }  \left\| \nabla^k \{\mathbf{I}-\mathbf{P}\} f^\varepsilon \right\|^2_{N^s_\gamma} \nonumber \\
\lesssim\;& \mathcal{E}^{1/2}_{N,l}(t)\bigg\{ \sum_{1 \leq k \leq N+1} \left\| \nabla^k \nabla_x \phi^\varepsilon \right\|^2
+  \sum_{2 \leq k \leq N} \left\| \nabla^k \mathbf{P}f^\varepsilon \right\|^2 \bigg\}. \label{1-N estimate decay}
\end{align}
\end{lemma}

\begin{proof}
Applying $\nabla^k$ to the first equation of \eqref{rVPB} and taking the inner product with $\nabla^k f^\varepsilon$ over $\mathbb{R}_x^3 \times \mathbb{R}_v^3$, one has
\begin{align} \label{pure spatial estimate}
&\frac{1}{2}\frac{\d}{\d t} \left\{ \left\| \nabla^k f^\varepsilon \right\|^2 + \left\| \nabla^k \nabla_x \phi^\varepsilon \right\|^2 \right\}
+\frac{\lambda}{\varepsilon^2}  \left\| \nabla^k \{\mathbf{I}-\mathbf{P}\} f^\varepsilon \right\|^2_{N^s_\gamma} \nonumber\\
=\;& \left( q_0 \nabla^k \left( \nabla_x \phi^\varepsilon \cdot \nabla_v f^\varepsilon\right), \nabla^k f^\varepsilon \right)
-\left( \frac{q_0}{2} \nabla^k \left(\nabla_x \phi^\varepsilon \cdot v f^\varepsilon \right), \nabla^k f^\varepsilon \right)
+\left( \frac{1}{\varepsilon} \nabla^k \Gamma(f^\varepsilon, f^\varepsilon),\nabla^k \{\mathbf{I}-\mathbf{P}\}f^\varepsilon \right) \nonumber\\
:=\;& J_{6,1}+ J_{6,2}+ J_{6,3}.
\end{align}
Here, we have used the collision invariant property.

Firstly, we estimate $ J_{6,1}$--$J_{6,3} $ with $2 \leq k \leq N$ for soft potentials.
For the term $J_{6,1}$, by employing the H\"{o}lder inequality, the Sobolev inequalities \eqref{Sobolev-ineq} and the Cauchy--Schwarz inequality, one has
\begin{align}
\begin{split}\label{J61 estimate}
J_{6,1} \lesssim\;& \sum_{1 \leq j \leq k} \left\| \nabla^j \nabla_x \phi^\varepsilon \right\|_{L^3}
\left\| \nabla^{k-j} (\mu^\delta f^\varepsilon) \right\|_{L^6_x L^2_v} \left\| \nabla^k \mathbf{P}f^\varepsilon\right\| \\
&+\sum_{1 \leq j \leq k-1} \left\| \nabla^j \nabla_x \phi^\varepsilon \right\|_{L^\infty}
\left\| \langle v \rangle ^{-\frac{\gamma+2s}{2}}\nabla^{k-j} \nabla_v f^\varepsilon \right\| \left\| \nabla^k \{\mathbf{I}-\mathbf{P}\}f^\varepsilon\right\|_{N^s_\gamma}\\
&+\sum_{ j =k} \left\| \nabla^j \nabla_x \phi^\varepsilon \right\|_{L^3}
\left\| \langle v \rangle ^{-\frac{\gamma+2s}{2}} \nabla_v f^\varepsilon \right\|_{L^6_x L^2_v} \left\| \nabla^k \{\mathbf{I}-\mathbf{P}\}f^\varepsilon\right\|_{N^s_\gamma}\\
\lesssim\;& \mathcal{E}^{1/2}_{N,l}(t) \bigg\{ \sum_{1 \leq \ell \leq N+1 }\left\| \nabla^\ell \nabla_x\phi^\varepsilon \right\|^2
+ \left\| \nabla^k \mathbf{P}f^\varepsilon \right\|^2
+ \frac{1}{\varepsilon^2}\left\| \nabla^k \{\mathbf{I}-\mathbf{P}\}f^\varepsilon \right\|_{N^s_\gamma}^2 \bigg\}.
\end{split}
\end{align}
Similarly, $J_{6,2}$ has the same upper bound as \eqref{J61 estimate}.

For the term $J_{6,3}$, we denote $J^{\prime}_{6,3}$, $J^{\prime\prime}_{6,3}$, $J^{\prime\prime\prime}_{6,3}$, $J^{\prime\prime\prime\prime}_{6,3}$ to be the terms corresponding to the decomposition \eqref{gamma decomposition}. For the term $J^{\prime}_{6,3}$,  making use of the H\"{o}lder inequality, the Sobolev inequalities \eqref{Sobolev-ineq} and the Cauchy--Schwarz inequality, we obtain the following result
\begin{align*}
J^{\prime}_{6,3}=\;& \left( \frac{1}{\varepsilon} \nabla^k \Gamma( \mathbf{P}f^\varepsilon, \mathbf{P}f^\varepsilon),
\nabla^k \{\mathbf{I}-\mathbf{P}\}f^\varepsilon \right)\\
\lesssim\;& \frac{1}{\varepsilon} \left\| \mathbf{P}f^\varepsilon \right\|_{L^\infty_x L^2_v}\left\| \nabla^k \mathbf{P}f^\varepsilon \right\|
\left\| \nabla^k \{\mathbf{I}-\mathbf{P}\}f^\varepsilon \right\|_{N^s_\gamma} \\
&+ \frac{1}{\varepsilon} \sum_{1 \leq j \leq k-1}
\left\| \nabla^j \mathbf{P}f^\varepsilon \right\|_{L^6_x L^2_v}\left\| \nabla^{k-j} \mathbf{P}f^\varepsilon \right\|_{L^3_x L^2_v}
\left\| \nabla^k \{\mathbf{I}-\mathbf{P}\}f^\varepsilon \right\|_{N^s_\gamma}\\
\lesssim\;&\mathcal{E}^{1/2}_{N,l}(t) \bigg\{ \sum_{2 \leq \ell \leq N} \left\| \nabla^\ell \mathbf{P}f^\varepsilon\right\|^2
+\frac{1}{\varepsilon^2} \left\| \nabla^k \{\mathbf{I}-\mathbf{P}\}f^\varepsilon\right\|_{N^s_\gamma}^2 \bigg\}.
\end{align*}
In a similar fashion, for the term $J^{\prime \prime}_{6,3}$, selecting the first term within the minimum function in \eqref{soft gamma1} yields
\begin{align}\nonumber
J^{\prime \prime}_{6,3} \lesssim\; \mathcal{E}^{1/2}_{N,l}(t) \bigg\{ \sum_{1 \leq \ell \leq N} \left\| \nabla^\ell \{\mathbf{I}-\mathbf{P}\}f^\varepsilon\right\|_{N^s_\gamma}^2
+\frac{1}{\varepsilon^2} \left\| \nabla^k \{\mathbf{I}-\mathbf{P}\}f^\varepsilon\right\|_{N^s_\gamma}^2 \bigg\}.
\end{align}
Considering the term $J^{\prime \prime \prime}_{6,3}$ and opting for the second term inside the minimum function in \eqref{soft gamma1}, we arrive at the following outcome
\begin{align}
\begin{split}\label{J63 1}
J^{\prime \prime \prime}_{6,3} \lesssim\;& \frac{1}{\varepsilon}
\left\| \{\mathbf{I}-\mathbf{P}\} f^\varepsilon \right\|_{L^\infty_x L^2_v}\left\| \nabla^{k} \mathbf{P}f^\varepsilon \right\|
\left\| \nabla^k \{\mathbf{I}-\mathbf{P}\}f^\varepsilon \right\|_{N^s_\gamma}\\
&+\frac{1}{\varepsilon} \sum_{1 \leq j \leq k-1}
\left\| \nabla^j \{\mathbf{I}-\mathbf{P}\} f^\varepsilon \right\|_{L^3_x L^2_v}\left\| \nabla^{k-j} \mathbf{P}f^\varepsilon \right\|_{L^6_x L^2_v}
\left\| \nabla^k \{\mathbf{I}-\mathbf{P}\}f^\varepsilon \right\|_{N^s_\gamma}\\
&+\frac{1}{\varepsilon}\left\| \nabla^k \{\mathbf{I}-\mathbf{P}\} f^\varepsilon \right\|\left\|  \mathbf{P}f^\varepsilon \right\|_{L^\infty_x L^2_v}
\left\| \nabla^k \{\mathbf{I}-\mathbf{P}\}f^\varepsilon \right\|_{N^s_\gamma}.
\end{split}
\end{align}
The final term on the right-hand side of the above inequality can be bounded by
\begin{align}\label{J63 2}
\frac{1}{\varepsilon} \left\|\nabla^k \{\mathbf{I}-\mathbf{P}\} f^\varepsilon \right\|_{N^s_\gamma}^{\frac{3}{2}}
\left\| \langle v \rangle ^{-\frac{\gamma+2s}{2}}\nabla^k \{\mathbf{I}-\mathbf{P}\} f^\varepsilon \right\|^{\frac{1}{2}}
\left\|  \nabla_x \mathbf{P}f^\varepsilon \right\|^{\frac{1}{2}}
\left\|  \nabla_x^2 \mathbf{P}f^\varepsilon \right\|^{\frac{1}{2}}.
\end{align}
Consequently, thanks to $0 < \varepsilon \leq 1$, by combining \eqref{J63 1} and \eqref{J63 2}, we can derive
\begin{align}\nonumber
J^{\prime \prime \prime}_{6,3} \lesssim \mathcal{E}^{1/2}_{N,l}(t) \bigg\{ \sum_{2 \leq \ell \leq N} \left\| \nabla^\ell \mathbf{P}f^\varepsilon\right\|^2 + \frac{1}{\varepsilon^2} \left\| \nabla^k \{\mathbf{I}-\mathbf{P}\}f^\varepsilon\right\|_{N^s_\gamma}^2 \bigg\}.
\end{align}
Regarding the term $J^{\prime \prime \prime \prime}_{6,3}$, it is crucial to estimate it with great care.

{\emph{Case 1.}}
$j=0$.\; By considering the first term inside the minimum function in $\eqref{soft gamma1}$, we have
\begin{align}\nonumber
J^{\prime \prime \prime \prime}_{6,3}
\lesssim\;& \frac{1}{\varepsilon} \left\| \{\mathbf{I}-\mathbf{P}\}f^\varepsilon \right\|_{L^\infty_x L^2_v}
\left\| \nabla^k \{\mathbf{I}-\mathbf{P}\}f^\varepsilon \right\|^2_{N^s_\gamma} \\
&+\frac{1}{\varepsilon} \sum_{|\beta^\prime| \leq 2}\left\| \partial_{\beta^\prime}\{\mathbf{I}-\mathbf{P}\}f^\varepsilon \right\|_{L^\infty_x L^2_v}
\left\| \nabla^k \{\mathbf{I}-\mathbf{P}\}f^\varepsilon \right\|^2_{N^s_\gamma} \\
\lesssim\;& \mathcal{E}^{1/2}_{N,l}(t)\frac{1}{\varepsilon^2} \left\| \nabla^k \{\mathbf{I}-\mathbf{P}\}f^\varepsilon \right\|^2_{N^s_\gamma}. \nonumber
\end{align}

{\emph{Case 2.}}
$1 \leq j \leq k-3$. \;By selecting the first term from within the minimum function in $\eqref{soft gamma1}$, we obtain
\begin{align}
J^{\prime \prime \prime \prime}_{6,3}
\lesssim\;&\frac{1}{\varepsilon} \left\| \nabla^j \{\mathbf{I}-\mathbf{P}\}f^\varepsilon \right\|_{L^\infty_x L^2_v}
\left\| \nabla^{k-j} \{\mathbf{I}-\mathbf{P}\}f^\varepsilon \right\|_{N^s_\gamma}
\left\| \nabla^k \{\mathbf{I}-\mathbf{P}\}f^\varepsilon \right\|_{N^s_\gamma}\nonumber \\
&+\frac{1}{\varepsilon} \sum_{|\beta^\prime| \leq 2} \left\| \nabla^j \partial_{\beta^\prime}\{\mathbf{I}-\mathbf{P}\}f^\varepsilon \right\|_{L^3_x L^2_v}
\left\| \nabla^{k-j} \{\mathbf{I}-\mathbf{P}\}f^\varepsilon \right\|_{L^6_x N^s_\gamma}
\left\| \nabla^k \{\mathbf{I}-\mathbf{P}\}f^\varepsilon \right\|_{N^s_\gamma} \nonumber\\
\lesssim\;& \mathcal{E}^{1/2}_{N,l}(t) \bigg\{ \sum_{3 \leq \ell \leq N}
\left\| \nabla^\ell \{\mathbf{I}-\mathbf{P}\}f^\varepsilon \right\|^2_{N^s_\gamma}
+\frac{1}{\varepsilon^2} \left\| \nabla^k \{\mathbf{I}-\mathbf{P}\}f^\varepsilon \right\|^2_{N^s_\gamma} \bigg\}. \nonumber
\end{align}

{\emph{Case 3.}}
$j = k-2$. \;  Singling out the first term inside the minimum function in $\eqref{soft gamma1}$, we arrive at
\begin{align}
J^{\prime \prime \prime \prime}_{6,3}
\lesssim\;&\frac{1}{\varepsilon} \left\| \nabla^j \{\mathbf{I}-\mathbf{P}\}f^\varepsilon \right\|_{L^\infty_x L^2_v}
\left\| \nabla^{k-j} \{\mathbf{I}-\mathbf{P}\}f^\varepsilon \right\|_{N^s_\gamma}
\left\| \nabla^k \{\mathbf{I}-\mathbf{P}\}f^\varepsilon \right\|_{N^s_\gamma}\nonumber \\
&+\frac{1}{\varepsilon} \sum_{|\beta^\prime| \leq 2} \left\| \nabla^j \partial_{\beta^\prime}\{\mathbf{I}-\mathbf{P}\}f^\varepsilon \right\|
\left\| \nabla^{k-j} \{\mathbf{I}-\mathbf{P}\}f^\varepsilon \right\|_{L^\infty_x N^s_\gamma}
\left\| \nabla^k \{\mathbf{I}-\mathbf{P}\}f^\varepsilon \right\|_{N^s_\gamma} \nonumber\\
\lesssim\;& \mathcal{E}^{1/2}_{N,l}(t) \bigg\{ \sum_{2 \leq \ell \leq 4}
\left\| \nabla^\ell \{\mathbf{I}-\mathbf{P}\}f^\varepsilon \right\|^2_{N^s_\gamma}
+\frac{1}{\varepsilon^2} \left\| \nabla^k \{\mathbf{I}-\mathbf{P}\}f^\varepsilon \right\|^2_{N^s_\gamma} \bigg\}. \nonumber
\end{align}

{\emph{Case 4.}}
$j = k-1$.\; By choosing the second term inside the minimum function in $\eqref{soft gamma1}$ and setting $m=-\frac{\gamma+2s}{2}$, we get
\begin{align}
J^{\prime \prime \prime \prime}_{6,3}
\lesssim\;&\frac{1}{\varepsilon} \left\| \nabla^j \{\mathbf{I}-\mathbf{P}\}f^\varepsilon \right\|_{L^3_x L^2_v}
\left\| \nabla^{k-j} \{\mathbf{I}-\mathbf{P}\}f^\varepsilon \right\|_{L^6_x N^s_\gamma}
\left\| \nabla^k \{\mathbf{I}-\mathbf{P}\}f^\varepsilon \right\|_{N^s_\gamma}\nonumber \\
&+\frac{1}{\varepsilon} \left\| \langle v \rangle ^ {\frac{\gamma+2s}{2}}\nabla^j \{\mathbf{I}-\mathbf{P}\}f^\varepsilon \right\|_{L^6_x L^2_v}
\sum_{|\beta^\prime| \leq 2}\left\| \nabla^{k-j} \partial_{\beta^\prime}\{\mathbf{I}-\mathbf{P}\}f^\varepsilon \right\|_{L^3_x N^s_\gamma}
\left\| \nabla^k \{\mathbf{I}-\mathbf{P}\}f^\varepsilon \right\|_{N^s_\gamma} \nonumber\\
\lesssim\;& \mathcal{E}^{1/2}_{N,l}(t) \bigg\{
\left\| \nabla^2 \{\mathbf{I}-\mathbf{P}\}f^\varepsilon \right\|^2_{N^s_\gamma}
+\frac{1}{\varepsilon^2} \left\| \nabla^k \{\mathbf{I}-\mathbf{P}\}f^\varepsilon \right\|^2_{N^s_\gamma} \bigg\}, \nonumber
\end{align}
where we have used
\begin{align} \nonumber
\sum_{|\beta^\prime| \leq 2}\left\| \nabla^{k-j} \partial_{\beta^\prime}\{\mathbf{I}-\mathbf{P}\}f^\varepsilon \right\|_{L^3_x N^s_\gamma}
\lesssim \left\| \{\mathbf{I}-\mathbf{P}\}f^\varepsilon \right\|_{H^2_x H^3_v}
\lesssim \mathcal{E}^{1/2}_{N,l}(t).
\end{align}

{\emph{Case 5.}}
$j = k$. \; By selecting the second term within the minimum function in $\eqref{soft gamma1}$ and setting $m=-\frac{\gamma+2s}{2}$, we can deduce
\begin{align}
J^{\prime \prime \prime \prime}_{6,3}
\lesssim\;&\frac{1}{\varepsilon} \left\| \nabla^k \{\mathbf{I}-\mathbf{P}\}f^\varepsilon \right\|
\left\|  \{\mathbf{I}-\mathbf{P}\}f^\varepsilon \right\|_{L_x^\infty N^s_\gamma}
\left\| \nabla^k \{\mathbf{I}-\mathbf{P}\}f^\varepsilon \right\|_{N^s_\gamma}\nonumber \\
&+\frac{1}{\varepsilon} \left\| \langle v \rangle ^ {\frac{\gamma+2s}{2}}\nabla^k \{\mathbf{I}-\mathbf{P}\}f^\varepsilon \right\|
\sum_{|\beta^\prime| \leq 2}\left\| \partial_{\beta^\prime}\{\mathbf{I}-\mathbf{P}\}f^\varepsilon \right\|_{L^\infty_x N^s_\gamma}
\left\| \nabla^k \{\mathbf{I}-\mathbf{P}\}f^\varepsilon \right\|_{N^s_\gamma} \nonumber\\
\lesssim\;& \mathcal{E}^{1/2}_{N,l}(t) \bigg\{ \sum_{1 \leq \ell \leq 2}
\left\| \nabla^\ell \{\mathbf{I}-\mathbf{P}\}f^\varepsilon \right\|^2_{N^s_\gamma}
+\frac{1}{\varepsilon^2} \left\| \nabla^k \{\mathbf{I}-\mathbf{P}\}f^\varepsilon \right\|^2_{N^s_\gamma} \bigg\}. \nonumber
\end{align}
Therefore, combining all the estimates mentioned above with the a priori assumption \eqref{priori assumption}, we can conclude
\begin{align}
&\frac{\d}{\d t} \bigg\{ \sum_{2 \leq k \leq N} \left\| \nabla^k f^\varepsilon \right\|^2 + \sum_{2 \leq k \leq N} \left\| \nabla^k \nabla_x \phi^\varepsilon \right\|^2 \bigg\}
+ \frac{\lambda}{\varepsilon^2} \sum_{2 \leq k \leq N }  \left\| \nabla^k \{\mathbf{I}-\mathbf{P}\} f^\varepsilon \right\|^2_{N^s_\gamma} \nonumber \\
\lesssim\;& \mathcal{E}^{1/2}_{N,l}(t)\bigg\{ \sum_{1 \leq k \leq N+1} \left\| \nabla^k \nabla_x \phi^\varepsilon \right\|^2
+  \sum_{2 \leq k \leq N} \left\| \nabla^k \mathbf{P}f^\varepsilon \right\|^2
+ \left\| \nabla_x \{\mathbf{I}-\mathbf{P}\}f^\varepsilon \right\|^2_{N^s_\gamma} \bigg\}. \label{k=2-N estimate decay}
\end{align}

Now, let's estimate \eqref{pure spatial estimate} for $k=0$. Under the a priori assumption \eqref{priori assumption}, by combining \eqref{hard gamma1} and \eqref{soft gamma1} along with the Sobolev inequalities \eqref{Sobolev-ineq}, it is easy to derive that
\begin{align}\label{k=0 estimate decay}
\frac{1}{2}\frac{\d}{\d t} \left\{ \left\| f^\varepsilon \right\|^2 + \left\|  \nabla_x \phi^\varepsilon \right\|^2 \right\}
+\frac{\lambda}{\varepsilon^2}  \left\|  \{\mathbf{I}-\mathbf{P}\} f^\varepsilon \right\|^2_{N^s_\gamma}
\lesssim \mathcal{E}^{1/2}_{N,l}(t)\left\{ \left\| \nabla_x \phi^\varepsilon \right\|^2
+ \left\| \nabla_x \mathbf{P}f^\varepsilon \right\|^2 \right\}.
\end{align}
Similarly, we can deduce that for $k=1$, there holds
\begin{align}\label{k=1 estimate decay}
\!\!\frac{1}{2}\frac{\d}{\d t} \left\{ \left\| \nabla_x f^\varepsilon \right\|^2 + \left\| \nabla^2_x \phi^\varepsilon \right\|^2 \right\}
+\frac{\lambda}{\varepsilon^2}  \left\| \nabla_x \{\mathbf{I}-\mathbf{P}\} f^\varepsilon \right\|^2_{N^s_\gamma}
\lesssim \mathcal{E}^{1/2}_{N,l}(t) \left\{ \left\| \nabla^2_x\phi^\varepsilon \right\|^2
+ \left\| \nabla^2_x \mathbf{P}f^\varepsilon \right\|^2 \right\}.
\end{align}

Therefore, for the soft potential case, the desired estimates \eqref{0-N estimate decay} and \eqref{1-N estimate decay} follow directly by \eqref{k=2-N estimate decay}, \eqref{k=0 estimate decay} and \eqref{k=1 estimate decay}.

Then, we turn to the hard potential case. In fact,  for the hard potential case, the estimates of $ J_{6,1}$--$J_{6,3}$ in \eqref{pure spatial estimate} with $2 \leq k \leq N$ can be derived in a similar but simpler way than the soft potential case. Moreover, $l = l_2 \geq N+ \max\{1, \gamma+2s\}$ will be used in the process of estimating $J_{6,2}$ and $J_{6,3}$. The details are omitted for simplicity. This completes the proof of Lemma \ref{0-N 1-N lemma}.
\end{proof}
\medskip

\begin{lemma}\label{k-N decay lemma}
There exist the energy functional $\mathcal{E}^k_{N}(t)$ and the corresponding energy dissipation rate functional $\mathcal{D}^k_{N}(t)$ satisfying \eqref{negative sobolev energy} and \eqref{negative sobolev dissipation} respectively such that
\begin{align}\label{k-N estimate}
\frac{\d}{\d t} \mathcal{E}^k_{N}(t)+ \lambda \mathcal{D}^k_{N}(t) \leq 0
\end{align}
holds for $k=0,1$ and all $0 \leq t \leq T$.

Furthermore, we can obtain that for $k=0, 1$,
\begin{align} \label{k-N decay}
\mathcal{E}^k_{N}(t) \lesssim (1+t)^{-(k+\varrho)} \sup_{0 \leq \tau \leq t}\widetilde{\mathcal{E}}_{N,l}(\tau).
\end{align}
Here $l=l_2$ for hard potentials $\gamma+2s \geq 0$, while $l= l_0+l_1$ for soft potentials $-3 < \gamma <-2s$.
\end{lemma}

\begin{proof}
Similar to Lemma \ref{macroscopic estimate}, by recalling \eqref{macro energy2} for $\mathcal{E}^N_{int}(t)$ and considering the a priori assumption \eqref{priori assumption}, we obtain a more precise estimate compared to \eqref{macro estimate},
\begin{align}
\begin{split}\label{0-N macro 1}
&\varepsilon \frac{\d}{\d t} \mathcal{E}^N_{int}(t) + \sum_{1 \leq k \leq N} \left\| \nabla^k \mathbf{P}f^\varepsilon\right\|^2
+\sum_{ k \leq N+1} \left\| \nabla^k \nabla_x \phi^\varepsilon\right\|^2\\
\lesssim\;&\frac{1}{\varepsilon^2} \sum_{k \leq N} \left\| \nabla^k \{\mathbf{I}-\mathbf{P}\}f^\varepsilon\right\|_{N^s_\gamma}^2
+\mathcal{E}_{N}(t)\frac{1}{\varepsilon^2} \sum_{ k \leq N} \left\| \nabla^k \{\mathbf{I}-\mathbf{P}\}f^\varepsilon\right\|_{N^s_\gamma}^2.
\end{split}
\end{align}

Moreover, similar to Lemma \ref{macroscopic estimate}, there is an interactive energy functional $\mathcal{E}^{1 \rightarrow N}_{int}(t)$ satisfying
\begin{align}\nonumber
\left| \mathcal{E}^{1 \rightarrow N}_{int}(t) \right| \lesssim \sum_{1 \leq |\alpha| \leq N+1} \left\| \partial^\alpha \nabla_x \phi^{\varepsilon}\right\|^2
+\sum_{1 \leq |\alpha| \leq N} \left\| \partial^\alpha f^{\varepsilon}\right\|^2
\end{align}
such that
\begin{align}
\begin{split} \label{1-N macro 1}
&\varepsilon \frac{\d}{\d t} \mathcal{E}^{1 \rightarrow N}_{int}(t) + \sum_{2 \leq k \leq N} \left\| \nabla^k \mathbf{P}f^\varepsilon\right\|^2
+\sum_{1\leq k \leq N+1} \left\| \nabla^k \nabla_x \phi^\varepsilon\right\|^2\\
\lesssim\;&\frac{1}{\varepsilon^2} \sum_{1 \leq k \leq N} \left\| \nabla^k \{\mathbf{I}-\mathbf{P}\}f^\varepsilon\right\|_{N^s_\gamma}^2
+\mathcal{E}_{N}(t) \frac{1}{\varepsilon^2}\sum_{1 \leq k \leq N} \left\| \nabla^k \{\mathbf{I}-\mathbf{P}\}f^\varepsilon\right\|_{N^s_\gamma}^2.
\end{split}
\end{align}

Therefore, recalling \eqref{low k energy} for $\mathcal{E}^k_N(t)$, \eqref{k-N estimate} for $k=0$ can be derived from \eqref{0-N estimate decay}, \eqref{0-N macro 1} and the a priori assumption \eqref{priori assumption}.
Similarly, \eqref{k-N estimate} for $k=1$ follows from \eqref{1-N estimate decay}, \eqref{1-N macro 1} and the a priori assumption \eqref{priori assumption}.

To deduce \eqref{k-N decay}, for $k=0,1$, one has
\begin{align}\nonumber
\left\| \nabla^k \mathbf{P}f^\varepsilon \right\|
\lesssim \left\| \nabla^{k+1} \mathbf{P}f^\varepsilon \right\|^{\frac{k+\varrho}{k+1+\varrho}}
\left\| \Lambda^{-\varrho} \mathbf{P}f^\varepsilon \right\|^{\frac{1}{k+1+\varrho}}.
\end{align}
The above inequality together with the fact that
\begin{align}\nonumber
\left\| \nabla^m \{\mathbf{I}-\mathbf{P}\}f^\varepsilon \right\|
\lesssim\;& \left( \frac{1}{\varepsilon} \left\| \langle v \rangle ^{\frac{\gamma+2s}{2}}\nabla^m \{\mathbf{I}-\mathbf{P}\}f^\varepsilon\right\|\right)^{\frac{k+\varrho}{k+1+\varrho}}
\left( \varepsilon^{\frac{1}{2}}\left\| \langle v \rangle ^{-\frac{\gamma+2s}{2}(k+\varrho)}\nabla^m \{\mathbf{I}-\mathbf{P}\}f^\varepsilon\right\|\right)^{\frac{1}{k+1+\varrho}}
\end{align}
implies
\begin{align} \nonumber
\mathcal{E}^k_{N}(t) \lesssim \left\{\widetilde{\mathcal{E}}_{N,l}(t)\right\}^{\frac{1}{k+1+\varrho}}
\left\{ \mathcal{D}^k_{N}(t) \right\}^{\frac{k+\varrho}{k+1+\varrho}}
\lesssim \left\{\sup_{0 \leq \tau \leq t}\widetilde{\mathcal{E}}_{N,l}(\tau) \right\} ^{\frac{1}{k+1+\varrho}}
\left\{ \mathcal{D}^k_{N}(t) \right\}^{\frac{k+\varrho}{k+1+\varrho}}.
\end{align}
Hence, we deduce that
\begin{align}
\frac{\d}{\d t} \mathcal{E}^k_{N}(t)+ \lambda \left\{ \sup_{0 \leq \tau \leq t}\widetilde{\mathcal{E}}_{N,l}(\tau) \right\} ^{-\frac{1}{k+\varrho}}
\left\{ \mathcal{E}^k_{N}(t) \right\}^{1+\frac{1}{k+\varrho}} \leq 0.
\end{align}
Therefore, \eqref{k-N decay} can be deduced by solving the above inequality directly. This completes the proof of Lemma \ref{k-N decay lemma}.
\end{proof}
\medskip

\subsection{Proof of Global Existence for Soft Potentials}
\hspace*{\fill}

In this subsection, we prove Theorem \ref{mainth1} by closing the a priori estimates on $X(t)$ in the following lemma. The idea of the following proof is similar to the Boltzmann case in \cite{Strain2012}.
\begin{lemma}\label{soft close priori estimate}
Let $0 < T \leq \infty $. Assume that there exists a suitable small constant $\delta_0 > 0$ such that the a priori assumption \eqref{priori assumption} holds. Then one has
\begin{align}\label{close priori estimate}
X(t) \lesssim \widetilde{\mathcal{E}}_{N,l_0+l_1}(0)+ X^{\frac{3}{2}}(t),
\end{align}
for any $0 \leq t \leq T$.
\end{lemma}

\begin{proof}
We divide it into three steps.

\text{{Step 1.}} \  It follows from the first equation of \eqref{macro equation 5} and \eqref{a equality} that
\begin{align}
\partial_t \phi^\varepsilon = -\Delta_x^{-1}\partial_t(a_{+}^\varepsilon-a_{-}^\varepsilon)
=\frac{1}{\varepsilon} \Delta^{-1}_x \nabla_x \cdot G^\varepsilon,
\end{align}
where $G^\varepsilon$ is given by \eqref{G define}. Then by the Sobolev inequality and the Riesz inequality, we have
\begin{align}\label{partial t phi}
\begin{split}
\varepsilon^2 \left\| \partial_t \phi^\varepsilon\right\|^2_{L^\infty}
&\lesssim \left\| G^\varepsilon \right\| \left\| \nabla_x G^\varepsilon\right\|
\lesssim \mathcal{E}^h_{N,l}(t)
\lesssim X(t)(1+t)^{-(\varrho+p)},\\
\left\| \nabla_x \phi^\varepsilon \right\|^2_{L^\infty}
&\lesssim \left\| \nabla^2_x \phi^\varepsilon \right\|
\left\| \nabla^3_x \phi^\varepsilon \right\|
\lesssim \mathcal{E}^h_{N,l}(t)
\lesssim X(t)(1+t)^{-(\varrho+p)}.
\end{split}
\end{align}
Here $1 < \varrho < 3/2$, $ 1/2 < p < 1$  and $\varrho+p >2$.
By applying Proposition  \ref{negative sobolev energy estimates total} in the case when $l=l_0+l_1$, one has from \eqref{a priori estimates3} that
\begin{align}\nonumber
\frac{\d}{\d t} \widetilde{\mathcal{E}}_{N,l_0+l_1}(t) + \lambda  \widetilde{\mathcal{D}}_{N,l_0+l_1}(t) \lesssim  \left\{ \varepsilon \left\| \partial_t \phi^\varepsilon \right\|_{L^\infty}+ \left\| \nabla_x \phi^\varepsilon \right\|_{L^\infty} \right\} \mathcal{E}_{N,l_0+l_1}(t).
\end{align}
This, together with \eqref{partial t phi} and the Gronwall inequality, as well as \eqref{priori assumption}, implies that
\begin{align}\label{l0+l1}
\widetilde{\mathcal{E}}_{N,l_0+l_1}(t)
+\int_{0}^t \widetilde{\mathcal{D}}_{N,l_0+l_1}(\tau) \d \tau
\lesssim \widetilde{\mathcal{E}}_{N,l_0+l_1}(0)e^{C\int_{0}^t \{ \varepsilon \left\| \partial_t \phi^\varepsilon \right\|_{L^\infty}
+ \left\| \nabla_x \phi^\varepsilon \right\|_{L^\infty} \}  \d \tau }
\lesssim \widetilde{\mathcal{E}}_{N,l_0+l_1}(0)
\end{align}
for any $0 \leq t \leq T$.

\text{{Step 2.}} \   We begin with \eqref{a priori estimates2} with $l=l_0$ in Proposition \ref{high energy estimates total}, i.e.
\begin{align}\label{high l0}
\frac{\d}{\d t} \mathcal{E}^h_{N,l_0}(t) + \lambda  \mathcal{D}_{N,l_0}(t) \lesssim
\left\|\nabla_x \mathbf{P}f^\varepsilon \right\|^2
+\left\{\varepsilon \left\| \partial_t \phi^\varepsilon \right\|_{L^\infty}+ \left\|\nabla_x \phi^{\varepsilon}\right\|_{L^\infty}\right\} \mathcal{E}^h_{N,l_0}(t).
\end{align}
As in \cite{Strain2012, SG2008ARMA}, at time $t$, we split the velocity space $\mathbb{R}^3_v$ into
\begin{align}\label{E,EC}
\mathbb{E}(t)=\left\{ \langle v \rangle ^{-(\gamma+2s)} \leq t^{1-p}\right\}, \quad
\mathbb{E}^c(t)=\left\{ \langle v \rangle ^{-(\gamma+2s)} > t^{1-p}\right\},
\end{align}
where $-3 < \gamma < -2s$ and $1/2 < p < 1$. Corresponding to this splitting, we define $\mathcal{E}^{h,low}_{N,l_0}(t)$ to be the restriction of $\mathcal{E}^h_{N,l_0}(t)$ to $\mathbb{E}(t)$ and similarly $\mathcal{E}^{h,high}_{N,l_0}(t)$ to be the restriction of $\mathcal{E}^h_{N,l_0}(t)$ to $\mathbb{E}^c(t)$.
Then, thanks to \eqref{E,EC} and \eqref{high energy functional}, it follows from \eqref{high l0} that
\begin{align}\nonumber
\frac{\d}{\d t} \mathcal{E}^h_{N,l_0}(t) + \lambda p t^{p-1}  \mathcal{E}^h_{N,l_0}(t) \lesssim
\left\|\nabla_x \mathbf{P}f^\varepsilon\right\|^2
+\left\{\varepsilon \left\| \partial_t \phi^\varepsilon \right\|_{L^\infty}+ \left\|\nabla_x \phi^{\varepsilon}\right\|_{L^\infty}\right\} \mathcal{E}^h_{N,l_0}(t) + t^{p-1} \mathcal{E}^{h,high}_{N,l_0}(t),
\end{align}
which by solving the ODE inequality, gives
\begin{align}\label{high l0 1}
\mathcal{E}^h_{N,l_0}(t) \lesssim\; &\!\int_0^t \!e^{-\lambda (t^p-\tau^p)}\!
\left( \!\left\|\nabla_x \mathbf{P}f^\varepsilon\right\|^2\!
+\left\{\varepsilon \left\| \partial_t \phi^\varepsilon \right\|_{L^\infty}\!+ \!\left\|\nabla_x \phi^{\varepsilon}\right\|_{L^\infty}\!\right\} \mathcal{E}^h_{N,l_0}(\tau) + \!\tau^{p-1} \mathcal{E}^{h,high}_{N,l_0}(\tau)\right) \d \tau \nonumber\\
& +e^{-\lambda t^p}\mathcal{E}^h_{N,l_0}(0).
\end{align}

In what follows we estimate these three terms in the time integral on the right-hand side of \eqref{high l0 1}.
First of all, making use of \eqref{partial t phi}, one has
\begin{align}\label{partial t phi decay}
\{ \varepsilon \left\| \partial_t \phi^\varepsilon \right\|_{L^\infty}+ \left\|\nabla_x \phi^{\varepsilon}\right\|_{L^\infty}\} \mathcal{E}^h_{N,l_0}(t)
\lesssim \{\mathcal{E}^h_{N,l_0}(t)\} ^{\frac{3}{2}}
\lesssim X^{\frac{3}{2}}(t)(1+t)^{-\frac{3}{2}(\varrho+p)}.
\end{align}
Next, one can prove
\begin{align}\label{E h,high decay}
\mathcal{E}^{h,high}_{N,l_0}(t) \leq \widetilde{\mathcal{E}}_{N,l_0+l_1}(0) (1+t)^{-(\varrho+p)}.
\end{align}
In fact, it is straightforward to see from \eqref{l0+l1} that for any $0 \leq t \leq T$,
\begin{align}\nonumber
\mathcal{E}^{h,high}_{N,l_0}(t) \lesssim \mathcal{E}^{h}_{N,l_0}(t) \lesssim \mathcal{E}_{N,l_0}(t)
\lesssim \mathcal{E}_{N,l_0+l_1}(t) \lesssim \widetilde{\mathcal{E}}_{N,l_0+l_1}(t) \lesssim \widetilde{\mathcal{E}}_{N,l_0+l_1}(0).
\end{align}
On the other hand, noticing that $t^{\varrho+p} \leq  \langle v \rangle^{-\frac{\varrho+p}{1-p}(\gamma+2s)}$ on $\mathbb{E}^c(t)$ and $l_1$ is defined in \eqref{soft assumption},
one has from \eqref{l0+l1} that
\begin{align}\nonumber
\mathcal{E}^{h,high}_{N,l_0}(t) \lesssim \widetilde{\mathcal{E}}_{N,l_0+l_1}(0) t^{-(\varrho+p)}.
\end{align}
Therefore, \eqref{E h,high decay} holds true. Finally, by employing \eqref{k-N decay} and \eqref{l0+l1}, we obtain
\begin{align}\label{nabla Pf decay}
\left\| \nabla_x \mathbf{P}f^\varepsilon\right\|^2 \lesssim (1+t)^{-(1+\varrho)} \sup_{0 \leq \tau \leq t}\widetilde{\mathcal{E}}_{N,l_0+l_1}(\tau)
\lesssim \widetilde{\mathcal{E}}_{N,l_0+l_1}(0) (1+t)^{-(1+\varrho)}.
\end{align}
Notice that one has the following inequalities:
\begin{align}
\begin{split}\label{decay inequality}
&\int_{0}^t e^{-\lambda (t^p-\tau^p)}(1+\tau)^{-\frac{3}{2}(\varrho+p)} \d \tau
\lesssim (1+t)^{-(\frac{3\varrho}{2}-1+\frac{5p}{2})}, \\
&\int_{0}^t e^{-\lambda (t^p-\tau^p)}\tau^{p-1}(1+\tau)^{-(\varrho+p)} \d \tau
\lesssim (1+t)^{-(\varrho+p)},  \\
&\int_{0}^t e^{-\lambda (t^p-\tau^p)}(1+\tau)^{-(1+\varrho)} \d \tau
\lesssim (1+t)^{-(\varrho+p)}.
\end{split}
\end{align}
Then, plugging \eqref{partial t phi decay}, \eqref{E h,high decay} and \eqref{nabla Pf decay} into \eqref{high l0 1} and using \eqref{decay inequality}, one has
\begin{align}\label{N l0 high decay}
\mathcal{E}^h_{N,l_0}(t) \lesssim \left\{ \widetilde{\mathcal{E}}_{N,l_0+l_1}(0)+X^{\frac{3}{2}}(t)\right\} (1+t)^{-(\varrho+p)}
\end{align}
for any $0 \leq t \leq T$. Here we have used $1< \varrho < 3/2$ and $1/2 < p < 1$.

\text{{Step 3.}}  \  By employing \eqref{k-N decay} and \eqref{l0+l1}, one can obtain
\begin{align}\label{Pf decay}
\left\| \mathbf{P}f^\varepsilon \right\|^2 \lesssim  (1+t)^{-\varrho} \sup_{0 \leq \tau \leq t}\widetilde{\mathcal{E}}_{N,l_0+l_1}(\tau)
\lesssim \widetilde{\mathcal{E}}_{N,l_0+l_1}(0) (1+t)^{-\varrho}.
\end{align}
Noticing $\mathcal{E}_{N,l_0}(t) \approx \left\| \mathbf{P}f^\varepsilon\right\|^2+ \mathcal{E}^h_{N,l_0}(t)$, \eqref{Pf decay} together with \eqref{N l0 high decay} give
\begin{align}\label{N l0 decay}
\mathcal{E}_{N,l_0}(t) \lesssim \left\{ \widetilde{\mathcal{E}}_{N,l_0+l_1}(0)+X^{\frac{3}{2}}(t)\right\} (1+t)^{-\varrho}.
\end{align}

Now, recalling \eqref{X define}, the desired estimate \eqref{close priori estimate} follows from \eqref{l0+l1}, \eqref{N l0 high decay} and \eqref{N l0 decay}. This completes the proof of Lemma \ref{soft close priori estimate}.
\end{proof}
\medskip

\begin{proof}[{\bf Proof of Theorem \ref{mainth1}.}] \  It follows immediately from the a priori estimate \eqref{soft close priori estimate} that $$
X(t) \lesssim \widetilde{\mathcal{E}}_{N,l_0+l_1}(0)
$$
holds for any $0 \leq t \leq T$, as long as $\widetilde{\mathcal{E}}_{N,l_0+l_1}(0)$ is sufficiently small. The rest is to prove the local existence and uniqueness of solutions in terms of the energy norm $\widetilde{\mathcal{E}}_{N,l_0+l_1}(t)$ and the non-negativity of $F^\varepsilon=\mu+\varepsilon \mu^{1/2}f^\varepsilon$. One can use the iteration method with the iterating sequence $f_n^\varepsilon$ $(n \geq 0)$ on the system
\begin{align}
	\left\{\begin{array}{l}
\displaystyle \partial_{t} f_{n+1}^{\varepsilon}+\frac{1}{\varepsilon}v \cdot \nabla_{x} {f}_{n+1}^{\varepsilon}+\frac{1}{\varepsilon} \nabla_{x} \phi_{n+1}^{\varepsilon} \cdot v \mu^{1/2}q_1+\frac{1}{\varepsilon^{2}} L {f}_{n+1}^{\varepsilon}
\\[3mm]
\displaystyle \qquad\qquad\qquad\qquad\qquad=q_0\nabla_{x} \phi_{n}^{\varepsilon} \cdot \nabla_{v} {f}_{n+1}^{\varepsilon}-\frac{1}{2}q_0 \nabla_{x} \phi_{n}^{\varepsilon} \cdot v {f}_{n+1}^{\varepsilon}+\frac{1}{\varepsilon} \Gamma\left({f}_{n}^{\varepsilon}, {f}_{n}^{\varepsilon}\right),  \\ [2mm]
\displaystyle -\Delta_{x} \phi_{n+1}^{\varepsilon}=\int_{\mathbb{R}^3}f_{n+1}^{\varepsilon}\cdot q_1 \mu^{1/2}\d v ,
\quad \lim_{|x|\rightarrow \infty}\phi_{n+1}^\varepsilon=0,  \\ [3mm]
\displaystyle f_{n+1}^\varepsilon(0,x,v)=f^\varepsilon_0(x,v).
	\end{array}\right.
\end{align}
The details of proof are omitted for brevity; see \cite{GS2011, Guo2012JAMS} and  \cite{JL2019}. Therefore, the global existence of solutions follows with the help of the continuity argument, and the estimates \eqref{thm1 N l0 decay}, \eqref{thm1 N h l0 decay} and \eqref{thm1 widetilde N l0+l1 decay} hold by the definition of $X(t)$ in \eqref{X define}. This completes the proof of Theorem \ref{mainth1}.
\end{proof}
\medskip

\subsection{Proof of Global Existence for Hard Potentials}
\hspace*{\fill}

In this subsection, we prove Theorem \ref{mainth2} by closing the a priori estimates on $X(t)$ in the following lemma. We provide a simpler proof than that of Lemma \ref{soft close priori estimate}, but it is only applicable to the hard potential case.

\begin{lemma}\label{hard close priori estimate}
Let $0 < T \leq \infty $. Assume that there exists a suitable small constant $\delta_0 > 0$ such that the a priori assumption \eqref{priori assumption} holds. Then we have
\begin{align}\label{close priori estimate2}
X(t) \lesssim \widetilde{\mathcal{E}}_{N,l_2}(0),
\end{align}
for any $0 \leq t \leq T$.
\end{lemma}

\begin{proof}
We divide the proof into two steps.

\text{{Step 1.}} \
Similar to \eqref{partial t phi}, due to $l=l_2$ for hard potentials, we obtain
\begin{align}\label{partial t phi1}
\begin{split}
\varepsilon^2 \left\| \partial_t \phi^\varepsilon\right\|^2_{L^\infty}
\lesssim \left\| G^\varepsilon \right\| \left\| \nabla_x G^\varepsilon\right\|
\lesssim \mathcal{D}_{N,l_2}(t).
\end{split}
\end{align}
By applying Proposition  \ref{negative sobolev energy estimates total}, thanks to $l=l_2$, one has from \eqref{a priori estimates3}, \eqref{partial t phi1} and the Cauchy--Schwarz inequality that
\begin{align}\nonumber
\frac{\d}{\d t} \widetilde{\mathcal{E}}_{N,l_2}(t) + \lambda  \widetilde{\mathcal{D}}_{N,l_2}(t) \lesssim   \varepsilon \left\| \partial_t \phi^\varepsilon \right\|_{L^\infty} \mathcal{E}_{N,l_2}(t)
\lesssim \mathcal{D}^{1/2}_{N,l_2}(t) \mathcal{E}_{N,l_2}(t)
\lesssim \frac{\lambda}{2} \mathcal{D}_{N,l_2}(t)
+\left\{\mathcal{E}_{N,l_2}(t)\right\}^2,
\end{align}
which, together with \eqref{X define} and the Gronwall inequality, as well as \eqref{priori assumption}, implies that
\begin{align}\label{l2 widetilde}
\widetilde{\mathcal{E}}_{N,l_2}(t)+\int_0^t \widetilde{\mathcal{D}}_{N,l_2}(\tau) \d \tau \lesssim \widetilde{\mathcal{E}}_{N,l_2}(0)e^{C\int_{0}^t \mathcal{E}_{N,l_2}(\tau) \d \tau}
\lesssim \widetilde{\mathcal{E}}_{N,l_2}(0)
\end{align}
for any $0 \leq t \leq T$.

\text{{Step 2.}} \
 It follows from \eqref{a priori estimates} with $l=l_2$ in Proposition \ref{energy estimates total}, \eqref{partial t phi1} and the Cauchy--Schwarz inequality that
\begin{align}\label{l2 1}
\frac{\d}{\d t} \mathcal{E}_{N,l_2}(t) + \lambda  \mathcal{D}_{N,l_2}(t) \lesssim
\left\{\mathcal{E}_{N,l_2}(t)\right\}^2.
\end{align}
From \eqref{energy functional} and \eqref{dissipation functional} for hard potentials, we have
\begin{align}\label{energy dissipation}
\mathcal{E}_{N,l_2}(t) \lesssim \left\| \mathbf{P}f^\varepsilon \right\|^2+ \mathcal{D}_{N,l_2}(t).
\end{align}
Combining \eqref{l2 1} with \eqref{energy dissipation}, one has
\begin{align}\nonumber
\frac{\d}{\d t} \mathcal{E}_{N,l_2}(t) + \lambda \mathcal{E}_{N,l_2}(t)
\lesssim \left\| \mathbf{P}f^\varepsilon \right\|^2
+\left\{\mathcal{E}_{N,l_2}(t)\right\}^2.
\end{align}
By using the a priori assumption \eqref{priori assumption} and solving the ODE inequality, we can deduce
\begin{align}\label{N l2 decay}
\mathcal{E}_{N,l_2}(t) \lesssim e^{-\lambda t}\mathcal{E}_{N,l_2}(0)
+\int_0^t e^{-\lambda(t-\tau)} \left\| \mathbf{P}f^\varepsilon \right\|^2 \d \tau.
\end{align}
Making use of \eqref{k-N decay}, \eqref{l2 widetilde} and \eqref{N l2 decay}, we obtain
\begin{align}\label{l2 decay}
\mathcal{E}_{N,l_2}(t) \lesssim \widetilde{\mathcal{E}}_{N,l_2}(0) (1+t)^{-\varrho},
\end{align}
where we used the following inequality
\begin{align}\nonumber
\int_0^t e^{-\lambda(t-\tau)} (1+\tau)^{-\varrho} \d \tau \lesssim (1+t)^{-\varrho}.
\end{align}

Now, recalling \eqref{X define}, the desired estimate \eqref{close priori estimate2} follows from \eqref{l2 widetilde} and \eqref{l2 decay}. This completes the proof of Lemma \ref{hard close priori estimate}.
\end{proof}
\medskip

\begin{proof}[{\bf Proof of Theorem \ref{mainth2}.}] \
The proof of Theorem \ref{mainth2} is similar to the proof of Theorem \ref{mainth1}. We omit the details for brevity.
\end{proof}
\medskip

\section{Limit to Two-fluid Incompressible NSFP System with Ohm's Law}\label{Limit section}

In this section, our goal is to derive the two-fluid incompressible NSFP system with Ohm's law from the VPB system \eqref{rVPB} as $\varepsilon \to  0$. For brevity, we only consider soft potentials. The proof about hard potentials is similar and we omit the details. The approach is similar to that one in \cite{JL2019}, but here  we  provide a full proof for completeness.

Firstly, we introduce the linearized Boltzmann collision operators $\mathcal{L}$ and $\mathscr{L}$ by
\begin{align}
\mathcal{L}h:=\;&-\mu^{-1/2}\left\{Q\left(\mu, \mu^{1/2}h\right)+Q\left(\mu^{1/2}h,\mu\right)\right\},\label{L usual define}\\
\mathscr{L}h:=\;&-2\mu^{-1/2}Q\left(\mu, \mu^{1/2}h\right)\label{scr L define}
\end{align}
for a scalar function $h$. The linearized Boltzmann collision operators $\mathcal{L}$ and $\mathscr{L}$ are also non-negative and self-adjoint. Furthermore, their null spaces are given by
$$
\mathcal{N}(\mathcal{L})=\mathrm{span}\left\{\mu^{1/2},v_{i}\mu^{1/2}(1\leq i\leq 3), |v|^2\mu^{1/2}\right\},\;\;\;\;\;\;
\mathcal{N}(\mathscr{L})=\mathrm{span}\left\{\mu^{1/2}\right\}.
$$
For a given scalar function $h$, there is macro-micro decomposition
\begin{equation}\label{h decomposition}
h=P_0h+\{I-P_0\}h,
\end{equation}
first introduced in \cite{Guo2004}.
Here, $P_0$ denotes the orthogonal projection from $L^2(\mathbb{R}_{v}^3) $ to $\mathcal{N}(\mathcal{L})$, that is,
\begin{equation}\label{P0h define}
P_0 h:=\left\{\widetilde{a}(t,x)+v \cdot \widetilde{b}(t,x)+\frac{|v|^2-3}{2}\widetilde{c}(t,x)\right\}\mu^{1/2},
\end{equation}
where
\begin{align}
\widetilde{a}(t,x):=\;&\langle \mu^{1/2}, h \rangle =\langle \mu^{1/2}, P_0 h \rangle, \nonumber \\
\widetilde{b}(t,x):=\;&\langle v\mu^{1/2}, h\rangle =\langle v\mu^{1/2}, P_0 h\rangle, \nonumber \\
\widetilde{c}(t,x):=\;&\frac{1}{3}\big\langle (|v|^2-3)\mu^{1/2},h \big\rangle=\frac{1}{3} \big\langle(|v|^2-3)\mu^{1/2},P_0 h \big\rangle. \nonumber
\end{align}

\subsection{Limits from the Energy Estimate}
\hspace*{\fill}

Based on \eqref{thm1 widetilde N l0+l1 decay} in Theorem \ref{mainth1},  the system (\ref{rVPB})
admits a global solution $f^\varepsilon \in L^\infty(\mathbb{R}^+; H^N_x L^2_v)$ and $\nabla_x\phi^\varepsilon \in L^\infty(\mathbb{R}^+; H^{N+1}_x)$, so there exists a positive constant $C$ which is independent of $\varepsilon$, such that
\begin{align}
&\sup_{t\geq 0} \left\{ \left\| f^{\varepsilon}(t) \right\|^2_{H^N_x L^2_v}+ \left\| \nabla_x \phi^{\varepsilon}(t) \right\|^2_{H^{N+1}_x}\right\} \leq  C,
\label{limit:1-1}\\
&\int_0^\infty \left\| \nabla_x \phi^{\varepsilon}(\tau) \right\|^2_{H^{N+1}_{x}}\d \tau\leq C,
\label{limit:1-4}\\
&\int_0^\infty \left\| \{\mathbf{I}-\mathbf{P}\}{f}^{\varepsilon}(\tau) \right\|^2_{H^N_x N^s_\gamma}\d \tau\leq C \varepsilon^2, \label{limit:1-2}\\
&\int_0^\infty \left\| \{\mathbf{I}-\mathbf{P}\}{f}^{\varepsilon}(\tau) \right\|^2_{H^N_{x,v}}\d \tau\leq C \varepsilon.
\label{limit:1-3}
\end{align}
Here, we have used the fact that for $l=l_0+l_1$,
$$
\left\| \{\mathbf{I}-\mathbf{P}\}{f}^{\varepsilon} \right\|^2_{H^N_{x,v}}
\lesssim \sum_{\substack{|\alpha|+|\beta| \leq N \\
|\alpha| \leq N-1}}\left\|w_l (\alpha, \beta) \partial_{\beta}^{\alpha}
\{\mathbf{I}-\mathbf{P}\} f^{\varepsilon}\right\|^2_{N^s_{\gamma}}
+\sum_{|\alpha|=N} \left\|w_l (\alpha, 0) \partial^{\alpha}
\{\mathbf{I}-\mathbf{P}\} f^{\varepsilon}\right\|^2_{N^s_{\gamma}}.
$$
From the energy bound \eqref{limit:1-1}, there exist $f\in L^\infty(\mathbb{R}^+; H^N_x L^2_v)$ and $\nabla_x\phi\in L^\infty(\mathbb{R}^+;H^{N+1}_x)$  such that
\begin{align}
f^{\varepsilon} &\to  f \quad \qquad\;\text{~~weakly}\!-\!* ~\text{in}~ L^\infty(\mathbb{R}^+; H^N_x L^2_v), \label{limit:3-1}\\
\nabla_x\phi^{\varepsilon}  &\to \nabla_x\phi\qquad \text{~~weakly}\!-\!* ~\text{in} ~L^\infty(\mathbb{R}^+;H^{N+1}_x) \label{limit:3-2}
\end{align}
as $\varepsilon \to 0$. We still employ the original notation of sequences to denote the subsequences for convenience,
although the limit may hold for some subsequences.
From the dissipation bound \eqref{limit:1-3}, we deduce that
\begin{align}
\label{limit:4}
&\{\mathbf{I}-\mathbf{P}\}f^{\varepsilon} \to  0 \;\;\;\text{~~strongly} {\text{~in}~L^2(\mathbb{R}^+;H^N_{x,v})}
\end{align}
as $\varepsilon \to 0$. We thereby deduce from the convergences \eqref{limit:3-1} and \eqref{limit:4} that
\begin{align}
\{\mathbf{I}-\mathbf{P}\}f=0.
\end{align}
This implies that there exist functions $\rho_{+}$, $\rho_{-}$, $u$, $\theta$ $\in L^\infty(\mathbb{R}^+;H^N_{x})$ such that
\begin{align}\label{limit:6}
\begin{split}
f(t,x,v)=\;&\rho_{+}(t,x)\frac{q_1+q_2}{2}\mu^{1/2} +\rho_{-}(t,x)\frac{q_2-q_1}{2}\mu^{1/2} \\
&+ u(t,x) \cdot v q_2 \mu^{1/2} + \theta (t,x)\big(\frac{|v|^2}{2}-\frac{3}{2}\big)q_2 \mu^{1/2},
\end{split}
\end{align}
where $q_1=[1,-1]$, $q_2=[1,1]$.

Next, we introduce the following fluid variables
\begin{align}
\begin{split}\label{limit:7}
&\rho^\varepsilon =\frac{1}{2}\Big\langle f^{\varepsilon}, q_2\mu^{1/2} \Big\rangle,\;\;\;\;
u^\varepsilon= \frac{1}{2} \Big \langle f^{\varepsilon}, q_2 v\mu^{1/2} \Big \rangle,\;\;\;\;
\theta^\varepsilon = \frac{1}{2} \Big\langle f^{\varepsilon}, q_2 \big(\frac{|v|^2}{3}-1 \big)\mu^{1/2}\Big\rangle, \\
&n^\varepsilon = \Big\langle f^{\varepsilon}, q_1\mu^{1/2} \Big\rangle, \;\;\;\;\;\;\;
j^\varepsilon = \frac{1}{\varepsilon} \Big\langle f^{\varepsilon}, q_1 v\mu^{1/2} \Big\rangle, \;\;\;
\omega^\varepsilon = \frac{1}{\varepsilon} \Big\langle f^{\varepsilon}, q_1\big(\frac{|v|^2}{3}-1\big)\mu^{1/2} \Big\rangle.
\end{split}
\end{align}
Multiplying the first equation of \eqref{rVPB} by $\frac{1}{2}q_2\mu^{1/2}$, $ \frac{1}{2}q_2 v\mu^{1/2}$, $\frac{1}{2}q_2(\frac{|v|^2}{3}-1)\mu^{1/2}$, $q_1 \mu^{1/2}$ respectively and integrating  over $\mathbb{R}^3_v$,  we derive the following local conservation laws
\begin{equation}
 \left\{
\begin{array}{ll}\label{limit:8-1:1}
\displaystyle
\partial_t \rho^\varepsilon+\frac{1}{\varepsilon}\nabla_x\cdot u^\varepsilon=0,~\\[2mm]
\displaystyle\partial_t u^\varepsilon+\frac{1}{\varepsilon}\nabla_x(\rho^\varepsilon+\theta^\varepsilon)+\frac{1}{\varepsilon}
\nabla_x\cdot \Big\langle \widehat{A}(v)\mu^{1/2}, \mathcal{L} \big(\frac{f^{\varepsilon} \cdot q_2}{2}\big)\Big\rangle =-\frac{1}{2}n^\varepsilon \nabla_x\phi^{\varepsilon},~\\[2mm]
\displaystyle\partial_t \theta^\varepsilon+\frac{2}{3\varepsilon}\nabla_x\cdot u^\varepsilon+\frac{2}{3\varepsilon}\nabla_x\cdot \Big\langle \widehat{B}(v)\mu^{1/2},
\mathcal{L} \big(\frac{f^{\varepsilon} \cdot q_2}{2}\big) \Big\rangle=-\frac{\varepsilon}{3}j^\varepsilon \cdot \nabla_{x}\phi^\varepsilon, \\[3mm]
\displaystyle\partial_t n^\varepsilon + \nabla_x \cdot j^\varepsilon =0,
\end{array}
\right.
\end{equation}
where
\begin{align}
A (v)=v\otimes v -\frac{|v|^2}{3}\mathbb{I}_{3\times 3},~\quad B (v)=v\Big(\frac{|v|^2}{2}-\frac{5}{2}\Big),
\end{align}
$\widehat{A} (v)$ is such that $\mathcal{L}\big( \widehat{A}(v) \mu^{1/2}\big)= A(v) \mu^{1/2}$ with $\widehat{A}(v) \mu^{1/2} \in \mathcal{N}^{\bot}(\mathcal{L})$,
$\widehat{B} (v)$ is such that $\mathcal{L}\big( \widehat{B}(v) \mu^{1/2}\big)$ $= B(v) \mu^{1/2}$ with $\widehat{B}(v) \mu^{1/2} \in \mathcal{N}^{\bot}(\mathcal{L})$ and $\mathcal{L}$ is the linearized Boltzmann collision operator defined in \eqref{L usual define}.
From the second equation of \eqref{rVPB}  and the definition of $n^\varepsilon$ in \eqref{limit:7}, we also find the equation of $\nabla_x\phi^{\varepsilon}$
\begin{align}
\label{limit:8-1:4}
&-\Delta_x\phi^{\varepsilon}=n^\varepsilon.
\end{align}

Furthermore, via the definitions of $\rho^\varepsilon$, $u^\varepsilon$, $\theta^\varepsilon$ and $n^\varepsilon$ in \eqref{limit:7} and the uniform energy bound \eqref{limit:1-1}, we obtain
\begin{align}\label{limit:9}
\sup_{t\geq 0}\Big\{\| \rho^\varepsilon(t)\|_{H^N_{x}}+ \| u^\varepsilon(t)\|_{H^N_{x}}+\| \theta^\varepsilon(t)\|_{H^N_{x}}
+\| n^\varepsilon(t)\|_{H^N_x}\Big\}\leq C.
\end{align}
From the convergence \eqref{limit:3-1} and the limit function $f(t,x,v)$ given in \eqref{limit:6}, we thereby deduce the following convergences
\begin{align}
&\rho^{\varepsilon}=\frac{1}{2}\Big\langle f^{\varepsilon}, q_2\mu^{1/2} \Big\rangle \to  \frac{1}{2}\Big\langle f, q_2\mu^{1/2} \Big\rangle
=\frac{1}{2}(\rho_{+}+\rho_{-}) := \rho,~\label{limit:10-1}\\
&u^{\varepsilon}=\frac{1}{2} \Big \langle f^{\varepsilon}, q_2 v\mu^{1/2} \Big \rangle \to \frac{1}{2} \Big \langle f, q_2 v\mu^{1/2} \Big \rangle=u,~ \label{limit:10-2}\\
&\theta^{\varepsilon} =\frac{1}{2} \Big\langle f^{\varepsilon}, q_2 \big(\frac{|v|^2}{3}-1 \big)\mu^{1/2}\Big\rangle \to \frac{1}{2} \Big\langle f, q_2 \big(\frac{|v|^2}{3}-1 \big)\mu^{1/2}\Big\rangle =\theta,~ \label{limit:10-3}\\
&n^\varepsilon=\Big\langle f^\varepsilon, q_1 \mu^{1/2} \Big\rangle \to \Big\langle f, q_1 \mu^{1/2} \Big\rangle=\rho_{+}-\rho_{-}
:= n \label{limit:10-4}
\end{align}
$\text{weakly}\!-\!*~ \text{in}~ L^\infty(\mathbb{R}^+,H^N_{x})$ as $\varepsilon \to 0$. It remains to find the limits of $j^\varepsilon$ and $w^\varepsilon$. Noticing that
\begin{align}\nonumber
q_1\big(\frac{|v|^2}{3}-1 \big)\mu^{1/2} \bot\; \mathcal{N}(L),\qquad q_1 v_i \mu^{1/2} \bot\; \mathcal{N}(L),i=1,2,3,
\end{align}
it follows from the dissipation bound \eqref{limit:1-2}, the definitions of $j^\varepsilon$ and $\omega^\varepsilon$ in \eqref{limit:7} and the  H{\"{o}}lder inequality that
\begin{align}
&\left\| j^\varepsilon \right\|_{L^2(\mathbb{R}^{+};H^N_x)}^2 + \left\| \omega^\varepsilon \right\|_{L^2(\mathbb{R}^{+};H^N_x)}^2  \nonumber\\
=\;& \frac{1}{\varepsilon^2}\int_{0}^\infty \left\{ \Big\| \big\langle f^{\varepsilon}(\tau), q_1 v\mu^{1/2} \big\rangle \Big\|_{H^N_x}^2
+\Big\| \Big\langle f^{\varepsilon}(\tau), q_1\big(\frac{|v|^2}{3}-1 \big)\mu^{1/2} \Big\rangle \Big\|_{H^N_x}^2 \right\} \d \tau \nonumber\\
=\;& \frac{1}{\varepsilon^2}\int_{0}^\infty \left\{\Big\| \big\langle \{\mathbf{I}-\mathbf{P}\}f^{\varepsilon}(\tau), q_1 v\mu^{1/2} \big\rangle \Big\|_{H^N_x}^2
+ \Big\| \Big\langle \{\mathbf{I}-\mathbf{P}\}f^{\varepsilon}(\tau), q_1\big(\frac{|v|^2}{3}-1 \big)\mu^{1/2} \Big\rangle \Big\|_{H^N_x}^2\right\} \d \tau \nonumber\\
\leq\;&  \frac{C}{\varepsilon^2}\int_{0}^\infty \left\| \{\mathbf{I}-\mathbf{P}\}f^{\varepsilon}(\tau) \right\|_{H^N_x N^s_\gamma} \d \tau \nonumber\\
\leq\;&  C. \label{limit:11}
\end{align}
Consequently, there exist functions $j$, $\omega$ $\in L^2(\mathbb{R}^{+}; H^N_x)$ such that
\begin{align}
(j^\varepsilon,\omega^\varepsilon) \to (j,\omega)  \;\;\;\text{~~weakly} {\text{~in}~L^2(\mathbb{R}^+;H^N_{x})} \label{limit:12}
\end{align}
 as $\varepsilon \to 0$.

\subsection{Convergences to the Limiting System}
\hspace*{\fill}

In this subsection, we deduce the two-fluid incompressible NSFP system with Ohm's law from the local conservation laws \eqref{limit:8-1:1}, \eqref{limit:8-1:4}, the convergences \eqref{limit:3-1}--\eqref{limit:4} \eqref{limit:10-1}--\eqref{limit:10-4} and \eqref{limit:12} obtained in the previous subsection.

\subsubsection{Incompressibility and Boussinesq relation.}
\hspace*{\fill}

From the first equation of \eqref{limit:8-1:1} and the energy uniform bound \eqref{limit:9}, it is easy to deduce
\begin{align}\nonumber
\nabla_x \cdot u^{\varepsilon }=-\varepsilon \partial_t \rho^{\varepsilon}\to 0  \;\;\; \text{~~in the sense of distributions}
\end{align}
as $\varepsilon \to 0$. Combining with the convergence \eqref{limit:10-2}, we get
\begin{align}\label{limit:13}
\nabla_x \cdot u = 0.
\end{align}
The second equation of \eqref{limit:8-1:1}  yields
\begin{align}\nonumber
&\nabla_x(\rho^\varepsilon +\theta^\varepsilon)=-\varepsilon \partial_t u^\varepsilon
-  \nabla_x\cdot \Big\langle\widehat{ {A}}(v)\mu^{1/2}, \mathcal{L} \big(\frac{f^{\varepsilon} \cdot q_2}{2}\big) \Big\rangle-\frac{\varepsilon}{2}n^\varepsilon \nabla_x\phi^{\varepsilon}.
\end{align}
The bound \eqref{limit:9} implies
$$
\varepsilon \partial_t u^\varepsilon \to 0 \;\;\; \text{~~in the sense of distributions}
$$
as $\varepsilon\to 0$.
With the aid of the self-adjointness of $\mathcal{L}$, we derive from the dissipation bound \eqref{limit:1-2} and the  H{\"{o}}lder inequality that
\begin{align}
\begin{split}\nonumber
\int_0^\infty
\!\Big\|\nabla_x\cdot \Big\langle\widehat{ {A}}(v)\mu^{1/2}, \mathcal{L} \big(\frac{f^{\varepsilon} \cdot q_2}{2}\big) \Big\rangle \Big\|_{H^{N-1}_x}^2\d \tau
\!=\;&\int_0^\infty
\Big\|\nabla_x\cdot \Big\langle\widehat{ {A}}(v)\mu^{1/2}, \mathcal{L} \big(\frac{\{\mathbf{I}-\mathbf{P}\}f^{\varepsilon} \cdot q_2}{2}\big) \Big\rangle \Big\|_{H^{N-1}_x}^2\d \tau \\
=\;&\int_0^\infty
\Big\|\nabla_x\cdot \Big\langle{ {A}}(v)\mu^{1/2}, \frac{\{\mathbf{I}-\mathbf{P}\}f^{\varepsilon} \cdot q_2}{2} \Big\rangle \Big\|_{H^{N-1}_x}^2\d \tau\\
\leq \;& C\int_0^\infty
\| \{\mathbf{I}-\mathbf{P}\}f^\varepsilon(\tau)\|_{H^N_{x}N^s_\gamma}^2\d \tau\\
\leq\;&C\varepsilon^2.
\end{split}
\end{align}
Therefore, we have
\begin{align}\nonumber
\nabla_x\cdot \Big\langle \widehat{A}(v)\mu^{1/2}, \mathcal{L} \big(\frac{f^{\varepsilon} \cdot q_2}{2}\big) \Big\rangle \to 0
\;\;\;\text{~~strongly} {\text{~in}~L^2(\mathbb{R}^+;H^{N-1}_{x})}
\end{align}
as $\varepsilon\to 0$.
Moreover, the uniform energy bounds \eqref{limit:1-1} and \eqref{limit:9} reveal that
\begin{align}\nonumber
\sup_{t \geq 0} \left\|\frac{\varepsilon}{2} n^\varepsilon \nabla_x\phi^{\varepsilon}(t)\right\|_{H^{N}_x}
\leq C \varepsilon ,
\end{align}
which means that
\begin{align}\nonumber
\frac{\varepsilon}{2} n^\varepsilon \nabla_x\phi^{\varepsilon}\to 0
\;\;\;\text{~~strongly} {\text{~in}~L^\infty(\mathbb{R}^+;H^{N}_{x})}
\end{align}
as $\varepsilon \to 0$.
In summary, we have
\begin{align}\label{limit:15}
\nabla_x(\rho^\varepsilon + \theta^\varepsilon) \to 0 \;\;\; \text{~~in the sense of distributions}
\end{align}
as $\varepsilon \to 0$.
Therefore, combining the convergences \eqref{limit:10-1}, \eqref{limit:10-3} and \eqref{limit:15}, we obtain
\begin{align}\nonumber
\nabla_x(\rho+\theta)=0.
\end{align}
From the integrability of functions, we can deduce the Boussinesq relation
\begin{align}\label{limit:16}
\rho+\theta=0.
\end{align}

\subsubsection{ Convergences of  $\frac{3}{5} \theta^\varepsilon-\frac{2}{5}\rho^\varepsilon$, $\mathcal{P}u^\varepsilon$, $n^\varepsilon$, $\nabla_x \phi^\varepsilon$.}
\hspace*{\fill}

First of all, we introduce the following Aubin--Lions--Simon Theorem, a fundamental result of compactness in the study of nonlinear evolution problems, which can be referred in \cite{BF13, Simon1987}.
\begin{lemma} [Aubin--Lions--Simon Theorem] \label{Aubin--Lions--Simon Theorem}
Let $B_0 \subset B_1 \subset B_2$ be three Banach spaces. We assume that the embedding of $B_1$ in $B_2$ is continuous and the embedding of $B_0$ in $B_1$ is compact. Let $p, r$ be such that $1 \leq p,r \leq \infty$. For $T>0$, we define
\begin{align}\nonumber
E_{p,r}=\Big\{ u \in L^p(0,T; B_0), \partial_t u \in L^r(0,T;B_2) \Big\}.
\end{align}
\begin{itemize}
\setlength{\leftskip}{-6mm}
\item[(1)] If $p < +\infty$, the embedding of $E_{p,r}$ in $L^p(0,T; B_1)$ is compact.
\item[(2)] If $p = +\infty$ and if $r > 1$, the embedding of $E_{p,r}$ in $C(0,T; B_1)$ is compact.
\end{itemize}
\end{lemma}

Now we consider the convergence of $\frac{3}{5}\theta^\varepsilon-\frac{2}{5}\rho^\varepsilon$. It follows from the first equation and the third equation of $\eqref{limit:8-1:1}$ that
\begin{align}\label{rho theta equation}
\partial_t\Big(\frac{3}{5}\theta^\varepsilon -\frac{2}{5}\rho^\varepsilon\Big)+\frac{2}{5\varepsilon}\nabla_x \cdot \Big\langle \widehat{B}(v)\mu^{1/2}, \mathcal{L}\big(\frac{f^{\varepsilon}\cdot q_2}{2}\big)\Big\rangle=-\frac{\varepsilon}{5}j^\varepsilon\cdot \nabla_x\phi^{\varepsilon}.
\end{align}
Thereby, by utilizing the equation \eqref{rho theta equation}, the energy bound \eqref{limit:1-1}, the dissipation bounds \eqref{limit:1-2} and \eqref{limit:11}, one has
\begin{align}
\begin{split}\label{partial rho theta 1}
	&\Big\|\partial_t\Big(\frac{3}{5}\theta^\varepsilon -\frac{2}{5}\rho^\varepsilon\Big)\Big\|_{L^2(0,T;H^{N-1}_x)}\\
	=\;&\Big\|-\frac{\varepsilon}{5}j^\varepsilon\cdot \nabla_x\phi^{\varepsilon}-\frac{2}{5\varepsilon} \nabla_x \cdot \Big\langle \widehat{B}(v)\mu^{1/2}, \mathcal{L}\big(\frac{f^{\varepsilon}\cdot q_2}{2}\big)\Big\rangle \Big\|_{L^2(0,T;H^{N-1}_x)}\\
	=\;&\Big\|-\frac{\varepsilon}{5}j^\varepsilon\cdot \nabla_x\phi^{\varepsilon}-\frac{2}{5\varepsilon} \nabla_x \cdot \Big\langle \widehat{B}(v)\mu^{1/2}, \mathcal{L}\big(\frac{\{\mathbf I-\mathbf P\} f^{\varepsilon}\cdot q_2}{2}\big)\Big\rangle \Big\|_{L^2(0,T;H^{N-1}_x)}\\
	=\;&\Big\|-\frac{\varepsilon}{5}j^\varepsilon\cdot \nabla_x\phi^{\varepsilon}-\frac{1}{5\varepsilon} \nabla_x \cdot \langle B(v)\mu^{1/2},\{\mathbf I-\mathbf P\} f^{\varepsilon}\cdot q_2 \rangle \Big\|_{L^2(0,T;H^{N-1}_x)}\\
	\leq \;& C \varepsilon \left\|j^\varepsilon \right\|_{L^2(\mathbb{R}^{+};H^{N}_x)} \left\|\nabla_x\phi^{\varepsilon} \right\|_{L^\infty(\mathbb{R}^{+};H_x^{N})}+\frac{C}{\varepsilon}\left\| \{\mathbf{I}-\mathbf{P}\}{f}^{\varepsilon}\right\|
_{L^2(\mathbb{R}^{+}; H^N_x N^s_\gamma)}\\
\leq\;& C
\end{split}
\end{align}
for any $T > 0$ and $0 < \varepsilon \leq 1$. On the other hand, we can derive from the energy bound \eqref{limit:9} that
\begin{align}\label{partial rho theta 2}
\Big\|\frac{3}{5}\theta^\varepsilon -\frac{2}{5}\rho^\varepsilon\Big\|_{L^\infty(0,T;H^{N}_x)} \leq C
\end{align}
for any $T > 0$ and $0 < \varepsilon \leq 1$. Notice the fact that
\begin{align}\label{embedding compact}
H^N  \hookrightarrow H^{N-1}_{loc} \hookrightarrow H^{N-1}_{loc},
\end{align}
where the embedding of $H^N$ in $H^{N-1}_{loc}$ is compact by the Rellich--Kondrachov Theorem (see Theorem 6.3 of \cite{AF2003}, for instance) and the embedding of $H^{N-1}_{loc}$ in $H^{N-1}_{loc}$ is naturally continuous. Then, from Aubin--Lions--Simon Theorem in Lemma \ref{Aubin--Lions--Simon Theorem}, the bounds \eqref{partial rho theta 1}, \eqref{partial rho theta 2} and the embedding \eqref{embedding compact}, we deduce that there exists a function $\widetilde{\theta} \in L^\infty(\mathbb{R}^+; H^N_x) \cap C(\mathbb{R}^+; H^{N-1}_{loc})$ such that
\begin{align}\nonumber
\frac{3}{5}\theta^\varepsilon-\frac{2}{5}\rho^\varepsilon \to \widetilde{\theta}
\;\;\;\text{~~strongly} {\text{~in}~C(0,T; H^{N-1}_{loc})}
\end{align}
for any $T > 0$ as $\varepsilon \to 0$. Combining with the convergences \eqref{limit:10-1} and \eqref{limit:10-3}, we deduce that $\widetilde{\theta}=\frac{3}{5}\theta-\frac{2}{5}\rho$. Moreover, it follows that $\widetilde{\theta}=\theta$ from the Boussinesq relation \eqref{limit:16} and $\theta=\frac{3}{5}\theta-\frac{2}{5}\rho+\frac{2}{5}(\rho+\theta)$. As a result, we obtain that
\begin{align}\nonumber
\frac{3}{5}\theta^\varepsilon-\frac{2}{5}\rho^\varepsilon \to \theta
\;\;\;\text{~~strongly} {\text{~in}~C(\mathbb{R}^+; H^{N-1}_{loc})}
\end{align}
as $\varepsilon \to 0$, where $\theta \in L^\infty(\mathbb{R}^+; H^N_x) \cap C(\mathbb{R}^+; H^{N-1}_{loc})$. Noticing that $\theta^\varepsilon=\frac{3}{5}\theta^\varepsilon-\frac{2}{5}\rho^\varepsilon+\frac{2}{5}(\rho^\varepsilon+\theta^\varepsilon)$ and combining with the convergence \eqref{limit:10-3}, one has
\begin{align}\nonumber
\rho^\varepsilon+\theta^\varepsilon \to 0
\;\;\; \text{weakly}\!-\!* \text{~in~} L^\infty(\mathbb{R}^+;H^N_{x}),\;
\text{~~strongly} \text{~in}~ C(\mathbb{R}^+;H^{N-1}_{loc})
\end{align}
as $\varepsilon \to 0$.

Next, we consider the convergence of $\mathcal{P}u^\varepsilon$, where $\mathcal{P}u^\varepsilon$ is the Leray projection operator defined by $\mathcal{P}= I-\nabla_x \Delta_x^{-1} \nabla_x \cdot$. By taking $\mathcal{P}$ on the second equation of \eqref{limit:8-1:1}, we have
\begin{align}\label{Pu equation}
&\partial_t \mathcal{P} u^\varepsilon +\frac{1}{\varepsilon}\mathcal{P}
\nabla_x\cdot \Big\langle \widehat{A}(v)\mu^{1/2}, \mathcal{L} \big(\frac{f^{\varepsilon} \cdot q_2}{2}\big)\Big\rangle =-\frac{1}{2}\mathcal{P}(n^\varepsilon \nabla_x\phi^{\varepsilon}).
\end{align}
Hence, it follows from the equation \eqref{Pu equation}, the uniform dissipation bounds \eqref{limit:1-4}, \eqref{limit:1-2}, the uniform energy bound \eqref{limit:9} and the H{\"{o}}lder inequality that
\begin{align}
\begin{split}\label{partial u 1}
&\left\|\partial_t \mathcal{P} u^\varepsilon \right\|_{L^2(0,T;H^{N-1}_x)}\\
=\;&\Big\|-\frac{1}{2}\mathcal{P}(n^\varepsilon \nabla_x\phi^{\varepsilon})-\frac{1}{\varepsilon}\mathcal{P} \nabla_x \cdot \Big\langle \widehat{A}(v)\mu^{1/2}, \mathcal{L} \big(\frac{f^{\varepsilon}\cdot q_2}{2}\big)\Big\rangle \Big\|_{L^2(0,T;H^{N-1}_x)}\\
\lesssim\:& \|-n^\varepsilon \nabla_x\phi^{\varepsilon}\|_{L^2(0,T;H^{N-1}_x)}
+\Big\| \frac{1}{\varepsilon} \nabla_x \cdot \Big\langle \widehat{A}(v)\mu^{1/2}, \mathcal{L}\big(\frac{\{\mathbf{I}-\mathbf{P}\}f^{\varepsilon}\cdot q_2}{2}\big)\Big\rangle \Big\|_{L^2(0,T;H^{N-1}_x)}\\
\leq\:& C \left\|n^\varepsilon \right\|_{L^\infty(\mathbb{R}^+; H^{N}_x)} \left\|\nabla_x\phi^{\varepsilon} \right\|_{L^2(\mathbb{R}^+; H^{N}_x)}
+\frac{C}{\varepsilon} \left\|\{\mathbf I-\mathbf P\} f^{\varepsilon}\right\|_{L^2(\mathbb{R}^+; H^{N}_{x} N^s_\gamma)}\\
\leq\;& C
\end{split}
\end{align}
for any $T > 0$ and $0 < \varepsilon \leq 1$. Furthermore, we can derive from the definition of $u^\varepsilon$ in \eqref{limit:7} and the uniform energy bound \eqref{limit:1-1} that for any $T > 0$ and $0 < \varepsilon \leq 1$,
\begin{align}\label{partial u 2}
\left\| \mathcal{P}u^\varepsilon \right\|_{L^\infty (0,T; H^N_x)} \leq C.
\end{align}
Then, from Aubin--Lions--Simon Theorem in Lemma \ref{Aubin--Lions--Simon Theorem}, the bounds \eqref{partial u 1}, \eqref{partial u 2} and the embedding \eqref{embedding compact}, we deduce that there exists a function $\widetilde{u} \in L^\infty(\mathbb{R}^+; H^N_x) \cap C(\mathbb{R}^+; H^{N-1}_{loc})$ such that
\begin{align}\nonumber
\mathcal{P}u^\varepsilon \to \widetilde{u}
 \;\;\;\text{~~strongly} {\text{~in}~ C(0,T; H^{N-1}_{loc})}
\end{align}
for any $T > 0$ as $\varepsilon \to 0$. Moreover, from the convergence \eqref{limit:10-2} and the incompressibility \eqref{limit:13}, we have
\begin{align}\nonumber
\widetilde{u}=\mathcal{P}u=u.
\end{align}
Consequently,
\begin{align}\nonumber
\mathcal{P}u^\varepsilon \to u
\;\;\; \text{weakly}\!-\!*~ \text{in~} L^\infty(\mathbb{R}^+;H^N_{x}),\;
\text{~~strongly} {\text{~in}~C(\mathbb{R}^+; H^{N-1}_{loc})}
\end{align}
as $\varepsilon \to 0$, where $u \in L^\infty(\mathbb{R}^+; H^N_x) \cap C(\mathbb{R}^+; H^{N-1}_{loc})$. Therefore, we have that
\begin{align}\nonumber
\mathcal{P}^{ \bot } u^\varepsilon \to 0
\;\;\; \text{weakly}\!-\!*~ \text{in~} L^\infty(\mathbb{R}^+;H^N_{x}),\;
\text{~~strongly} \text{~in}~ C(\mathbb{R}^+;H^{N-1}_{loc})
\end{align}
as $\varepsilon \to 0$. Here, $\mathcal{P}^{ \bot } := I-\mathcal{P}$.

Moreover, we consider the convergence of $n^\varepsilon$. From the fourth equation of \eqref{limit:8-1:1} and the uniform dissipation bound \eqref{limit:11}, one has
\begin{align}\label{partial n 1}
\left\| \partial_t n^\varepsilon \right\|_{L^2(0,T;H^{N-1}_x)}= \left\| \nabla_x \cdot j^\varepsilon \right\|_{L^2(0,T;H^{N-1}_x)}
\leq \left\| j^\varepsilon \right\|_{L^2(0,T;H^{N}_x)} \leq C
\end{align}
for any $T > 0$ and $0 < \varepsilon \leq 1$. Furthermore, from the bound \eqref{limit:9}, we deduce that for any $T > 0$ and $0 < \varepsilon \leq 1$,
\begin{align}\label{partial n 2}
\left\| n^\varepsilon \right\|_{L^\infty(0,T;H^N_x)} \leq C.
\end{align}
Then, from Aubin--Lions--Simon Theorem in Lemma \ref{Aubin--Lions--Simon Theorem}, the bounds \eqref{partial n 1}, \eqref{partial n 2} and the embedding \eqref{embedding compact}, we deduce that there exists a function $\widetilde{n} \in L^\infty(\mathbb{R}^+; H^N_x) \cap C(\mathbb{R}^+; H^{N-1}_{loc})$ such that
\begin{align}\nonumber
n^\varepsilon \to \widetilde{n}
\;\;\;\text{~~strongly} {\text{~in}~C(0,T; H^{N-1}_{loc})}
\end{align}
for any $T > 0$ as $\varepsilon \to 0$. Combining with the convergence \eqref{limit:10-4}, one has
\begin{align}\nonumber
n^\varepsilon \to n
\;\;\; \text{weakly}\!-\!*~ \text{in~} L^\infty(\mathbb{R}^+;H^N_{x}),\;
\text{~~strongly} {\text{~in}~C(\mathbb{R}^{+}; H^{N-1}_{loc})}
\end{align}
as $\varepsilon \to 0$, where $n \in L^\infty(\mathbb{R}^+; H^N_x) \cap C(\mathbb{R}^+; H^{N-1}_{loc})$.

Finally, we consider the convergence of $\nabla_x \phi^\varepsilon$. From the fourth equation of \eqref{limit:8-1:1}, \eqref{limit:8-1:4} and the uniform dissipation bound \eqref{limit:11}, we have
\begin{align}\label{partial nabla phi 1}
\left\| \partial_t \nabla_x \phi^\varepsilon \right\|_{L^2(0,T;H^{N}_x)} =  \left\| \nabla_x \Delta_x^{-1} \nabla_x \cdot j^\varepsilon \right\|_{L^2(0,T;H^{N}_x)} \leq \left\| j^\varepsilon \right\|_{L^2(0,T;H^{N}_x)} \leq C
\end{align}
for any $T > 0$ and $0 < \varepsilon \leq 1$. Furthermore, from the bound \eqref{limit:1-1}, we deduce that for any $T > 0$ and $0 < \varepsilon \leq 1$
\begin{align}\label{partial nabla phi 2}
\left\| \nabla_x \phi^\varepsilon \right\|_{L^\infty(0,T;H^{N+1}_x)} \leq C.
\end{align}
Then, from Aubin--Lions--Simon Theorem in Lemma \ref{Aubin--Lions--Simon Theorem} and the bounds \eqref{partial nabla phi 1},
\eqref{partial nabla phi 2}, we deduce that
\begin{align}\nonumber
\nabla_x \phi^\varepsilon \to \nabla_x \phi
\;\;\; \text{weakly}\!-\!*~ \text{in~} L^\infty(\mathbb{R}^+;H^{N+1}_{x}),\;
\text{~~strongly} {\text{~in}~C(\mathbb{R}^+; H^{N}_{loc})}
\end{align}
as $\varepsilon \to 0$, where $\nabla_x \phi \in L^\infty(\mathbb{R}^+; H^{N+1}_x) \cap C(\mathbb{R}^+; H^{N}_{loc})$.

In summary, we have deduced the following convergences:
\begin{align}\label{limit:17}
\!\!\!\Big( \frac{3}{5}\theta^\varepsilon-\frac{2}{5}\rho^\varepsilon, \mathcal{P}u^\varepsilon, n^\varepsilon\Big)
\to ( \theta, u, n)
\;\;\; \text{weakly}\!-\!*~ \text{in~} L^\infty(\mathbb{R}^+;H^{N}_{x}),\;
\text{~~strongly} {\text{~in}~C(\mathbb{R}^+; H^{N-1}_{loc})}
\end{align}
as $\varepsilon \to 0$, where $( \theta, u, n) \in L^\infty(\mathbb{R}^+; H^N_x) \cap C(\mathbb{R}^+; H^{N-1}_{loc})$ and
\begin{align}\label{limit:17-1}
\nabla_x \phi^\varepsilon &\to \nabla_x \phi
\;\;\; \text{weakly}\!-\!*~ \text{in~} L^\infty(\mathbb{R}^+;H^{N+1}_{x}),\;
 \text{~~strongly} \text{~in}~C(\mathbb{R}^+; H^{N}_{loc})
\end{align}
as $\varepsilon \to 0$, where $ \nabla_x \phi \in L^\infty(\mathbb{R}^+; H^{N+1}_x) \cap C(\mathbb{R}^+; H^{N}_{loc})$ and
\begin{align}\label{limit:18}
\left( \rho^\varepsilon + \theta^\varepsilon, \mathcal{P}^{\bot}u\right) \to (0,0)
\;\;\; \text{weakly}\!-\!*~ \text{in~} L^\infty(\mathbb{R}^+;H^N_{x}),\;
\text{~~strongly} \text{~in}~ C(\mathbb{R}^+;H^{N-1}_{loc})
\end{align}
as $\varepsilon \to 0$, and
\begin{align}\label{limit:19}
(j^\varepsilon,\omega^\varepsilon) \to (j,\omega)   \;\;\;\text{~~weakly} {\text{~in}~L^2(\mathbb{R}^+;H^N_{x})}
\end{align}
as $\varepsilon \to 0$, where $(j,\omega) \in L^2(\mathbb{R}^+; H^N_x)$.

\subsubsection{Equations of  $\rho$, $u$, $\theta$, $n$, $j$ and $\omega$}
\hspace*{\fill}

We first calculate the term
\begin{align}\nonumber
 \frac{1}{\varepsilon} \Big \langle\widehat{\Xi}(v) \mu^{1/2}, \mathcal{L}\big(\frac{f^{\varepsilon} \cdot q_2}{2}\big)\Big\rangle,
\end{align}
where $\Xi=A$ or $B$. Via multiplying the first equation of \eqref{rVPB}  by $q_2$ and direct calculations, we deduce that
\begin{align}
& \partial_t\left(f^{\varepsilon} \cdot q_2\right)+\frac{1}{\varepsilon} v \cdot \nabla_x\left(f^{\varepsilon} \cdot q_2\right)+\frac{2}{\varepsilon^2} \mathcal{L}\left(f^{\varepsilon} \cdot q_2\right) \nonumber\\
& \quad=  \nabla_x \phi^{\varepsilon} \cdot \nabla_v\left(f^{\varepsilon} \cdot q_1\right)
-\frac{1}{2} \nabla_x \phi^{\varepsilon} \cdot v \left(f^{\varepsilon} \cdot q_1\right)
+\frac{1}{\varepsilon} \mathcal{T}\left(f^{\varepsilon} \cdot q_2, f^{\varepsilon} \cdot q_2\right). \nonumber
\end{align}
Following the standard formal derivations of fluid dynamic limits of the Boltzmann equation (see \cite{BGL93} for instance), we obtain
\begin{align}
\label{A(v):weiguan:zuoyong}
\displaystyle\frac{1}{\varepsilon} \Big\langle \widehat{A}(v)\mu^{1/2}, \mathcal{L}\big(\frac{f^{\varepsilon} \cdot q_2}{2}\big) \Big\rangle
=\;& u^{\varepsilon}\otimes u^{\varepsilon}-\frac{|u^{\varepsilon}|^2}{3}\mathbb{I}_{3 \times 3}-\nu\Sigma(u^{\varepsilon})-R^\varepsilon_A,\\[2mm]
\displaystyle\frac{1}{\varepsilon}\Big\langle \widehat{B}(v)\mu^{1/2}, \mathcal{L}\big(\frac{f^{\varepsilon} \cdot q_2}{2}\big)\Big\rangle
=\;&\frac{5}{2}u^{\varepsilon}\theta^{\varepsilon}-\frac{5}{2}\kappa\nabla_x\theta^{\varepsilon}-R^\varepsilon_B, \label{B(v):weiguan:zuoyong}
\end{align}
where
\begin{align}
\nu := \;&\frac{1}{20}\Big\langle \widehat{A}(v) \mu^{1/2}, A(v)\mu^{1/2} \Big\rangle, \label{nu define}\\
\kappa := \;&\frac{1}{15}\Big\langle \widehat{B}(v) \mu^{1/2}, B(v)\mu^{1/2} \Big\rangle, \label{kappa define}\\
\Sigma(u^{\varepsilon}) := \;&\nabla_x u^{\varepsilon}+({\nabla_xu^{\varepsilon}})^{T}-\frac{2}{3}\nabla_x\cdot u^{\varepsilon}\mathbb{I}_{3 \times 3}, \nonumber\\
R^\varepsilon_\Xi := \;& \frac{\varepsilon}{2} \Big\langle \widehat{\Xi}(v)\mu^{1/2},
\partial_t \big(\frac{f^{\varepsilon} \cdot q_2}{2}\big) \Big\rangle
+\frac{1}{2}\Big\langle \widehat{\Xi}(v)\mu^{1/2}, v\cdot\nabla_x \big(\frac{\{\mathbf{I}-\mathbf{P}\}f^{\varepsilon} \cdot q_2}{2}\big) \Big\rangle \nonumber\\
\;&-\frac{\varepsilon}{2} \Big\langle \widehat{\Xi}(v)\mu^{1/2},  \nabla_x\phi^\varepsilon \cdot
\nabla_v \big(\frac{ \{\mathbf{I}-\mathbf{P}\} f^{\varepsilon} \cdot q_1}{2}\big) \Big\rangle
+\frac{\varepsilon}{2} \Big\langle \widehat{\Xi}(v)\mu^{1/2},  \frac{1}{2}\nabla_x\phi^\varepsilon \cdot
v \big(\frac{ \{\mathbf{I}-\mathbf{P}\} f^{\varepsilon} \cdot q_1}{2}\big) \Big\rangle \nonumber\\
\;&-\Big\langle \widehat{\Xi}(v)\mu^{1/2},
\mathcal{T}\Big( \{I-P_0\} \big( \frac{f^{\varepsilon} \cdot q_2}{2} \big) , \{I-P_0\} \big( \frac{f^{\varepsilon} \cdot q_2}{2} \big) \Big) \Big\rangle \nonumber\\
\;&-\Big\langle \widehat{\Xi}(v)\mu^{1/2},
\mathcal{T}\Big( \{I-P_0\} \big( \frac{f^{\varepsilon} \cdot q_2}{2} \big), P_0 \big(\frac{f^{\varepsilon} \cdot q_2}{2} \big) \Big) \Big\rangle \nonumber\\
\;&-\Big\langle \widehat{\Xi}(v)\mu^{1/2},
\mathcal{T}\Big(P_0 \big(\frac{f^{\varepsilon} \cdot q_2}{2} \big), \{I-P_0\} \big( \frac{f^{\varepsilon} \cdot q_2}{2} \big) \Big) \Big\rangle \nonumber
\end{align}
with $\Xi = A$ or $B$. Here, we have used the relation $P_0(f^\varepsilon \cdot q_2)= \mathbf{P}f^\varepsilon \cdot q_2$ and $P_0$ is defined in \eqref{P0h define}.

The relation \eqref{A(v):weiguan:zuoyong}, the decomposition $u^{\varepsilon}=\mathcal{P}u^{\varepsilon}+\mathcal{P}^{\bot} u^{\varepsilon}$ and the equation \eqref{Pu equation} yield that
\begin{equation}\label{jixian:u:zongjie}
\partial_t\mathcal{P}u^{\varepsilon}+\mathcal{P}\nabla_x \cdot (\mathcal{P}u^{\varepsilon}\otimes\mathcal{P}u^{\varepsilon})
-\nu\Delta_x\mathcal{P}u^{\varepsilon}=-\frac{1}{2}\mathcal{P}(n^\varepsilon \nabla_x\phi^\varepsilon)+R^{\varepsilon}_u,
\end{equation}
where $R^\varepsilon_u$ is given by
\begin{equation}\label{jixian:Ru:zongjie}
R^\varepsilon_u := \mathcal{P}\nabla_x\cdot R^\varepsilon_A-\mathcal{P}\nabla_x\cdot
(\mathcal{P}u^{\varepsilon}\otimes\mathcal{P}^\perp u^{\varepsilon}+\mathcal{P}^\bot u^{\varepsilon}\otimes\mathcal{P}u^{\varepsilon}+\mathcal{P}^\perp u^{\varepsilon}\otimes\mathcal{P}^\perp u^{\varepsilon}).
\end{equation}

The relation \eqref{B(v):weiguan:zuoyong}, the decomposition $u^{\varepsilon}=\mathcal{P}u^{\varepsilon}+\mathcal{P}^{\bot} u^{\varepsilon}$ and the equation \eqref{rho theta equation} yield that
\begin{equation}\label{jixian:theta:zongjie}
\partial_t\Big(\frac{3}{5}\theta^{\varepsilon}-\frac{2}{5}\rho^{\varepsilon}\Big)
+\nabla_x \cdot \Big[\mathcal{P} u^\varepsilon \Big(\frac{3}{5}\theta^\varepsilon -\frac{2}{5}\rho^\varepsilon\Big)\Big]
-\kappa \Delta_x \Big(\frac{3}{5}\theta^\varepsilon -\frac{2}{5}\rho^\varepsilon\Big)=R^\varepsilon_\theta,
\end{equation}
where
\begin{align}
\begin{split}\label{jixian:Rtheta:zongjie}
R^\varepsilon_\theta := \;&\frac{2}{5}\nabla_x \cdot R^\varepsilon_B
-\frac{2}{5} \nabla_x \cdot \big[\mathcal{P} u^\varepsilon(\rho^\varepsilon+\theta^\varepsilon)\big]
-\frac{2}{5}\nabla_x \cdot\big[\mathcal{P}^{\bot} u^\varepsilon(\rho^\varepsilon+\theta^\varepsilon) \big]\\
&-\nabla_x \cdot \Big[\mathcal{P}^{\bot} u^\varepsilon \Big(\frac{3}{5}\theta^\varepsilon -\frac{2}{5}\rho^\varepsilon \Big)\Big]
+\frac{2}{5}\kappa \Delta_x(\rho^\varepsilon+\theta^\varepsilon)
-\frac{\varepsilon}{5}j^\varepsilon\cdot \nabla_x \phi^{\varepsilon}.
\end{split}
\end{align}

Now, we take the limit from the equation \eqref{jixian:u:zongjie} to obtain the $u$-equation of \eqref{INSFP limit}. For any given $T > 0$, choose a vector-valued test function $\Omega^1(t,x) \in C^1(0,T;C_0^\infty(\mathbb{R}^3))$ with $\nabla_x \cdot \Omega^1=0$, $\Omega^1(0,x)=\Omega^1_0(x) \in C_0^\infty(\mathbb{R}^3)$ and $\Omega^1(t,x)=0$ for $t \geq T^\prime$, where $T^\prime < T$. Multiplying \eqref{jixian:u:zongjie} by  $\Omega^1(t,x)$ and integrating over $(t,x) \in [0,T] \times \mathbb{R}^3$, we have
\begin{align}
\begin{split}\nonumber
&\int_0^T\int_{\mathbb{R}^3}\partial_t\mathcal{P}u^{\varepsilon}\cdot \Omega^1(t,x) \d x\d t \\	
=\;&-\int_{\mathbb{R}^3}\mathcal{P}u^{\varepsilon}(0,x)\cdot \Omega^1(0,x) \d x
-\int_0^T\int_{\mathbb{R}^3}\mathcal{P}u^{\varepsilon}\cdot\partial_t \Omega^1(t,x) \d x \d t\\
=\;&-\int_{\mathbb{R}^3}\frac{1}{2}\mathcal{P} \Big\langle f^{\varepsilon}_{0},q_2 v \mu^{1/2} \Big\rangle \cdot\Omega^1_0(x)\d x
-\int_0^T\int_{\mathbb{R}^3}\mathcal{P}u^{\varepsilon}\cdot\partial_t \Omega^1(t,x) \d x\d t.
\end{split}
\end{align}
The initial conditions in Theorem \ref{mainth3} and the convergence \eqref{limit:17} give that
\begin{align*}
\int_{\mathbb{R}^3}\frac{1}{2}\mathcal{P}\Big\langle f^{\varepsilon}_{0}, q_2v \mu^{1/2}\Big\rangle\cdot \Omega^1_0(x)\d x
\to \int_{\mathbb{R}^3} \frac{1}{2} \mathcal{P} \Big\langle f_{0},q_2v\mu^{1/2} \Big\rangle\cdot \Omega^1_0(x)\d x
=\int_{\mathbb{R}^3} \mathcal{P} u_{0} \cdot \Omega^1_0(x)\d x
\end{align*}
and
\begin{align*}
\int_0^T\int_{\mathbb{R}^3}\mathcal{P}u^{\varepsilon}\cdot\partial_t \Omega^1(t,x) \d x\d t
\to \int_0^T\int_{\mathbb{R}^3}u\cdot\partial_t \Omega^1(t,x) \d x\d t
\end{align*}
as $\varepsilon \to 0$. That is, we deduce
\begin{align}\label{jixian:partial u 1}
\int_0^T\int_{\mathbb{R}^3}\partial_t\mathcal{P}u^{\varepsilon}\cdot \Omega^1(t,x) \d x\d t
\to-\int_{\mathbb{R}^3}\mathcal{P}u_{0}\cdot \Omega^1_0(x)\d x-\int_0^T\int_{\mathbb{R}^3}u\cdot\partial_t \Omega^1(t,x) \d x\d t
\end{align}
as $\varepsilon \to 0$.
With the aid of the bounds \eqref{limit:1-1}--\eqref{limit:1-2}, \eqref{limit:9} and the convergences \eqref{limit:17}--\eqref{limit:18}, we show the following convergences
\begin{equation}
 \left\{
\begin{array}{ll}\label{jixian:partial u 2}
\mathcal{P}\nabla_x\cdot(\mathcal{P}u^{\varepsilon}\otimes\mathcal{P}u^{\varepsilon}) \to
\mathcal{P}\nabla_x\cdot(u\otimes u) \quad\text{strongly in}~ \; C(\mathbb{R}^+;H^{N-2}_{loc}),\\[2mm]
\nu\Delta_x\mathcal{P}u^{\varepsilon}\to
\nu\Delta_x u \;\;\;\qquad\qquad\qquad\quad\quad \text{~~in the sense of distributions},\\[2mm]
\mathcal{P}(n^{\varepsilon}\nabla_x\phi^{\varepsilon})\to
\mathcal{P}(n\nabla_x\phi) \;\;\;\;\quad\quad\quad\quad\quad\;\text{strongly in}~\; C(\mathbb{R}^+;H_{loc}^{N-1}),\\[2mm]
R^\varepsilon_u \to 0 \;\quad\quad\quad\quad\quad\quad\quad\quad\quad\quad\quad\quad\quad \text{~~in the sense of distributions}
\end{array}
\right.
\end{equation}
as $\varepsilon \to 0$. Therefore, the convergences \eqref{jixian:partial u 1}, \eqref{jixian:partial u 2} and the incompressibility \eqref{limit:13} imply that $u \in L^\infty(\mathbb{R}^+;H^{N}_x) \cap C(\mathbb{R}^+; H^{N-1}_{loc})$  subjects to
\begin{equation}
\left\{
\begin{array}{ll}\label{jixian:u equation}
\partial_t u+\mathcal{P}\nabla_x\cdot(u\otimes u)-\nu\Delta_x u =-\frac{1}{2}\mathcal{P}( n \nabla_x\phi ),\\[2mm]
\nabla_x \cdot u=0,
\end{array}
\right.
\end{equation}
with initial data $u(0,x)=\mathcal{P}u_0(x)$.

Next, we take the limit from the equation \eqref{jixian:theta:zongjie} to gain the $\theta$-equation of \eqref{INSFP limit}. For any given $T > 0$, choose a test function $\Omega^2(t,x)  \in C^1(0,T;C_0^\infty(\mathbb{R}^3))$ with $\Omega^2(0,x)=\Omega^2_0(x) \in C_0^\infty(\mathbb{R}^3)$ and $\Omega^2(t,x)=0$ for $t \geq T^\prime$, where $T^\prime < T$. Then from the initial conditions in Theorem \ref{mainth3} and the convergence \eqref{limit:17}, we have
\begin{align}
\begin{split}\label{jixian:partial theta 1}
&\int_{0}^{T}\int_{\mathbb{R}^3}\partial_t\Big(\frac{3}{5}\theta^{\varepsilon}-\frac{2}{5}\rho^{\varepsilon}\Big) \Omega^2(t,x) \d x \d t \\
=\;&-\int_{\mathbb{R}^3}\frac{1}{2} \Big\langle f^{\varepsilon}_{0},q_2(\frac{|v|^2}{5}-1) \mu^{1/2} \Big\rangle \Omega^2_0(x)\d x
-\int_{0}^{T}\int_{\mathbb{R}^3} \Big(\frac{3}{5}\theta^{\varepsilon}-\frac{2}{5}\rho^{\varepsilon}\Big)\partial_t \Omega^2(t,x)  \d x \d t\\
\to\;&-\int_{\mathbb{R}^3}\frac{1}{2} \Big\langle f_{0},q_2(\frac{|v|^2}{5}-1) \mu^{1/2} \Big\rangle \Omega^2_0(x)\d x
-\int_{0}^{T}\int_{\mathbb{R}^3}  \theta(t,x)\partial_t \Omega^2(t,x)  \d x \d t\\
=\;&-\int_{\mathbb{R}^3}\Big(\frac{3}{5}\theta_0-\frac{2}{5}\rho_0\Big) \Omega^2_0(x)\d x
-\int_{0}^{T}\int_{\mathbb{R}^3}  \theta(t,x)\partial_t \Omega^2(t,x)  \d x \d t
\end{split}
\end{align}
as $\varepsilon \to 0$. Furthermore, with the help of the bounds \eqref{limit:1-1}--\eqref{limit:1-2}, \eqref{limit:9} and the convergences \eqref{limit:17}--\eqref{limit:18}, we derive the following convergences
\begin{equation}
 \left\{
\begin{array}{ll}\label{jixian:partial theta 2}
\nabla_x\cdot\Big[\mathcal{P}u^{\varepsilon}\Big(\frac{3}{5}\theta^{\varepsilon}-\frac{2}{5}\rho^{\varepsilon}\Big)\Big]
\to	\nabla_x\cdot(u\theta)\quad \,\,\text{strongly in}~\; C(\mathbb{R}^+;H_{loc}^{N-2}), \\[2mm]
\kappa\Delta_x\Big(\frac{3}{5}\theta^{\varepsilon}-\frac{2}{5}\rho^{\varepsilon}\Big)\to \kappa\Delta_x\theta
\;\;\;\;\,\,\qquad\qquad \text{~~in the sense of distributions},\\[2mm]
R^\varepsilon_\theta \to 0  \,\;\;\;\quad\quad\quad\quad\quad\quad\quad\quad\qquad\quad\quad\quad\text{in the sense of distributions}
\end{array}
\right.
\end{equation}
as $\varepsilon \to 0$.
Therefore, the convergences \eqref{jixian:partial theta 1}, \eqref{jixian:partial theta 2} and the incompressibility \eqref{limit:13} imply that that $\theta \in L^\infty(\mathbb{R}^+;H^{N}_x) \cap C(\mathbb{R}^+; H^{N-1}_{loc})$ obeys
\begin{align}\label{jixian:theta equation}
\partial_t\theta+ u \cdot \nabla_x\theta-\kappa \Delta_x\theta=0
\end{align}
with the initial data $\theta(0,x)=\frac{3}{5}\theta_{0}(x)-\frac{2}{5}\rho_{0}(x)$.

Next, we take the limit from the fourth equation in $\eqref{limit:8-1:1}$ to deduce the $n$-equation of \eqref{INSFP limit}.
For any $T>0$, choose a test function $\Omega^3(t,x)\in C^1(0,T;C_0^\infty(\mathbb{R}^3))$ with $\Omega^3(0,x)=\Omega^3_0(x)\in C_0^\infty(\mathbb{R}^3)$, and $\Omega^3(t,x)=0$ for $t\geq T^\prime$, where $T^\prime<T$.
Then from the initial conditions in Theorem \ref{mainth3} and the convergence \eqref{limit:17}, we obtain
\begin{align}
\begin{split}\label{jixian:partial n 1}
&\int_0^T\int_{\mathbb{R}^3}\partial_t n^{\varepsilon} \Omega^3(t,x)\d x\d t \\
=\;&-\int_{\mathbb{R}^3} \big\langle f^\varepsilon_0, q_1 \mu^{1/2} \big\rangle \Omega^3_0(x)\d x
-\int_0^T\int_{\mathbb{R}^3} n^{\varepsilon} \partial_t\Omega^3(t,x)\d x\d t \\
\to\;&-\int_{\mathbb{R}^3} \big\langle f_0, q_1 \mu^{1/2} \big\rangle \Omega^3_0(x)\d x
-\int_0^T\int_{\mathbb{R}^3} n(t,x) \partial_t\Omega^3(t,x)\d x\d t \\
=\;&-\int_{\mathbb{R}^3} n_{0} \Omega^3_0(x)\d x-\int_0^T\int_{\mathbb{R}^3}n(t,x)\partial_t\Omega^3(t,x)\d x\d t
\end{split}
\end{align}
as $\varepsilon \to 0$.
Moreover, the convergence \eqref{limit:19} implies that
\begin{align}\label{jixian:partial n 2}
\nabla_x \cdot j^\varepsilon \to \nabla_x \cdot j    \;\;\;\text{~~weakly} {\text{~in}~L^2(\mathbb{R}^+;H^{N-1}_{x})}
\end{align}
as $\varepsilon \to 0$.
Then, the convergences \eqref{jixian:partial n 1}, \eqref{jixian:partial n 2} and the fourth equation in \eqref{limit:8-1:1} imply that
\begin{align}\label{jixian:n equation}
\partial_{t} {n}+\nabla_{x} \cdot j=0.
\end{align}
Furthermore, the convergences \eqref{limit:17}, \eqref{limit:17-1} and the equation
\eqref{jixian:u equation} show that
\begin{align}
-\Delta_x \phi=n
\end{align}
with the initial date $n_0(x)=-\Delta_x \phi_0(x)$.

Finally, we derive the Ohm's law and energy equivalence relation in the second to the last line of \eqref{INSFP limit}.
To do so, we refer back to \eqref{scr L define} for the linearized Boltzmann collision operator $\mathscr{L}$.
Before proceeding, let's introduce some qualities and their properties.
For the qualities $\Phi(v)=v\mu^{1/2}$ and $\Psi(v)=(\frac{|v|^2}{2}-\frac{3}{2})\mu^{1/2}$, there exist unique $\widetilde{\Phi}(v)$ and $\widetilde{\Psi}(v)$ such that $\mathscr{L}\widetilde{\Phi}(v)=\Phi(v)$ with $\widetilde{\Phi}(v) \in \mathcal{N}^{\bot}(\mathscr{L})$ and $\mathscr{L}\widetilde{\Psi}(v)=\Psi(v)$ with $\widetilde{\Psi}(v) \in \mathcal{N}^{\bot}(\mathscr{L})$.
Furthermore, there exist two scalar functions $\alpha$, $\beta: \mathbb{R}^{+} \to \mathbb{R}$ (see \cite{AS2019}) such that
$$
\widetilde{\Phi}(v)=\alpha(|v|)\Phi(v) \;\;\text{and}\;\; \widetilde{\Psi}(v)=\beta(|v|)\Psi(v).
$$
Via multiplying the first equation of \eqref{rVPB} by $q_1$ and direct calculations, we obtain
\begin{align}
\begin{split}\label{j omega jisuan 1}
\frac{1}{\varepsilon} \mathscr{L} ( f^{\varepsilon} \cdot q_1)
=\;&-\varepsilon \partial_t (f^{\varepsilon} \cdot q_1) - v\cdot \nabla_x (f^{\varepsilon} \cdot q_1)
- 2 \nabla_x \phi^\varepsilon \cdot v \mu^{1/2}
+ \varepsilon \nabla_x \phi^\varepsilon \cdot \nabla_v (f^{\varepsilon}\cdot q_2) \\
&- \frac{\varepsilon}{2}\nabla_x \phi^\varepsilon \cdot v(f^{\varepsilon}\cdot q_2)
+ \mathcal{T}(f^{\varepsilon} \cdot q_2, f^{\varepsilon}\cdot q_1).
\end{split}
\end{align}
Taking the $L^2_v$ inner product with \eqref{j omega jisuan 1} by $\widetilde{\Phi}(v)$, $\frac{2}{3}\widetilde{\Psi}(v)$, we obtain
\begin{align} \label{j jisuan 2}
j^\varepsilon=\;& n^\varepsilon \mathcal{P} u^\varepsilon-\sigma(\frac{1}{2}\nabla_x n^\varepsilon +\nabla_x \phi^\varepsilon)+R^\varepsilon_{j},\\
\omega^\varepsilon=\;& n^\varepsilon \Big(\frac{3}{5}\theta^\varepsilon-\frac{2}{5}\rho^\varepsilon\Big)+R^\varepsilon_{\omega},
\label{omega jisuan 2}
\end{align}
where
\begin{align}\label{sigma define}
\sigma :=\;& \frac{2}{3} \Big\langle \widetilde{\Phi}(v), \Phi(v) \Big\rangle,\\
R^\varepsilon_{j} := \;& n^\varepsilon \mathcal{P}^\bot u^\varepsilon
-\varepsilon\Big\langle \partial_t(f^\varepsilon \cdot q_1), \widetilde{\Phi}(v) \Big\rangle
-\Big\langle v \cdot \nabla_x(\{\mathbf{I}-\mathbf{P}\}f^\varepsilon \cdot q_1), \widetilde{\Phi}(v) \Big\rangle \nonumber\\
\;&+\varepsilon\Big\langle \nabla_x \phi^\varepsilon \cdot \nabla_v (f^{\varepsilon}\cdot q_2), \widetilde{\Phi}(v) \Big\rangle
-\frac{\varepsilon}{2}\Big\langle\nabla_x \phi^\varepsilon \cdot v(f^{\varepsilon}\cdot q_2), \widetilde{\Phi}(v) \Big\rangle \nonumber\\
\;&+\Big\langle \mathcal{T}\big(\mathbf{P}f^{\varepsilon} \cdot q_2, \{\mathbf{I}-\mathbf{P}\}f^{\varepsilon}\cdot q_1 \big),
\widetilde{\Phi}(v) \Big\rangle
+\Big\langle \mathcal{T}\big(\{\mathbf{I}-\mathbf{P}\}f^{\varepsilon} \cdot q_2, \mathbf{P}f^{\varepsilon}\cdot q_1 \big),
\widetilde{\Phi}(v) \Big\rangle \nonumber\\
\;&+\Big\langle \mathcal{T}\big(\{\mathbf{I}-\mathbf{P}\}f^{\varepsilon} \cdot q_2, \{\mathbf{I}-\mathbf{P}\}f^{\varepsilon}\cdot q_1 \big),
\widetilde{\Phi}(v) \Big\rangle, \nonumber\\
R^\varepsilon_{\omega} := \;& n^\varepsilon (\rho^\varepsilon+\theta^\varepsilon)
-\frac{2\varepsilon}{3} \Big\langle \partial_t(f^\varepsilon \cdot q_1), \widetilde{\Psi}(v) \Big\rangle
-\frac{2}{3}\Big\langle v \cdot \nabla_x(\{\mathbf{I}-\mathbf{P}\}f^\varepsilon \cdot q_1), \widetilde{\Psi}(v) \Big\rangle \nonumber\\
\;&+\frac{2\varepsilon}{3} \Big\langle \nabla_x \phi^\varepsilon \cdot \nabla_v (f^{\varepsilon}\cdot q_2), \widetilde{\Psi}(v) \Big\rangle
-\frac{\varepsilon}{3} \Big\langle\nabla_x \phi^\varepsilon \cdot v(f^{\varepsilon}\cdot q_2), \widetilde{\Psi}(v) \Big\rangle \nonumber\\
\;&+\frac{2}{3}\Big\langle \mathcal{T}\big(\mathbf{P}f^{\varepsilon} \cdot q_2, \{\mathbf{I}-\mathbf{P}\}f^{\varepsilon}\cdot q_1 \big),
\widetilde{\Psi}(v) \Big\rangle
+\frac{2}{3}\Big\langle \mathcal{T}\big(\{\mathbf{I}-\mathbf{P}\}f^{\varepsilon} \cdot q_2, \mathbf{P}f^{\varepsilon}\cdot q_1 \big),
\widetilde{\Psi}(v) \Big\rangle \nonumber\\
\;&+\frac{2}{3}\Big\langle \mathcal{T}\big(\{\mathbf{I}-\mathbf{P}\}f^{\varepsilon} \cdot q_2, \{\mathbf{I}-\mathbf{P}\}f^{\varepsilon}\cdot q_1 \big), \widetilde{\Psi}(v) \Big\rangle. \nonumber
\end{align}
Here, we make use of the definition of $\mathbf{P}f^\varepsilon$, the fact $\mathcal{T}(\mu^{1/2}, \mu^{1/2})=0$,
$\big\langle \Phi(v), \widetilde{\Psi}(v) \big\rangle =0$ and the self-adjointness of $\mathscr{L}$.
Applying the bounds \eqref{limit:1-1}--\eqref{limit:1-2}, \eqref{limit:9} and the convergences \eqref{limit:17}--\eqref{limit:18}, we obtain the following convergences
\begin{equation}
 \left\{
\begin{array}{ll}\label{jixian:j omega 1}
n^\varepsilon \mathcal{P}u^{\varepsilon}
\to	nu \qquad\qquad\qquad\qquad\qquad\qquad\quad \,\;\text{strongly in}\; C(\mathbb{R}^+;H_{loc}^{N-1}), \\[2mm]
\sigma(\frac{1}{2}\nabla_x n^\varepsilon +\nabla_x \phi^\varepsilon) \to \sigma(\frac{1}{2}\nabla_x n +\nabla_x \phi)
\quad \quad \; \text{strongly in}~~\; C(\mathbb{R}^+;H_{loc}^{N-2}),\\[2mm]
R^\varepsilon_j \to 0  \;\;\;\qquad\qquad\qquad\qquad\quad\quad\qquad\quad\quad\;\,\;\, \text{in the sense of distributions},\\[2mm]
n^\varepsilon \Big(\frac{3}{5}\theta^\varepsilon-\frac{2}{5}\rho^\varepsilon\Big)
\to	n\theta \qquad\qquad\qquad\qquad\quad \,\;\;\text{strongly in}~~\; C(\mathbb{R}^+;H_{loc}^{N-1}), \\[2mm]
R^\varepsilon_\omega \to 0  \;\;\qquad\qquad\qquad\qquad\quad\quad\qquad\quad\quad\quad \text{in the sense of distributions}
\end{array}
\right.
\end{equation}
as $\varepsilon \to 0$. As a result, the convergences \eqref{limit:19}, \eqref{jixian:j omega 1} and the equation \eqref{j jisuan 2} imply that
\begin{align}
j=n u-\sigma\Big(\nabla_{x} \phi+\frac{1}{2} \nabla_{x} n\Big)
\end{align}
with the initial data
\begin{align*}
j(0,x)=\;&n_0(x) u_0(x)-\sigma\Big(\nabla_{x} \phi_0(x)+\frac{1}{2} \nabla_{x} n_0(x)\Big)\\
=\;&-\Delta_x \phi_0(x) \mathcal{P} u_0(x) - \sigma \Big(\nabla_{x} \phi_0(x)- \frac{1}{2} \nabla_{x} \Delta_x \phi_0(x) \Big).
\end{align*}
Moreover,  the convergences \eqref{limit:19}, \eqref{jixian:j omega 1} and the equation \eqref{omega jisuan 2} show that
\begin{align}
\omega=n\theta
\end{align}
with
the initial data
\begin{align*}
\omega(0,x)=n(0,x)\theta(0,x)=-\Delta_x \phi_0(x) \Big(\frac{3}{5}\theta_0(x)-\frac{2}{5}\rho_0(x)\Big).
\end{align*}

\subsubsection{Summarization}
\hspace*{\fill}

Collecting all above convergence results, we have shown that
\begin{align*}
&\left(\rho, u, \theta, n, j, \omega\right) \in L^\infty(\mathbb{R}^+; H^N_{x}) \cap  C(\mathbb{R}^+; H^{N-1}_{loc}),\\
&\nabla_x \phi(x) \in L^\infty(\mathbb{R}^+; H^{N+1}_{x}) \cap C(\mathbb{R}^+; H^{N}_{loc})
\end{align*}
satisfy the following two-fluid incompressible NSFP system with Ohm's law
\begin{equation*}
\left\{
\begin{array}{ll}
\displaystyle \partial_{t} u+u \cdot \nabla_{x} u-\nu \Delta_{x} u+\nabla_{x}P=-\frac{1}{2} n \nabla_{x} \phi,  \;\;\;\;&\nabla_x \cdot u=0,\\[2mm]
\displaystyle \partial_{t} \theta+u \cdot \nabla_{x} \theta-\kappa \Delta_{x} \theta=0, &\rho+\theta=0,\\[2mm]
\displaystyle \partial_{t} {n}+ u \cdot \nabla_x n - \frac{\sigma}{2} \Delta_x n + \sigma n=0, &-\Delta_{x} \phi=n,\\ [2mm]
\displaystyle j=n u-\sigma\Big(\nabla_{x} \phi+\frac{1}{2} \nabla_{x} n\Big), &\omega=n \theta,
\end{array} \right.
\end{equation*}
with the initial data
\begin{align}
\begin{split}\nonumber
u(0,x)=\mathcal{P}u_0(x),\;\; \theta (0,x)=\frac{3}{5}\theta_0(x)-\frac{2}{5}\rho_0(x),\;\; \nabla_x \phi(0,x)=\nabla_x \phi_0(x).
\end{split}
\end{align}

Actually, by mollifying the above system and combining the local existence theory, the standard energy method and the continuity argument,  we can show that $(\rho, u, \theta, n, \omega) \in C(\mathbb{R}^{+}; H_x^N) $, $\nabla_x \phi \in C(\mathbb{R}^{+}; H_x^{N+1})$ and $j \in C(\mathbb{R}^{+}; H_x^{N-1})$, cf. Chapter 3.2 in \cite{Majda-Burtozzi}, and then \eqref{jixian solution space} holds true.  Moreover, by approximating $\mathbb{R}^3_x$ by bounded
region \cite{CK2006} and the uniform bounds of functions and their limits proved above, the convergences in \eqref{theorem1.3 4} are thus verified.
This completes the proof of Theorem \ref{mainth3}.

\bigskip

\noindent\textbf{Declarations of competing interest}

The authors declare that they have no conflicts of interest.
\medskip

\noindent\textbf{Data availability}

No data was used for the research described in the article.
\medskip

\noindent\textbf{Acknowledgements}

The research is supported by NSFC under the grant number 12271179, NSFC key project under the grant number 11831003,
Basic and Applied Basic Research Project of Guangdong under the grant number 2022A1515012097.

\end{document}